\documentclass[1 [leqno,10pt]{amsart}
\usepackage{amsmath,mathrsfs}
\usepackage{amssymb, amsmath}
\usepackage{graphicx}

\newtheorem{defi}{Definition}[section]
\newtheorem{thm}{Theorem}[section]
\newtheorem{lem}{Lemma}[section]
\newtheorem{rmk}{Remark}[section]
\newtheorem{cor}{Corollary}[section]
\newtheorem{prop}{Proposition}[section]

\numberwithin{equation}{section}

\newcommand{\eps}{\ep}

\let\f=\frac
\newcommand{\lr}[1]{\langle #1 \rangle}
\newcommand{\lf}[1]{#1_{\phi}}
\newcommand{\hf}[1]{#1^{\phi}}

\newcommand{\sgn}{\mathop{\mathrm{sgn}}}
\newcommand{\wcT}{\widetilde{\mathcal{T}}}
\newcommand{\beq}{\begin{equation}}
\newcommand{\eeq}{\end{equation}}
\newcommand{\ben}{\begin{eqnarray}}
\newcommand{\een}{\end{eqnarray}}
\newcommand{\beno}{\begin{eqnarray*}}
\newcommand{\eeno}{\end{eqnarray*}}

\def\pa{\partial}

\def\virgp{\raise 2pt\hbox{,}}
\def\cdotpv{\raise 2pt\hbox{;}}
\def\eqdef{\buildrel\hbox{\footnotesize d\'ef}\over =}
\def\eqdefa{\buildrel\hbox{\footnotesize def}\over =}

\def\C{\mathop{\mathbb C\kern 0pt}\nolimits}
\def\DD{\mathop{\mathbb D\kern 0pt}\nolimits}
\def\EE{\mathop{{\mathbb E \kern 0pt}}\nolimits}
\def\K{\mathop{\mathbb K\kern 0pt}\nolimits}
\def\N{\mathop{\mathbb N\kern 0pt}\nolimits}
\def\Q{\mathop{\mathbb Q\kern 0pt}\nolimits}
\def\R{\mathop{\mathbb R\kern 0pt}\nolimits}
\def\SS{\mathop{\mathbb S\kern 0pt}\nolimits}
\def\NN{\mathop{\mathbb N\kern 0pt}\nolimits}
\def\ZZ{\mathop{\mathbb Z\kern 0pt}\nolimits}
\def\TT{\mathop{\mathbb T\kern 0pt}\nolimits}
\def\P{\mathop{\mathbb P\kern 0pt}\nolimits}
\def\na{\nabla}

\def\ep{\epsilon}

\begin{document}
\title[Global dynamics of Boltzmann equations]
{On the  global dynamics of  the  inhomogeneous Boltzmann equation without angular cutoff: Hard potentials and Maxwellian molecules}

\author{Ling-Bing He}
\address{Department of Mathematical Sciences, Tsinghua University, Beijing, 100084, P.R.China}
\email{hlb@mail.tsinghua.edu.cn}

\author{Jin-Cheng Jiang}
\address{Department of Mathematics, National Tsing Hua University, Hsinchu, Taiwan 30013, R.O.C}
\email{jcjiang@math.nthu.edu.tw}
\maketitle

\begin{abstract}
 We consider the global dynamics of the original Boltzmann equation without angular cutoff 
on the torus  for the hard potentials  and Maxwellian molecules. The new idea to solve the problem is the energy-entropy method  which characterizes the propagation of the regularity,  $H$-theorem and the interplay between the energy and the entropy.   Our  main results are as follows:

 (1).  We present a unified framework to prove the well-posedness for the original Boltzmann equation  for both angular cutoff and without cutoff  in weighted Sobolev spaces with polynomial weights. As a consequence, we obtain the propagation of the regularity and an explicit formula for the asymptotics of the  equation from angular cutoff to non-cutoff.

(2). We describe the global dynamics of the equation under the almost optimal assumption on the solution which
ensures that the Boltzmann collision operator behaves like a fractional Laplace operator for the velocity variable. In particular,
we obtain a new mechanism for the convergence 
of the solution to its equilibrium with quantitative estimates. 

(3). We prove that any global and smooth solution to the equation is stable, i.e., any perturbed solution will remain close to the reference solution if initially they are close to each other. Here we  remove  the assumption that perturbed solution and the reference solution should have the same associated equilibrium.

Our approach incorporates almost all the fundamental properties of the equation: the entropy production inequality, the  immediately appearance of pointwise lower bound of the solution, the smoothing property of the positive part of the collision operator, averaging lemma for the transport equation, the Povnzer inequality for $L^1$ moment and the sharp bounds for the collision operator.  

\medskip

{\bf MSC2010:} 	35Q20;35A01 A02 A09;35B35 B40.

\end{abstract}

\tableofcontents

\section{Introduction} 

We investigate the global dynamics of the Boltzmann equation without angular cutoff in this  paper and an 
upcoming paper~\cite{hjz}. Due to the different structures of the collision operator, only
the hard potentials and Maxwellian molecules  will be discussed in this paper and the soft potential
case is left to~\cite{hjz}.    
The new idea of this program is an introduction of the energy-entropy method which characterizes the propagation of the regularity,  $H$-theorem and the interplay between the energy and the entropy.  
\subsection{Boltzmann equation: basic assumptions and properties.} We first recall that
the spatially inhomogeneous Boltzmann equation with the initial data $f_0$ reads:
\ben\label{Boltz eq}  \left\{
\begin{aligned}
	& \partial_t f+v\cdot\nabla_x f=Q(f,f),\\
	&f|_{t=0}=f_0.\end{aligned}\right. 
\een
where $f(t,x,v)\geq 0$ is a   distribution
function  at time
$t\geq 0$   of  colliding particles at position $x\in  \TT^3=[-\f12,\f12]^3$  with velocity
$v\in\R^3$.    Here  the collision operator $Q$ is a bilinear
operator which acts only on the velocity variable, that is,
\beq\label{COQ} Q(g,f)(v)\eqdefa
\int_{\R^3}\int_{\SS^{2}}B(v-v_*,\sigma)(g'_*f'-g_*f)d\sigma dv_*.
\eeq In the above we use the standard shorthand notations  $f=f(t,x,v)$, $g_*=g(t,x, v_*)$,
$f'=f(t,x,v')$, $g'_*=g(t,x,v'_*)$  where $(v,v_*)$ and $(v',v_*')$ are the velocities of particles before and after the collision. Here $v'$ and $v_*'$ are given by
\beq\label{sigma-rep}
v'=\frac{v+v_{*}}{2}+\frac{|v-v_{*}|}{2}\sigma\ ,\ \ \
v'_{*}=\frac{v+v_{*}}{2}-\frac{|v-v_{*}|}{2}\sigma\
,\qquad\sigma\in\SS^{2}
\eeq
which follows from the parametrization of the
set of solutions of the physical laws of elastic collision: \beno
v+v_*=v'+v_*',\quad |v|^2+|v_*|^2=|v'|^2+|v'_*|^2. \eeno

The non-negative function $B(v-v_*,\sigma)$, Boltzmann collision kernel, in the collision
operator is always assumed to depend only on $|v-v_{*}|$ and 
$\frac{v-v_{*}}{|v-v_{*}|}\cdot\sigma$. Usually, we introduce the angle variable $\theta$ through
$\cos\theta=\frac{v-v_{*}}{|v-v_{*}|}\cdot\sigma$.  Without loss of generality, we may
assume that $B(v-v_{*},\sigma)$ is supported in the set
 $0\leq\theta\leq\frac{\pi}{2}$ , i.e,
$\frac{v-v_{*}}{|v-v_{*}|}\cdot\sigma\ge0
$,   otherwise, $B$ can be replaced by its symmetrized form:\beno
\bar{B} (v-v_{*},\sigma)=[B (v-v_{*},\sigma)+B(v-v_{*},-\sigma)]\mathbf{1}_{\{\frac{v-v_{*}}{|v-v_{*}|}\cdot\sigma\ge0\}}.\eeno
Here, $\mathbf{1}_A$ is the characteristic function of the set $A$.
 By change of variables, it is easy to check that
 \[
\lr{Q(f,g),h}_v=\iiint f_*g(h'-h)B(v-v_*,\sigma) d\sigma dv_*dv.
\]
In this paper, we assume  that  the collision kernel
satisfies the following assumptions:
\begin{itemize}
\item  The cross-section $B(v-v_{*},\sigma)$ takes a product form as
\ben \label{a1} B(v-v_{*},\sigma)=\Phi(|v-v_*|)b(\cos\theta),\een
 where both $\Phi$ and $b$ are non-negative functions.
\item The angular function $b(\cos\theta)$ is not   integrable and it
satisfies  for $\theta\in [0, \pi/2]$\begin{eqnarray}\label{a2} K\theta^{-1-2s} \le\sin\theta
b(\cos\theta)\le K^{-1}
 \theta^{-1-2s},  \quad
\mbox{with}\   0<s<1,\ K>0.
\end{eqnarray}
\item The kinetic factor $\Phi$ takes the form \ben\label{a3}
\Phi(|v-v_*|)=|v-v_{*}|^{\gamma}, \,\mbox{with}\,\, 0\le \gamma\le2.\een

\end{itemize}
\medskip

When the deviation angle $\theta$ has a lower bound, i.e., $\theta\ge C\epsilon>0$ which corresponds to the famous Grad's  cutoff assumption, the  equation is called the Boltzmann equation with angular cutoff. In this case, the collision operator $Q$ is turned to be $Q^\eps$ defined by:
\beq\label{COQeps} Q^\eps(g,f)(v)\eqdefa
\int_{\R^3}\int_{\SS^{2}}B^\eps(v-v_*,\sigma)(g'_*f'-g_*f)d\sigma dv_*
\eqdefa Q^\eps_+(g,f)- Q^\eps_-(g,f),
\eeq 
where $B^\eps(v-v_*,\sigma)=b^\eps(\cos\theta)|v-v_*|^\gamma$ with $b^\eps(\cos\theta)=b(\cos\theta)(1-\psi)((\sin\f{\theta}2)/\ep)$.  
The bump function $\psi$ with support around $0$ is defined in \eqref{defpsivarphi}.
Then the  Boltzmann equation with angular cutoff and initial data $f_0$ is written as
\ben\label{NepsB}  \left\{
\begin{aligned}
	& \pa_t f+v\cdot \na_x f =Q^{\ep}(f,f),\\
	&f|_{t=0}=f_0.\end{aligned}\right. \een

\begin{rmk} 
The parameter $\eps$  is used to emphasize that the deviation angle $\theta$ has the lower bound that $\theta\ge C \eps$. In this paper, we will consider the case that $\eps$ is sufficiently small. The mathematical problem of the asymptotics of the Boltzmann equation from angular cutoff to non cutoff  is formulated by taking the limit in which the parameter $\eps$ 
goes to zero. Although there is a slightly difference, it is convenient to
regard this process intuitively as a limit from short-range interactions to long-range 
interactions.
\end{rmk}
The solutions of the Boltzmann equation \eqref{Boltz eq} or \eqref{NepsB} enjoy the fundamental properties of the conservation of  mass,  momentum and  kinetic energy, that is, for all $t\ge0$,
\beno  \iint_{\TT^3\times\R^3} f(t,x,v)\phi(v)dvdx=\iint_{\TT^3\times\R^3} f(0,x, v)\phi(v)dvdx,\quad \phi(v)=1,v,|v|^2.\eeno

\begin{defi}  Suppose  $\bar{\rho}, \bar{u}, \bar{T}$ are constants. We call  the function $M_{\bar{\rho}, \bar{u}, \bar{T}}$ to be a global Maxwellian if $M_{\bar{\rho}, \bar{u}, \bar{T}}\eqdefa\f{\bar{\rho} e^{\f{|v-\bar{u}|^2}{2\bar{T}}}}{(2\pi \bar{T})^{\f32}} $. For given a distribution function $f$,  $M_f$ is called to be a global Maxwellian associated to $f$ if  $M_f$ is a global Maxwellian and  has  the same mass, momentum and the kinetic energy as those for $f$. 
\end{defi}

\begin{defi} The hydrodynamical fields: the density $\rho$, mean value velocity $u$ and temperature $T$, associated to the distribution $f(x,v)$  are defined by 
  \ben\label{hydrofield}
\rho(x)=\int_{\R^3} f dv,\quad (\rho u)(x)=\int_{\R^3} fvdv,\quad 3(\rho T)(x)=\int_{\R^3} f|v-u|^2dv.
\een 
Then the local Maxwellian $M_{\rho,u,T}^f$ associated to $f$ can be defined  by  $M_{\rho,  u, T}^f\eqdefa\f{\rho e^{\f{|v-u|^2}{2T}}}{(2\pi T)^{\f32}} $. 
\end{defi}

Next we introduce the relative 
  entropy $H(f|M_f)$  which is  defined by  $$H(f|M_f)(t)\eqdefa \iint_{\TT^3\times\R^3} (f \ln\f{f}{M_f}-f+M_f ) dvdx.$$ 
Then the  Boltzmann's $H$-theorem can be stated   as follows 
\beno \f{d}{dt}H(f|M_f)(t)= -D(f)\eqdefa \iint_{\TT^3\times\R^3} Q (f ,f )\ln f   dvdx\le0,\eeno
which  predicts  that  the entropy is decreasing over time. It is not difficult to check that $D(f)=0$ is equivalent to $f=M^f_{\rho,u,T}$.

\subsection{Short review of the problem} The global dynamics of the Boltzmann equation is to describe the behavior of the solution of the Cauchy problem when the initial data is given.   Mathematically it can be formulated to prove the propagation of the  regularity of the solution and give the quantitative estimate for the convergence to the equilibrium. To solve the problem, we introduce the energy-entropy method which characterizes the propagation of the regularity, $H$-theorem and the interplay between the energy and the entropy.    We mention that this new method is motivated by the following  aspects:

\smallskip

(i). As a well known fact,  $H$-theorem indicates that the relative entropy $H(f|M_f)$ will never stop decreasing whenever $D(f)>0$.   Thus to get the longtime behavior of the solution, the essential part is to make full use of the entropy dissipation $D(f)$. However it is very difficult to derive useful information from $D(f)$ for the inhomogeneous equation due to the complexity of the structure.  The authors in \cite{DV} overcome the difficulty by  imposing the regularity assumptions on the solution to derive the explicit dissipation estimate (see \eqref{DispE}) which is the key to get the desired result. This shows that the propagation of the regularity is helpful to get the convergence to the equilibrium. Thus it is natural to  take the energy into consideration and investigate the interplay between the energy and the entropy.     
\smallskip

(ii). In \cite{cclu}, Carlen, Carvalho and Lu constructed a large class of the solutions verifying that the rate of convergence to equilibrium is only algebraic for the  Boltzmann equation with or without angular cutoff if the initial data only have  finite $L^1$-moment. Yet it is proved only for the moderate soft potentials in the spatially homogeneous setting. But we believe that this phenomenon is universal  for   soft potentials.   It  discloses that the  rate of the convergence has strong connection with  the initial data. Roughly speaking, the initial data with  polynomial moment induces the polynomial convergent rate while the initial data with  exponential moment induces the exponential convergent rate.  We stress that the propagation of the  regularity should be involved in the strategy to catch the rich phenomenon of the longtime behavior of the solution. Only in this way, we can quantify the dependence of the rate of the convergence on the initial data.
\smallskip

There are huge number of literatures on the global dynamics of the Boltzmann equation. In the following, we only review the results on the inhomogeneous equation.
 \subsubsection{Linearized theory} 
The first breakthrough to the global dynamics is due to the linearized theory. 
The main idea of the linearized theory is looking for a special solution near the  global Maxwellian which is the steady solution to the equation. Suppose $f=\mu+\mu^{\f12}F$ where $\mu=M_{1,0,1}$. Then the equation is transformed into 
\beno \pa_t F+v\cdot\na_x F+\mathcal{L}_B F=\Gamma(F,F),\eeno
where $\mathcal{L}_Bf\eqdefa -(\mu^{-\f12}(Q(\mu,\mu^{\f12} f)+Q(\mu^{\f12} f,\mu)))$  and  $\Gamma(h,f)\eqdefa \mu^{-\f12}Q(\mu^{\f12}h,\mu^{\f12} f)$.  By the standard perturbation theory, the problem is essentially reduced to  the coercivity estimate for the self-adjoint operator $\mathcal{L}_B$ and the upper bounds for  the nonlinear term $\Gamma(h,f)$.  We remark that the entropy  is useless in this   framework. The typical explanation is that in this situation all  information is contained in the linearized collision operator $\mathcal{L}_B$.  We refer   readers to \cite{amuxy2, cal, dhwy, gs1, guo1, guo2, gktt,  ukai2} and references therein for more details.  

Based on the linearized theory, the stability results may be relaxed to   so-called weakly inhomogeneous data. It corresponds to the perturbation of any homogeneous solution. We refer readers to \cite{asp,  momi, guoliu, wennberg}.  We comment that
\begin{enumerate}
\item These results more or less depend on the linearized theory. In particular, the linearized theory is widely used in the construction of the existence of the solution;
\item The equation is treated as a semi-linear equation by the standard perturbation theory;
\item These results require that the perturbed solution and the reference solution have the same associated global Maxwellian.
\end{enumerate}
 
\subsubsection{Entropy method} The entropy method is built on the $H$-theorem. The ambitious goal of this method is to show  how the solution converges to its associated equilibrium when the initial data is far away from the equilibrium.   The main difficult part of the method lies in the entropy dissipation estimate:
\ben\label{DispE} D(f)\gtrsim C(f)H(f|M^f_{\rho,u,T})^{1+\epsilon}.\een 
It only measures  the distance between the current state and the local Maxwellian which is not enough to get the convergence to the equilibrium.  The breakthrough by this method is due to Desvillettes and Villani in \cite{DV}. 
Suppose that the solution is smooth and uniformly bounded for all time. Then based on {\bf the instability of the hydrodynamic description}, that is, ruling out the eventuality that the solution spends much of its time close to a local (non global) Maxwellian,  
the authors derived estimates on the rate of the convergence to equilibrium.  Technically, the method relies on the analysis of a differential inequality like
\ben\label{SecIne} h''(t)+Ch(t)^{1-\epsilon}\ge c>0.\een
This pioneer result gives the first positive answer for the convergence to the equilibrium without the restriction that the initial data is near equilibrium. 

 However the result in \cite{DV} is not completely satisfactory because of the strong assumptions on the regularity. Moreover  the statement of the result  was not given in a quantitative way. It only indicates that the high regularity will induce the high rate of the convergence. Some comments are in order:
 \begin{enumerate}
\item We first note that in the reality the assumption on the solution as in \cite{DV} is too strong. For instance, when the solution verifies that  the rate of the convergence to the equilibrium is  algebraic,  we can only have the uniform control of the solution with finite derivative.  

\item Secondly, from the point view of the dynamics, the method used in \cite{DV} only emphasized the role of the entropy. It neglected the propagation of the regularity and so the interplay between the entropy and the energy. As a result, the pure entropy method cannot quantify the dependence of the rate of the convergence on the initial data. 

\item Thirdly,   in the reality,  the propagation of the regularity and the convergence to the equilibrium should be treated as an organic whole (see Section 1.3.3 for  details in the below).   
\end{enumerate}

\subsubsection{Semi-group method}   The semi-group method was introduced by  Gualdani, Mischler and Mouhot  in \cite{momi}.  The goal of the method is to build a bridge between the entropy method and the linearized theory to obtain the optimal rate of the convergence  beyond the linearized setting. It is a remarkable method when the solution is near the equilibrium in the sense that $f-M_f\sim o(1)$.  We emphasize that it is different from the words ``near the equilibrium" in the linearized theory. In the latter case,  we mean $(f-M_f)/\sqrt{M_f}\sim o(1)$.

\subsubsection{Existence results}  There are few results on the local well-posedness for the original equation~\eqref{Boltz eq}.  We refer to \cite{ukai1}  for the local well-posedness result in Gevrey space with exponential weight, to \cite{amuxy1} for the  local solution in Sobolev spaces if the initial data has exponential moment and to \cite{my} for the local existence of polynomial decay solutions to the Boltzmann equation for soft potentials. 
It is also worthy to mention the result on the global renormalized solution
in  \cite{dl} and \cite{AV02} for the  equation with and without angular cutoff respectively.

\subsection{Difficulties, new ideas and strategies}  The difficulties of the problem of the global dynamics lies in   three parts. The first one is on  the propagation of the  regularity. The second one comes from the hypofunction of \eqref{DispE} which only measures the distance between the current state and the local Maxwellian.  We need a new ingredient to  ensure that the solution will converge to its associated equilibrium. The third one is  the  strategy to   catch the interplay between the energy and the entropy.  

\subsubsection{New ideas and strategies (I): propagation of the regularity}

 The main obstacles for showing the propagation of the regularity for the original Boltzmann equation without angular cutoff lie in three aspects: the increase in weight from the lower and upper bounds for the collision operator, getting the non-negativity of the solution and the robustness of the method which can extend the local existence to be  global.  In what follows, we will explain them in detail. 
\smallskip

 {\it (i).  The increase in weight.} Roughly speaking,  if  $g$ is a non-negative and smooth function and have some  lower bound of the density,  then the lower and upper bounds for the collision operator in Sobolev spaces can be stated as follows:
\ben\label{boundWSP} C_g|f|_{H^s_{\gamma/2}}^2-|g|_{L^1}|f|_{L^2_{\gamma/2}}^2 \le\langle -Q(g, f), f\rangle_v \lesssim |g|_{L^1_{\gamma+2s}} |f|_{H^s_{\gamma/2+s}}^2,\een 
 where the weighted Sobolev norm $|\cdot|_{H^m_l}$ is  defined in Section \ref{funspace}. It is obvious that
additional weight is required for $f$ in the upper bound compared to that in the lower bound. The reason results from the fractional Laplace-Beltrami operator $(-\triangle_{\SS^2})^s$ which exists in the structure of the operator(see \cite{HE16} for details). Moreover when $\gamma>0$,  the lower bound in \eqref{boundWSP} is not handy for the energy estimates. For instance, from  the basic $L^2$ energy estimate, we have
\beno \f{d}{dt}\|f\|_{L^2}^2+ C_f \|f\|_{H^s_{\gamma/2}}^2\le \|f
\|_{L^\infty_xL^1_{\gamma+2s}}\|f\|_{L^2_{\gamma/2}}^2. \eeno
It is not sure that  $\|f
\|_{L^\infty_xL^1_{\gamma+2s}}\|f\|_{L^2_{\gamma/2}}^2$ can be controlled by the dissipation  $C_f \|f\|_{H^s_{\gamma/2}}^2$. On the other hand, since $\gamma>0$, we cannot apply the Gronwall inequality to close the estimate.  It is not difficult to see that  the problem, the increase in weight, occurs in each step of $L^2$ and high order energy estimates. Note that this difficulty does not exist for the linearized theory  thanks to the coercivity estimates(see \cite{amuxy2, gs1} for  details).

Now let us review how this difficulty is  overcome in the previous work. In \cite{ukai1}, the author  designed a proper   Gevrey space with exponential weighted function to solve the problem. In \cite{amuxy1}, by assuming that the initial data $f_0$ has exponential moment, i.e., $ f_0e^{a\lr{v}^2}\in L^2$, and taking the transform: $f=e^{-(a-t)\lr{v}^2}g$, the equation \eqref{Boltz eq} will be reduced to the   equation for $g$ verifying
\beno \pa_tg+v\cdot \na_xg+\lr{v}^2g=\Gamma^t(g,g).\eeno
Now the problem of the increase in weight is absorbed by the damping term in the equation  if $\gamma+2s<1$. Then the well-posedness for $g$ is obtained. However in this case the lifespan of the well-posedness depends not only  on the initial data but also on the transform from $f$ to $g$. Indeed the lifespan will not be longer than $a$ because the equation for $g$ is valid only in the time interval $[0, a)$.  In \cite{my}, the  authors constructed a local solution to the equation with soft potentials in weighted Sobolev spaces with $s\in (0,\f12]$ and $\gamma+2s\le 0$. Let us explain how to establish the local existence in this case. From the sharp upper bounds for the collision operator $Q$ in Theorem \ref{thmub}, the lower bound in \eqref{boundWSP} and the observation that 
\beno \lr{ \pa^{\alpha}Q(f, f) ,\pa^{\alpha}f}_v=\lr{Q(f,  \pa^{\alpha}f),  \pa^{\alpha}f}_v+\sum_{|\alpha_1|\ge1,\alpha_1+\alpha_2=\alpha}\lr{Q( \pa^{\alpha_1}f,  \pa^{\alpha_2}f),  \pa^{\alpha}f}_v, \eeno
 we derive that 
$ |\lr{ \pa^{\alpha}Q(f, f) ,\pa^{\alpha}f}_v|\lesssim\sum_{|\alpha_1|\ge1,\alpha_1+\alpha_2=\alpha}  |\pa^{\alpha_1}f|_{L^2_{7}} |\pa^{\alpha_2}f|_{H^{2s}_{\gamma+2s}}|\pa^{\alpha}f|_{L^2}. $
Thanks to the restrictions that $s\in (0,\f12]$ and $\gamma+2s\le 0$, the above inequality implies that 
\beno |\lr{ \pa^{\alpha}Q(f, f) ,\pa^{\alpha}f}_v|&\lesssim&\sum_{|\alpha_1|\ge1,\alpha_1+\alpha_2=\alpha}  |\pa^{\alpha_1}f|_{L^2_{7}} |\pa^{\alpha_2}f|_{H^1} |\pa^{\alpha}f|_{L^2}. \eeno
By the standard estimates for the commutator between the weight function and the operator $Q$, the energy estimate can be concluded as 
$$ \f{d}{dt} X(t)\le X(t)^{\f32},$$ where $X(t)$ denotes the proper energy functional for the solution.  It is enough to close the argument and get the local existence. Thus the increase in weight does not occur in the energy estimates  due to 
the restrictions   $s\in (0,\f12]$ and $\gamma+2s\le 0$.

 {\it (ii). Getting the non-negativity.}  As for the non-negativity of the solution, since the Boltzmann equation is a non-local parabolic type equation,   the maximum principle is not available and so it is not so easy to prove the desired result.  In \cite{amuxy2}, the authors use proper energy estimates to prove the non-negativity but with the restriction that the solution should have the exponential moment. While in \cite{ukai1,my},  authors proved it by the cutoff approximation thanks to  the Duhamel formula. Here we will follow the latter method to prove the non-negativity by reducing the problem to the study of the asymptotics of the equation from angular cutoff to  non-cutoff.

  {\it (iii). Robustness of the method.}  As we mentioned before, 
  we request that  the method is   robust such that  the local solution can be extended to be the global one. It is not difficult to check that the methods in \cite{amuxy2, my,ukai2} are not robust. We need some new idea. 

\medskip

Our new idea on the proof of the propagation of the regularity is a combination of  the smoothing property of the positive part of the collision operator, averaging lemma for the transport equation, the Povnzer inequality for $L^1$ moment and the sharp bounds for the collision operator.

To make the method robust, we will consider the equation:\ben\label{h-eq} \pa_th+v\cdot \na_x h=Q(f,h)+Q(h,g), \een where
 $h=f-g$ and $f$ and $g$ are  non-negative smooth and bounded functions verifying that their densities have lower bounds.  
The equation \eqref{h-eq} is related to the original equation if $g=0$, related to the consideration of the long time behavior of the solution if $g=M_f$ and related to the strong stability result if $g$ is a  reference solution  to the equation.   
 In what follows, we will  assume that $f$ and $g$ are given functions and  focus on the energy estimates for the linear equation \eqref{h-eq}.
 
 Next we will sketch our strategy and explain what is new in the proof and how to overcome the difficulties mentioned before.   Our strategy can be concluded as follows:
 \begin{enumerate}
 \item  We first consider the propagation of $L^1$-moment. By revisiting the Povnzer inequality (see Lemma \ref{Povnzer2}), we prove  that   
 $$ \int_{\sigma\in \SS^2} (-\lr{v'}^l+\lr{v}^l)b(\cos\theta) d\sigma=O(l^s)\lr{v}^l+L.O.T.$$
  Compared to the previous result, we get the sharp coefficient of $\lr{v}^l$ which is of order $l^s$.  This fact will be used to prove the gain of the moment and also the uniqueness result considering that $l$ can be chosen arbitrarily large.  Based on this new estimate, we prove the propagation of   $L^1$ moments.
  
 \item When the propagation of $L^1$-moment is available, we are ready to prove the  propagation of $L^2$-moment.  To get rid of the increase in weight in the lower bound of  \eqref{boundWSP}, we perform a decomposition of the collision operator and make full use of the smoothing property of the positive part of the collision operator.  Roughly speaking,  $Q$ can be decomposed into three parts:
 \beno Q=\f12Q+\f12(Q^\delta_+-Q^\delta_-)+\f12(Q-Q^\delta), \eeno 
where $\delta$ is sufficiently small and $Q^\delta$ is defined in \eqref{COQeps}. Then the coercivity estimate can be improved by   
\beno \int_{\TT^3}\lr{-Q(f,hW_l),hW_l}_vdx\gtrsim \|h\|_{H^s_{l+\gamma/2}}^2+(\delta^{-2s}-c_f) \|h\|_{L^2_{l+\gamma/2}}^2-\delta^{-6-6s}\|h\|_{L^1_{2l+\gamma}}\|h\|_{L^\infty_xL^2_{2}},\eeno  
 where $W_l=\lr{v}^l$. It is clear that  there is no increase in weight in the estimate.
  \item As we explain it before, we will meet the increase in weight in each time of using the energy estimates. To prove the propagation of the high order regularity, the key observation lies in the result from the regularity theory.  In fact, from $L^2$-moment estimate and the averaging lemma, we can get the smoothing estimates with respect to $x$ variable as follows
 \beno \int_0^t \|h\|_{H^\varrho_x L^2_l}^2 d
 \tau \lesssim \|h_0\|_{L^2_{l+2}}^2+\int_0^t [\|h\|_{L^2_{\gamma+2s+l}}^2+\|h\|_{H^s_{l+\gamma/2}}^2]d\tau+C(\|h\|_{L^1_{2l+\gamma}}, \|h\|_{L^\infty_xL^2_{2}}), \eeno  
where $\varrho\in(0,1)$.   
It seems that  the minus term in the coercivity estimate \eqref{boundWSP} can be absorbed by the smoothing estimate if we replace $h$ by  $|D_x|^\rho h$ in \eqref{h-eq}. In other words the  fractional regularity with respect to $x$ variable for the solution can be propagated. By using the same idea to bootstrap the order, we   can  prove the propagation of the regularity for $x$ variable. To overcome the increase in weight in the upper bound of \eqref{boundWSP}, the weight function and the regularity should be designed well which obeys the law that the high order regularity is equipped with the low order weight function. 
\item Once we have the control of the regularity with respect to $x$ variable, the equation will behave like a homogeneous Boltzmann equation. Then it is not difficult to prove the propagation of the regularity for $v$ variable. This completes the energy estimates. 
\end{enumerate}

To apply the strategy to the non-linear equation \eqref{Boltz eq},  the proper bootstrap assumptions for the solution itself should be given and then applied in the iteration scheme.  In this sense our approach is really quasi-linear. Moreover in the proof, we find that   the lower bound of the density   and   upper bound of the solution in the space $ H^{\f32+\delta}_xL^2_4$  will control the propagation of the regularity.  These conditions are comparable to the minimal assumptions on the solution to ensure that the collision operator behaves like a fractional Laplace operator for $v$ variable. 
 \smallskip
 
Finally let us illustrate how to prove the non-negativity of the solution. Our method is to reduce the problem to the study of the asymptotics of the equation from angular cutoff to  non-cutoff. If  such kind of  asymptotics can be justified, we will automatically get the desired result. However it is not so easy to do that because
the collision operator $Q^\eps$ behaves like a fractional Laplace operator only in the low frequency part and remains the hyperbolic structure in the high frequency part. Thus to show that  the strategy explained in the above is stable for the equation \eqref{NepsB}, we face the following difficulties:
\begin{enumerate}
	\item Prove the  propagation of the regularity noticing that the equation is  hyperbolic in the  high frequency part;
	\item Modify the energy functional to ensure that all the estimates obtained in the strategy   are  uniformly  bounded with respect to the parameter $\eps$.
	\end{enumerate}
Both of above request sharp bound estimates for the collision operators which are uniform with respect to $\eps$, the commutator estimates and the localization of the equation \eqref{NepsB} in the frequency space. 
 
 \subsubsection{New ideas and strategies (II): revisiting the entropy dissipation} 
To bypass the difficulty caused by \eqref{DispE},  the authors in \cite{DV} derived a second-order in time differential inequality on $\|f-M^f_{\rho,u,T}\|$ to prove that the solution $f$ cannot stay too close to local Maxwellians.  Roughly speaking, they repeatedly used the inequality \eqref{SecIne} to show that
 the solution will depart from of the set of local Maxwellians except  one point, the global Maxwellian. It ensures that the solution  converges to its equilibrium. 
 \smallskip
 
 Our new idea for using the entropy is  to show that the relative entropy  will never stop decreasing  until it vanishes.  Assume that $M_f=M_{1,0,1}=M$ and $H(f|M)(t_1)=H(f|M)(t_2)$ with $t_1< t_2$. Then the entropy dissipation estimate \eqref{DispE} implies that $f=M^f_{\rho,u,T}$ for $t\in[t_1,t_2]$.  Our strategy is carried out by three steps:

{\it Step 1:} From the  equation for $f-M^f_{\rho,u,T}$, we prove that there exists a function $F_1=F_1(\rho,u,T)$ such that
\beno  \f{d}{dt}(f-M^f_{\rho,u,T}, F_1(\rho,u,T))+r_1\|T-\lr{T}_x\|_{L^2}^2\le C\|f-M^f_{\rho,u,T}\|_{L^2};\eeno
where $\lr{T}_x=\int_{\TT^3} Tdx$. It implies that $f=M^f_{\rho,u,T}=M^f_{\rho,u,\lr{T}_x}$ for $t\in[t_1,t_2]$. 

{\it Step 2:} By checking the equation for $f-M^f_{\rho,u,\lr{T}_x}$, we obtain 
that there exists a function $F_2=F_2(\rho,u,\lr{T}_x)$ such that
\beno  \f{d}{dt}(f-M^f_{\rho,u,\lr{T}_x}, F_2(\rho,u,\lr{T}_x))+r_2\|u-\lr{u}_x\|_{L^2}^2\le C\|f-M^f_{\rho,u,\lr{T}_x}\|_{L^2};\eeno
 which implies that $f=M^f_{\rho,\lr{u}_x,\lr{T}_x}$ for $t\in[t_1,t_2]$. 

{\it Step 3:} From the equation for $f-M^f_{\rho,\lr{u}_x,\lr{T}_x}$, we derive 
that there exists a function $F_3=F_3(\rho,\lr{u}_x,\lr{T}_x)$ such that
\beno  \f{d}{dt}(f-M^f_{\rho,\lr{u}_x,\lr{T}_x}, F_3(\rho,\lr{u}_x,\lr{T}_x))+r_3\|\rho-1\|_{L^2}^2\le C\|f-M^f_{\rho,\lr{u}_x,\lr{T}_x}\|_{L^2};\eeno  which yields  $f=M^f_{1,\lr{u}_x,\lr{T}_x}$ for $t\in[t_1,t_2]$. 

Due to the conservation of the mass, momentum and the energy,  we can derive that $f=M_{1,0,1}$ for $t\in[t_1,t_2]$. It implies that the entropy is a strictly decreasing function until it vanishes.  In other words, the solution never arrives at the local Maxwellian $M^f_{\rho,u,T}$ until it reaches the equilibrium. Moreover this method provides some kind of the dissipation estimates for hydrodynamical fields which are crucial to the quantitative estimate on the convergence. 

\subsubsection{New ideas and strategies (III): new mechanism for the convergence} Now we are in a position to introduce the energy-entropy method to catch the propagation of the regularity, entropy dissipation and the interplay between the energy and the entropy.

To explain our idea in detail,  we go back to \eqref{h-eq} with $g=M_f$. Let  $X(t)$, $D(t)$  and $H(t)$  denote the energy and dissipation functionals  for $h$ and the relative entropy $H(f|M_f)$ respectively.  Then the energy-entropy method relies on the following first-order system:
\ben 
 \label{E:Energy}\f{d}{dt}X(t)+D(t)&\lesssim& H(t),\\
\label{E:Entropy} \f{d}{dt}H(t)+c_1H(f|M^f_{\rho,u,T})(t)&\lesssim& c_2(H(t)^a+X(t)),\\
\label{E:Hydro} \f{d}{dt}M_h(t)+(\|(\rho-1)(t)\|_{L^2}^2+\|u(t)\|_{L^2}^2+\|(T-1)(t)\|_{L^2}^2)&\lesssim& H(f|M^f_{\rho,u,T})^b(t)X^{1-b}(t),
\een
where $0<a\le1$, $c_2\ll c_1$,  $|M_h(t)|\le X(t)$ and $b\in (0,1)$.  Let us explain where they come from.

\smallskip

{\it (i). Energy inequality \eqref{E:Energy}.}  We first note that   \eqref{E:Energy} comes from the energy estimates for $h$-equation, i.e.,
\ben\label{energyX} \f{d}{dt}X(t)+D(t)\lesssim \|h\|_{L^1}^2, \een 
where $D(t)\ge X(t)$ for hard potentials and $D(t)\le X(t)$ for soft potentials.
Generally it is impossible to eliminate the righthand side term otherwise we will get the convergence to the equilibrium without using the entropy.  Notice that  the term in the righthand side of \eqref{energyX}  can be bounded by the relative entropy thanks to  the Csiszar-Kullback-Pinsker inequality. Thus the energy inequality can be rewritten as \eqref{E:Energy}.

{\it (ii). Entropy dissipation \eqref{E:Entropy} and \eqref{E:Hydro}.} Without loss of generality, we assume that $M_f=M_{1,0,1}\eqdefa M$. Since now the regularity of the solution is propagated at least within a time interval, we may get the pointwise lower bound of the solution thanks to the main results in \cite{Mouhot}. Then the entropy method is invoked. By a slight modification of the entropy dissipation inequality (see \cite{DV}), we get  \eqref{E:Entropy}. 
Thanks to the dissipation estimates for the hydrodynamical fields derived in the last subsection, it is not difficult to obtain \eqref{E:Hydro}. 

 \smallskip

Putting together all the estimates, eventually we  arrive at  the first-order system (\ref{E:Energy}-\ref{E:Hydro}).
 Since the energy is involved in the system, we can quantify the dependence of the rate of the convergence  on the initial data.  And in particular, for $\gamma=2$, we recover the exponential rate of the convergence to the equilibrium.

\medskip

Finally let us give some comments on the  the mechanism of the convergence for hard potentials and soft potentials.

(i). Roughly speaking, for the hard potentials, the mechanism for the convergence can be concluded as follows. By the energy estimates, we first have the control of the energy thanks to the conservation of the mass and the energy estimate \eqref{energyX}:
$$\f{d}{dt}X+X\lesssim 1.$$ Then due to the main theorem in \cite{Mouhot}, we have the pointwise lower bound of the solution for any positive time.  It in turn gives the entropy dissipation estimate.  Finally the interplay between the energy and the entropy results in the convergence to the equilibrium.

(ii). For the soft potentials, the   mechanism is more subtle because in this case the dissipation in the energy estimate is weaker than that for the hard potentials.   In fact, if the initial data only has finite moment, then the dissipation functional $D(t)$ verifies $D(t)<X(t)$. In this case, we only get 
 $$ \f{d}{dt}X\lesssim 1,$$  which implies that we cannot get the uniform bounds for the solution first. It seems that the strategy applied for hard potentials does not work anymore. Therefore for soft potentials, we should go back to  the first-order system to treat the global existence, pointwise lower bound  and the convergence to the equilibrium as a organic whole. We refer readers to \cite{hjz}
 for details.

\subsubsection{Summary} The key point to understand the global dynamics of the equation relies on characterizing the propagation of the regularity,  $H$-theorem and the interplay between the entropy and the energy. Let us summarize our results in this paper. 

We first consider the propagation of the  regularity. The quasi-linear method instead of the standard linearization method is used and meanwhile the approach is   stable in  the  asymptotics of the  equation from angular cutoff to non-cutoff. As a byproduct, we prove that
\begin{enumerate}
	\item  \eqref{Boltz eq} (or \eqref{NepsB}) admits a non-negative and unique solution in weighted Sobolev spaces with polynomial weight. The lifespan of the solution is totally determined by the initial data.
	\item The lower bound of the density together with the upper bound of the solution in the space $ H^{\f32+\delta}_xL^2_{\gamma+4}$  control the propagation of the   regularity.
\item  We derive the first explicit formula on the asymptotics of the Boltzmann equation from angular cutoff to non-cutoff.
\end{enumerate}	
  
In the next, under the assumptions that we have the control of the density and the upper bound for the solution in the space $ H^{\f32+\delta}_xL^2_{\gamma+4}$, we obtain the global dynamics of the equation. In particular, we derive a new mechanism for the convergence. More precisely, we obtain

\begin{enumerate}
\item  the propagation of the regularity or smoothing estimates uniformly in time;
	\item the relative entropy is a decreasing function until it vanishes. 
	\item the quantitative estimates for the dependence of the rate of the convergence on the initial data.
\end{enumerate}

 As a corollary of the local-wellposedness and the new mechanism for the convergence, we can prove the   general strong stability. Roughly speaking, we  prove that any  small perturbation for a reference solution initially will generate a global solution to the equation and   these two   solutions will remain close to each other for all time.   Compared to the previous work, our stability result removes the assumption that perturbed solution and the reference solution should have the same associated equilibrium. Hence the result  implies that the set of smooth and bounded solutions to the equation is  open.

The strategy of the proof for the stability falls into three steps:
\begin{enumerate}
	\item By the local well-posedness for the equation \eqref{h-eq},  we can show that the perturbed solution $f$ will keep close to the reference solution $g$ for a long time if initially they are close. 
	\item The  mechanism for the convergence implies that the reference solution is close to its associated global equilibrium after a long time. 
	\item  Combining these two facts, we can find a time $t_0$ such that $t_0$ is far away from the initial time and  $f(t_0)$ is close to its equilibrium $M_f$. Then it is not difficult to prove the global existence  in the close-to-equilibrium setting.
	\end{enumerate}

 \subsection{Notations, function spaces and main results} \label{funspace} We first  list some notations which will be used in the paper.  We denote the multi-index $\alpha =(\alpha _1,\alpha _2,\alpha _3)$ with
$|\alpha |=\alpha _1+\alpha _2+\alpha _3$. We write $a\lesssim b$ to indicate that  there is a
uniform constant $C,$ which may be different on different lines,
such that $a\leq Cb$.  We use the notation $a\sim b$ whenever $a\lesssim b$ and $b\lesssim
a$. The notation $a^+$ means the maximum value of $a$ and $0$ and $[a]$ denotes the maximum integer which does not exceed $a$.   The Japanese bracket $\lr{\cdot}$ is defined by $\lr{v}=(1+|v|^2)^{\frac{1}{2}}$. We denote $C(\lambda_1,\lambda_2,\cdots, \lambda_n)$ by a constant depending on   parameters $\lambda_1,\lambda_2,\cdots, \lambda_n$. The notations  $\lr{f,g}_v\eqdef \int_{\R^3}f(v)g(v)dv$ and $(f,g)\eqdefa \int_{\R^3\times\TT^3} fgdxdv$ are used to denote the inner products for $v$ variable and for $x,v$ variables respectively. We also set $\lr{f}_x\eqdefa \int_{\TT^3} f(x)dx$ to denote the average value of the function $f$ over the domain $\TT^3$ recalling that $|\TT^3|=1$.

\subsubsection{Function spaces}
  For a distribution function $f(v)$, we have the following definitions.
\begin{enumerate}
\item For   real number $m, l $, we define the weighted Sobolev space
\begin{equation*}
H^{m}_l \eqdefa\bigg\{f(v):  |f|^2_{H^m_l}=\int_{\R^3_v} |\langle D\rangle^m \langle v\rangle^l f(v)|^2 dv
 <+\infty\bigg\},
\end{equation*} Here $a(D)$ is a   differential operator with the symbol
$a(\xi)$ defined by
\beno  \big(a(D)f\big)(v)\eqdefa\f1{(2\pi)^3}\int_{\R^3}\int_{\R^3} e^{i(v-y)\xi}a(\xi)f(y)dyd\xi.\eeno
Similarly we can define a  differential operator $D_x^a$ on the position variables $x$ in the tours $\TT^3$,
$ D^{a}_xf\eqdefa \sum_{q\in\ZZ^3,q\neq 0} |q|^{a}\hat{f}(q)e^{2\pi iq\cdot x},$
where  $\hat{f}$ denotes the Fourier transform with respect to $x$ variables and $a\in \R$.

\item The general weighted Sobolev space $W^{N,p}_l$  with $p\in [1, \infty)$ is defined as follows
\beno
W^{N,p}_l\eqdefa \bigg\{f(v): |f|_{W^{N,p}_l}=\sum_{|\alpha|\le N} \bigg(\int_{\R^3}  |\partial^\alpha f(v)|^p\langle v\rangle^{lp}dv \bigg)^{1/p}<\infty
\bigg\}.
\eeno
In particular, if $N=0$, we  introduce the weighted $L^p_l$ space as
\beno
L^p_l\eqdefa \bigg\{f(v): |f|_{L^p_l}=\bigg(\int_{\R^3} |f(v)|^p\langle v
\rangle^{l p}dv\bigg)^{\f1{p}}<\infty \bigg\}.
\eeno
\item The $L\log L$ space  is defined by
\beno L\log L\eqdefa \bigg\{f(v):  |f|_{L\log L}=\int_{\R^3}
|f|\log (1+|f|)dv<\infty \bigg\}.\eeno
\end{enumerate}

For a distribution function $f(x,v)$, we use the following weighted Sobolev  spaces
with weight  on velocity variable.

\begin{enumerate}
\item For $N_1,N_2\in \N$, the general weighted Sobolev space $H^{N_2}_xW^{N_1,p}_l$  with $p\in [1, \infty)$ is defined by
\ben
H^{N_2}_xW^{N_1,p}_l\eqdefa \bigg\{f(x,v): \|f\|^2_{H^{N_2}_xW^{N_1,p}_l}=\sum_{|\alpha|\le N_2} \int_{\TT^3} |\pa^\alpha_x f|_{W^{N_1,p}_l}^2 dx <\infty
\bigg\}.
\een
When $N_2=0$, we  introduce the weighted $L^{p_1}_xW^{N_1,p_2}_l$ space as
\ben\label{funsp2}
L^{p_1}_xW^{N_1,p_2}_l\eqdefa \bigg\{f(x,v): \|f\|_{L^{p_1}_xW^{N_1,p_2}_l}=\bigg(\int_{\TT^3} |f|^{p_1}_{W^{N_1,p_2}_l}dx \bigg)^{\f1{p_1}}<\infty \bigg\}.
\een
When $p_1=p_2=p$ and $N_1=0$, for simplicity we will use the notation $L^p_l$ to denote $L^p_xL^p_l$. In this case, we set that $\|f\|_{L^p_l}\eqdefa\|f\|_{L^p_xL^p_l}$.

\item For   real numbers $m,n, l $ with $m\ge0$, we define the weighted Sobolev space $H^n_xH^{m}_l$ as
\ben\label{funsp1} &&H^n_xH^{m}_l \eqdefa \bigg\{f(x,v):  \|f\|^2_{H^n_xH^m_l}=\|\lr{D}^mW_l f\|_{L^2_xL^2_v}^2+\mathrm{1}_{n\in \N}\sum_{|\alpha|= n}\|\lr{D}^mW_l\pa^\alpha_x f\|_{L^2_xL^2_v}^2 \notag\\&&\quad+\mathrm{1}_{n\notin \N}\sum_{|\alpha|= [n]} \iiint_{\R^3\times\TT^3\times\TT^3} \f{|(\lr{D}^mW_l\pa_x^\alpha f)(x+k,v)-(\lr{D}^mW_l\pa_x^\alpha f)(x,v)|^2}{|k|^{3+2(n-[n])}}dxdkdv <\infty\bigg\}.
\een 
 Due to the fact that $\int_{\TT^3}  \f{|e^{iq\cdot k}-1|^2}{|k|^{3+2a}}dk\sim |q|^{2a}$ for $a\in(0,1)$, it is easy to check that
   $\|f\|^2_{H^n_xH^m_l}\sim \sum_{q\in\ZZ^3}(1+|q|^{2})^{n} |\hat{f}(q)|^2_{H^m_l},$
where  $\hat{f}$ denotes the Fourier transform with respect to $x$ variables.  
\end{enumerate}

Let $X$ be a function space defined in (\ref{funsp1}-\ref{funsp2}), then $L^2([0,T]; X)$ is defined by
\beno L^2([0,T]; X)\eqdefa \{f(t,x,v) \big|\|f\|_{L^2([0,T]; X)}^2= \int_0^T \|f(t)\|^2_{X} dt<\infty\}. \eeno

\subsubsection{Two types of the dyadic decomposition}
We first list some basic knowledge on the Littlewood-Paley decomposition. Let $B_{\frac{4}{3}}\eqdefa  \{\xi\in \mathrm{R}^3 ~|~ |\xi|\leq \frac{4}{3} \} $
and $ \emph{C}\eqdefa     \{\xi\in \mathrm{R}^3 ~|~ \frac{3}{4} \leq |\xi|\leq \frac{8}{3} \} $. Then one may introduce two radial functions $ \psi \in
C_0^{\infty}(B_{\frac{4}{3}})$ and $ \varphi \in C_0^{\infty}(\emph{C} )$  which satisfy
\begin{eqnarray}\label{defpsivarphi} \psi, \varphi \ge0,\quad\mbox{and}\quad \psi(\xi) + \sum_{j\geq 0} \varphi(2^{-j} \xi) =1,~~\xi \in \mathrm{R}^3.  \end{eqnarray}

We  first   introduce the dyadic decomposition  in the phase space.  Let $N_0$ be a integer. The dyadic operator in the phase space   $ \mathcal{P}_j$ can be defined as
  \begin{eqnarray*} \mathcal{P}_{-1}f(x) =  \psi(x)f(x),~~~~
\mathcal{P}_{j}f(x) = \varphi(2^{-j}x)f(x)  ,~(j\geq 0). \end{eqnarray*}
 Let
$\tilde{ \mathcal{P}}_{j}f(x)=\sum_{|k-j|\le N_0}\mathcal{P}_{k}f(x) $ and $  \mathcal{U}_{j}f(x) = \sum_{k\le j}\mathcal{P}_{k}f(x)$ where $N_0$ verifies $\mathcal{P}_{j}\mathcal{P}_{k}=0$ if $|j-k|>N_0$. For any smooth function $f$, we have
$ f = \mathcal{P}_{-1} f +\sum_{j\geq 0} \mathcal{P}_j f.$

Next we  introduce the dyadic decomposition  in the frequency space.   We denote $ \tilde{m}\eqdefa\mathfrak{F} ^{-1} \psi $ and $\tilde{\phi} \eqdefa \mathfrak{F}^{-1} \varphi $, where they are the inverse Fourier Transform of $\varphi$ and $\psi$.
If we set $\tilde{\phi}_j(x)\eqdefa2^{3j}\tilde{\phi}(2^{j}x)$, then the dyadic operator  in the frequency space $ \mathfrak{F}_j$ can be defined as
follows \begin{eqnarray*} \mathfrak{F}_{-1}f(x) = \int_{\mathrm{R}^3}\tilde{m}(x-y) f(y)dy,~~~~
\mathfrak{F}_{j}f(x) =  \int_{\mathrm{R}^3}\tilde{\phi}_j(x-y) f(y)dy,~(j\geq 0).  \end{eqnarray*}Let  $\tilde{\mathfrak{F}}_{j}f(x)=\sum_{|k-j|\le 3N_0}\mathfrak{F}_{k}f(x)
$ and $\mathcal{S}_{j}f(x) = \sum_{k\le j}\mathfrak{F}_{k}f . $
Then for any $f \in \mathcal{S}'(\mathbb{R}^3)$, it holds
$ f = \mathfrak{F}_{-1} f +\sum_{j\geq 0} \mathfrak{F}_j f. $
\subsubsection{The symbol of the collision operator}  We first give the definition on the symbol $S^{m}_{1,0}$.
\begin{defi}\label{psuopde} A smooth function $a(v,\xi)$ is said to a symbol of type $S^{m}_{1,0}$ if   $a(v,\xi)$  verifies for any multi-indices $\alpha$ and $\beta$,
\beno |(\pa^\alpha_\xi\pa^\beta_v a)(v,\xi)|\le C_{\alpha,\beta} \langle \xi\rangle^{m-|\alpha|}, \eeno
where $C_{\alpha,\beta}$ is a constant depending only on   $\alpha$ and $\beta$.
\end{defi}
To analyze the operator $Q^\ep$, we introduce 
\beno W^\ep_{q}(\xi)\eqdefa \lr{\xi}^q(1-\phi(\ep \xi))+\epsilon^{-q}\phi(\epsilon \xi),\quad  W^\ep_{q+\log}(\xi)=\phi(\ep\xi)\lr{\xi}^{q}\log\lr{\xi}+(1-\phi(\ep\xi))\eps^{-q}|\log \epsilon|
 ,\eeno with $q\in\R^+$ 
and  $\phi\eqdefa 1-\psi$. We  also set functions $ f_\phi\eqdefa(1-\phi(\ep D))f $ and $   f^\phi\eqdefa\phi(\ep D)f$.
 We emphasize that $W^\eps_q(D)$ and $\phi(\ep D)$ are pseudo-differential operators acting only on   $v$ variable.

\subsubsection{Well-prepared sequences of Weight functions}\label{wswf} To prove the propagation of the   regularity for the equation, two types of
sequences of weight functions will be introduced. These sequences of weighted functions
have decreasing orders. 

\begin{defi}$W_l(v)$ is called to be a weight function if $W_l(v)\eqdefa\lr{v}^l$.
	\end{defi}

 Let $\varrho\in (0,1)$  be a parameter related to the hypoellipticity of the equation. It is easy to check that if $\f1{2\varrho}\notin \N$, then $0<N_d<\varrho$ where $N_d\eqdefa(N_{\varrho,2}+1)\varrho-\f12 $ with $N_{\varrho,2}\eqdefa [1/(2\varrho)]$.
  
\begin{defi}\label{ws1} Let  $N\in \N$ and $\varrho,\kappa,\delta_1\in (0,1)$ verify 
	
	 \noindent $\bullet$ (P-1) $\varrho\le \f{s}{4(s+4)}$, $\f1{2\varrho}, \f1{\varrho} \notin\N$;
	
	\noindent $\bullet$ (P-2)
	$3\delta_1\le N_d/2$ and $[(\f12+2\delta_1)/\varrho]=N_{\varrho,2}$;
	
	\noindent $\bullet$  (P-3)  $N+\kappa\ge \f32+2\delta_1$ and $\f{\kappa}{\varrho}\notin\N$.

 Assume that  $2s<q_2<q_1<1+s$ and  \beno
 N_{\varrho,1}\eqdefa[1/\varrho], N_{\varrho,\kappa}\eqdefa [\kappa/\varrho];   N_{\varrho,\delta_1}\eqdefa[\log_2(\varrho/\delta_1+2)]+1,   N_{\delta_1}\eqdefa[\log_2(3(2\delta_1)^{-1}+1)]+1, \\
 N_{q,s,1}\eqdefa\max\{[\log_2(q+s)]+1,[\log_2\big(\f{q+s}{2s}(1+\delta_1)\big)]+1\},  N_{q_1,s,2}\eqdefa \big[\log_2(\f{q_1+s}{q_1-s})\big]+1, \\N_s\eqdefa\f{2s}{1-s},N_{q_2,s,2}\eqdefa \big[\log_2(\f{q_2}{q_2-2s}(1+\delta_1))\big]+1,
 N_{q_1,q_2,s}\eqdefa[\log_2(\f{q_2+s}{q_1-q_2}+1)]+1.\eeno
 
 Suppose that  $  \mathbb{I}_x(N,\kappa)\eqdefa\big\{   \{0,1,\cdots,N-1\}\times \{-1,0,\cdots, N_{\varrho,1}+1\}\cup  \{\{N\}\times\{-1, 0,\cdots,N_{\varrho,\kappa}+1\}\big\}$ and $W_{m,n}\eqdefa W_{l_{m,n}}=\lr{v}^{l_{m,n}}$.  
  Then $\mathbb{W}_{I}(N,\kappa, \varrho,
  \delta_1,q_1,q_2)\eqdefa\{W_{1,\f12+\delta_1}, W_{1,\f12+2\delta_1}\}\cup \{W_{m,n}\}_{(m,n)\in \mathbb{I}_x(N,
	\kappa)}$ is called to be a well-prepared sequence of the  weighted functions of {\bf Type I} if it verifies the following conditions:

\noindent $\bullet$ (W-1) If $W_{0,-1}=W_{l_1}, W_{0,0}=W_{l_2}$, then $l_1\ge \min\{ N_s+2, 2l_2+\gamma\}$;

\noindent $\bullet$ (W-2) if $m\in[0,N-1]$, $W_{m,N_{\varrho,1}+1}=W_{m+1,0}$, $W_{N, N_{\varrho,\kappa}+1}=W_{N,\kappa}\eqdefa W_{l_{N,\kappa}}\ge W_{\gamma+4}$;

\noindent $\bullet$ (W-3) if $m\in [1,N]$, $ W_{m,-1}=W_{m-1,N_{\varrho,1}}$;

\noindent $\bullet$ (W-4)
$\max\{W_{m,n+1}W_{\f32\gamma+2s+d_1+d_2}W_{d_3}, W_{d_3}W_{m,n+1}(W_{2s})^{2^{N_{\varrho,\delta_1}}}\}\le W_{m,n}$;

\noindent $\bullet$ (W-5) $ W_{1,\f12+\delta_1}=W_{l_{1,\f12+\delta_1}}=W_{1,N_{\varrho,2}+1}W_{d_2/2}$, $ \max\{W_{1,\f12+\delta_1}W_{\f32\gamma+2s+d_1+d_2}W_{d_3},W_{d_3}W_{1,\f12+\delta_1}(W_{2s})^{2^{N_{\varrho,\delta_1}}}\}\le W_{1,N_{\varrho,2}}$,
where $d_1>\f32,d_2>\gamma,d_3\in\R^+$ and $W_{1,\f12+2\delta_1}=W_{\gamma+4}$;

\noindent $\bullet$ (W-6)
$\gamma/2+\max\{\f522^{N_{q_1,s,1}},\f522^{N_{q_2,s,1}}, \f{\gamma}22^{N_{q_1,s,2}}(2s+2),  2^{N_{\delta_1}+N_{q_1,q_2,s}}(2s)\}\le l_2$.\\
For simplicity, we   use notation $\mathbb{W}_{I}$ to denote $\mathbb{W}_{I}(N,\kappa, \varrho,
\delta_1,q_1,q_2)$.
 \end{defi}

\begin{defi}\label{ws2}Let  $N\in \N$ and $\varrho,\kappa,\delta_1\in (0,1)$ verify (P-1)-(P-3). Then $\mathbb{W}_{II}(N,\kappa, \varrho,
	\delta_1)\eqdefa \\ \{W_{1,\f12+\delta_1},W_{1,\f12+2\delta_1}\}\cup \{W_{m,n}\}_{(m,n)\in \mathbb{I}_x(N,
		\kappa)}$ is called to be a well-prepared sequence of the  weighted functions of {\bf Type II} if it verifies (W-1)-(W-5). For simplicity, we   use the notation $\mathbb{W}_{II} $ to denote $\mathbb{W}_{II}(N,\kappa, \varrho,\delta_1)$.
\end{defi}
Some remarks are   in order:

 \begin{rmk}
The well-prepared sequences of weight functions depend on the parameters $\gamma, s, N, \kappa, \varrho, q_1, q_2$ and $ \delta_1$. Even when all the parameters are fixed, the well-prepared sequences of weighted functions are not unique.   Let us explain the meaning of the parameters. The sum of parameters $N$ and $\kappa$ is the regularity index for $x$  variable. Parameters $q_1$ and $q_2$ are the regularity indices for the velocity variable $v$.
\end{rmk}

\begin{rmk} Conditions ((W-1)-(W-6)) reflect the fact that the propagation of the regularity depends heavily on the $L^1$ moment. Indeed, Conditions ((W-1)-(W-5))  are used to prove the propagation of the regularity for the spatial variable $x$.  While Condition  (W-6)  is used to prove the propagation of the   regularity for the velocity variable $v$. It obeys the rule that the high moment we have, the more regularity can be propagated and produced(thanks to the hypo-elliptic property of the equation).
\end{rmk}

\begin{rmk} Due to the lower and upper bounds for the original Boltzmann collision operator, we believe that the design of the well-prepared sequences of the weighted functions is compulsory to catch the propagation of the regularity.
	\end{rmk}

 \subsubsection{Energy spaces and the dissipation functionals}\label{notspace} In this subsection, we introduce the energy spaces and the related dissipation functionals. We emphasize that all the definitions are based on the well-prepared sequences of weight functions.  

(i). {\it Notation $E^{m,n,\eps}$.}  In the procedure of applying energy estimates to the equation, the inductive method will be used.
To catch the smoothing effect or the propagation of the regularity for  $x$ variable in each step, 
 we introduce the notations $E^{m,n,\eps}(f), D^{m,n,\eps}_2(f)$ and $D^{m,n,\eps}_3(f)$. Let $W_{m,n}\in \mathbb{W}_{II}(N,\kappa, \varrho,
 \delta_1)$(or $\mathbb{W}_{I}(N,\kappa, \varrho,\delta_1,q_1,q_2)$). Then    $E^{m,n,\eps}(f)\eqdefa\|W_{m,n}f\|^2_{H^{m+n\varrho}_xL^2}$ and
\beno D^{m,n,\eps}_2(f)&\eqdefa&\|W^\ep_s(D)W_{m,n}W_{\gamma/2}f\|^2_{H^{m+n\varrho}_xL^2} +\int_{\TT^3}  \mathcal{E}^{0,\eps}_\mu(W_{m,n}W_{\gamma/2}|D_x|^{m+n\varrho}f)dx,\\
D^{m,n,\eps}_3(f)&\eqdefa& \|W_{-d_1}(W_{m,n+1}W_{\gamma/2+d_1+d_2}f)_\phi\|^2_{H^{m+(n+1)\varrho}_xL^2},
\eeno
where $\mu\eqdefa M_{1,0,1}$ and \ben\label{defEg} \mathcal{E}^{\gamma,\eps}_{g}(f)\eqdefa\f12 \int_{\sigma,v_*,v} |v-v_*|^\gamma b^\ep(\cos\theta) g_*(f'-f)^2 d\sigma dv_*dv.\een
We remark that they are used  in the $(m+n)$-th step of the energy estimates for the equation.

Similar notations can be defined for $E^{1,\f12+\delta_1,\eps}, E^{1,\f12+2\delta_1,\eps} ,E^{N,\kappa,\eps}$ and $D^{1,\f12+\delta_1,\eps}_2, D^{1,\f12+2\delta_1,\eps}_2, D^{N,\kappa,\eps}_2$. For instance, $E^{N,\kappa,\eps}(f)\eqdefa \|W_{N,\kappa}f\|_{H^{N+\kappa}_xL^2}$ and
\beno D^{N,\kappa,\eps}_2(f)&\eqdefa&\|W^\ep_s(D)W_{N,\kappa}W_{\gamma/2}f\|^2_{H^{N+\kappa}_xL^2} +\int_{\TT^3}  \mathcal{E}^{0,\eps}_\mu(W_{N,\kappa}W_{\gamma/2}|D_x|^{N+\kappa}f)dx.
\eeno
When $\epsilon=0$, we simplify the notations  $E^{m,n,0}(f), D^{m,n,0}_2(f)$ and $D^{m,n,0}_3(f)$ to   $E^{m,n}(f), D^{m,n}_2(f)$ and $D^{m,n}_3(f)$.

 (ii).  {\it Energy spaces $V^{q,\eps}$ and $V^q$.} To prove the smoothing effect or the propagation of the regularity for $v$ variable, we introduce the notation:
 $V^{q,\eps}(f)\eqdefa   \sum\limits_{j\le |\log \epsilon|}2^{2qj} \|\mathfrak{F}_j h\|_{L^2}^2+\sum\limits_{j\ge |\log \epsilon|} \ep^{-2q} \|\mathfrak{F}_j h\|_{L^2}^2\sim\|W^\eps_q(D)f\|_{L^2}^2$ and $V^{q}(f)\eqdefa \sum\limits_{j\ge-1}2^{2qj} \|\mathfrak{F}_j h\|_{L^2}^2\sim \|\lr{D}^q f\|_{L^2}^2$.  Energy spaces $V^{q,\eps}$ and $V^q$ are defined by
 \beno V^{q,\eps}=\{f|V^{q,\eps}(f)<\infty\}, V^{q}=\{f|V^{q}(f)<\infty\}. \eeno

(iii). {\it Energy space $\mathbb{E}^{N,\kappa,\eps}(\mathbb{E}^{N,\kappa})$ and the dissipation functional $\mathbb{D}^{N,\kappa,\eps}(\mathbb{D}^{N,\kappa})$.} 
   The
  energy space  $\mathbb{E}^{N,\kappa,\eps}$  is called to be a function space associated to $\mathbb{W}_{I}(N,\kappa,\varrho,\delta_1,q_1,q_2)$ if it is defined by
 $\mathbb{E}^{N,\kappa,\eps}=\{f|\mathbb{E}^{N,\kappa,\eps}(f)<\infty\}, $
 where \beno
&& \mathbb{E}^{N,\kappa,\eps}(f)\eqdefa  \|W_{l_1}f\|_{L^1}+\|W_{l_1-\gamma}f\|_{L^1}^2 +\sum_{(m,n)\in [0, N-1]\times[0, N_{\varrho,1}] }E^{m,n,\eps}(f)+\sum_{n\in [0, N_{\varrho,\kappa}]} E^{N,n,\eps}(f)+ E^{N,\kappa,\eps}(f)\\&&\,\,+V^{q_1,\eps}(f) +E^{1,\f12+\delta_1,\eps}(f)+E^{1,\f12+2\delta_1,\eps}(f)
+\mathrm{1}_{N+\kappa\ge \f52+\delta_1}\|W^\eps_{q_2}(D)f\|_{H^{\f32+\delta_1}_xL^2}^2,\eeno with $W_{m,n}\in \mathbb{W}_{I}(N,\kappa,\varrho, \delta_1,q_1,q_2)$ and $ W_{0,-1}=W_{l_1}, W_{0,0}=W_{l_2}$.
We remark that this kind of space is used to prove the well-posedness for the equation. 

The  dissipation functional $\mathbb{D}^{N,\kappa,\eps}$   consists of four parts:
 \beno  \mathbb{D}^{N,\kappa,\eps}(f)&=&\mathbb{D}^{N,\kappa,\eps}_1(f)+\mathbb{D}^{N,\kappa,\eps}_2(f)+\mathbb{D}^{N,\kappa,\eps}_3(f)+\mathbb{D}_g^{N, \kappa,\eps}(f),\eeno
 where \beno
 \mathbb{D}^{N,\kappa,\eps}_1(f)&\eqdefa& l_1^{s}\|W_{l_1+\gamma}f\|_{L^1}+l_1^{s}\|W_{l_1-\gamma}f\|_{L^1}\|W_{l_1}f(t)\|_{L^1}+\delta^{-2s}\|W_{l_2+\gamma/2}f\|_{L^2}^2,  \\ \mathbb{D}^{N,\kappa,\eps}_2(f)&\eqdefa& \|W^\ep_{q_1+s}(D)W_{\gamma/2}f\|_{L^2}^2
 +\sum_{(m,n)\in [0, N-1]\times[0, N_{\varrho,1}]} D_2^{m,n,\eps}(f)
 +\sum_{n\in [0, N_{\varrho,\kappa}]}D_2^{N,n,\eps}(f) \\&&
 +D_2^{N,\kappa,\eps}(f)
 +D_2^{1,\f12+\delta_1,\eps}(f)+D_2^{1,\f12+2\delta_1,\eps}(f)
   +\mathrm{1}_{N+\kappa\ge \f52+\delta_1}\|W^\ep_{q_2+s}(D)W_{\gamma/2}f\|_{H^{\f32+\delta_1}_xL^2}^2,\\
    \mathbb{D}^{N,\kappa,\eps}_3(f)&\eqdefa&\sum_{0\le m\le N-1,0<n\le N_{\varrho,1}}D^{m,n,\eps}_3(f)+\sum_{m=1}^{N}D^{m,0,\eps}_3(f)
 +\sum_{n\in [1, N_{\varrho,\kappa}+1]}D^{N,n,\eps}_3(f)\\&&+\|W_{-d_1}(W_{1,N_{\varrho,2}+1}W_{\gamma/2+d_1+d_2}f)_\phi\|^2_{H^{1+(N_{\varrho,2}+1)\varrho}_xL^2}, \\
  \mathbb{D}_g^{N, \kappa,\eps}(f)&\eqdefa& \int_{\TT^3} \bigg( 
 \sum_{(m,n)\in [0, N-1]\times[0,N_{\varrho,1}] }  \mathcal{E}^{\gamma,\eps}_{g}(W_{m,n}|D_x|^{m+n\varrho}f) 
 +\sum_{n\in [0, N_{\varrho,\kappa}]}\mathcal{E}^{\gamma,\eps}_{g}(W_{N,n}|D_x|^{N+n\varrho}f)\\&&\quad+\mathcal{E}^{\gamma,\eps}_{g}(W_{N,\kappa}|D_x|^{N+\kappa}f)  +\mathcal{E}^{\gamma,\eps}_{g}(W_{1,\f12+\delta_1}|D_x|^{\f32+\delta_1}f)+\mathcal{E}^{\gamma,\eps}_{g}(W_{1,\f12+2\delta_1}|D_x|^{\f32+2\delta_1}f) \bigg)dx.\eeno
We explain that   $\mathbb{D}^{N,\kappa,\eps}_1,\mathbb{D}^{N,\kappa,\eps}_2,\mathbb{D}^{N,\kappa,\eps}_3$ and $ \mathbb{D}_g^{N, \kappa,\eps}$ correspond to the gain of the weights, gain of the regularity for $v$ variable, gain of the regularity for $x$ variable and gain of the sharp regularity  respectively. The parameter $\delta$ in $\mathbb{D}^{N,\kappa,\eps}_1$ will be determined by the initial data of the equation.

When $\eps=0$, we use the notation $\mathbb{E}^{N,\kappa}\eqdefa \mathbb{E}^{N,\kappa,0}$ to denote the total energy functional for the Boltzmann equation without angular cutoff.
 Similarly the   notations   $\mathbb{D}^{N,\kappa}_1,\mathbb{D}^{N,\kappa}_2,\mathbb{D}^{N,\kappa}_3$ and $ \mathbb{D}_g^{N, \kappa}$ will be used  in the context.

(iv). {\it Energy space $\mathbb{P}_e^{N,\kappa}$  and the dissipation functional $\mathbb{D}^{N,\kappa}_{e}$.} The energy space $\mathbb{P}_e^{N,\kappa}$ is called to be a function space associated to $\mathbb{W}_{II}(N,\kappa,\varrho,\delta_1)$ if it is defined by $\mathbb{P}_e^{N,\kappa}=\{f|\mathbb{P}_e^{N,\kappa}(f)<\infty\}$, where
\beno &&\mathbb{P}_e^{N,\kappa}(f)\eqdefa  \|W_{l_1}f\|_{L^1}^2+\ \sum_{(m,n)\in [0, N-1]\times[0, N_{\varrho,1}]}E^{m,n}(f)+\sum_{n\in [0, N_{\varrho,\kappa}]} E^{N,n }(f)\\&&\qquad+E^{N,\kappa}(f) +E^{1,\f12+\delta_1}(f)+E^{1,\f12+2\delta_1}(f)
 , \eeno with $W_{m,n}\in \mathbb{W}_{II}(N,\kappa,\varrho, \delta_1)$ and $ W_{l_1}=W_{0,-1}, W_{l_2}=W_{0,0}$. Then the corresponding dissipation functional can be defined by \beno &&\mathbb{D}^{N,\kappa}_{e}(f)\eqdefa \|W_{l_1}f\|_{L^1}\|W_{l_1+\gamma}f\|_{L^1}
 +\sum_{(m,n)\in [0, N-1]\times[0, N_{\varrho,1}] } D_2^{m,n}(f)
 +\sum_{n\in [0, N_{\varrho,\kappa}]}D_2^{N,n}(f)
 \\&&\qquad+  D_2^{N,\kappa}(f)
 +D_2^{1,\f12+\delta_1}(f)+D_2^{1,\f12+2\delta_1}(f)
 .\eeno
The energy space $\mathbb{P}_e^{N,\kappa}$ is used to show the propagation of the regularity for  $x$  variable.

(v). {\it Energy space $\mathbb{X}^{N,\kappa,q}$.} Suppose that $\mathbb{P}_e^{N,\kappa}$ is associated to $\mathbb{W}_{II}(N,\kappa,\varrho,\delta_1)$.  The energy space $\mathbb{X}^{N,\kappa,q}$ is defined by $\mathbb{X}^{N,\kappa,q}=\{f|\mathbb{X}^{N,\kappa,q}(f)<\infty\}$, where $q$ verifies that $\f522^{N_{q,s,1}}\le l_{0,0}$ and 
$ \mathbb{X}^{N,\kappa,q}(f)\eqdefa  \mathbb{P}_e^{N,\kappa}(f)+V^q(f).$ It is easy to check that $\mathbb{X}^{N,\kappa,q}= \mathbb{P}_e^{N,\kappa}\cap V^q$.
The space $\mathbb{X}^{N,\kappa,q}$ is used to consider  the propagation of full regularity.

 \subsubsection{Main results} Now we are in a position to state our main results. We begin with the   well-posedness for the original Boltzmann equation.

\begin{thm}[Well-posedness]\label{thmwepo} Let $\mathbb{E}^{N,\kappa,\eps}(\mathbb{E}^{N,\kappa})$ and $\mathbb{E}^{1,\f12+2\delta_1,\eps}(\mathbb{E}^{1,\f12+2\delta_1})$ be function spaces associated to $\mathbb{W}_{I}(N,\kappa,\varrho,\delta_1,q_1,q_2)$ where  $  \mathbb{W}_{I}(N,\kappa,\varrho,\delta_1,q_1,q_2)=\{W_{1,\f12+\delta_1}, W_{1,\f12+2\delta_1}\}\cup \{W_{m,n}\}_{(m,n)\in \mathbb{I}_x(N,
		\kappa)}$ with $N+\kappa\ge \f52+\delta_1$ and $W_{0,-1}=W_{l_1}, W_{0,0}=W_{l_2},W_{1,\f12+\delta_1}=W_{l_{1,\f12+\delta_1}}$. Assume that $f_0$ is a non-negative function   verifying $\|f_0\|^2_{H^{\f32+2\delta_1}_xL^2_{\gamma+4}}\le c_2/2$ and  $\inf\limits_{x\in\TT^3}|f_0|_{L^1}\ge  2c_1$.
 Let  $A_i(c_1,c_2)(i=1,\dots,9)$ be the constants which can be calculated explicitly from Proposition \ref{L1i} to Proposition \ref{Enmix} and \eqref{defiAi}. Suppose that 	$\mathbb{W}_{I}(N,\kappa,\varrho,\delta_1,q_1,q_2), c_1$  and $c_2$ verify \ben\label{unicon1}&&\quad \max\{A_1(c_1,c_2)^{-\f1s},  \mathrm{1}_{\gamma\neq0}8C_I (c_2^{\f12}+V^{q_1}(f_0)),N_s+2\} <  l_{1,\f12+\delta_1}2^{-N_{q_2,s,2}}-N_s-2-\gamma,\\
	&&\quad l_1^{1+s}\ge 40C_E\sqrt{c_2}A_8^{-1},\quad \delta^{-2s}\ge\max\{40C_E c_2A_8^{-1}, \f{4\mathcal{C}_2}{\mathcal{C}_3},\f{12c_2}{\mathcal{C}_5}\} \label{Rstrdelta},\een where $C_I$ is defined in \eqref{defici}, $ N_s, N_{q_2,s,2}$ are  defined in Section \ref{wswf},   $C_E$ is a universal constant appearing in Proposition \ref{Enmix}, $\delta$ is a constant appearing in the definition of $\mathbb{D}_1^{N,\kappa}$ and $\mathcal{C}_i(i=2,3,5)$ are constants in Theorem \ref{thmlb1}, Proposition \ref{lbLdelta} and Theorem \ref{thmlb2}. Then  if $\mathbb{E}^{1,\f12+2\delta_1}(f_0)\le M_1$ and $ \mathbb{E}^{N,\kappa}(f_0)\le M_2$, 
	there exists a time $T^*=T^*(c_1,c_2,M_1,M_2,\mathbb{W}_{I})$ such that
	
	{\bf (i).}  \eqref{NepsB} admits a unique and non-negative solution $f^\eps$ in the function space  $C([0,T^*];   \mathbb{E}^{N,
	 	\kappa,\eps})$ if \ben\label{Rstreps} &&\eps\le \min\{[(\f{A_5c_o  A_7}{200A_6(N_{\varrho,1}+2)(2C_Ec_0A_5+1)})^{1+4/s}c_2^{-1}]^{1/(2(1-s))}, \notag\\&&\qquad\qquad
(\f{\min\{A_7,1\}}{20(N_{\varrho,1}+2)C_{E,1}(A_6+4M_1+4+c_2C_{E,1})})^{1/(2s)} ,l_1^{-\f12-\eta},\f12\delta\},  \een where $\eta>0$, $C_{E,1}$ and $c_o$  are universal constants appearing in Proposition \ref{Enmix}  and \eqref{e321}.  Moreover $f^\eps$ satisfies  that 
	\beno &&\inf_{x\in\TT^3, t\in[0,T^*]}|f^\eps|_{L^1}\ge  c_1; \sup_{ t\in[0,T^*]}\|f^\eps\|_{H^{\f32+2\delta_1}_xL^2_{\gamma+4}}^2\le c_2;\sup_{ t\in[0,T^*]}V^{q_1,\eps}(f(t))\le 4V^{q_1,\eps}(f_0);\\
&& \sup_{t\in[0,T^*]} \mathbb{E}^{1,\f12+2\delta_1,\eps}(f^\eps(t))+A_7(c_1,c_2)\int_0^{T^*}\mathbb{D}_2^{1,\f12+2\delta_1,\eps}(f^\eps(\tau))d\tau\nonumber\\&&\quad+A_8(c_1,c_2)\int_0^{T^*}\mathbb{D}_1^{1,\f12+2\delta_1,\eps}(f^\eps(\tau))d\tau\le 4M_1;
\int_0^{T^*} \mathbb{D}^{1,\f12+2\delta_1,\eps}_3(f^\eps(\tau))d\tau \le 4A_{9}(c_1,c_2) M_1;\\
&& \sup_{t\in[0,T^*]} \mathbb{E}^{N,\kappa,\eps}(f^\eps(t)) +A_7(c_1,c_2)\int_0^{T^*}\mathbb{D}_2^{N, \kappa,\eps}(f^\eps(\tau))d\tau  \le 4M_2;
\int_0^{T^*} \mathbb{D}^{N,\kappa,\eps}_3(f^\eps(\tau))d\tau \le 4A_{9}(c_1,c_2) M_2.
\eeno
If additionally  $f_0\in V^q$ with $2\le q\le N+\kappa$, $\f522^{N_{q,s,1}}\le l_2$ and $(\gamma+5)q\le l_{1,\f12+\delta_1}$, then $f^\eps\in C([0,T^*]; V^q)$. 
	
	{\bf (ii).} \eqref{Boltz eq} admits a unique and non-negative solution $f$ in the function space  $C([0,T^*];   \mathbb{E}^{N,
		\kappa})$  and  for $t\in[0,T^*]$, \beno &&\inf_{x\in\TT^3, t\in[0,T^*]}|f|_{L^1}\ge  c_1; \sup_{ t\in[0,T^*]}\|f\|_{H^{\f32+2\delta_1}_xL^2_{\gamma+4}}^2\le c_2;  \sup_{ t\in[0,T^*]}V^{q_1}(f(t))\le 4V^{q_1}(f_0);\\
	&& \sup_{t\in[0,T^*]} \mathbb{E}^{1,\f12+2\delta_1}(f(t))+A_7(c_1,c_2)\int_0^{T^*}\mathbb{D}_2^{1,\f12+2\delta_1}(f(\tau))d\tau\nonumber\\&&\quad+A_8(c_1,c_2)\int_0^{T^*}\mathbb{D}_1^{1,\f12+2\delta_1}(f (\tau))d\tau\le 4M_1;
\int_0^{T^*} \mathbb{D}^{1,\f12+2\delta_1}_3(f(\tau))d\tau \le 4A_{9}(c_1,c_2) M_1;\\
&& \sup_{t\in[0,T^*]} \mathbb{E}^{N,\kappa}(f(t)) +A_7(c_1,c_2)\int_0^{T^*}\mathbb{D}_2^{N, \kappa}(f(\tau))d\tau  \le 4M_2;
\int_0^{T^*} \mathbb{D}^{N,\kappa}_3(f(\tau))d\tau \le 4A_{9}(c_1,c_2) M_2.
\eeno

{\bf (iii).}  $\forall t\in [0,T^*], \qquad\|f-f^\eps\|_{L^1}=O(\eps^{2-2s})$.
\end{thm}

\begin{rmk}  Function spaces $\mathbb{E}^{1,\f12+2\delta_1,\eps}$ and $\mathbb{E}^{1,\f12+2\delta_1}$ are used to prove the propagation of the key quantity $\|f\|_{H^{\f32+2\delta_1}_xL^2_{\gamma+4}}$. And the function spaces $ \mathbb{E}^{N,
	 	\kappa,\eps}$ and $ \mathbb{E}^{N,
	 	\kappa}$ are used to prove the propagation of the lower bound of the density.  When these two quantities are controllable,  the regularity (or the smoothing estimates) can be propagated(or produced).    
\end{rmk}

\begin{rmk}  Conditions \eqref{unicon1} and \eqref{Rstrdelta} are crucial to prove the stability and the uniqueness results for the equation. They also indicate the dependence of $\mathbb{W}_{I}(N,\kappa,\varrho,\delta_1,q_1,q_2)$ and  $ \mathbb{E}^{N,
		\kappa}$($ \mathbb{E}^{N,
		\kappa,\eps}$)    on the initial data. As a result,  conditions \eqref{unicon1} and \eqref{Rstrdelta} should be verified  when we want to use the continuity argument to improve the lifespan of the solution. 
\end{rmk}

\begin{rmk}    The estimates in result $(i)$ are the bootstrap assumptions for the  existence theory.  To the best of our knowledge, the   formula $f-f^\eps=O(\eps^{2-2s})$ is the first result on the asymptotics of the Boltzmann equation from angular cutoff to non-cutoff.
\end{rmk}

Our second result is on  the global dynamics of the equation.

 \begin{thm}[Global dynamics] \label{thmdyna} Suppose that $f$ is a  non-negative and smooth  solution to the equation \eqref{Boltz eq} verifying  that for $t\ge0$ and $x\in \TT^3$,
 	\ben\label{dynacondi}\rho(t,x)\eqdefa\int_{\R^3} fdv\ge 2c_1>0, \quad \|f(t)\|^2_{H^{\f32+2\delta_1}_xL^2_{\gamma+4}}\le  c_2/2, \een where $c_i(i=1,2)$ are universal constants. Let  $h=f-M_f$.
 	\begin{enumerate}
 	\item {\bf  ($\gamma=0$).}   Let $6<q\le N+\kappa$ and $\eta_1\ll1$. 
 		 \begin{enumerate}
 		 	
 		 	\item {\bf (Propagation of the regularity and the lower bound of the solution)}  Suppose that $\mathbb{X}^{N,\kappa, q}=\mathbb{P}_{e}^{N,\kappa}\cap V^q$ and $\mathbb{P}_e^{1,\f12+2\delta_1}$ are function spaces associated to $\mathbb{W}_{II}(N,\kappa,\varrho,\delta_1)$.  
 		 	 Then for $t\ge0$,
 		\beno \mathbb{X}^{N,\kappa,q}(h(t)) +\int_t^{t+1}( \mathbb{D}^{N,\kappa}_{e}(h)+\mathbb{D}^{N,\kappa}_3(h)+ V^{q+s}(W_{\gamma/2}h))d\tau\le C(c_1,c_2,  \mathbb{X}^{N,\kappa,q}(h_0) ).\eeno	
 		Moreover, if $t\ge (N-1)(N_{\varrho,1}+1)+N_{\varrho,\kappa}+[q/s]+3$, it holds
 		\beno \mathbb{X}^{N,\kappa,q}(h(t))+\int_t^{t+1}( \mathbb{D}^{N,\kappa}_{e}(h)+\mathbb{D}^{N,\kappa}_3(h)+ V^{q+s}(W_{\gamma/2}h))d\tau\le C(c_1,c_2,  \mathbb{P}_{e}^{1,\f12+2\delta_1}(h_0)).\eeno
 		As a direct consequence, we have the pointwise lower bound:
 		$f\ge K_0 \exp\{-A_0 |v|^{q_0}\}$ where $K_0, A_0, q_0$ are constants depending only on $c_1,c_2,  \mathbb{P}_e^{1,\f12+2\delta_1}(h_0)$.
 		
      \item {\bf (Control of the entropy)} For $t> (N-1)(N_{\varrho,1}+1)+N_{\varrho,\kappa}+[q/s]+3$, $H(f|M_f)(t)$ is a strictly decreasing function until it vanishes and verifies $\lim_{t\rightarrow\infty} H(f|M_f)(t)=0$ and $H(f|M_f)(t)\le C(c_1,c_2,  \mathbb{P}_e^{1,\f12+2\delta_1}(h_0)) \|f-M_f\|_{L^1_{q_0+2}}$.

     \item {\bf   (Long time behavior)}  Suppose that
     $\mathbb{X}_{1}^{N,\kappa,q}=\mathbb{P}_{e,1}^{N,\kappa}\cap V^q$ and $\mathbb{X}_{2}^{1,\f12+2\delta_1, q_1}=\mathbb{P}_{e,2}^{1,\f12+2\delta_1}\cap V^{q_1}$ are
     the function spaces associated to  $\mathbb{W}_{II}^{(1)}(N,\kappa,\varrho,\delta_1)$ and $\mathbb{W}_{II}^{(2)}(1,\f12+2\delta_1^{},\varrho,\delta_1^{ })$ respectively.
     Assume that $h_0\in \mathbb{P}_{e,1}^{1,\f12+2\delta_1}$, $W_{m,n}^{(2)}(=W_{l^{(2)}_{m,n}})\le W_{m,n}^{(1)}(=W_{l^{(1)}_{m,n}})$ if $(m,n)\in \mathbb{I}_x(1,\f12+2\delta_1)$, $W_{1,\f12+\delta_1}^{(2)}\le W_{1,\f12+\delta_1}^{(1)}$ and 
     \ben\label{evolcondi5}&&  8C_E( \|M_f\|_{ H^{\eta_1}_{ \gamma+4}}+c_2^{\f12})\le \f12(l^{(2)}_{0,-1})^{1+s}A_1(c_1,c_2).\een    Let 
      $\bar{m}, q_1\in \R^+,a,\theta_1\in (0,1)$ verify
     \ben
   \label{evolcondi}&& 1<q_1<1+s,\quad     102^{N_{	q_1}}\le l_{0,0}^{(2)},\quad \mbox{where}\quad N_{q_1}=[\log_2 (\f{q_1}{q_1-1})]+1, \\
     \label{evolcondi1}&& \max\{2(1-a)^{-1}(q_0+\bar{m}+1)+q_0+2, q_0+2\bar{m}+4\} \le l_{0,0}^{(2)},\qquad 5q_0+4\bar{m}+6\le l^{(2)}_{1,\f12+2\delta_1},  \een  and the interpolation inequality
     \ben\label{evolcondi3} \big(\mathbb{X}^{1,\f12+2\delta_1,q_1}_{2}(h) \big)^{\f12}\le \|h\|_{L^1}^{\theta_1}\big(\mathbb{X}^{N,\kappa,q}_{1}(h) \big)^{(1-\theta_1)/2}. \een 
 There exists a constant $c=\min\{(1-\theta_1)^{-1}, \f{\f32-A}{2(1-a)}( A+\eta(\f32-A))^{-1}\}$ with $A=\f{3}{2(\bar{m}+2)}$ and $\eta<\f12$  
     such that for $t\ge (N-1)(N_{\varrho,1}+1)+N_{\varrho,\kappa}+[q/s]+3$,
     \beno H(f|M_f)(t)+\mathbb{P}_{e,2}^{1,\f12+2\delta_1}(h(t))\lesssim (1+t)^{-c}. \eeno
      \end{enumerate}
  
  	\item {\bf ($\gamma>0$).} Let $t\ge t_0>0$.
  \begin{enumerate}
  	
  	\item {\bf (Smoothing estimates and low bound of the solution)} For any $N,q,l\in\R^+$,
  	\beno  \|h(t)\|_{H^N_xH^q_l}\le C(t_0,N,q,l,c_1,c_2).\eeno As a corollary, for $t\ge t_0>0$, there exist   constants  $K_0, A_0, q_0$ depending only on $c_1,c_2, t_0$ such that
  	$f\ge K_0 \exp\{-A_0 |v|^{q_0}\}$.
  	
  	\item {\bf (Control of the entropy)} For $t\ge t_0>0$, $H(f|M_f)$ is a strictly decreasing function until it vanishes
  	and  verifies $\lim_{t\rightarrow\infty} H(f|M_f)(t)=0$ and $H(f|M_f)(t)\lesssim \|f-M_f\|_{L^1_{q_0+2}}$.
  	
  	\item {\bf (Long time behavior)} For any sufficiently small $\eta>0$,
  	\beno H(f|M_f)(t) \lesssim (1+t)^{-\f1{\eta}}. \eeno
  	If $\gamma=2$, then there exists a universal constant $c$ such that 
 $H(f|M_f)(t) \lesssim e^{-ct}.$
  \end{enumerate}
  \end{enumerate}
 \end{thm}

Some remarks are in order.
\begin{rmk}   The solutions constructed in \cite{gs1} and \cite{HE12}  verify  \eqref{dynacondi}.
	The lower bound of the density and the condition $\|f(t)\|_{L^{\infty}_x(L^1_{\gamma+2s}\cap L\log L)}\le  c_2^{\f12}$ are the minimal assumptions  to prove that the collision operator $Q$ behaves  like a fractional Laplace operator for velocity variable. Obviously   \eqref{dynacondi} is a little stronger than the optimal one.  Considering that $\|f\|_{L^{\infty}_x(L^1_{\gamma+2s}\cap L\log L)}$ is comparable to $\|f\|_{H^{\f32+\delta}_xL^2_{\gamma+2s+2}}$, we think that the assumption \eqref{dynacondi}  is  almost optimal.  It is not clear how to relax it. Some new idea should be introduced, for instance,      the recent development of the De Giorgi's method for the Boltzmann equation in   \cite{ImbSil}. \end{rmk}

 \begin{rmk} Since our new mechanism for the convergence is still related to the entropy method, we do not get the exponential  rate (except for $\gamma=2$) of the convergence compared to the linearized theory.   Noting the recent development of the semi-group's method introduced in \cite{momi},  for $\gamma>0$,  one may obtain the optimal decay rate by  construction of the proper energy functional when the solution is sufficiently close to its equilibrium(see \cite{ht}). \end{rmk}
 
\begin{rmk}
	 It is not   clear that the proof in \cite{Mouhot} can be applied  to \eqref{NepsB} to get the pointwise lower bound for the solution $f^\eps$ uniformly in $\eps$. If it holds,  slight modification of our proof will yield the   rate of the convergence to the equilibrium for \eqref{NepsB} which will be uniform in $\eps$. 
\end{rmk}

Our final  result is concerned with the global-in-time strong stability for the equation. 
\begin{thm}[Global-in-time strong stability] \label{thmstab} Suppose that $g$ is a  non-negative and  smooth  solution to the equation \eqref{Boltz eq} verifying \eqref{dynacondi}
and $g\ge K_0 \exp\{-A_0 |v|^{q_0}\}$. Let $6<q\le N+\kappa, \eta_1\ll1$. 
\begin{enumerate}
\item{\bf ($\gamma=0$)} Suppose that
$\mathbb{X}_{1}^{N,\kappa,q}=\mathbb{P}_{e,1}^{N,\kappa}\cap V^q(\mathbb{E}^{N,\kappa},\mathbb{P}_{e,1}^{1,\f12+2\delta_1})$ and $\mathbb{X}_{2}^{1,\f12+2\delta_1,q_1}=\mathbb{P}_{e,2}^{1,\f12+2\delta_1}\cap V^{q_1}$ are
the function spaces associated to  $\mathbb{W}_{I}^{(1)}(N,\kappa,\varrho,\delta_1,q^{(1)}_1,q^{(1)}_2)$ and $\mathbb{W}_{II}^{(2)}(1,\f12+2\delta_1^{},\varrho,\delta_1^{ })$ respectively. Assume that    $\mathbb{W}_{I}^{(1)}$  and $\mathbb{W}_{II}^{(2)}$ verify all the assumptions stated in result $(c)$  of Theorem \ref{thmdyna} for $\gamma=0$. Moreover, $\mathbb{W}_{I}^{(1)}$ satisfies \eqref{unicon1} and \eqref{Rstrdelta}.  Let   $h_0,g_0\in \mathbb{E}^{N,\kappa}\cap V^q$ and  for $W_{m,n}^{(2)}\in \mathbb{W}_{II}^{(2)}$,
\ben\label{evolcondi2}&&\qquad  8C_E(\sup_{t\ge0}\|g(t)\|_{H^{\f32+\delta_1}_xH^{\eta_1}_{ \gamma+4}}+\|M_g\|_{ H^{\eta_1}_{ \gamma+4}}+c_2^{\f12})\le \f12(l^{(2)}_{0,-1})^{1+s}A_1(c_1,c_2); \|W_{0,-1}^{(2)}W_{2+N_s}g\|^2_{H^{\f32+\delta_1}_xH^{2s+\eta_1}}\\  &&\qquad\qquad\label{evolcondi4}+  \|W_{m,n}^{(2)}W_{\f32\gamma+2s+4} g\|^2_{H_x^{m+n\varrho}H^s}+\sum_{i=1}^2\|W_{1,\f12+i\delta_1}^{(2)}W_{\f32\gamma+2s+4} g\|^2_{H_x^{\f32+i\delta_1}H^s}\lesssim  \mathbb{X}_{1}^{N,\kappa,q}(g).  \een	
  There exists a sufficiently small constant $\eta$  depending only on $c_1,c_2,  \mathbb{P}_{e,1}^{N,\kappa}(g_0)$ such that if $\mathbb{P}_{e,1}^{1,\f12+2\delta_1}(h_0)\le \eta$, then
	 \eqref{Boltz eq}
admits a unique, non-negative and smooth solution $f$ with the initial data   $f_0=g_0+h_0$ in the space $C([0,\infty);\mathbb{E}^{N,\kappa}\cap V^q)$. Moreover, for any $t\ge 0$, we have
\beno  \mathbb{P}_{e,2}^{1,\f12+2\delta_1}((f-g)(t))\lesssim \min\{(1+ |\ln \eta|)^{-c}, (1+ t)^{-c}+\eta\},
\eeno
where the constant $c$ is stated in Theorem \ref{thmdyna}.

 \item{\bf  ($\gamma>0$)} Suppose $\mathbb{E}^{2,\f12+\delta_1}(\mathbb{P}_{e,1}^{1,\f12+2\delta_1})$  and $\mathbb{X}_{2}^{1,\f12+2\delta_1,q_1}=\mathbb{P}_{e,2}^{1,\f12+2\delta_1}\cap V^{q_1}$ are the function spaces associated to $\mathbb{W}_I^{(1)}(2,\f12+\delta_1,\varrho,\delta_1,q_1,q_2)$ and $\mathbb{W}_{II}^{(2)}(2,\f12+\delta_1,\varrho,\delta_1)$ respectively. Assume that $\mathbb{W}_{II}^{(2)}$ verify \eqref{evolcondi2}  and \ben\label{evolcondi6} &&\sup_{t\ge0}\big(V^{q_1}(g(t))+\|W_{0,-1}^{(2)}W_{2+N_s}g(t)\|^2_{H^{\f32+\delta_1}_xH^{2s+\eta_1}}+  \sum_{(m,n)\in \mathbb{I}_x(1,\f12+2\delta_1)}\|W_{m,n}^{(2)}W_{\f32\gamma+2s+4} g(t)\|^2_{H_x^{m+n\varrho}H^s}\notag\\&&\quad+\|W_{1,\f12+\delta_1}^{(2)}W_{\f32\gamma+2s+4} g(t)\|^2_{H_x^{\f32+\delta_1}H^s}+\|W_{1,\f12+2\delta_1}^{(2)}W_{\f32\gamma+2s+4} g(t)\|^2_{H_x^{\f32+2\delta_1}H^s}\big) \le \bar{\mathcal{C}},\een Moreover,  $\mathbb{W}_{I}^{(1)}$  verifies \eqref{Rstrdelta}, $W_{m,n}^{(2)}(=W_{l^{(2)}_{m,n}})\le W_{m,n}^{(1)}(=W_{l^{(1)}_{m,n}})$ if $(m,n)\in \mathbb{I}_x(1,\f12+2\delta_1)$, $W_{1,\f12+\delta_1}^{(2)}\le W_{1,\f12+\delta_1}^{(1)}$ and \ben\label{unicon2}\qquad  &&\max\{A_1(c_1,c_2)^{-\f1s},   8C_I ( \bar{\mathcal{C}}+b_2(1+N)+C(  \bar{\mathcal{C}},c_1,c_2)^{\f12},N_s+2\}  < l^{(1)}_{1,\f12+\delta_1}2^{-N_{q_2,s,2}}-N_s-2-\gamma,\een
 where $b_2, C$ are defined in Lemma \ref{refEn} and $N>0$.  Let $h_0\in\mathbb{E}^{2,\f12+\delta_1}$. 
There exists a sufficiently small constant  $\eta$ depending only on $g$ such that if   $\mathbb{P}_{e,1}^{1,\f12+2\delta_1}(h_0)\le \eta$ and $V^{q_1}(h_0)\le N$,  then  \eqref{Boltz eq}
admits a unique, non-negative  and smooth solution $f$ with the initial data $f_0=g_0+h_0$ in the space $C([0,\infty);\mathbb{E}^{2,\f12+\delta_1})$. Moreover, for any $t\ge t_0>0$,  
$\mathbb{P}_{e,1}^{1,\f12+2\delta_1}((f-g)(t))\le O( |\ln \eta|^{-\infty}).$
\end{enumerate}
\end{thm}

\begin{rmk} The perturbation is performed in  function space $\mathbb{P}_{e}^{\f12+2\delta_1}$. But we still request that the perturbed solution is in the energy space  $\mathbb{E}^{2,\f12+\delta_1}$ to get the uniqueness. Condition \eqref{unicon2} is used to ensure that \eqref{unicon1} holds for the perturbed solution for all time. It will allow us to use continuity argument  to prove the global stability.         
\end{rmk}

\subsection{Organization of the paper} Section 2 is devoted to the analysis of the collision operator and the hypo-elliptic estimate for the transport equation. Then  the proof of the well-posedness for the equation with and without angular cutoff will be given in Section 3 and Section 4. In   Section 5, we will prove the global dynamics for the equation under the assumption \eqref{dynacondi}. We will prove the global-in-time stability for the equation in Section 6. Technical lemmas and auxiliary theorems will be list in the appendix.

\section{Analysis of the collision operator and the transport equation}
In this section, we will make a full analysis of the collision operator and the transport equation in order to get some key estimates for proving the main theorems.

\subsection{The function spaces related to the symbol $W^{\eps}_q$}
We will show some useful equivalences and inequalities for the function spaces related to the symbol $W^{\eps}_q$.

\begin{lem} \label{func} Let $f$ be a smooth function defined in $\R^3$. Then for $l \in \R$ and $p,m,q\ge0$,  it hold
	\ben \label{func1} |f|_{H^m_l}&\sim& |f^\phi|_{H^m_l}+|f_\phi|_{H^m_l};\\
	|f_\phi|_{H^m_l}&\lesssim& \epsilon^{-p}|f_{\phi}|_{H^{(m-p)}_l},  \quad\mbox{for} \quad p\le m; \label{func2}\\
	|W^\ep_q(D)(W_l f  )|_{H^m}&\sim& \ep^{-q}|f^\phi|_{H^m_l}+|f_\phi|_{H^{m+s}_l}\sim|W^\ep_q(D)f  |_{H^m_l};\label{func3}\\
	|W^\eps_q(D)W_l f|_{L^2}^2&\lesssim& |W^\eps_{q-\eta}(D)W_{2l} f|_{L^2}|W^\eps_{q+\eta}(D) f|_{L^2},\quad\mbox{for} \quad \eta\le q.\label{func4}
	  \een
\end{lem}
\begin{proof}  (i) {\it Proof of \eqref{func1}}: By Lemma \ref{baslem2} and the definition of $f_\phi$, We infer that
	\beno |\lf{f}|_{H^m_l}&\sim &|W_l \lr{D}^m(1-\phi(\ep D))f|_{L^2}  
	\lesssim| \lr{D}^m(1-\phi(\ep D))W_lf|_{L^2}+|[W_l, \lr{D}^m(1-\phi(\ep D))]f|_{L^2}\\
	&\lesssim& | \lr{D}^m W_lf|_{L^2}+|f|_{H^{m-1}_{l-1}}\lesssim |f|_{H^m_l}.\eeno
	The same argument can be applied to $\hf{f}$ which together with the above  imply the first equivalence \eqref{func1} since the inverse inequality is trivial by the triangle inequality.

	(ii) {\it Proof of \eqref{func2}}:   Due to the facts $(1-\phi(\ep \xi))(1-\phi(\f12\ep\xi))=1-\phi(\ep \xi)$ and $ \ep^{p}\lr{\xi}^m(1-\phi(\f12\ep\xi))\in S^{m-p}_{1,0}$(see Definition \ref{psuopde}) and Lemma \ref{baslem1}, we deduce that
	\beno |\lf{f}|_{H^m_l}&\sim &|W_l \lr{D}^m(1-\phi(\ep D))f|_{L^2} 
	=|W_l \lr{D}^m(1-\phi(\f12\ep D)) f_\phi|_{L^2} \\
	&\lesssim& \ep^{-p}| \lr{D}^{m-p} W_l f_\phi|_{L^2}\sim \ep^{-p} |\lf{f}|_{H^{m-p}_l}.\eeno
	
	(iii) {\it Proof of \eqref{func3}}: We claim that
	\ben \label{wsqequi1} |W^\ep_q(D)(W_l f)|_{H^m}\sim |W^\eps_q(D)f|_{H^m_l}.  \een
	
	Let $A^\eps(\xi)\eqdefa W^\ep_q(\xi)\lr{\xi}^m$ and $B^\eps(\xi)\eqdefa \big((1-\phi(\eps\xi))\lr{\xi}^{-q}+\eps^{q}\phi(\eps\xi)\big)\lr{\xi}^{-m}$. It is easy to check that $A^\eps B^\eps, (A^\eps)^{-1}(B^\eps)^{-1}\in S^{0}_{1,0}$.
	We first prove that $| A^\eps(D)W_lf|_{L^2}\lesssim |W_lA^{\eps}(D)f|_{L^2}$. Observe that
$A^\eps(D)W_l=P_1P_2W_lA^{\eps}(D)$,
	where $P_1=A^\eps(D)W_lB^\eps(D) W_{-l}$ and $P_2=W_l(B^\eps(D))^{-1}(A^{\eps}(D))^{-1}W_{-l}$. We reduce the desired estimate  to show that
	$P_1$ and $P_2$ are bounded operators in $L^2$.
	
	Due to Lemma \ref{baslem2}, we have
	$ B^\eps(D) W_{-l}=W_{-l}B^\eps(D)+\sum\limits_{1\le\alpha\le N-1}\f{1}{\alpha!} (\pa^\alpha W_{-l})(\pa^\alpha B^\eps)(D)+r_{N}(v,D)$, 
	where $\lr{v}^{l+N}r_N(v,D)\in S^{-N}_{1,0}$.
	Then we deduce that
	\beno A^\eps(D)W_lB^\eps(D) W_{-l}=A^\eps(D)B^\eps(D)+\sum_{1\le\alpha\le N-1} \f1{\alpha!} A^\eps(D)W_l(\pa^\alpha W_{-l})(\pa^\alpha B^\eps)(D)+A^\eps(D)W_lr_{N}(v,D). \eeno
	Notice that if $N\ge m+q$,
   $ |A^\eps(D)W_lr_{N}(v,D)f|_{L^2}\lesssim |\lr{D}^{m+q}W_lr_N(v,D)f|_{L^2}
	\lesssim |f|_{L^2}. $
	Then to show that $P_1$ is a bounded operator, it suffices to prove that for $1\le\alpha\le N-1$,
  \ben\label{wsq1} |A^\eps(D)W_l(\pa^\alpha W_{-l})(\pa^\alpha B^\eps)(D)f|_{L^2}\lesssim |f|_{L^2}.\een
	By the similar expasion, we have
	\beno A^\eps(D)W_l(\pa^\alpha W_{-l})(\pa^\alpha B^\eps)(D)&=&W_l(\pa^\alpha W_{-l})A^\eps(D)(\pa^\alpha B^\eps)(D)+
	\sum_{1\le|\beta|\le N_1-1} \f1{\beta!} \pa^\beta\big(W_l(\pa^\alpha W_{-l}))\\&&\times\big(\pa^\beta A^\eps\big)(D)(\pa^\alpha B^\eps)(D)+r_{N_1}(v,D)(\pa^\alpha B^\eps)(D), \eeno
	where  $\lr{v}^{|\alpha|}r_{N_1}(v,D)\in S^{m+q-N_1}_{1,0}$.
	It implies \eqref{wsq1}.  Thus we obtain that $|P_1 f|_{L^2}\lesssim |f|_{L^2}.$
	
	Note that $(A^\eps)^{-1}(B^\eps)^{-1}\in S^0_{1,0}$. With the help of Lemma \ref{baslem1}, it is easy to check that
	$ |P_2f|_{L^2}\lesssim |f|_{L^2}, $
	which gives $| A^\eps(D)W_lf|_{L^2}\le |W_lA^{\eps}(D)f|_{L^2}$.
	
	To prove the inverse inequality of \eqref{wsqequi1}, we notice that
   $ W_lA^{\eps}=P_3P_4A^\eps(D)W_l,$
	where $P_3=W_lA^{\eps}W_{-l}B^\eps(D)$ and $P_4=(B^\eps(D))^{-1}(A^{\eps}(D))^{-1}$.
	Following the similar argument just used before, we can show that $P_3$ and $P_4$ are bounded operators in $L^2$. Thus we have $|W_lA^{\eps}(D) f|_{L^2}\le |A^\eps(D)W_lf|_{L^2},$ which completes the proof to \eqref{wsqequi1}.
	
	Now we are in a position to prove \eqref{func3}. Thanks to Lemma \ref{baslem2}, we get
	\beno W_l\lr{D}^{m+q}\phi(\eps D)W_{-l}=W_lW_{-l}\lr{D}^{m+q}\phi(\eps D)
	+\sum_{1\le|\alpha|\le N-1}\f{1}{\alpha!} W_l(\pa^\alpha W_l)
	\big(\pa^{\alpha}(\lr{\cdot}^{m+q} \phi(\eps \cdot)) \big)(D)+W_lr_N(v,D),\eeno
	where  ${v}^{l+N}r_N(v,D)\in S^{m+q-N}_{1,0}$ with $N\ge m+q$.
	It implies that
	$|f_\phi|_{H^{m+q}_l}\sim|(W_l\lr{D}^{m+q}\phi(\eps D)W_{-l})W_lf|_{L^2}
	\lesssim |W_q^\eps(D)W_lf|_{H^m}. $
	Similarly we can obtain that
	$\eps^{-q}|f^\phi|_{H^m_l}\lesssim |W_q^\eps(D)W_lf|_{H^m}.$
These two inequalities yield that
     $|f_\phi|_{H^{m+q}_l}+ \eps^{-q}|f^\phi|_{H^m_l}\lesssim |W_q^\eps(D)W_lf|_{H^m}.$
    Observe that $ |W_q^\eps(D)f|_{H^m_l}=|\lr{D}^{m+q}f_\phi+\eps^{-q}\lr{D}^mf^\phi|_{L^2_l}
 $.
     We   derive that
     \beno |W_q^\eps(D)f|_{H^m_l}\lesssim |f_\phi|_{H^{m+q}_l}+ \eps^{-q}|f^\phi|_{H^m_l}\lesssim|W_q^\eps(D)W_lf|_{H^m}\sim |W_q^\eps(D)f|_{H^m_l}, \eeno
    where we use  \eqref{wsqequi1} in the last step. It completes the proof to \eqref{func3}.

     (iv).{\it Proof of \eqref{func4}}: It is not difficult to check that
    $ \lr{W^\eps_q(D)W_lf, W^\eps_q(D)W_lf}_v=\lr{ W_lW^\eps_{q-\eta}(D)W_lf,\\ W_{-l}W^\eps_{q+\eta}(D)W_lf}_v. $
     Thanks to \eqref{func3}, we have
     \beno  |W^\eps_q(D)W_lf|^2_{L^2}&\lesssim& |W_lW^\eps_{q-\eta}(D)W_lf|_{L^2}|W_{-l}W^\eps_{q+\eta}(D)W_lf|_{L^2} 
     \lesssim | W^\eps_{q-\eta}(D)W_{2l}f|_{L^2}| W^\eps_{q+\eta}(D) f|_{L^2}. \eeno
     We complete the proof of the lemma.
\end{proof}

 \begin{cor} Suppose $\Phi(v)\in S^l_{1,0}$.   For $q\ge0$,  $A^\eps(\xi)$ and $W^\eps_{q}(\xi)$ verify $|A^\eps(\xi)|\le W^\eps_{q}(\xi)$   and
 	$ |(\pa^\alpha A^\eps)(\xi)|\le W^\eps_{(q-|\alpha|)^+}(\xi).$
 	Then we have
 	\ben\label{func5} |\Phi A^\eps(D)f|_{H^m}+| A^\eps(D)\Phi f|_{H^m}\lesssim |W^\eps_q(D)W_lf|_{H^m}.
 	\een
As applications, one has	
 	\ben\sum_{k\ge-1}|[W^\ep_q(D), \mathcal{P}_k2^{kp}W_l]f|_{L^2}^2&\lesssim& |W^\eps_{(q-1)^+}(D)W_{l+p-1}f|_{L^2}^2  ;
 	\label{func6}\\
 	|[W^\ep_{q}(D), \mathcal{U}_k]f|_{L^2}&\lesssim& |W^\eps_{(q-1)^+} (D)f|_{L^2};
 	\label{func7}\\
 	\sum_{j\le [\log_2 \f1\epsilon]} 2^{2qj}|W_l\mathfrak{F}_jf|_{L^2}^2+\sum_{j>[\log_2 \f1\epsilon] } \eps^{-2q} |W_l\mathfrak{F}_jf|_{L^2}^2&\sim& |W^\eps_{q}(D)W_lf|_{L^2}^2.\label{func8} \een
\end{cor}
\begin{proof} We set $B^\eps_q(\xi)=(1-\phi(\eps \xi))\lr{\xi}^{-q}+\phi(\eps \xi)\eps^q$. It is easy to check that $B^\eps_q A^\eps,(B^\eps_q)^{-1} (W^\eps_q)^{-1} \in S^{0}_{1,0}$.
	
(i). {\it Proof of \eqref{func5}}: We first note that
\beno W_l\lr{D}^m A^\eps(D)= P_5(B^\eps_q(D))^{-1}\lr{D}^mW_l, \lr{D}^m A^\eps(D)\Phi=P_6P_7W_l\lr{D}^m W^\eps_q(D) \eeno	
	where $P_5=W_l\lr{D}^m A^\eps(D) W_{-l}B^\eps_q(D)\lr{D}^{-m}$, $P_6=\lr{D}^m A^\eps(D)\Phi B^\eps_q(D)\lr{D}^{-m}W_{-l}$ and\\ $P_7=W_l(B^\eps_q(D))^{-1}(W^{\eps}_q(D))^{-1}W_{-l}$.
By using the similar argument applied to $P_i(i=1,2,3,4)$, we can prove that
$P_5,P_6$ and $P_7 $ are bounded operators in $L^2$. Thus we have
\beno &&|\Phi A^\eps(D)f|_{H^m}\lesssim |W_l \lr{D}^m A^\eps(D)f|_{L^2}\lesssim
|(B^\eps_q(D))^{-1}\lr{D}^mW_l f|_{L^2}\lesssim |W^\eps_q(D)W_lf|_{H^m},\\	
 &&| A^\eps(D)\Phi f|_{H^m}\lesssim |W_l\lr{D}^m W^\eps_q(D)f|_{L^2}\sim |W^\eps_q(D)W_lf|_{H^m}. \eeno

(ii). {\it Proof of \eqref{func6}}: We first observe that
\beno [W^\eps_q(D), \mathcal{P}_k2^{kp}W_l]&=&2^{-k}[W^\eps_q(D), \mathcal{P}_k2^{k(p+1)}W_l]\\&=&2^{-k}\big(\sum_{1\le |\alpha|\le N-1}\f{1}{\alpha!} \pa^\alpha(\mathcal{P}_k2^{k(p+1)}W_l)(\pa^\alpha W^{\eps}_q)(D)+r_N(v,D)\big), \eeno
where $\lr{v}^{N-(l+p+1)}r_N\in S^{q-N}_{1,0}$ with $N\ge l+p+1$.
Then we have
\beno |[W^\eps_q(D), \mathcal{P}_k2^{kp}W_l]f|_{L^2}\lesssim
\sum_{1\le |\alpha|\le N-1} |\pa^\alpha(\mathcal{P}_k2^{kp}W_l)(\pa^\alpha W^{\eps}_q)(D)f|_{L^2}+2^{-k}|f|_{L^2},  \eeno
which together with \eqref{func5} imply that
\ben\label{dd} \sum_{k\ge-1}|[W^\eps_q(D), \mathcal{P}_k2^{kp}W_l]f|_{L^2}^2&\lesssim& |W_{l+p-1} W^{\eps}_{(q-1)^+}(D)f|^2_{L^2}+|f|_{L^2}^2.  \een

To prove the desired result, we observe   the facts that
$ |[W^\eps_q(D), \mathcal{P}_k2^{kp}W_l]f|_{L^2}=|[W^\eps_q(D), \mathcal{P}_k2^{kp}W_l]W_{-(l+p-1)}\\(W_{l+p-1}f)|_{L^2}$
and \beno
[W^\eps_q(D), \mathcal{P}_k2^{kp}W_l]W_{-(l+p-1)}&=&[W^\eps_q(D), \mathcal{P}_k2^{kp}W_{-(p-1)}]+\sum_{1\le|\alpha|\le N-1} \f{1}{\alpha!}
\mathcal{P}_k2^{kp}W_l\pa^\alpha W_{-(l+p-1)}\\&&\times(\pa^\alpha W^\eps_q)(D)+\mathcal{P}_k2^{kp}W_lr_N(v,D),
\eeno
where $\lr{v}^{N+l+p+1}r_N(v,D)\in S^{q-N}_{1,0}$. By \eqref{func5},  we first have
\beno \sum_{k\ge-1} |[W^\eps_q(D), \mathcal{P}_k2^{kp}W_l]f|_{L^2}^2\le \sum_{k\ge-1} \|[W^\eps_q(D), \mathcal{P}_k2^{kp}W_{-(p-1)}](W_{l+p-1}f)\|_{L^2}^2+
\|W^{\eps}_{(q-1)^+}(D) W_{l+p-1}f\|_{L^2}^2. 
\eeno   Then \eqref{func6} is derived by \eqref{dd}.

 (iii). The result of \eqref{func7} is easily derived  by \eqref{func5} and the  expansion
$ [W^\ep_{q}(D), \mathcal{U}_k]=\sum_{1\le|\alpha|\le N-1} \f{1}{\alpha!}
 \pa^\alpha  \mathcal{U}_k\\(\pa^\alpha W^\eps_q)(D)+r_N(v,D), $
 where  $\lr{v}^{N}r_N(v,D)\in S^{q-N}_{1,0}$.

 (iv).  {\it Proof of \eqref{func8}:} We first  notice that
 \ben &&\sum_{j\le ã[\log_2 \f1\epsilon]} 2^{2qj}|W_l\mathfrak{F}_jf|_{L^2}^2+\sum_{j> [\log_2 \f1\epsilon] } \eps^{-2q} |W_l\mathfrak{F}_jf|_{L^2}^2\notag\\
 &&\sim\sum_{j\le [\log_2 \f1\epsilon]} 2^{2qj}|W_l\mathfrak{F}_j(1-\phi(\f14\eps  D))f|_{L^2}^2+\sum_{j>[\log_2 \f1\epsilon] } \eps^{-2q} |W_l\mathfrak{F}_j\phi(4\eps D)f|_{L^2}^2\label{iv1}\\&&
=\sum_{j\ge-1} \big(2^{2qj}|W_l\mathfrak{F}_j(1-\phi(\f14\eps  D))f|_{L^2}^2+  \eps^{-2q} |W_l\mathfrak{F}_j\phi(4\eps D)f|_{L^2}^2\big)\label{iv2}\\&&- \sum_{j\ge[\log_2 \f1\epsilon]} 2^{2qj}|W_l\mathfrak{F}_j(1-\phi(\f14\eps  D))f|_{L^2}^2-\sum_{j\le[\log_2 \f1\epsilon] } \eps^{-2q} |W_l\mathfrak{F}_j\phi(4\eps D)f|_{L^2}^2.\notag
\een
Thanks to Theorem \ref{baslem3} and \eqref{func5}, \eqref{iv1} implies that 
 \beno \sum_{j\le [\log_2 \f1\epsilon]} 2^{2qj}|W_l\mathfrak{F}_jf|_{L^2}^2+\sum_{j> [\log_2 \f1\epsilon] } \eps^{-2q} |W_l\mathfrak{F}_jf|_{L^2}^2\lesssim |(1-\phi(\f14\eps  D))f|_{H^q_l}^2+\eps^{-2q}|\phi(4\eps D)f|_{L^2_l}^2\lesssim|W^\eps_q(D)f|_{L^2_l}^2. \eeno
 
 To prove the inverse inequality,  we claim that
 \beno&& \sum_{j\le ã[\log_2 \f1\epsilon]} 2^{2qj}|W_l\mathfrak{F}_jf|_{L^2}^2+\sum_{j> [\log_2 \f1\epsilon] } \eps^{-2q} |W_l\mathfrak{F}_jf|_{L^2}^2 
 \gtrsim
 |(1-\phi(\f14\eps  D))f|_{H^q_l}^2+\eps^{-2q}(|\phi(4\eps D)f|_{L^2_{l}}^2\\
 &&+ |W_l(\phi(4\eps D)-\phi(\f14\eps D))f|_{L^2}^2)\gtrsim
 |(1-\phi(\f14\eps  D))f|_{H^q_l}^2+\eps^{-2q}|\phi(\f14\eps D)f|_{L^2_{l}}^2\sim
 |W^\eps_q(D)W_lf|_{L^2}^2.
 \eeno
We first note that if $j\sim [\log_2 \f1\epsilon]$, then
\beno &&|W_l\mathfrak{F}_j(1-\phi(\f14\eps  D))f|_{L^2}=|W_l\tilde{\mathfrak{F}}_j(1-\phi(\f14\eps  D))\mathfrak{F}_jf|_{L^2}\\
 &&\le |\tilde{\mathfrak{F}}_j(1-\phi(\f14\eps  D))W_l\mathfrak{F}_jf|_{L^2}+|[W_l,\tilde{\mathfrak{F}}_j(1-\phi(\f14\eps  D))]\mathfrak{F}_jf|_{L^2}\le (1+2^{-j})|W_l\mathfrak{F}_jf|_{L^2},\eeno
which yields that \beno &&\sum_{j\ge[\log_2 \f1\epsilon]} 2^{2qj}|W_l\mathfrak{F}_j(1-\phi(\f14\eps  D))f|_{L^2}^2=\sum_{j\sim [\log_2 \f1\epsilon]} 2^{2qj}|W_l\mathfrak{F}_j(1-\phi(\f14\eps  D))f|_{L^2}^2 \\ &&\lesssim \sum_{j\le [\log_2 \f1\epsilon]} 2^{2qj}|W_l\mathfrak{F}_jf|_{L^2}^2+\sum_{j> [\log_2 \f1\epsilon] } \eps^{-2q} |W_l\mathfrak{F}_jf|_{L^2}^2. \eeno  Similarly we can prove 
 $  \sum_{j\le[\log_2 \f1\epsilon] } \eps^{-2q} |W_l\mathfrak{F}_j\phi(4\eps D)f|_{L^2}^2\lesssim \sum_{j\le [\log_2 \f1\epsilon]} 2^{2qj}|W_l\mathfrak{F}_jf|_{L^2}^2+\sum_{j> [\log_2 \f1\epsilon] } \eps^{-2q} |W_l\mathfrak{F}_jf|_{L^2}^2. $
 From these two estimates together with \eqref{iv2} and Theorem \ref{baslem3}, we have 
   \ben\label{d} &&\sum_{j\le ã[\log_2 \f1\epsilon]} 2^{2qj}|W_l\mathfrak{F}_jf|_{L^2}^2+\sum_{j> [\log_2 \f1\epsilon] } \eps^{-2q} |W_l\mathfrak{F}_jf|_{L^2}^2 \gtrsim
 |(1-\phi(\f14\eps  D))f|_{H^q_l}^2+\eps^{-2q}|\phi(4\eps D)f|_{L^2_{l}}^2.\een 
Secondly we have
\beno |W_l(\phi(4\eps D)-\phi(\f14\eps D))f|_{L^2}&=& |W_l(\phi(4\eps D)-\phi(\f14\eps D))\sum_{[\log_2\f1{\eps}]-N_0\le j\le [\log_2\f1{\eps}]+N_0 }\mathfrak{F}_jf|_{L^2}
\\&\lesssim &| \sum_{[\log_2\f1{\eps}]-N_0\le j\le [\log_2\f1{\eps}]+N_0 }\mathfrak{F}_jf|_{L^2_l}+\eps^{}| \sum_{[\log_2\f1{\eps}]-N_0\le j\le [\log_2\f1{\eps}]+N_0 }\mathfrak{F}_jf|_{L^2_{l-1}},
\eeno
which implies that
 \beno \eps^{-2q}|W_l(\phi(4\eps D)-\phi(\f14\eps D))f|_{L^2}^2\lesssim \sum_{j\le ã[\log_2 \f1\epsilon]} 2^{2qj}|W_l\mathfrak{F}_jf|_{L^2}^2+\sum_{j> [\log_2 \f1\epsilon] } \eps^{-2q} |W_l\mathfrak{F}_jf|_{L^2}^2, \eeno
 from which together with \eqref{d}, we conclude the claim. Then
  we complete  the proof to \eqref{func8}.
\end{proof}

\subsection{Lower  and upper bounds for the collision operator $Q^\ep$} 
We will give various lower and upper bounds for the collision operator $Q^\ep$ in weighted Sobolev spaces.

\subsubsection{Lower bounds for the collision operator $Q^\ep$ } 
We begin with a useful proposition which is related to the symbol of the collision operator $Q^\ep$.
\begin{prop}\label{symbol} Suppose $ A^\ep(\xi)\eqdefa \int_{\sigma\in\SS^2} b^\epsilon(\f{\xi}{|\xi|}\cdot \sigma)\min\{ |\xi^-|^2,1\} d\sigma$,
where $\xi^-=(\xi-|\xi|\sigma)/2$. Then we have
$A^\ep(\xi)+1\sim (W^\ep_s(\xi))^2.$
\end{prop}
\begin{proof} By definition, we first get
$A^\ep(\xi)=2\pi\int_0^{\pi/2} \sin\theta b(\cos\theta)\phi(\sin\f{\theta}2/\ep)\min\{|\xi|^2\sin^2(\theta/2),1\} d\theta. $
By change of variable: $t=\sin(\theta/2)$, we have
\beno A^\ep(\xi)&\sim& \int_0^\f12 t^{-1-2s}\phi(t/\ep)\min\{ |\xi|^2t^2,1\}dt 
= |\xi|^{2s} \int_0^{|\xi|/2} t^{-1-2s}\phi(\ep^{-1}t|\xi|^{-1})\min\{ t^2,1\}dt.
\eeno
It is easy to check there exist  constants $\bar{c}_1$ and $\bar{c}_2$ such that $\bar{c}_1<\bar{c}_2$ and
\beno
|\xi|^{2s}\int_{\bar{c}_2\ep|\xi|}^{|\xi|/2} t^{-1-2s} \min\{ t^2,1\}dt\lesssim A^\ep(\xi)\lesssim |\xi|^{2s}\int_{\bar{c}_1\ep|\xi|}^{|\xi|/2} t^{-1-2s} \min\{ t^2,1\} dt.
\eeno
Now we focus on the quantity
$ I(\xi)\eqdefa |\xi|^{2s}\int_{c\ep|\xi|}^{|\xi|/2} t^{-1-2s} \min\{ t^2,1\} dt.$
\begin{enumerate}
	\item For the case of $|\xi|\le2$, we have
	$I(\xi)=|\xi|^{2s} \int_{c\ep|\xi|}^{|\xi|/2} t^{1-2s}  dt\sim (1-s)^{-1}|\xi|^2.$
		\item For the case of $2<|\xi|\le (c\ep)^{-1}$, we have \beno
		I(\xi)=|\xi|^{2s} \big(\int_{c\ep|\xi|}^{1} t^{1-2s}  dt+ \int_{1}^{|\xi|/2} t^{-1-2s}  dt\big)
		\sim (1-s)^{-1}|\xi|^{2s}(1-(c\ep |\xi|)^{2-2s})+|\xi|^{2s}(1-(2|\xi|^{-1})^{2s}). \eeno
		\item For the case of $|\xi|\ge (c\ep)^{-1}$,
		we have
	$ I(\xi)=|\xi|^{2s}   \int_{c
			\ep |\xi|}^{|\xi|/2} t^{-1-2s}  dt \sim \ep^{-2s}. $
\end{enumerate}

Patching together  all the estimates, we arrive at $ A^\ep(\xi)+1\sim I(\xi)+1\sim (W^\ep_s)^2,$
which concludes the desired result.
\end{proof}

Now we can state the coercivity estimate for the  Boltzmann collision operator $Q^\ep$:
\begin{thm}\label{thmlb1} Suppose that the non-negative function $g$ verifies the conditions
\ben\label{lbcondi} |g|_{L^1}\ge c_1, |g|_{L^1_2}+|g|_{L\log L}< c_2.\een
Then for any smooth function $f$, there exist constants $\mathcal{C}_1$ and $\mathcal{C}_2$ depend only on $c_1$ and $c_2$ such that
\ben\label{lbQe}  \lr{-Q^\ep(g,f), f}_v \ge  \mathcal{C}_1(c_1,c_2)\big(\mathcal{E}_{\mu}^{0,\eps}(W_{\gamma/2}f)+|W^\ep_s(D)W_{\gamma/2} f|_{L^2}^2\big)-\mathcal{C}_2(c_1,c_2)|f|^2_{L^2_{\gamma/2}}.\een
\end{thm}

\begin{proof}
It is easy to derive that
$\lr{-Q^\ep(g,f), f}_v=\mathcal{E}^{\gamma,\eps}_{g}(f)-\mathcal{N}^{\gamma,\eps}(f),$
where $\mathcal{E}^{\gamma,\eps}_{g}(f)$ is defined in \eqref{defEg} and $  
\mathcal{N}^{\gamma,\eps}(f)\eqdefa\f12 \int_{\sigma,v_*,v} |v-v_*|^\gamma b^\ep(\cos\theta) g_*(f'^2-f^2) d\sigma dv_*dv.$
We recall that the cancellation lemma(see \cite{ADVW00}) can be stated as follows: if $A(v-v_*,\sigma)=A(|v-v_*|,\cos\theta)$ with $\cos\theta=\f{v-v_*}{|v-v_*|}\cdot \sigma$, then
\ben\label{canclem} 
\int_{\sigma, v} A(v-v_*,\sigma)(f'-f)d\sigma dv=(f*S)(v_*),
\een
where $S(z)=|\SS^{1}|\int_0^{\pi/2} [\cos^{-3}(\theta/2)B(|z|/\cos(\theta/2),\cos\theta)-B(|z|,\cos\theta)]\sin\theta d\theta$. 
It implies  
$| \mathcal{N}^{\gamma,\eps}(f)|\lesssim |g|_{L^1_\gamma}|f|^2_{L^2_{\gamma/2}}.$  Next we   concentrate on the functional $\mathcal{E}^{\gamma,\eps}_{g}(f)$. We begin with  the case $\gamma=0$. From the computation in \cite{ADVW00}, one has\ben\label{lb00}
\mathcal{E}^{0,\eps}_{g}(f)+|f|_{L^2}^2\gtrsim \mathcal{C}(g) \int_{\R^3} (A^\ep(\xi)+1)|\hat{f}(\xi)|^2 d\xi\ge  \mathcal{C}(c_1,c_2)|W^\ep_s(D) f|_{L^2}^2. \een
By Lemma 3.4 in \cite{HE16}, one has
$ \mathcal{E}^{\gamma,\eps}_{g}(f)+|f|^2_{L^2_{\gamma/2}}\ge C(c_1,c_2) \mathcal{E}^{0,\eps}_{\mu}(W_{\gamma/2}f).$
 From this  together with \eqref{lb00}, we complete the proof of the theorem.
\end{proof}

In order to get sharp bounds for $Q^\ep$, we perform the following decomposition: $Q=Q_\delta+Q^\ep_r$ defined by
$Q_\delta(g,h)\eqdefa \int_{v_*,\sigma} B^\ep(|v-v_*|,\sigma)\phi(\f{\sin(\theta/2)}{\delta}) (g_*'h'-g_*h)d\sigma dv_*,$ and $
Q^\ep_r(g,h)\eqdefa \int_{v_*,\sigma} B^\ep(|v-v_*|,\sigma)\big(1-\phi(\f{\sin(\theta/2)}{\delta}) \big)(g_*'h'-g_*h)d\sigma dv_*.$
By the definition of $\phi$,   for $\delta>2\ep$, one has $\phi(\f{\sin(\theta/2)}{\delta})\phi(\f{\sin(\theta/2)}{\eps})=\phi(\f{\sin(\theta/2)}{\delta})$, which implies that 
\beno Q_\delta(g,h)&=& \int_{v_*,\sigma} B(|v-v_*|,\sigma)\phi(\f{\sin(\theta/2)}{\delta}) (g_*'h'-g_*h)d\sigma dv_*. \eeno

Let $\kappa\eqdefa c_1^{\f13}(3\exp\{3c_2/c_1+3\})^{-\f13}$. Then   $Q_\delta$ we has the further decomposition:
$ Q_\delta(g,h)= Q_{\delta,\kappa}^+(g,h)-L^{\delta,\kappa}(g)h+Q^{\kappa}_\delta(g,h), $
where 
$Q_{\delta,\kappa}^+(g,h)\eqdefa \int_{v_*,\sigma} B(|v-v_*|,\sigma)\phi(\f{\sin(\theta/2)}{\delta})\phi(\f{ |v-v_*|}{2\kappa}) g_*'h'd\sigma dv_*$, 
    $L^{\delta,\kappa}(g)h\eqdefa \int_{v_*,\sigma} B(|v-v_*|,\sigma)\phi(\f{\sin(\theta/2)}{\delta})\phi(\f{ |v-v_*|}{2\kappa}) g_*hd\sigma dv_*$ and 
  $Q_{\kappa}^\delta(g,h)\eqdefa\int_{v_*,\sigma} B(|v-v_*|,\sigma)\phi(\f{\sin(\theta/2)}{\delta})(1-\phi(\f{ |v-v_*|}{2\kappa}))\\ (g_*'h'-g_*h)d\sigma dv_*$.
\smallskip

In summary, we have a new decomposition for $Q^\eps$: \ben\label{decomQeps} Q^\ep(g,h)=Q_{\delta,\kappa}^+(g,h)-L^{\delta,\kappa}(g)h+Q_{\kappa}^\delta(g,h)+Q^\ep_r(g,h).\een  In what follows, we will focus on the estimates for     $L^{\delta,\kappa}$ and $Q_{\delta,\kappa}^+$.

\begin{prop}\label{lbLdelta} Let $\delta>2\ep$. Suppose that the non-negative function $g$ verifies the condition \eqref{lbcondi}. Then there exists a constant $\mathcal{C}_3$ depending only on $c_1,c_2$ such that
\beno  L^{\delta,\kappa}(g)\ge \mathcal{C}_3(c_1,c_2) \delta^{-2s}\lr{v}^{\gamma}.\eeno
\end{prop}
\begin{proof}  We first recall that
$ L^{\delta,\kappa}(g)\ge\big(\int_{|v-v_*|\ge \kappa}|v-v_*|^\gamma g_*dv_*\big)\big(\int_{\sigma} b(\cos\theta) \phi(\f{\sin(\theta/2)}{\delta})d\sigma\big)$. It implies that
$ L^{\delta,\kappa}(g)\ge \delta^{-2s} \int_{|v-v_*|\ge \kappa}|v-v_*|^\gamma g_*dv_*$. Let  $M=\exp\{3c_2/c_1+1\}$. 
 
Since $\kappa=c_1^{\f13}(3\exp\{3c_2/c_1+3\})^{-\f13}$, on one hand, it is easy to check that
\beno \int_{|v-v_*|\ge\kappa} |v-v_*|^\gamma g_*dv_*\ge \kappa^{\gamma}(|g|_{L^1}-\kappa^3M-(\log M)^{-1}|g|_{L\log L})\ge \kappa^{\gamma}\f{c_1}3.  \eeno
On the other hand, we can derive that if $R=\max\{\sqrt{2c_2/c_1}, 2\kappa\}$, then
\beno \int_{|v-v_*|\ge\kappa} |v-v_*|^\gamma g_*dv_*\mathrm{1}_{|v|\ge R}&\ge &(\f12|v|)^\gamma\mathrm{1}_{|v|\ge R}\int_{|v_*|\le R/2} g_*dv_*\\
\\&\ge&   \mathrm{1}_{|v|\ge R}(\f12|v|)^\gamma (|g|_{L^1}-R^{-2}|g|_{L^1_2})\ge \f{c_1}2\mathrm{1}_{|v|\ge R}(\f12|v|)^\gamma.\eeno
The desired result follows these two inequalities  .
\end{proof}

For   $Q_{\delta,\kappa}^+$, we will use the Randon transform to capture the smoothing property of the operator. In order to do that, we use $\omega$-representation to rewrite $Q_{\delta,\kappa}^+$ as follows:
\begin{equation}\label{D:gain-term}
Q_{\delta,\kappa}^+(f,g)(v)=\int_{v_*\in\mathbb{R}^3}\int_{\omega\in S^2}f(v')g(v'_*)\tilde{B}_\delta(|v-v_*|,\omega)\phi(\f{|v-v_*|}{2\kappa}) d\Omega(\omega)dv_*
\end{equation}
where
$v'=v-((v-v_*)\cdot\omega)\omega\;,\;v_*'=v_*+((v-v_*)\cdot\omega)\omega.$
Here  $\tilde{B}_\delta$ is of the form
\[
\tilde{B}_\delta(v-v_*,\omega)=|v-v_*|^{\gamma}b_\delta(\cos\theta)
\eqdefa|v-v_*|^{\gamma}\tilde{b}(\cos\theta)\phi(\f{\cos\theta}{\delta})\;,\;\cos\theta=\frac{v-v_*}{|v-v_*|}\cdot\omega,
\]
where the angular function $\tilde{b}$ is defined by
\begin{equation}\label{E:angular}
\tilde{b}(\cos\theta)=4(\cos\theta)b(\frac{v-v_*}{|v-v_*|}\cdot \frac{v'-v'_*}{|v'-v'_*|}).
\end{equation}
 We remark that now the singularity of the cross-section occurs near  $\theta=\pi/2 $. More precisely, one has 
 $\tilde{b}(\cos\theta)\sim (\cos\theta)^{-1-2s}$.

 \medskip
  Before giving the upper bound for $Q_{\delta,\kappa}^+$, we state a crucial lemma:
 
\begin{lem} \label{ubQ+} 
	Let $\mu>0$ be a small number.  Suppose $c_{\mu}(|x|)$ is a 
	positive smooth bump function which equals to $1$ when $|x|>\mu$ and 
	$0$ when $|x|<\mu/2$. Let
$
	Th(x)\eqdefa\int_{S^2} c_{\mu} (|x|)b_\delta(\cos\theta)h(x-(x\cdot\omega)\omega)d\Omega(\omega)$ and $\cos\theta=\lr{x,\omega}/|x|.
$
	Then  \begin{equation}\label{E:T-smooth}
	|Th|_{L^2}\lesssim \delta^{-2-2s}\mu^{-1}|h|_{H^{-1}}.
	\end{equation}
\end{lem}

 \begin{proof}
	We rewrite $T$ as
	\begin{equation}\label{D:T}
	Th(x)=(2\pi)^{-3}\int_{\mathbb{R}^3} e^{ix\cdot\xi}a(x,\xi)\widehat{h}(\xi)d\xi
	\end{equation}
	where 
	$
	a(x,\xi)=\int_{S^2} e^{-i(x\cdot\omega)(\xi\cdot\omega)}c_{\mu}(|x|)b_\delta(\cos\theta) 
	d\Omega(\omega).
	$
	
	We split the operator $T$ to the one restricted to $|x|>1$ and the other 
	to $\mu< |x|< 3/2$ by multiplying it with suitable bump functions.

	{\bf Part I.} $|x|>1$.\par
	The operator  is denoted by $T$ again. 
	We should evaluate $a(x,\xi)$ on three different domains. Thus the operator $T$ 
	is split into three operators accordingly.
	For simplicity of representation, we will denote the function on three domains by $a(x,\xi)$ again.
	
	{\bf Case 1.} $|x|\geq 1,|\xi|\geq 1$. \par
	
	The operator $T$ restricted to this domain is again denoted by $T$ and we 
	wish to show
	\begin{equation}\label{E:T-1}
	|T h|_{L^2}\leq C(\delta)|h|_{H^{-1}}
	\end{equation}
	where $C(\delta)$ is of order $\delta^{-2-2s}$.
	
	The calculation of $a(x,\xi)$ on this domain can be done by the same method as 
	that in~\cite{JC12} which has its origin from~\cite{Lio94}. Hence
	we sketch the calculation of $a(x,\xi)$ and the estimate of $T$ quickly. 
	With these in mind, we can show the 
	estimate still holds for the case 2 by modifying the argument of case 1. This new 
	argument to the case 2 was not seen in~\cite{JC12}.     
	
	In order to prove~\eqref{E:T-1}, we further split the phase space $\{(x,\xi)||x|\geq 1,|\xi|\geq 1\}$ into cones 
	by letting $m\in\mathbb{N}$ and
  $\Gamma_0={\Big \{} (x,\xi)|\;  2\delta \leq \theta_0 \leq {\pi}-2\delta {\Big \}}$,  
	 $\Gamma_{m}={\Big \{} (x,\xi)|\; \pi-\frac{\delta}{2^{m-3}} < \theta_0 \leq \pi-\frac{\delta}{2^{m-1}} {\Big \}}$,
$ \Gamma_{-m}={\Big \{} (x,\xi)|\;  \frac{\delta}{2^{m-1}} \leq \theta_0 < \frac{\delta}{2^{m-3}} {\Big \}}$, where $\theta_0$ is the angle spanned by $x$ and $\xi$.
	
	By a similar calculation  as~\cite{JC12}, using stationary phase formula, we obtain that $a(x,\xi)$ equals 
	\begin{equation}\label{E:a-g0}
	\left\{
	\begin{array}{lll}
	e^{-i\phi_+(x,\xi)}c_{+}(\theta_0) p_+(x,\xi)&+e^{-i\phi_-(x,\xi)}c_{-}(\theta_0)
	p_-(x,\xi),&\;{\rm if}\; 
	(x,\xi)\in\Gamma_0\\
	p_{-\infty}(x,\xi)&+e^{-i\phi_-(x,\xi)}c_{-}(\theta_0)p_-(x,\xi),&\;{\rm if}\; 
	(x,\xi)\in\Gamma_m\\
	e^{-i\phi_+(x,\xi)}c_{+}(\theta_0)p_+(x,\xi)&+p_{-\infty}(x,\xi),&\;{\rm if}\; 
	(x,\xi)\in\Gamma_{-m}
	\end{array}
	\right.
	\end{equation}
	where 
	$
	\phi_{\pm}(x,\xi)=\frac{1}{2}[x\cdot\xi\pm |x||\xi|],
$
	$p_{\pm}(x,\xi)\in S_{1,0}^{-1}$ are in the class of symbol of order $-1$  and 
	$p_{-\infty}\in S^{-\infty}_{1,0}$ is the symbol of the smooth operator. Please note that the coefficients
	$c_{+}(\theta_0),c_{-}(\theta_0)$ and their derivatives with respect to $x,\xi$ 
	are bounded by $C\delta^{-(2+2s)}$. The upper bound of $c_{\pm}(\theta_0)$ and their
	derivatives tends to this order when $\theta_0$ tends to $0$ or $\pi$.  
	
	Absorbing the factor $\lr{\xi}^{-1}$ into $\widehat{h}(\xi)$ and take out 
	$\delta^{-(2+2s)}$, the proof of~\eqref{E:T-1} is then reduced to the proof of $L^2$ boundedness of 
	integral operator
$
	T_{\pm}h(x)=\int_{\mathbb{R}^3} e^{\frac{1}{2}i(x\cdot\xi\mp|x||\xi|)}
	p_{\pm}(x,\xi)\widehat{h}(\xi)d\xi
	$
	on cones $\Gamma_j,j\in\mathbb{Z}$ where $p_{\pm}(x,\xi)$ are symbols of order $0$. 
	We note that the phase functions 
	$\psi_{\pm}(x,\xi)=\frac{i}{2}(x\cdot\xi\mp|x||\xi|)$ of operators satisfy
	the non-degeneracy condition
	\begin{equation}\label{E:non-degeneracy}
	{\Big |}\det \frac{\partial^2 \psi_{\pm}(x,\xi)}{\partial x_j\partial \xi_k} {\Big |} >c>0
	\end{equation}
	on $\Gamma_0$.
	Hence the operator satisfies~\eqref{E:T-1} on the cone $\Gamma_0$ by the Theorem 2.1 of~\cite{RS06}. Its proof relies on the localization of operator $T$, integration by parts and Coltar-Stein lemma. Let $d\in C^{\infty}_{0}(\mathbb{R}^3)$ be a real valued positive function such that $\{d_{k}(x)\}_{k\in\mathbb{Z}^3}$ forms a partition of unity where 
	$d_{k}(x)=d(x-k)$. For example we can decompose the operator $T_+$ as
	\begin{equation}\label{E:decomposition_of_T}
	T_+=\sum_{(j,l)\in{\mathbb Z}^3\times{\mathbb Z}^3} T_{(j,l)}
	\end{equation}
	where $T_{(j,l)}=d_{j}T_+d_{l}$, that is
	$
	T_{(j,l)}u(x)=d_{j}(x)\int e^{i\psi_+(x,\xi)}p_+(x,\xi)d_{l}(\xi)u(\xi)d\xi.
	$
	The adjoint of $T_{j,l}$, denoted by $T^{*}_{j,l}$ is
	 $T^{*}_{(j,l)}v(\xi)=d_{l}(\xi)\int e^{-i\psi_+(y,\xi)}\overline{p_+(y,\xi)}v(y)d_{j}(y)dy.$   The non-degeneracy condition and integration by parts give
	$|T_{(j,l)}T^{*}_{(k,m)}|_{L^2\rightarrow L^2}\leq CA^2\frac{h(l-m)}{1+|j-k|^7},$
	and
	$|T^{*}_{(j,l)}T_{(k,m)}|_{L^2\rightarrow L^2}\leq CA^2\frac{h(j-k)}{1+|l-m|^7}.$
	Thus the operator $T^+$ is $L^2$ bounded on $\Gamma_0$ by Coltar-Stein lemma.

	The estimates of operator $T_{\pm}$ on cones $\Gamma_j$ is based on the estimate
	on $\Gamma_0$.  The fact that constant $c$ in the non-degeneracy 
	condition~\eqref{E:non-degeneracy}  tends to $0$ as $|j|\rightarrow\infty$ means
	the decay rate of the kernel of $T$ on $\Gamma_{|j+1|}$ is one half of that 
	on $\Gamma_{|j|}$. On the other hand,  the span of angle $\theta_0$ on  
	$\Gamma_{|j+1|}$ is one half of that on $\Gamma_{|j|}$.   
	Combine these observations and use the argument in~\cite{JC12},  we can see that 
	the bounds on $\Gamma_j$ form a geometric series as $|j|\rightarrow\infty$ and we conclude the result for this case.

	{\bf Case 2.} $|x|\geq 1, |x||\xi|\geq 2$.\par
	
	As before, we may restrict the operator $T$ on the domain $E=\{|x||\xi|>2\}-\{|x|>1,|\xi|>2\}$ by multiplying a smooth 
	bump function to $a(x,\xi)$. 
	We denote the resulting function  by $a(x,\xi)$ again for simplicity and recognize that it is restricted to the domain $E$. 
	
	We note that if the Fourier variable of a function$h$ is restricted to 
	low frequency $|\xi|<2$, then we have 
	$
	|h|_{H^{-1}}\simeq |h|_{L^2}.
$
	Hence we only have to show 
	$
	|Th|_{L^2}\leq |h|_{L^2}
$
	for $T$ restricted to $E$.
	
	By Plancherel theorem,  it 
	equals to consider $\mathcal{T}$ defined by 
	\begin{equation}\label{E:mathcal_T}
	\mathcal{T}f(x)=\int_{\mathbb{R}^3} e^{ix\cdot\xi}a(x,\xi) f(\xi)d\xi,
	\end{equation} 
	where 
	\begin{equation}\label{E:a_def_2}
	a(x,\xi)=\int_{\omega\in S^2}e^{-i(x\cdot\omega)
		(\xi\cdot\omega)} b_\delta(\cos\theta) d\Omega(\omega).
	\end{equation}
	
	Let $\sum_{i=1}^{\infty}\gamma_{i}(x)$ be a partition of unity on $\{x ||x|> 2\}$  where $\gamma_{i}(x)=\gamma_{0}(2^{-i}x)$ for 
	some nonnegative smooth $\gamma_0$ whose support is in $[1/2,2]$. Then $\sum_{i=0}^{\infty}\gamma_{-i}(\xi)$ is a partition of 
	unity on $\{\xi| |\xi|<2\}$. We decompose the operator $\mathcal{T}$ as
$
	\mathcal{T}=\sum \mathcal{T}_{(j,-l)}\;,\;(j,l)\in
	\{\mathbb{N}\times(\mathbb{N}\cup \{0\}) \}
	$
	where 
	\begin{equation}\label{E:T_mathcal}
	\mathcal{T}_{(j,-l)}f(x)=\int_{\mathbb{R}^3} e^{ix\cdot\xi} a_{(j,-l)}(x,\xi)f(\xi)d\xi
	\end{equation}
	and 
	\begin{equation}\label{E:a_j_l}
	a_{(j,-l)}(x,\xi)=\gamma_{j}(x)\gamma_{-l}(\xi)a(x,\xi).
	\end{equation}
	We note that the condition $(x,\xi)\in E$ implies that 
	$a_{(j,-l)}(x,\xi)=0$ if $j-l<0$. In other words, we must have 
	\begin{equation}\label{E:l_j_relation}
	j\geq l+1.
	\end{equation}
	As the proof of {\bf case 1}, the $L^2$ boundedness of $\mathcal{T}$ can be proved through the 
	estimates of $|\mathcal{T}_{(j,-l)}\mathcal{T}^*_{(k,-m)}|_{L^2\rightarrow L^2}$ and 
	$|\mathcal{T}^*_{(j,-l)}\mathcal{T}_{(k,-m)}|_{L^2\rightarrow L^2}$ where $\mathcal{T}^*_{(k,-m)}$ is the adjoint of $\mathcal{T}_{(-k,m)}$ and $(k,m)\in \{\mathbb{N}\times(\mathbb{N}\cup \{0\}) \},k\geq m+1$.  By symmetry, it suffices to 
	study the latter, i.e. operators of the form
	\begin{equation}\label{E:mathcal_T*_T}
	\mathcal{T}^*_{(j,-l)}\mathcal{T}_{(k,-m)}g(\xi)=\int \mathcal{K}_{(j,-l),(k,-m)}(\xi,\eta)g(\eta)d\eta
	\end{equation}
	where 
	\begin{equation}\label{E:mathcal_K}
	\mathcal{K}_{(j,-l),(k,-m)}(\xi,\eta)=\int e^{ix\cdot(\eta-\xi)}\overline{a}_{(j,-l)}(x,\xi)a_{(k,-m)}(x,\eta)dx.
	\end{equation}
	
	If $|k-j|\geq 2$ then we have $\mathcal{K}_{(j,-l),(k,-m)}(\xi,\eta)=0$ by $\gamma_{j}(x)\gamma_{k}(x)=0$ and~\eqref{E:a_j_l}. 
	Without loss of generality, we assume $k=j+1$.
	Let $x=2^{j}\tilde{x},\xi=2^{-j}\tilde{\xi},\eta=2^{-j}\tilde{\eta}$. We observe that 
	this change of variables does not change the direction of vectors $x,\xi,\eta$, thus it is easy to see that from~\eqref{E:a_def_2} 
	we have
$
	a(x,\xi)=a(\tilde{x},\tilde{\xi}).
	$
	Applying it to~\eqref{E:a_j_l}, we have 
	\begin{equation}\label{E:a_x_xi_tilde}
	\overline{a}_{(j,-l)}(x,\xi)=\overline{a}_{(0,j-l)}(\tilde{x},\tilde{\xi})\;,\;a_{(k,-m)}(x,\eta)=a_{(1,j-m)}(\tilde{x},\tilde{\eta}),
	\end{equation}
	and
	\[
	\mathcal{K}_{(j,-l),(k,-m)}(\xi,\eta)=(2^{j})^3 \mathcal{K}_{(0,j-l),(1,j-m)}(\tilde{\xi},\tilde{\eta}).
	\]
	Let $g(2^{-j}\tilde{\xi})=g_{-j}(\tilde{\xi})$, we have 
	\[
	\begin{split}
	&\mathcal{T}^*_{(j,-l)}\mathcal{T}_{(k,-m)}g(\xi)=\int \mathcal{K}_{(j,-l),(k,-m)}(\xi,\eta)g(\eta)d\eta\\
	&=\int \mathcal{K}_{(0,j-l),(1,j-m)}(\tilde{\xi},\tilde{\eta})g_{-j}(\tilde{\eta})d\tilde{\eta}=\mathcal{T}^*_{(0,j-l)}\mathcal{T}_{(1,j-m)}g_{-j}(\tilde{\xi}).
	\end{split}
	\]
	Since
	$
	|\mathcal{T}^*\mathcal{T}g_{-j}(\tilde{\xi})|_{L^2_{\tilde{\xi}}}=(2^{-j})^{-3/2}|\mathcal{T}^*\mathcal{T}g(\xi)|_{L^2_{\xi}} \;,\; |g_{-j}(\tilde{\eta})|_{L^2_{\tilde{\eta}}}=(2^{-j})^{-3/2}|g(\eta)|_{L^2_{\eta}}
	$,
	we see that\\ $|\mathcal{T}^*_{(j,-l)}\mathcal{T}_{(k,-m)}|_{L^2_{\eta}\rightarrow L^2_{\xi}}=
	|\mathcal{T}^*_{(0,j-l)}\mathcal{T}_{(1,j-m)}|_{L^2_{\tilde{\eta}}\rightarrow L^2_{\tilde{\xi}}}.$ We note that
	the bound of $|\mathcal{T}^*_{(0,j-l)}\mathcal{T}_{(1,j-m)}|_{L^2_{\tilde{\eta}}\rightarrow L^2_{\tilde{\xi}}}$ is determined by the kernel
	of operator
	$\mathcal{K}_{(0,j-l),(1,j-k)}(\tilde{\xi},\tilde{\eta})$.   
	We also note that if we define 
	\[
	\wcT_{(1,j-m)}f(\tilde{x})=\int_{\mathbb{R}^3} e^{i\tilde{x}\cdot\tilde{\xi}}
	a_{(1,j-m)}(\tilde{x},\tilde{\xi})f(\tilde{\xi})d\tilde{\xi}.
	\]
	Then the operator$\wcT^*_{(0,j-l)}\wcT_{(1,j-m)}$ has kernel 
	$\mathcal{K}_{(0,j-l),(1,j-m)}(\tilde{\xi},\tilde{\eta})$. 
	
	The purpose of this change of variables is to give new indices $0,j-l,1,j-m$ which 
	are nonnegative by~\eqref{E:l_j_relation}.
	And it suffices to 
	consider $|\wcT^*_{(\cdot,\cdot)}\wcT_{(\cdot,\cdot)}|_{L^2\rightarrow L^2}$
	with nonnegative indices.
	Thus the operators $\wcT,\wcT^*$ are 
	defined in the domain $\{|\tilde{x}|\geq 1,|\tilde{\xi}|\geq 1\}$ and  their
	kernels $a,\overline{a}$ can be estimated by the method of {\bf Case 1}.
	We note that if $j-l$ or $j-m$ is strictly larger than 1, then the support set
	of $\gamma_{j-l}(\tilde{\xi})$ or $\gamma_{j-m}(\tilde{\xi})$ 
	is larger than that of $d$ function given in~\eqref{E:decomposition_of_T}.
	Since we may take a refine decomposition to $\gamma_{-l}(\xi),\gamma_{-m}(\xi)$
	such that the size of each support is of order $2^{-j}$ and relabel
	it  before change of variable, 
	the role of $\gamma_{(\cdot)}(\tilde{\xi})$ here will be later regarded 
	as that of function $d_{({\cdot})}$ in case 1. Therefore the proof 
	of this case just follows the case 1 and enjoys the same $L^2$ bound.

	{\bf Case 3.} $|x|\geq 1 , |x||\xi|< 3$\par

	Since we have $|\xi|<3$ on this domain, to prove~\eqref{E:T-1} is equivalent to show 
	\begin{equation}\label{E:T_case_3}
	|Th|_{L^2}\leq C|h|_{L^2}.
	\end{equation}
	Recall
$
	Th(x)=\int_{\mathbb{R}^3} e^{ix\cdot\xi}a(x,\xi)\hat{h}(\xi)d\xi
	$ 
	where for $|x|>1$, we have
	\begin{equation}\label{E:a_1}
	a(x,\xi)=\int_{S^2} e^{-i(x\cdot\omega)(\xi\cdot\omega)}b_\delta(\cos\theta) 
	d\Omega(\omega).
	\end{equation}
	By Plancherel theorem, it suffices to consider the  operator $T_1$ as
	\[
	T_1h(x)=\int_{\mathbb{R}^3} e^{i x\cdot\xi}a(x,\xi) h(\xi)d\xi
	=\int_{\mathbb{R}^3} K(x,\xi) h(\xi)d\xi.
	\]
	It is clear that $K$ satisfies
	\begin{equation}\label{E:kernel}
	|K(x,\xi)|\leq C\;,\;|x||\xi|<3.
	\end{equation}
	where $C\simeq (2s)^{-1}(\delta)^{-2s}$ is a uniform constant.

	Let $p(x)=(1+|x|)^{-1}$ and $q(\xi)=|\xi|^2$. For any fixed $|x_0|\geq 1$, 
	using spherically coordinate and~\eqref{E:kernel}, we have 
	\[ 
 \int_{\{|\xi|\leq 3|x_0|^{-1}\}} |K(x_0,\xi)|q(\xi)d\xi\\
	 \leq C_1 \int_0^{3|x_0|^{-1}} r^{-2}\cdot r^2 dr\\
	 \leq C_2|x_0|^{-1}\leq C_3p(x_0).
 	\]
	And for any fixed $|\xi_0|\leq 3$, we have 
	\[ 
	 \int_{\{1\leq |x|\leq 3|\xi_0|^{-1}\}} |K(x,\xi_0)|p(x)dx\\
 \leq C_1 \int_1^{3|\xi_0|^{-1}} r^{-1}\cdot r^2 dr\\
	 \leq C_2|\xi_0|^{-2}= C_2q(\xi_0).
	\]
	By Schur's test, we conclude that $T_1$ is bounded on $L^2$ and hence 
	~\eqref{E:T_case_3} holds.

{\bf Part II.} $\mu<|x|<3/2$.

Let $\psi(x)\in C^{\infty}_0((1,2))$ be a function such that sequence $\psi_j(x)=\psi(2^jx)\;,\;j\in{\mathbb Z} $ forms partition of unity in $\mathbb R$.
Then $c_{\mu}(|x|)=\sum_{j=-\infty}^N c_{\mu,j}(|x|)
=\sum_{j=-\infty}^Nc_{\mu}(|x|)\psi_j(x)$ where $N$ depends on $\mu$.

The estimate of Part I works for the operator $T$ with $a(x,\xi)$ being replaced by 
$c_{\mu,0}(|x|)a(x,\xi)$. Indeed we have 
\begin{equation}\label{E:high}
|Th|_{L^2}\lesssim \delta^{-2-2s}|h|_{\dot{H}^{-1}}.
\end{equation}
when $c_{\mu,0}(|x|)a(x,\xi)$ is further restricted to the domain $1<|x|<2, |x||\xi|>1$ and  
we have 
\begin{equation}\label{E:low}
|Th|_{L^2}\lesssim \delta^{-2-2s}|h|_{L^2}
\end{equation}
when $c_{\mu,0}(|x|)a(x,\xi)$ is  restricted to the domain $1<|x|<2, |x||\xi|<1$.

Now we consider the operator $T$ with $a(x,\xi)$ being replaced by 
$c_{\mu,1}(|x|)a(x,\xi)$. First we consider the domain where 
$c_{\mu,1}(|x|)a(x,\xi)$ is further restricted to the domain $|x||\xi|>2$.  By
scaling argument as the case 2 of part I, enlarge $x$ and shrink $\xi$ by scale 
$2$ and $2^{-1}$ respectively, we see that this operator again enjoys
~\eqref{E:high}. For the operator defined on the domain $|x||\xi|<2$, we can 
obtain~\eqref{E:low} by Schur's test with the bound twice large. Since the 
same argument works for the operator with $a(x,\xi)$ being replaced by 
$c_{\mu,j}(|x|)a(x,\xi)$, sum up   all the estimates and conclude the result.   
\end{proof}

Now we are in a position to give the upper bound for $Q_{\delta,\kappa}^+$.
\begin{thm}\label{thmubQ+}
For smooth functions $g, h$ and $f$ and $a,b\in\R$ with $a+b=\gamma$,  we have
\beno |\lr{Q_{\delta,\kappa}^+(g, h), f}_v|\le C(\kappa) \delta^{-2-2s} (|g|_{L^1}|h|_{L^2_{a}}|f|_{H^{-1}_{b}}+|g|_{L^1_{\gamma}}|h|_{L^2}|f|_{H^{-1}}).\eeno
\end{thm}

\begin{proof} By \eqref{D:gain-term}, we have 
\beno  &&\lr{Q_{\delta,\kappa}^+(g, h), f}_v=\int \phi(\f{|v-v_*|}{2\kappa})|v-v_*|^\gamma b_\delta(\cos\theta) g_*'h'fd\Omega(\omega)dv_*dv\\
 &&=\sum_{k\ge-1} 2^{k\gamma} \int   g_* h  [ \tau_{-v_*}\circ T \circ \tau_{v_*} f(v)] dv_*dv\eqdefa\sum_{k\ge-1} 2^{k\gamma} \mathcal{I}_k(g,h,f).
\eeno
where $(\tau_{v_*}h)(v)=h(v-v_*)$ and $Th(x)=\int_{S^2}   c_{\kappa}(|x|)b_\delta(\cos\theta)h(x-(x\cdot\omega)\omega)d\Omega(\omega)$ with
\, $c_{\kappa}(|x|)=2^{-k\gamma}|x|^{\gamma}\psi_k(|x|)\phi(|x|/(2\kappa))$.

Thanks to Lemma \ref{ubQ+}, we have
\beno  |\mathcal{I}_k(g,h,f)|&\lesssim& |g|_{L^1}|h|_{L^2}\sup_{v_*}|\tau_{-v_*}\circ T \circ \tau_{v_*} f|_{L^2}\lesssim\delta^{-2-2s} |g|_{L^1}|h|_{L^2}\sup_{v_*}|\tau_{v_*} f|_{H^{-1}} \\&\lesssim&\delta^{-2-2s} |g|_{L^1}|h|_{L^2}|f|_{H^{-1}}.\eeno
Suppose  $|v_*|\sim 2^j$ and $|v-v_*|\sim 2^k$. Then thanks to the fact $|v-v_*|\sim |v'-v_*|$, we have
\begin{itemize}
\item If $j\le k-N_0$, then  $|v|\sim |v'|\sim 2^k$;
\item If $j\ge k+N_0$, then $|v|\sim |v'|\sim 2^j$;
\item If $|j-k|< N_0$, then $|v|\le 2^{k+N_0}, |v'|\le 2^{k+N_0}.$
\end{itemize}
Due to this observation, we have the following decomposition:
\beno&& \lr{Q_{\delta,\kappa}^+(g, h), f}_v= \sum_{j\ge-1}\lr{Q_{\delta,\kappa}^+(\mathcal{P}_jg, h), f}_v \notag
=\sum_{k\ge N_0-1}2^{k\gamma}\mathcal{I}_k(\mathcal{U}_{k-N_0} g, \tilde{\mathcal{P}}_kh, \tilde{\mathcal{P}}_kf ) \\&&\qquad+
\sum_{j\ge k+N_0}2^{k\gamma}\mathcal{I}_k(\mathcal{P}_{j} g, \tilde{\mathcal{P}}_jh, \tilde{\mathcal{P}}_jf )+\sum_{|j-k|\le N_0}2^{k\gamma}\mathcal{I}_k( \mathcal{P}_{j} g, \mathcal{U}_{k+N_0}h, \mathcal{U}_{k+N_0}f), \eeno
which together with Theorem \ref{baslem3} imply that 
\beno | \lr{Q_{\delta,\kappa}^+(g, h), f}_v|&\lesssim& \delta^{-2-2s} |g|_{L^1}|h|_{L^2_{a}}\big[(\sum_{j\ge-1}2^{2jb}|\tilde{\mathcal{P}}_j f|^2_{H^{-1}})^{1/2}+|g|_{L^1_{\gamma}}|h|_{L^2}\sup_{k}|\mathcal{U}_{k+N_0}f|_{H^{-1}}\big]\\
&\lesssim&\delta^{-2-2s} (|g|_{L^1}|h|_{L^2_{a}}|f|_{H^{-1}_{b}}+|g|_{L^1_{\gamma}}|h|_{L^2_{}}|f|_{H^{-1}}). \eeno We complete the proof of the theorem.
\end{proof}

Combining  the previous results, we arrive at:

\begin{thm}\label{thmlb2} Suppose that the non-negative function $g$ verifies the conditions
\eqref{lbcondi} .
Then for any smooth function $f$, there exist constants $\mathcal{C}_i(i=4,5,6,7)$  depending only on $c_1$ and $c_2$ such that
\ben\label{lbQe2} (i).\, \lr{-Q^\ep(g,f), f}_v& \ge& \f13 \mathcal{E}^{\gamma,\eps}_{g}(f)+\mathcal{C}_4(c_1,c_2)( \mathcal{E}^0_{\mu}(W_{\gamma/2}f)+ |W^\ep_s(D)W_{\gamma/2} f|_{L^2}^2)+ \mathcal{C}_5(c_1,c_2)\delta^{-2s}|f|^2_{L^2_{\gamma/2}}\nonumber\\&&-
\mathcal{C}_6(c_1,c_2)\delta^{-6-6s}|f|_{L^1_{\gamma/2}}^2;\\
(ii).\,  \lr{-Q^\ep(g,f), f}_v& \ge& \f13 \mathcal{E}^{\gamma,\eps}_{g}(f)+\mathcal{C}_4(c_1,c_2)( \mathcal{E}^0_{\mu}(W_{\gamma/2}f)+ |W^\ep_s(D)W_{\gamma/2} f|_{L^2}^2)+ \mathcal{C}_5(c_1,c_2)\delta^{-2s}|f|^2_{L^2_{\gamma/2}}\nonumber\\&&-
\mathcal{C}_7(c_1,c_2)\delta^{-4-2s}|f|_{H^{-1}_{\gamma/2}}^2, \nonumber\een  where $2\epsilon<\delta<(\f{\mathcal{C}_3}{4\mathcal{C}_2})^{\f1{2s}}$.
\end{thm}

\begin{proof}   By \eqref{decomQeps}, we observe that
\beno &&\lr{-Q^\ep(g,f), f}_v=\f13\lr{-Q^\ep(g,f), f}_v+ \f13\lr{-Q^\ep(g,f), f}_v+\f13 \lr{ L^{\delta,\kappa}(g)f,f }_v\\&&\qquad+\f13\lr{-Q_{\delta,\kappa}^+(g,f),f}_v+\f13\lr{-Q_{\kappa}^\delta(g,f),f}_v+\f13\lr{-Q^\ep_r(g,f),f}_v\eqdefa\sum_{k=1}^6\mathcal{I}_k.
\eeno
For $\mathcal{I}_1$ and $ \mathcal{I}_2$, we apply Theorem \ref{thmlb1} directly. For $\mathcal{I}_5,\mathcal{I}_6$, by the proof of Theorem \ref{thmlb1}, we only need to use   \eqref{canclem} to give the upper bounds. We  have
\[ \mathcal{I}_1+ \mathcal{I}_2+\mathcal{I}_5+\mathcal{I}_6\ge \f13 \mathcal{E}^{\gamma,\eps}_{g}(f)+\f13 \mathcal{C}_1(c_1,c_2)(\mathcal{E}^0_{\mu}(W_{\gamma/2}f)+|W^\ep_s(D)W_{\gamma/2}f|_{L^2}^2)-\f43\mathcal{C}_2(c_1,c_2)|f|^2_{L^2_{\gamma/2}}.\]
Thanks to Proposition \ref{lbLdelta} and Theorem \ref{thmubQ+}, we have
\[\mathcal{I}_3+\mathcal{I}_4\ge \mathcal{C}_3(c_1,c_2)\delta^{-2s}|f|^2_{L^2_{\gamma/2}}-\delta^{-2-2s} \mathcal{C}_4(c_1,c_2)|f|_{L^2_{\gamma/2}}|f|_{H^{-1}_{\gamma/2}}. \]
Use the inequality
$ |f|_{L^2_{\gamma/2}}|f|_{H^{-1}_{\gamma/2}}\lesssim |f|_{L^2_{\gamma/2}}|f|_{L^2_{\gamma/2}}^{\f13}|f|^{\f23}_{L^1_{\gamma/2}}\lesssim \eta |f|_{L^2_{\gamma/2}}^2+\eta^{-2}|f|^{2}_{L^1_{\gamma/2}}$, then 
 we  are led to the desired result by putting together all the estimates.
\end{proof}

\subsubsection{Upper  bounds for the collision operator $Q^\epsilon$} In this subsection, we will give the upper bounds for the Boltzmann collision operator $Q^\epsilon$.

\begin{thm}\label{ubofQe}
Suppose that $g, h$ and $f$ are smooth functions. Then
\begin{enumerate}
 \item For $a,b\ge0$ with $a+b=\gamma$, 
$
|\lr{Q^\ep(g,h), f}|\lesssim |g|_{L^1_{\gamma+2s}}|W^\ep_s(D) W_{a+2s}h|_{L^2}|W^\ep_s(D) W_{b}f|_{L^2}.$
\item $ |Q^\eps(g,h)|_{L^2_2}\lesssim |g|_{L^1_{\gamma+2s+2}}|W^\eps_{2s}(D)h|_{L^2_{\gamma+2s+2}}. $
\item  If $2s\le1$, $
|\lr{Q^\ep(g,h), f}|\lesssim |g|_{L^1_{\gamma+2s}}| W_{\gamma+2s}h|_{L^2}|f|_{H^{2s}}.
$ If $2s>1$, then for $\eta>0$,
$ |\lr{Q^\ep(g,h), f}|\lesssim  \big((\eta+|g|_{L^1_{\gamma+2s}}\ep^{1-s})|W^\ep_s(D) W_{\gamma+2s}h|_{L^2}+ \eta^{-\f{2s-1}{1-s}}|g|_{L^1_{\gamma+2s}}^{\f{2s-1}{1-s}}|h|_{L^2_{\gamma+2s}}\big)|f|_{H^1}.$
\end{enumerate}
\end{thm}
\begin{proof} (i). {\it Proof of (1)}:
From Theorem \ref{thmub}, the upper bound for the collision operator with cutoff in \cite{mv} and the decomposition that
$\lr{Q^\ep(g,h), f}_v=\lr{Q^\ep(g,\lf{h}+\hf{h}),\lf{ f}+\hf{f}}_v, $
 we have
\beno &&|\lr{Q^\ep(g,\lf{h}),\lf{ f}}_v|\lesssim |g|_{L^1_{\gamma+2s}}|\lf{h}|_{H^s_{a+2s}}|\lf{f}|_{H^s_{b}},\,
 |\lr{Q^\ep(g,\lf{h}),\hf{ f}}_v|\lesssim |g|_{L^1_{\gamma+2s}}|\lf{h}|_{H^{2s}_{a+2s}}|\hf{f}|_{L^2_{b}},\\
 && |\lr{Q^\ep(g,\hf{h}),\lf{ f}}_v|\lesssim |g|_{L^1_{\gamma+2s}}|\hf{h}|_{L^2_{a+2s}}|\lf{f}|_{H^{2s}_{b}},\,
    |\lr{Q^\ep(g,\hf{h}),\hf{ f}}_v|\lesssim  \ep^{-2s}|g|_{L^1_{\gamma}}|\hf{h}|_{L^2_{a}}|\hf{f}|_{L^2_{b}}.\eeno
Thanks to Lemma \ref{func}, we get that
$|\lf{h}|_{H^{2s}_{a+2s}}\lesssim \ep^{-s} |\lf{h}|_{H^s_{a+2s}},$ which implies
\beno |\lr{Q^\ep(g,h), f}_v|&\lesssim& |g|_{L^1_{\gamma+2s}}\big(|\lf{h}|_{H^s_{a+2s}}+\ep^{-s}|\lf{h}|_{L^2_{a+2s}}\big)\big(|\lf{f}|_{H^s_{b}}+\ep^{-s}|\lf{f}|_{L^2_{b}}\big)\\ &\lesssim& |g|_{L^1_{\gamma+2s}}|W^\ep_s(D) W_{a+2s}h|_{L^2}|W^\ep_s(D) W_{b}f|_{L^2}.
\eeno It ends the proof of the first inequality.

(ii). {\it Proof of (2)}: We prove it by duality. By Theorem \ref{thmub} and the decomposition 
$ \lr{Q^\eps(g,h),f}_v=\lr{Q^\eps(g,h_\phi+h^\phi),f}_v,$  we get
\beno |\lr{Q^\eps(g,h_\phi ),f}_v|\lesssim  |g|_{L^1_{\gamma+2s+2}}|h_\phi|_{H^{2s}_{\gamma+2s+2}}|f|_{L^2_{-2}},\,
|\lr{Q^\eps(g,h^\phi ),f}_v|\lesssim \eps^{-2s} |g|_{L^1_{\gamma+2}}|h^\phi|_{L^2_{\gamma+2}}|f|_{L^2_{-2}},
 \eeno
which yields
$ |\lr{Q^\eps(g,h),f}_v|\lesssim |g|_{L^1_{\gamma+2s+2}}|W^\eps_{2s}(D)h|_{H^{2s}_{\gamma+2s+2}}|f|_{L^2_{-2}}.$

(iii). {\it Proof of (3)}:  The first inequality follows Theorem \ref{thmub}.
Thanks to Lemma \ref{func}, we deduce that if $2s>1$,
   \beno |\lr{Q^\ep(g,\lf{h}),\lf{ f}}_v|&\lesssim& |g|_{L^1_{\gamma+2s}}|\lf{h}|_{H^{2s-1}_{\gamma+2s}}|\lf{f}|_{H^1},\\
 |\lr{Q^\ep(g,\lf{h}),\hf{ f}}_v|&\lesssim& |g|_{L^1_{\gamma+2s}}|\lf{h}|_{H^{2s}_{\gamma+2s}}|\hf{f}|_{L^2_{}}\lesssim |g|_{L^1_{\gamma+2s}}|\lf{h}|_{H^{2s-1}_{\gamma+2s}}|\hf{f}|_{H^1},\\
  |\lr{Q^\ep(g,\hf{h}),\lf{ f}}_v|&\lesssim &|g|_{L^1_{\gamma+2s}}|\hf{h}|_{L^2_{\gamma+2s}}|\lf{f}|_{H^{2s}}\lesssim |g|_{L^1_{\gamma+2s}}\big(\epsilon^{-(2s-1)}|\hf{h}|_{L^2_{\gamma+2s}}\big)|\lf{f}|_{H^{1}}\\
\mbox{and} \quad   |\lr{Q^\ep(g,\hf{h}),\hf{ f}}_v|&\lesssim& \ep^{-2s}|g|_{L^1_{\gamma}}|\hf{h}|_{L^2_{\gamma}}|\hf{f}|_{L^2_{}}\lesssim |g|_{L^1_{\gamma}}\big(\ep^{-(2s-1)}|\hf{h}|_{L^2_{\gamma}}\big)|\hf{f}|_{H^1}.\eeno
 From this, we obtain that
 $ |\lr{Q^\ep(g,h), f}|\lesssim |g|_{L^1_{\gamma+2s}}|W^\ep_{2s-1}(D) W_{\gamma+2s}h|_{L^2}|f|_{H^1}.$
 By interpolation and Lemma \ref{func},  we have  \beno &&|W^\ep_{2s-1}(D) W_{\gamma+2s}h|_{L^2} \sim |\lf{h}|_{H^{2s-1}_{\gamma+2s}}+\ep^{-(2s-1)}|\hf{h}|_{L^2_{\gamma+2s}}
 \\&&\lesssim \eta^{-\f{2s-1}{1-s}}|\lf{h}|_{L^2_{\gamma+2s}}+\eta|\lf{h}|_{H^{s}_{\gamma+2s}}+\ep^{1-s} \ep^{-s}|\hf{h}|_{L^2_{\gamma+2s}}\lesssim  (\eta+\ep^{1-s})|W^\ep_s(D) W_{\gamma+2s}h|_{L^2}+ \eta^{-\f{2s-1}{1-s}}|h|_{L^2_{\gamma+2s}}.
 \eeno We get the desired result  by choosing $\eta:=\eta|g|_{L^1_{\gamma+2s}}^{-1}$. It ends the proof of the lemma.
\end{proof}

\subsection{Commutator estimates for $Q^\ep$} In this subsection, we want to give two types of estimates on the commutators.

\subsubsection{Commutator estimates for $Q^\ep$ with the  weight function $W_l$} We begin with

\begin{lem}\label{commforweight} Suppose $l>3+\gamma+s$ and $a, b\ge0$ with $a+b=\gamma$. Then for smooth functions $g, h$ and $f$, we have
 \beno  \big| \lr{Q^\eps(g, h)W_l-Q^\eps(g, hW_l) ,f}_v \big| &\lesssim& |g|_{L^1_{\gamma+2s}}|h|_{L^2_{l+a}}|f|_{L^2_b}+|g|_{L^1_{\gamma+2s}}|h|_{L^2_{l+a+2s-1}}|W^\ep_s(D)W_bf|_{L^2} \\&&+|g|_{L^2_{l+a}}|h|_{L^1_\gamma}|f|_{L^2_b}. \eeno
 Moreover, for $g\ge0$,   we   have
 \beno
 \big| \lr{Q^\eps(g, h)W_l-Q^\eps(g, hW_l) ,f}_v \big| &\lesssim&\eta \mathcal{E}^{\gamma,\eps}_{g}(f)+\eta^{-1} (|g|_{L^1_{\gamma+2s}}|h|_{L^2_{l-1}}^2+|g|_{L^1}|h|^2_{L^2_{l+\gamma/2}})\\&&+|g|_{L^1_{\gamma+2s}}|h|_{L^2_{l}}|f|_{L^2}+|g|_{L^1_\gamma}|h|_{L^2_{l+a}}|f|_{L^2_b}+|g|_{L^2_{l+a}}|h|_{L^1_\gamma}|f|_{L^2_b}.
 \eeno
\end{lem}
\begin{proof}   We perform  the following decomposition that
\beno &&\big| \lr{Q^\eps(g, h)W_l-Q^\eps(g, hW_l) ,f}_v |= \int |v-v_*|^\gamma b^\eps g_*hf'\big((W_l)'-W_l\big) (\mathrm{1}_{|v|\ge 4|v_*|}+\psi(|v-v_*|\sin(\theta/2))\\&&\qquad\times\mathrm{1}_{|v|\le 4|v_*|})d\sigma dv_*dv+\int |v-v_*|^\gamma b^\eps g_*hf'\big((W_l)'-W_l\big) (1-\psi(|v-v_*|\sin\f{\theta}2))\mathrm{1}_{|v|\le 4|v_*|}d\sigma dv_*dv\\
&&\qquad\eqdefa I+II.\eeno
 Let us give the estimates term by term. Before that, we remark that the following estimate will be frequently used in the proof: if $\kappa(v)=v+(1-\kappa)(v'-v)$ with $\kappa\in[0,1]$, then
 \ben\label{changev} \mathrm{1}_{|v|\ge 4|v_*|} |v|\sim \mathrm{1}_{|v|\ge 4|v_*|} |v-v_*|\sim  \mathrm{1}_{|v|\ge 4|v_*|} |\kappa(v)|.\een

{\it Step 1: Estimate of $I$.} Observe that 
$ (W_l)'-W_l=(\na W_l)\cdot(v'-v)+\f12\int_0^1 (1-\kappa)(\na^2 W_l)(\kappa(v)): (v'-v)\otimes(v'-v)d\kappa,$
where $\kappa(v)=v+(1-\kappa)(v'-v)$. Then $I$ has the further decomposition that $I=I_1+I_2$ where  
\beno I_1&=& \int |v-v_*|^\gamma b^\eps g_*h f'(\na W_l)\cdot (v'-v)  (\mathrm{1}_{|v|\ge 4|v_*|}+\psi(|v-v_*|\sin\f{\theta}2)\mathrm{1}_{|v|\le 4|v_*|}) d\sigma dv_* dv, \\ I_2&=&\f12\int(1-\kappa) |v-v_*|^\gamma b^\eps g_*h f'(\na^2 W_l)(\kappa(v)): (v'-v)\otimes(v'-v)  (\mathrm{1}_{|v|\ge 4|v_*|}\\&&+\psi(|v-v_*|\sin\f{\theta}2)\mathrm{1}_{|v|\le 4|v_*|}) d\sigma dv_* dvd\kappa 
\eqdefa I_2^1+I_2^2. \eeno

{\it Step1.1: Estimates of $I_1$.}
We will give two types of the estimates for $I_1$ due to the different property of the function $g$.
\smallskip

 (i). For the case of $g\ge0$, we see that
\beno  I_1&=&\int |v-v_*|^\gamma b^\eps g_*h (f'-f)(\na W_l)\cdot (v'-v) (\mathrm{1}_{|v|\ge 4|v_*|}+\psi(|v-v_*|\sin\f{\theta}2)\mathrm{1}_{|v|\le 4|v_*|}) d\sigma dv_* dv\\
&&+\int |v-v_*|^\gamma b^\eps g_*h f(\na W_l)\cdot (v'-v)  (\mathrm{1}_{|v|\ge 4|v_*|}+\psi(|v-v_*|\sin\f{\theta}2)\mathrm{1}_{|v|\le 4|v_*|})d\sigma dv_* dv\\
&\eqdefa& I_{1,1}+I_{1,2}.\eeno

{\it \underline{Estimate of $I_{1,1}$.}} By Cauchy-Schwartz inequality and \eqref{changev},  one has
\beno  |I_{1,1}|&\le& \big(\mathcal{E}^{\gamma,\eps}_{g}(f)\big)^{1/2}\big(\int |v-v_*|^\gamma b^\eps g_*h^2|(\na W_l)\cdot (v'-v)|^2  (\mathrm{1}_{|v|\ge 4|v_*|}+\psi(|v-v_*|\sin\f{\theta}2)\mathrm{1}_{|v|\le 4|v_*|})^2 d\sigma dv_* dv\big)^{1/2}\\
&\lesssim &\eta \mathcal{E}^{\gamma,\eps}_{g}(f)+\eta^{-1} (|g|_{L^1}|h|^2_{L^2_{l+\gamma/2}}+|g|_{L^1_{\gamma+2s}}|h|_{L^2_{l-1}}^2).\eeno

{\it \underline{ Estimate of $I_{1,2}$.}} By the facts that $|v-v_*|\sin\f{\theta}2=|v-v'|$ and 
\beno &&\int_{\SS^2}  b(
\frac{v-v_*}{|v-v_*|}\cdot \sigma )(v-v')\psi(|v-v'|)d\sigma=\int_{\SS^2}  b(  \frac{v-v_*}{|v-v_*|}\cdot\sigma )
\f{1-\langle
	\frac{v-v_*}{|v-v_*|},\sigma\rangle}{2} \psi (|v-v'|)d\sigma
(v-v_*),
\eeno
 we derive that
$ |I_{1,2}|\lesssim |g|_{L^1_{\gamma+2s}}|h|_{L^2_{l-1}}|f|_{L^2}+|g|_{L^1}|h|_{L^2_{l+a}}|f|_{L^2_b}. $

Now putting together these two estimates, we conclude that 
\[ |I_1|\lesssim  \eta \mathcal{E}^{\gamma,\eps}_{g}(f)+\eta^{-1} (|g|_{L^1}|h|^2_{L^2_{l+\gamma/2}}+|g|_{L^1_{\gamma+2s}}|h|_{L^2_{l-1}}^2)+ |g|_{L^1_{\gamma+2s}}|h|_{L^2_{l-1}}|f|_{L^2}+|g|_{L^1}|h|_{L^2_{l+a}}|f|_{L^2_b}.\]

(ii). For the general case,  $I_1$ is decomposed into tow parts:
\beno  I_1&=&\int |v-v_*|^\gamma b^\eps g_*h (f_\phi)'(\na W_l)\cdot (v'-v) (\mathrm{1}_{|v|\ge 4|v_*|}+\psi(|v-v_*|\sin\f{\theta}2)\mathrm{1}_{|v|\le 4|v_*|}) d\sigma dv_* dv\\&&+\int |v-v_*|^\gamma b^\eps g_*h (f^\phi)'(\na W_l)\cdot (v'-v)  (\mathrm{1}_{|v|\ge 4|v_*|}+\psi(|v-v_*|\sin\f{\theta}2)\mathrm{1}_{|v|\le 4|v_*|})d\sigma dv_* dv\\
&\eqdefa& I_{1,3}+I_{1,4}.\eeno

{\it \underline{ Estimate of $I_{1,4}$.}} Thanks to \eqref{changev}, one has \beno&& |v-v_*|^\gamma  |(\na W_l)\cdot (v'-v)|(\mathrm{1}_{|v|\ge 4|v_*|}+\psi(|v-v_*|\sin\f{\theta}2)\mathrm{1}_{|v|\le 4|v_*|})\\&&\lesssim  \mathrm{1}_{|v|\ge 4|v_*|} W_{l+a}(v)W_{b}(v') \theta+\mathrm{1}_{|v|\le 4|v_*|} \psi(|v-v_*|\sin\f{\theta}2)W_{l-1}|v-v_*|^{\gamma}(|v-v_*|\sin\theta).\eeno
Then by Cauchy-Schwartz inequality, we have
\beno  |I_{1,4}| &\lesssim& (|\log\eps|\mathrm{1}_{2s\le1}+\eps^{1-2s}\mathrm{1}_{2s>1})(|g|_{L^1_{\gamma+\min\{1,2s\}}}|h|_{L^2_{l-1}}|f^\phi|_{L^2}+ |g|_{L^1_\gamma}|h|_{L^2_{l+a}}|f^\phi|_{L^2_b}).\eeno

 {\it \underline{ Estimate of $I_{1,3}$.}} We split $I_{1,3}$ into two parts: $I_{1,3}=A_1+A_2$ where
\beno
A_1&\eqdefa& \int |v-v_*|^\gamma b^\eps g_*h( (f_\phi)'-f_\phi)(\na W_l)\cdot (v'-v) (\mathrm{1}_{|v|\ge 4|v_*|}+\psi(|v-v_*|\sin\f{\theta}2)\mathrm{1}_{|v|\le 4|v_*|}) d\sigma dv_*dv,\\
A_2&\eqdefa& \int |v-v_*|^\gamma b^\eps g_*h f_\phi(\na W_l)\cdot (v'-v) (\mathrm{1}_{|v|\ge 4|v_*|}+\psi(|v-v_*|\sin\f{\theta}2)\mathrm{1}_{|v|\le 4|v_*|}) d\sigma dv_* dv.
\eeno
We remark that the structure of   $A_2$ is very similar to that of  $I_{1,2}$. Thus we have
\beno |A_2|&\lesssim& |g|_{L^1_{ \gamma+2s}}|h|_{L^2_{l-1}}|f_\phi|_{L^2}+|g|_{L^1}|h|_{L^2_{l+a}}|f_\phi|_{L^2_b}.
\eeno
In the next we will give the estimate to $A_1$. Observe that
\beno A_1&=&\sum_{j\ge-1} \int |v-v_*|^\gamma b^\eps g_*h( (f_\phi)_j'-(f_\phi)_j)(\na W_l)\cdot (v'-v)(\mathrm{1}_{|v|\ge 4|v_*|}+\psi(|v-v_*|\sin\f{\theta}2)\mathrm{1}_{|v|\le 4|v_*|})  d\sigma dv_*dv
\\&=&\sum_{j\ge-1}  \bigg(\int |v-v_*|^\gamma b^\eps g_*h( (f_\phi)_j'-(f_\phi)_j)(\na W_l)\cdot (v'-v)\mathrm{1}_{\theta\le 2^{-j}|v-v_*|^{-1}}(\mathrm{1}_{|v|\ge 4|v_*|}\\
&&+\psi(|v-v_*|\sin\f{\theta}2)\mathrm{1}_{|v|\le 4|v_*|})  d\sigma dv_*dv+\int |v-v_*|^\gamma b^\eps g_*h( (f_\phi)_j'-(f_\phi)_j)(\na W_l)\cdot (v'-v) \mathrm{1}_{\theta>2^{-j}|v-v_*|^{-1}}\\
&&\times(\mathrm{1}_{|v|\ge 4|v_*|}+\psi(|v-v_*|\sin\f{\theta}2)\mathrm{1}_{|v|\le 4|v_*|}) d\sigma dv_*dv\bigg)
\eqdefa \sum_{j\ge-1} (A_{1,1}^j+A_{1,2}^j).
\eeno

For  $A_{1,1}^j$, due to  the fact $ |(\na W_l)\cdot (v'-v)|\lesssim W_{l-1}(v)|v-v_*|\theta$ and the mean value theorem, we have
\beno  |A_{1,1}^j|&\lesssim& \int |v-v_*|^{\gamma+2}\theta^2 b^\eps |g_*||W_{l-1}h||\big(\na (f_\phi)_j\big)(\kappa(v))| \mathrm{1}_{\theta\le 2^{-j}|v-v_*|^{-1}} d\sigma dv_*dvd\kappa\\
&\lesssim& \int |v-v_*|^{2-2s}(\mathrm{1}_{4|v_*|\le |v|} W_{2s+a}(v)W_{b}(\kappa(v))+ \mathrm{1}_{4|v_*|\ge |v|}  W_{\gamma+2s}(v_*))\theta^2 b^\eps |g_*|
\\&&\times|W_{l-1}h||\big(\na (f_\phi)_j\big)(\kappa(v))| \mathrm{1}_{\theta\le 2^{-j}|v-v_*|^{-1}} d\sigma dv_*dvd\kappa\\
&\lesssim&\big(\int  |v-v_*|^{2-2s} \theta^2 b^\eps\big(|g_*||W_{l+a+2s-1}h|^2  +|gW_{\gamma+2s}||W_{l-1}h|_{L^2}^2  \big)  \mathrm{1}_{\theta\le 2^{-j}|v-v_*|^{-1}}  d\sigma dv_*dvd\kappa\big)^{1/2} \\&&\times\big(\int |v-v_*|^{2-2s} \theta^2 b^\eps \big(|g_*||\big(\na (f_\phi)_jW_{b}\big)(\kappa(v))|^2 +|gW_{\gamma+2s}||\big(\na (f_\phi)_j\big)(\kappa(v))|^2 \big) \\&&\times \mathrm{1}_{\theta\le 2^{-j}|v-v_*|^{-1}} d\sigma dv_*dvd\kappa\big)^{1/2}
\lesssim 2^{(2s-2)j}|g|_{L^1_{\gamma+2s}}|h|_{L^2_{l+a+2s-1}} | (f_\phi)_j |_{H^{1}_b}.\eeno

For  $A_{1,2}^j$, we first have  
\beno && |A_{1,2}^j| \lesssim \int  |v-v_*|^{\gamma} b^\eps |g_*||W_{l-1}h|( |(f_\phi)_j'|+|(f_\phi)_j)|) |v-v_*|\sin\theta \mathrm{1}_{|v_*|\le |v|} \mathrm{1}_{\theta\ge 2^{-j}|v-v_*|^{-1}}d\sigma dv_*dv \\&&
+\int  |v-v_*|^{\gamma} b^\eps |g_*||W_{l-1}h|( |(f_\phi)_j'|+|(f_\phi)_j)|) (|v-v_*|\sin\theta )\mathrm{1}_{|v|\le 4|v_*|} \psi(|v-v_*|\sin\f{\theta}2) \mathrm{1}_{\theta\ge 2^{-j}|v-v_*|^{-1}}d\sigma dv_*dv. \eeno
By Cauchy-Schwartz inequality and  the upper bounds for the collision operator in \cite{mv}, one has\beno
&&|A_{1,2}^j|\lesssim  \big(\int  (\mathrm{1}_{2s>1}|v-v_*|^{2s+1}+\mathrm{1}_{2s\le1} |\log |v-v_*||) b^\eps |v-v_*|^{2a}|g_*||W_{l-1}h|^2\sin\theta  \mathrm{1}_{|v_*|\le |v|}\\&&\quad\times\mathrm{1}_{\theta\ge 2^{-j}|v-v_*|^{-1}} d\sigma dv_*dv\big)^{\f12}
\big(\int  (\mathrm{1}_{2s>1}|v-v_*|^{1-2s}+\mathrm{1}_{2s\le1}|\log |v-v_*||^{-1}) b^\eps |v-v_*|^{2b} |g_*| ( |(f_\phi)_j'|^2\\&&\quad+|(f_\phi)_j)|^2)  \sin\theta \mathrm{1}_{|v_*|\le |v|}\mathrm{1}_{\theta\ge 2^{-j}|v-v_*|^{-1}} d\sigma dv_*dv\big)^{\f12}
 +\big(\int  |v-v_*|^{\gamma} b^\eps |g_*||W_{l-1}h|^2 (\mathrm{1}_{2s>1}|v-v_*|\sin\theta \\&&\quad+\mathrm{1}_{2s\le1} |v-v_*|^{2s}\sin\theta^{2s}) \mathrm{1}_{|v|\le 4|v_*|} \mathrm{1}_{|v-v_*|^{-1}\theta\ge 2^{-j}|v-v_*|^{-1}}d\sigma dv_*dv\big)^{\f12}
  \big(\int  |v-v_*|^{\gamma} b^\eps |g_*| ( |(f_\phi)_j'|^2+|(f_\phi)_j)|^2)  \\&&\quad\times(\mathrm{1}_{2s>1}|v-v_*|\sin\theta+\mathrm{1}_{2s\le1} |v-v_*|^{2s}\sin\theta^{2s})\mathrm{1}_{|v|\le 4|v_*|} \mathrm{1}_{|v-v_*|^{-1}\theta\ge 2^{-j}|v-v_*|^{-1}}d\sigma dv_*dv\big)^{\f12}
\\ &&\quad\lesssim \mathrm{1}_{2s>1}2^{(2s-1)j}(|g|_{L^1_{\gamma}}|h|_{L^2_{l+a+2s-1}}|(f_\phi)_j|_{L^2_b} +|g|_{L^1_{\gamma+2s}}|h|_{L^2_{l-1}}|(f_\phi)_j|_{L^2})
+\mathrm{1}_{2s\le 1}j(|g|_{L^1_{\gamma}}|h|_{L^2_{l+a-1}} |(f_\phi)_j|_{L^2_b} \\&&\qquad+|g|_{L^1_{\gamma+2s}}|h|_{L^2_{l-1}}|(f_\phi)_j|_{L^2}).\eeno

Thanks to Theorem \ref{baslem3}, we have
$|A_{1}| \lesssim  |g|_{L^1_{\gamma+2s}}|h|_{L^2_{l+a+2s-1}} |f_\phi|_{H^s_b},$
  which together with the estimate of $A_2$ yield that $ |I_{1,3}|\lesssim |g|_{L^1_{ \gamma+2s}}|h|_{L^2_{l-1}}|f_\phi|_{L^2}+|g|_{L^1_{\gamma+2s}}|h|_{L^2_{l+a+2s-1}}|f_\phi|_{H^{s}_b} $.
Now combining with the  estimates to $I_{1,3}$ and $I_{1,4}$, we have 
\beno |I_1|\lesssim |g|_{L^1_{\gamma+2s}}|h|_{L^2_{l+a}}|f|_{L^2_b}+|g|_{L^1_{\gamma+2s}}|h|_{L^2_{l+a+2s-1}}|W^\ep_s(D)W_bf|_{L^2}. \eeno

\smallskip

{\it Step1.2: Estimates of $I_2$.} For the term $I_2$, it is easy to check that \beno |(\na^2 W_l)(\kappa(v)): (v'-v)\otimes(v'-v)|  
&\le& C_l( W_{l-2}(v)|v-v'|^2+|v-v'|^l). \eeno 

{\it \underline{ Estimate of $I_{2}^1$.}} Thanks to \eqref{changev}, we have
  \beno  \mathrm{1}_{|v|\ge 4|v_*|}|v-v_*|^\gamma ( W_{l-2}(v)|v-v'|^2+|v-v'|^l)\lesssim   W_{l+a}(v)W_b(v')\theta^2,\eeno
  which implies that
  $ |I_2^1|\lesssim  |g|_{L^1_b}|h|_{L^2_{l+a}}|f|_{L^2_b}.$

{\it \underline{ Estimate of $I_2^2$.}}
Observe that 
  \beno    &&\mathrm{1}_{|v|\le 4|v_*|}\psi(|v-v_*|\sin\f{\theta}2)|v-v_*|^\gamma ( W_{l-2}(v)|v-v'|^2+|v-v'|^l)\\
  &&\lesssim \mathrm{1}_{|v|\le 4|v_*|}\psi(|v-v_*|\sin\f{\theta}2)|v-v_*|^{\gamma+2}\theta^2W_{l-2}+|v-v_*|^{l+a}\theta^{l-b}(W_b(v)+W_b(v')).
  \eeno
Then one has
$  |I_2^2|\lesssim |g|_{L^1_{\gamma+2s}}|h|_{L^2_{l-2}}|f|_{L^2} +\int   b^\eps\theta^{l-b} |(gW_{l+a})_*||hW_b||(fW_b)'|d\sigma dv_* dv. $
If we denote the integration in the above  by $B$, then by Cauchy-Schwartz's inequality, we have
\beno B&\lesssim& \big(\int   b^\eps\theta^{2} |(gW_{l+a})_*|^2|hW_b| d\sigma dv_* dvd\kappa\big)^{1/2} \big( \int   b^\eps\theta^{2l-2b-2}  |hW_b|| (fW_b) '|^2d\sigma dv_* dvd\kappa\big)^{1/2}. \eeno
 Let $\cos\tilde{\theta}=\frac{v'-v}{|v'-v|}\cdot \sigma$, then we have $\tilde{\theta}+\theta/2=\pi/2$. It implies that $b^\eps\theta^{2l-2b-2}\sim (\cos\tilde{\theta})^{2l-2b-4-2s}$.
   Use the facts $l>3+\gamma+s$ and $\big|\f{\pa v_*}{\pa v'}\big|=\frac{4}{1-\frac{v-v_*}{|v-v_*|}\cdot \sigma},$
then we have
\beno B\lesssim  |g|_{L^1_{l+a}}|h|_{L^1}^{\f12}\bigg(\int (\cos\tilde{\theta})^{2l-2\gamma-6-2s}|hW_b|| (fW_b)'|^2d\tilde{\theta} dv'dv\bigg)^{\f12} \lesssim |g|_{L^1_{l+a}}|h|_{L^1_b}|f|_{L^2_b}^2.\eeno
It implies that
$|I_2^2|\lesssim |g|_{L^1_{\gamma+2s}}|h|_{L^2_{l-2}}|f|_{L^2}+|g|_{L^2_{l+a}}|h|_{L^1_b}|f|_{L^2_b}.$

Finally we derive that
\beno  |I_2|\lesssim  |g|_{L^1_b}|h|_{L^2_{l+a}}|f|_{L^2_b}+|g|_{L^1_{\gamma+2s}}|h|_{L^2_{l-2}}|f|_{L^2}+|g|_{L^2_{l+a}}|h|_{L^1}|f|_{L^2}.\eeno

{\it Step 2: Estimate of $II$.} 
We have the decomposition that $II=II_1+II_2$ where \beno 
II_1&\eqdefa&\int  |v-v_*|^\gamma b^\eps g_*hf' W_l  (1-\psi(|v-v_*|\sin\f{\theta}2))\mathrm{1}_{|v|\le 4|v_*|}d\sigma dv_*dv,\\ II_2&\eqdefa&\int  |v-v_*|^\gamma b^\eps g_*hf' (W_l)'  (1-\psi(|v-v_*|\sin\f{\theta}2))\mathrm{1}_{|v|\le 4|v_*|}d\sigma dv_*dv.\eeno
Let us give a short proof to the estimates of $II_1$ and $II_2$.

{\it  Step 2.1: Estimate of $II_1$.}  By Cauchy-Schwartz inequality, we have
\beno  |II_1|
&\lesssim&\big(\int |v-v_*|^\gamma b^\eps |g_*|h^2W_l^2(1-\psi(|v-v_*|\sin\f{\theta}2))\mathrm{1}_{|v|\le 4|v_*|}d\sigma dv_*dv\big)^{\f12}\\
&&\times\big(\int |v-v_*|^\gamma b^\eps |g_*|(f')^2 (1-\psi(|v-v_*|\sin\f{\theta}2))\mathrm{1}_{|v|\le 4|v_*|}d\sigma dv_*dv\big)^{\f12}\\
&\lesssim&  |g|_{L^1_{\gamma+2s}}|h|_{L^2_l}|f|_{L^2}, \eeno
where we use the change of variables and the fact $|v-v_*|\sim|v'-v_*|\lesssim |v_*|$.

{\it   Step 2.2: Estimate of $II_2$.} 
Observe that $(W_l)'\lesssim W_l+|v-v'|^l$ and
   \beno &&\mathrm{1}_{|v|\le 4|v_*|} |v-v_*|^\gamma |v-v'|^{l} \\
&\lesssim&  \mathrm{1}_{|v|\le 4|v_*|}|v-v_*|^{a+b}(|v'|^b+|v|^b)|v-v'|^{l-b}\lesssim |v_*|^{l+a} \theta^{l-b}(|v'|^b+|v|^b).\eeno
Then following the argument applied to $I_2^2$, we get that
\beno   |II_2|&\lesssim&   |g|_{L^1_{\gamma+2s}}|h|_{L^2_l}|f|_{L^2}+
|g|_{L^2_{l+a}}|h|_{L^1_{\gamma}}|f|_{L^2_b}.\eeno
 
Finally summing up all the estimates, we  derive the desired inequalities in the lemma.
\end{proof}

Combining the above estimates, we are led to the following corollary.

\begin{cor}\label{basic-estimates} Suppose $g, h$ and $f$ are smooth functions. We have
\beno  &&(i). \,|\lr{Q^\eps(g,h)W_l, f}_v|\lesssim|g|_{L^1_{\gamma+2s}}|W^\ep_s(D)W_{l+\gamma/2+2s} h|_{L^2}|W^\ep_s(D)W_{\gamma/2} f|_{L^2}
+|g|_{L^2_{l+\gamma/2}}|h|_{L^1_2}|f|_{L^2_{\gamma/2}};\\  
&&(ii).\, |\lr{Q^\eps(g,h)W_l, f}_v|\lesssim \big(|g|_{L^1_{\gamma+2s}}|W^\ep_s(D)W_{l+\gamma+2s}h|_{L^2}
+|g|_{L^2_{l+\gamma}}|h|_{L^1_2}\big)|W^\ep_s(D)f|_{L^2}; 
\\&&(iii). \,  |\lr{Q^\eps(g,h)W_l, f}_v|\lesssim \big( \mathrm{1}_{2s>1}(\eta+\eps^{1-s}|g|_{L^1_{\gamma+2s}})|W^\ep_s(D)W_{l+\gamma+2s}h|_{L^2}+
\eta^{-\f{2s-1}{1-s}}|g|_{L^1_{\gamma+2s}}^{\f{2s-1}{1-s}}|h|_{L^2_{l+\gamma+2s}}\\&&\qquad\qquad\qquad\qquad\qquad\qquad+|g|_{L^1_{\gamma+2s}}|h|_{L^2_{l+\gamma+2s}}
+|g|_{L^2_{l+\gamma}}|h|_{L^1_2}\big)|f|_{H^1}. \eeno
Moreover, if $2\epsilon<\delta<(\min\{\f{\mathcal{C}_3}{4\mathcal{C}_2},\f{\mathcal{C}_5}{12c_2}\})^{\f1{2s}}$ and $g$ is a non-negative function verifying the condition \eqref{lbcondi}, then it hold
\ben
(iv).\,\lr{Q^\eps(g,h)W_l, hW_l}_v&\lesssim& -\f16\mathcal{E}_{g}^\gamma(hW_l)-\mathcal{C}_4(c_1,c_2)(|W^\ep_s(D)W_{l+\gamma/2}h|_{L^2}^2+\mathcal{E}_{\mu}^{0,\eps}(W_{l+\gamma/2}h))-\f12\mathcal{C}_5(c_1,c_2)\nonumber\\&&\times\delta^{-2s}|h|_{L^2_{l+\gamma/2}}^2+\mathcal{C}_{6}(c_1, c_2)\delta^{-6-6s} |h|_{L^1_{2l}}|h|_{L^1_{\gamma}}+|g|^2_{L^1_{\gamma+2s}}|h|_{L^2_{l}}^2+|g|_{L^2_{l+\gamma/2}}^2|h|_{L^1_2}^2;\nonumber\\
(v).\,\lr{Q^\eps(g,h)W_l, hW_l}_v&\lesssim& -\f16\mathcal{E}_{g}^\gamma(hW_l)-\mathcal{C}_4(c_1,c_2)(|W^\ep_s(D)W_{l+\gamma/2}h|_{L^2}^2+\mathcal{E}_{\mu}^{0,\eps}(W_{l+\gamma/2}h))-\f12\mathcal{C}_5(c_1,c_2)\nonumber\\&&\times\delta^{-2s}|h|_{L^2_{l+\gamma/2}}^2+\mathcal{C}_7(c_1,c_2)\delta^{-4-2s}|h|_{H^{-1}_{\gamma/2}}^2+|g|^2_{L^1_{\gamma+2s}}|h|_{L^2_{l}}^2+|g|_{L^2_{l+\gamma/2}}^2|h|_{L^1_2}^2; \nonumber\\
 (vi).\, \lr{Q^\eps(g,h)W_l, hW_l}_v&\lesssim& -\f13\mathcal{E}_{g}^\gamma(hW_l)-\mathcal{C}_1(c_1,c_2)(|W^\ep_s(D)W_{l+\gamma/2}h|_{L^2}^2+\mathcal{E}_{\mu}^{0,\eps}(W_{l+\gamma/2}h))\nonumber\\
&&+
\mathcal{C}_2(c_1,c_2)|h|_{L^2_{l+\gamma/2}}^2+|g|^2_{L^1_{\gamma+2s}}|h|_{L^2_{l}}^2+|g|_{L^2_{l+\gamma/2}}^2|h|_{L^1_2}^2.\label{corbaesti1}
\een
\end{cor}

\subsubsection{Commutator estimates for $Q^\ep$ with the symbol $W^\ep_q(D)$}
By setting \ben\label{DefPhi} \Phi_k^\gamma(v)\eqdefa
\left\{\begin{aligned} & |v|^\gamma \varphi(2^{-k}|v|), \quad\mbox{if}\quad k\ge0;\\
& |v|^\gamma \psi( |v|),\quad\mbox{if}\quad k=-1.\end{aligned}\right.\een
  we  derive that $
\langle Q^\ep(g, h), f \rangle_v=\sum_{k=-1}^\infty \langle Q^\ep_k(g, h), f \rangle_v,  $
where
$
  Q^\ep_{k}(g, h)=\iint_{\sigma\in \SS^2,v_*\in \R^3} \Phi_k^\gamma(|v-v_*|)b^\eps(\cos\theta) (g'_*h'-g_*h)d\sigma dv_*. $

By Bony's decomposition,  we have
\beno Q^\ep_{k}(g, h)=\sum_{p\ge-1} \big[Q^\ep_{k}(\mathcal{S}_{p-N_0}g, \mathfrak{F}_p h)+Q^\ep_{k}(\mathfrak{F}_p g,\mathcal{S}_{p-N_0} h)\big]+\sum_{|p-p'|\le N_0}Q^\ep_{k}(\mathfrak{F}_{p'} g,\mathfrak{F}_{p}h),
 \eeno
where $N_0$ is a integer such that $\mathfrak{F}_p\mathfrak{F}_m=0$ if $|p-m|>N_0$.
We recall that   the Bobylev's formula of the operator can be stated as
\ben\label{bobylev}&& \qquad\langle\mathfrak{F}\big( Q_k(g, h)\big), \mathfrak{F}f \rangle\\&&=\iint_{\sigma\in \SS^2, \eta,\xi\in \R^3} b^\eps(\f{\xi}{|\xi|}\cdot \sigma)\big[ \mathfrak{F}(\Phi_k^\gamma ) (\eta-\xi^{-})-\mathfrak{F}(\Phi_k^\gamma)(\eta)\big](\mathfrak{F}g)(\eta)(\mathfrak{F}h)(\xi-\eta)\overline{(\mathfrak{F}f)}(\xi)d\sigma d\eta d\xi,\nonumber \een
where $\mathfrak{F}f$ denotes the Fourier transform of $f$ and $\xi^-\eqdefa\f{\xi-|\xi|\sigma}2$.   Then we get the following decomposition:
\ben  && \lr{\mathfrak{F}_{j}Q_k^\ep(g,h)-Q_k^\ep(g,\mathfrak{F}_{j}h), \mathfrak{F}_{j}f}_v \nonumber\\
&=& \sum_{|p-j|\le 2N_0}  \lr{\mathfrak{F}_{j}Q_k^\ep(\mathcal{S}_{p-N_0}g, \mathfrak{F}_p h)-Q_k^\ep(\mathcal{S}_{p-N_0}g,  \mathfrak{F}_{j}\mathfrak{F}_{p}h), \mathfrak{F}_{j}f}_v    +\sum_{|p-j|\le 2N_0} \lr{\mathfrak{F}_{j}Q_k^\ep(\mathfrak{F}_{p}g, \mathcal{S}_{p-N_0}h),  \mathfrak{F}_{j}f}_v\nonumber \\\label{Qkdp}&&
\quad+\sum_{|p-j|\le 2N_0}\sum_{|p-p'|\le N_0}  \lr{\mathfrak{F}_{j}Q_k^\ep(\mathfrak{F}_{p'} g, \mathfrak{F}_{p}h)-Q_k^\ep(\mathfrak{F}_{p'}g,  \mathfrak{F}_{j}\mathfrak{F}_{p}h),  \mathfrak{F}_{j}f}_v  \\&&\quad+\sum_{p>j+2N_0}\sum_{|p-p'|\le N_0}  \lr{\mathfrak{F}_{j}Q_k^\ep(\mathfrak{F}_{p'} g, \mathfrak{F}_{p}h), \mathfrak{F}_{j}f}_v
\eqdefa\sum_{i=1}^4 \mathcal{T}_i^j.\nonumber\een

We first have
\begin{prop}\label{comforsym1}    For smooth functions $g, h$ and $f$, we have
\begin{enumerate}
\item if $2^{j}\ge \f1{\ep}$, \beno  |\lr{ \mathfrak{F}_{j}Q_k^\ep( g, h)-Q_k^\ep( g,  \mathfrak{F}_{j}h), \mathfrak{F}_{j}f}_v|&\lesssim &\left\{\begin{aligned} &  2^{k(\gamma+\f32)}(\ep^{-2s+1}\mathrm{1}_{2s>1}+1_{2s=1}|\log \ep|+1_{2s<1})|g|_{L^2}|h|_{L^2}|\mathfrak{F}_jf|_{L^2};\\
&   2^{k(\gamma+\f32)}2^{-j}\ep^{-2s} |g|_{H^1}|h|_{L^2}|\mathfrak{F}_jf|_{L^2}. \end{aligned}\right.
 \eeno
  \item if $2^{j}\le \f1{\ep}$, \beno &&|\lr{ \mathfrak{F}_{j}Q_k^\ep( g, h)-Q_k^\ep( g,  \mathfrak{F}_{j}h), \mathfrak{F}_{j}f}_v|\\&&\lesssim
 2^{k(\gamma+\f52)}  |g|_{L^2}|h|_{L^2}( \mathrm{1}_{2s>1}|\mathfrak{F}_jf|_{H^{2s-1}}+1_{2s=1}j|\mathfrak{F}_jf|_{L^2}+1_{2s<1}|\mathfrak{F}_jf|_{L^2}).
\eeno
\end{enumerate}
\end{prop}

\begin{proof}
By Bobylev's formula, we observe that
 \beno  \lr{ \mathfrak{F}_{j}Q_k^\ep( g, h)-Q_k^\ep( g,  \mathfrak{F}_{j}h), \mathfrak{F}_{j}f}_v 
 &=&
\int_{\sigma\in \SS^2, \eta,\xi\in \R^3} b^\eps(\f{\xi}{|\xi|}\cdot \sigma)\big[ \mathfrak{F}(\Phi_k^\gamma ) (\eta-\xi^{-})-\mathfrak{F}(\Phi_k^\gamma)(\eta)\big](\mathfrak{F}g)(\eta)(\mathfrak{F}h)(\xi-\eta)\\&&\quad\times\varphi(2^{-j}\xi)\overline{(\mathfrak{F}f)}(\xi)(\varphi(2^{-j}\xi)-\varphi(2^{-j}(\xi-\eta))d\sigma d\eta d\xi\eqdefa \mathcal{A}. \eeno
We split the estimates into two cases.

{\it Case 1: $2^{j}\ge \f1{\ep}$.}
We have  \beno
 |\mathcal{A}|&\lesssim& \int_{\R^6\times\SS^2} 2^{-j}|\eta|b^\eps(\f{\xi}{|\xi|}\cdot \sigma)\big[ |\mathfrak{F}(\Phi_k^\gamma ) (\eta-\xi^{-})|+|\mathfrak{F}(\Phi_k^\gamma)(\eta)|\big]|(\mathfrak{F}g)(\eta)| |(\mathfrak{F}h)(\xi-\eta)||\varphi(2^{-j}\xi)(\mathfrak{F}f)(\xi)|d\sigma d\eta d\xi \\
 &\lesssim& \int_{\R^6\times\SS^2} 2^{-j}b^\eps(\f{\xi}{|\xi|}\cdot \sigma)\big[ |\mathfrak{F}(\Phi_k^\gamma ) (\eta-\xi^{-})||\eta-\xi^-|+|\mathfrak{F}(\Phi_k^\gamma)(\eta)||\eta|\big]|(\mathfrak{F}g)(\eta)|\\&&\quad\times|(\mathfrak{F}h)(\xi-\eta)||\varphi(2^{-j}\xi)(\mathfrak{F}f)(\xi)|d\sigma d\eta d\xi  +\int_{\R^6\times\SS^2} 2^{-j}b^\eps(\f{\xi}{|\xi|}\cdot \sigma) |\xi|\sin\f{\theta}2|\mathfrak{F}(\Phi_k^\gamma ) (\eta-\xi^{-})||\eta-\xi^-|\\&&\quad\times|(\mathfrak{F}g)(\eta)||(\mathfrak{F}h)(\xi-\eta)||\varphi(2^{-j}\xi)(\mathfrak{F}f)(\xi)|d\sigma d\eta d\xi.
 \eeno
 By Cauchy-Schwartz inequality, it holds
 \beno
  |\mathcal{A}|&\lesssim&2^{-j}\big( \int_{\R^6\times\SS^2}  b^\eps(\f{\xi}{|\xi|}\cdot \sigma)[|\mathfrak{F}(\Phi_k^\gamma ) (\eta-\xi^{-})||\eta-\xi^-|+|\mathfrak{F}(\Phi_k^\gamma)(\eta)||\eta|\big]|^2   |(\mathfrak{F}f)(\xi)\varphi(2^{-j}\xi)|^2d\sigma d\eta d\xi\big)^{\f12}\\&&\times \big( \int_{\R^6\times\SS^2} b^\eps(\f{\xi}{|\xi|}\cdot \sigma) |(\mathfrak{F}g)(\eta)|^2|(\mathfrak{F}h)(\xi-\eta)|^2 d\sigma d\eta d\xi\big)^{\f12} +
  \big(\int_{ \R^6\times\SS^2}  b^\eps(\f{\xi}{|\xi|}\cdot \sigma)\theta|\mathfrak{F}(\Phi_k^\gamma ) (\eta-\xi^{-})|^2\\&&\times |(\mathfrak{F}f)(\xi)\varphi(2^{-j}\xi)|^2d\sigma d\eta d\xi\big)^{\f12}  \big( \int_{\R^6\times\SS^2} b^\eps(\f{\xi}{|\xi|}\cdot \sigma)\theta |(\mathfrak{F}g)(\eta)|^2|(\mathfrak{F}h)(\xi-\eta)|^2 d\sigma d\eta d\xi\big)^{\f12}\\
&\lesssim& \ep^{-2s}2^{-j} 2^{k(\gamma+\f12)}|g|_{L^2}|h|_{L^2}|\mathfrak{F}_jf|_{L^2} + 2^{k(\gamma+\f32)}(\ep^{-2s+1}\mathrm{1}_{2s>1}+\mathrm{1}_{2s=1}|\log \epsilon|+\mathrm{1}_{2s<1})|g|_{L^2}|h|_{L^2}|\mathfrak{F}_jf|_{L^2}.
\eeno
Here we use the fact  $\| \Phi_k^\gamma\|_{L^2}\lesssim 2^{k(\gamma+\f32)}$ and
$ \int_{\R^3} |\mathfrak{F}(\Phi_k^\gamma)(\xi)|^2|\xi|^2d\xi\lesssim 2^{2k(\gamma+\f12)}.$ It completes the proof of the first estimate. If  we do not split $|\eta|$ into $|\eta-\xi^-|$ and $\xi^-$, then the second estimate in $(1)$   can be obtained directly by Cauchy-Schwartz inequality.

{\it Case 2: $2^{j}\le \f1{\ep}$.}
We split the domain into two parts: $2|\xi^-|\le \lr{\eta}$ and $2|\xi^-|> \lr{\eta}$. Then $\mathcal{A}$ can be decomposed into two parts:  $\mathcal{A}_1$ and  $\mathcal{A}_2$, which denote the integration  of $\mathcal{A}$ over the domains $2|\xi^-|\le \lr{\eta}$ and $2|\xi^-|> \lr{\eta}$ respectively. In what follows, we give the proof to the case that $2s\ge1$.

In the region $2|\xi^-|\le \lr{\eta}$, we have $\sin(\theta/2)\le \lr{\eta}/|\xi|$ and $\lr{\eta-t\xi^-}\sim \lr{\eta}$ for $t\in[0,1]$. By Taylor expansion,
 we have
 \beno
 &&|\mathcal{A}_1|\le \big|\int_{2|\xi^-|\le \lr{\eta}} b^\eps(\f{\xi}{|\xi|}\cdot \sigma)  (\na\mathfrak{F}(\Phi_k^\gamma ) )(\eta)\cdot \xi^- (\mathfrak{F}g)(\eta)(\mathfrak{F}h)(\xi-\eta) \varphi(2^{-j}\xi)\overline{(\mathfrak{F}f)}(\xi)(\varphi(2^{-j}\xi)\\&&\quad-\varphi(2^{-j}(\xi-\eta))d\sigma d\eta d\xi\big| +\big|\int_0^1\int_{2|\xi^-|\le \lr{\eta}} b^\eps(\f{\xi}{|\xi|}\cdot \sigma) (\na^2(\mathfrak{F}(\Phi_k^\gamma ))(\eta-t\xi^-):\xi^-\otimes\xi^-)(\mathfrak{F}g)(\eta)(\mathfrak{F}h)(\xi-\eta)\\&&\quad\times\varphi(2^{-j}\xi)\overline{(\mathfrak{F}f)}(\xi)(\varphi(2^{-j}\xi)-\varphi(2^{-j}(\xi-\eta))d\sigma d\eta d\xi dt\big|.
 \eeno
Since it hold $| (\na\mathfrak{F}(\Phi_k^\gamma ) )(\eta)|\lesssim 2^{k(\gamma+4)} \lr{2^k\eta}^{-(\gamma+4)}$ and 
 $|\na^2 (\mathfrak{F}(\Phi_k^\gamma ) )(\eta-t\xi^-)|\lesssim2^{k(\gamma+5)} \lr{2^k\eta}^{-(\gamma+5)},$
we get
 \beno
 &&|\mathcal{A}_1|\lesssim 2^{-j}2^{k(\gamma+4)}\int_{\xi,\eta} |\eta|\big(\lr{2^k\eta}^{-(\gamma+4)}|\xi|\min\{1,(\lr{\eta}/|\xi|
 )^{2-2s}\}+2^k\lr{2^k\eta}^{-(\gamma+5)}|\xi|^2(\lr{\eta}/|\xi|
 )^{2-2s}\big)|(\mathfrak{F}g)(\eta)|\\&&\qquad\times|(\mathfrak{F}h)(\xi-\eta)||\varphi(2^{-j}\xi)(\mathfrak{F}f)(\xi)|d\eta d\xi 
 \lesssim  2^{k(\gamma+\f52)}  |g|_{L^2}|h|_{L^2}|\mathfrak{F}_jf|_{H^{2s-1}}.
\eeno
In the region $2|\xi^-|> \lr{\eta}$, we have $\sin(\theta/2)\gtrsim \lr{\eta-\xi^-}/(3|\xi|)$ and  $\sin(\theta/2)\ge \lr{\eta}/(2|\xi|)$. We have
\beno
 |\mathcal{A}_2|&\lesssim&2^{-j}\big( \int_{2|\xi^-|> \lr{\eta}}  b^\eps(\f{\xi}{|\xi|}\cdot \sigma) (|\mathfrak{F}(\Phi_k^\gamma ) (\eta-\xi^{-})|^2 +|\mathfrak{F}(\Phi_k^\gamma)(\eta)|^2) |\eta|^{2-2s}  |(\mathfrak{F}f)(\xi)\varphi(2^{-j}\xi)|^2\\&&\quad\times\lr{\xi}^{2s} d\sigma d\eta d\xi\big)^{\f12} \big( \int_{2|\xi^-|> \lr{\eta}} b^\eps(\f{\xi}{|\xi|}\cdot \sigma) |(\mathfrak{F}g)(\eta)|^2|(\mathfrak{F}h)(\xi-\eta)|^2 \lr{\xi}^{-2s}|\eta|^{2s} d\sigma d\eta d\xi\big)^{\f12}\\
 &\lesssim& 2^{k(\gamma+\f32)}  |g|_{L^2}|h|_{L^2}(\mathrm{1}_{2s>1}|\mathfrak{F}_jf|_{H^{2s-1}}+\mathrm{1}_{2s=1}j|\mathfrak{F}_jf|_{L^2} ),
\eeno
where we use   change of variables from $(\eta, \xi)$ to $(\eta-\xi^-, \xi )$ if needed.

We conclude that for $2s\ge1$,
$ |\mathcal{A}| \lesssim 2^{k(\gamma+\f52)}  |g|_{L^2}|h|_{L^2}(\mathrm{1}_{2s>1}|\mathfrak{F}_jf|_{H^{2s-1}}+\mathrm{1}_{2s=1}j|\mathfrak{F}_jf|_{L^2}).$
We remark that the case of $2s<1$ can be treated in a similar way. We complete the proof of the proposition.
\end{proof}

Next we want to prove:
\begin{lem}\label{comforsym2}  Suppose $g, h$ and $f$ are smooth functions and $q\ge s$. We have  
 \beno &&(i). \sum_{j\le -\log_2 \epsilon} 2^{2qj} |\lr{\mathfrak{F}_{j}Q^\ep_k(g,h)-Q^\ep_k(g,\mathfrak{F}_{j}h), \mathfrak{F}_{j}f}_v|+\sum_{j\ge -\log_2 \epsilon} \ep^{-2q} |\lr{\mathfrak{F}_{j}Q^\ep_k(g,h)-Q^\ep_k(g,\mathfrak{F}_{j}h), \mathfrak{F}_{j}f}_v|\\
 &&\lesssim  2^{k(\gamma+\f52)}|g|_{L^2}( \mathrm{1}_{2s>1}|W^\ep_{q+s-1}(D)h|_{L^2}+\mathrm{1}_{2s=1}|W^\ep_{q-s+\log}(D)h|_{L^2}+\mathrm{1}_{2s<1}|W^\ep_{q-s}(D)h|_{L^2})\\&&\quad\times|W^\ep_{q+s}(D)f|_{L^2}+ \mathrm{1}_{k=-1} |W^\ep_{q+s-1 }(D)g|_{L^2}|h|_{L^2}|W^\ep_{q+s}(D)f|_{L^2}+\mathrm{1}_{k\ge0}|g|_{L^1}|h|_{L^2}|f|_{L^2}, \\
 &&(ii). \sum_{j\ge-1} 2^{2qj} |\lr{\mathfrak{F}_{j}Q^\ep_k(g,h)-Q^\ep_k(g,\mathfrak{F}_{j}h), \mathfrak{F}_{j}f}_v|\\ &&\lesssim   2^{k(\gamma+\f52)}(|g|_{H^1} |\lr{D}^{q-1}W^\ep_{s}(D)h|_{L^2}+|g|_{L^2}(\mathrm{1}_{2s=1} |W^\ep_{q-s+\log}(D)h|_{L^2}+\mathrm{1}_{2s<1}| W^\ep_{q-s}(D)h|_{L^2}))\\&&\quad\times|\lr{D}^qW^\ep_{s}(D)f|_{L^2}+ \mathrm{1}_{k=-1} |\lr{D}^{q-1}W^\ep_{s}(D)g|_{L^2}|h|_{L^2}|\lr{D}^qW^\ep_{s}(D)f|_{L^2}+\mathrm{1}_{k\ge0}|g|_{L^1}|h|_{L^2}|f|_{L^2}.
\eeno 

\end{lem}
 \begin{proof}  We only give the proof in the case $2s>1$. The other cases can be proved by the same argument. Due to \eqref{Qkdp}, it suffices to  give the estimates to $\mathcal{T}^j_i(i=1,2,3,4)$ term by term. For $\mathcal{T}^j_1$ and $\mathcal{T}^j_3$, by Proposition \ref{comforsym1}, we have
 \beno &&\sum_{j\le -\log_2 \epsilon} 2^{2qj}(\mathcal{T}^j_1+\mathcal{T}^j_3)+\sum_{j\ge -\log_2 \epsilon} \ep^{-2q}(\mathcal{T}^j_1+\mathcal{T}^j_3)
 \lesssim(2^{k(\gamma+\f32)}+2^{k(\gamma+2s+\f12)})\bigg( \sum_{j\le -\log_2 \epsilon} 2^{(2q+2s-1)j}\\&&\qquad\times\big[  \sum_{|p-j|\le 2N_0}   |
 \mathcal{S}_{p-N_0}g|_{L^2}|\mathfrak{F}_ph|_{L^2}|\mathfrak{F}_jf|_{L^2}+\sum_{|p-j|\le 2N_0}\sum_{|p-p'|\le N_0}  |\mathfrak{F}_{p'}g|_{L^2}|\mathfrak{F}_ph|_{L^2}|\mathfrak{F}_jf|_{L^2} \big]\eeno\beno&&\qquad +\sum_{j\ge -\log_2 \epsilon}\ep^{-2q-2s+1}\big[ \sum_{|p-j|\le 2N_0}   | \mathcal{S}_{p-N_0}g|_{L^2}|\mathfrak{F}_ph|_{L^2}|\mathfrak{F}_jf|_{L^2}+\sum_{|p-j|\le 2N_0}\sum_{|p-p'|\le N_0}  |\mathfrak{F}_{p'}g|_{L^2}|\mathfrak{F}_ph|_{L^2}|\\&&\qquad\times\mathfrak{F}_jf|_{L^2}  \bigg)
 \lesssim   2^{k(\gamma+\f52)}|g|_{L^2}|W^\ep_{q+s-1}(D)h|_{L^2}|W^\ep_{q+s}(D)f|_{L^2}.
 \eeno
 Next we turn to the terms $\mathcal{T}^j_2$ and $\mathcal{T}^j_4$. For $\mathcal{T}^j_2$,
 Thanks to \eqref{bobylev}, if $|p-j|\le 2N_0$ and $m\le p-N_0$, then
 \beno \lr{\mathfrak{F}_{j}Q_k^\ep(\mathfrak{F}_{p}g, \mathfrak{F}_{m}h),  \mathfrak{F}_{j}f}_v= \iint_{\sigma ,v_*,v} \big(\tilde{ \mathfrak{F}}_{p}\Phi_k^\gamma\big)(|v-v_*|)b(\cos\theta) (\mathfrak{F}_{p}g)_*(\mathfrak{F}_{m}h)\big[
 ( (\mathfrak{F}^2_{j}f)'-
 \mathfrak{F}_{j}^2f\big]d\sigma dv_* dv, \eeno
 which enjoys the same structure as that for $\mathfrak{M}_{k,p,l}^1$ defined in Lemma \ref{lemub1}. Then
 we conclude that in this case,  on one hand,
 \beno
 &&| \lr{\mathfrak{F}_{j}Q_k^\ep(\mathfrak{F}_{p}g, \mathfrak{F}_{m}h),  \mathfrak{F}_{j}f}_v|\lesssim \mathrm{1}_{k=-1}2^{2sm} 2^{-p}|\mathfrak{F}_{p}g|_{L^2}|\mathfrak{F}_{m}h|_{L^2}|\mathfrak{F}_{j}f|_{L^2} 
 +\mathrm{1}_{k\ge0}C_N2^{-pN}|\mathfrak{F}_{p}g|_{L^1}|\mathfrak{F}_{m}h|_{L^2}|\mathfrak{F}_{j}f|_{L^2}.
 \eeno 
 On the other hand, for $k=-1$, we may use the fact   $\|\tilde{ \mathfrak{F}}_{p}\Phi_k^\gamma\|_{L^2}\lesssim 2^{-p(\f32+\gamma)}$ and the Cauchy-Schwartz inequality to get $  
 | \lr{\mathfrak{F}_{j}Q_{-1}^\ep(\mathfrak{F}_{p}g, \mathfrak{F}_{m}h),  \mathfrak{F}_{j}f}_v|\lesssim  \eps^{-2s} 2^{-(\f32+\gamma)p}|\mathfrak{F}_{p}g|_{L^2}|\mathfrak{F}_{m}h|_{L^2}|\mathfrak{F}_{j}f|_{L^2}.$

 For $\mathcal{T}^j_4$, in the cae of $|p-p'|\le N_0$ and $p>j+2N_0$,  by \eqref{bobylev}, it is easy to see that the structure of $\lr{\mathfrak{F}_{j}Q_k^\ep(\mathfrak{F}_{p'} g, \mathfrak{F}_{p}h), \mathfrak{F}_{j}f}_v$ is   as the same as that for $\mathfrak{M}_{k,p,l,m}^4$ in Lemma \ref{lemub1}. Then for any $N\in \N$, one has
\beno &&|\lr{\mathfrak{F}_{j}Q_k^\ep(\mathfrak{F}_{p'} g, \mathfrak{F}_{p}h), \mathfrak{F}_{j}f}_v|\lesssim \mathrm{1}_{k=-1}2^{2sj}2^{-p}|\mathfrak{F}_{p'}g|_{L^2}|\mathfrak{F}_{p}h|_{L^2}|\mathfrak{F}_{j}f|_{L^2}+\mathrm{1}_{k\ge0}C_N2^{-pN}|\mathfrak{F}_{p'}g|_{L^2}|\mathfrak{F}_{p}h|_{L^2}|\mathfrak{F}_{j}f|_{L^2},\\
 &&|\lr{\mathfrak{F}_{j}Q_k^\ep(\mathfrak{F}_{p'} g, \mathfrak{F}_{p}h), \mathfrak{F}_{j}f}_v|\lesssim \mathrm{1}_{k=-1} \eps^{-2s}2^{-(\f32+\gamma)p}|\mathfrak{F}_{p'}g|_{L^2}|\mathfrak{F}_{p}h|_{L^2}|\mathfrak{F}_{j}f|_{L^2}+\mathrm{1}_{k\ge0}C_N2^{-pN}|\mathfrak{F}_{p'}g|_{L^1}|\mathfrak{F}_{p}h|_{L^2}|\mathfrak{F}_{j}f|_{L^2}.
 \eeno

Now putting together all the estimates, we infer that  
 \beno  &&\sum_{j\le -\log_2 \epsilon} 2^{2qj}(\mathcal{T}^j_2+\mathcal{T}^j_4)+\sum_{j\ge -\log_2 \epsilon} \ep^{-2q}(\mathcal{T}^j_2+\mathcal{T}^j_4)\\
 &&\lesssim \sum_{j\le -\log_2 \epsilon} \big[\sum_{|p-j|\le N_0}\sum_{m\le p-N_0} 2^{2qj}(\mathrm{1}_{k=-1}2^{2sm} 2^{-p}|\mathfrak{F}_{p}g|_{L^2}|\mathfrak{F}_{m}h|_{L^2}|\mathfrak{F}_{j}f|_{L^2}
+\mathrm{1}_{k\ge0}C_N2^{-pN}|\mathfrak{F}_{p}g|_{L^1}|\mathfrak{F}_{m}h|_{L^2}\\ &&\times|\mathfrak{F}_{j}f|_{L^2})
+\sum_{p>j+2N_0}\sum_{|p-p'|\le N_0} 2^{2qj}( \mathrm{1}_{k=-1}2^{2sj}2^{-p}|\mathfrak{F}_{p'}g|_{L^2}|\mathfrak{F}_{p}h|_{L^2}|\mathfrak{F}_{j}f|_{L^2}+\mathrm{1}_{k\ge0}C_N2^{-pN}|\mathfrak{F}_{p'}g|_{L^2}|\mathfrak{F}_{p}h|_{L^2}\\ &&\times|\mathfrak{F}_{j}f|_{L^2})\big]
 +\sum_{j\ge -\log_2 \epsilon}  \big[\sum_{|p-j|\le N_0}\sum_{m\le p-N_0} \eps^{-2q} ( \mathrm{1}_{k=-1} \eps^{-2s} 2^{-\f32p}|\mathfrak{F}_{p}g|_{L^2}|\mathfrak{F}_{m}h|_{L^2}|\mathfrak{F}_{j}f|_{L^2}+\mathrm{1}_{k\ge0}C_N2^{-pN}\\ &&\times|\mathfrak{F}_{p}g|_{L^1}|\mathfrak{F}_{m}h|_{L^2}|\mathfrak{F}_{j}f|_{L^2}) 
 +\sum_{p>j+2N_0}\sum_{|p-p'|\le N_0} \eps^{-2q}(\mathrm{1}_{k=-1} \eps^{-2s}2^{-\f32p}|\mathfrak{F}_{p'}g|_{L^2}|\mathfrak{F}_{p}h|_{L^2}|\mathfrak{F}_{j}f|_{L^2} +\mathrm{1}_{k\ge0}C_N\\ &&\times2^{-pN}|\mathfrak{F}_{p'}g|_{L^2}|\mathfrak{F}_{p}h|_{L^2}|\mathfrak{F}_{j}f|_{L^2})\big] 
\lesssim \mathrm{1}_{k=-1} |W^\ep_{q+s-1}(D)g|_{L^2}|h|_{L^2}|W^\ep_{q+s}(D)f|_{L^2}+\mathrm{1}_{k\ge0}|g|_{L^1}|h|_{L^2}|f|_{L^2}.\eeno

Combining the above estimates, we will derive the first result in the lemma. The second one can be obtained by the similar argument and we skip the details here. 
\end{proof}

Now we can state the main result in this subsection:

\begin{lem}\label{comforsym3}  Suppose $g, h$ and $f$ are smooth functions and $q\ge s$. We have
 \beno &&(i).\,\sum_{j\le -\log_2 \epsilon} 2^{2qj} |\lr{\mathfrak{F}_{j}Q^\ep(g,h)-Q^\ep(g,\mathfrak{F}_{j}h), \mathfrak{F}_{j}f}_v|+\sum_{j\ge -\log_2 \epsilon} \ep^{-2q} |\lr{\mathfrak{F}_{j}Q^\ep(g,h)-Q^\ep(g,\mathfrak{F}_{j}h), \mathfrak{F}_{j}f}_v|\\
 &&\lesssim  |g|_{L^2}( \mathrm{1}_{2s>1}|W^\ep_{q+s-1}(D)W_{\gamma/2+\f52}h|_{L^2}+\mathrm{1}_{2s=1}|W^\ep_{q-s+\log}(D)W_{\gamma/2+\f52}h|_{L^2}+\mathrm{1}_{2s<1}|W^{\eps}_{q-s}(D)W_{\gamma/2+\f52}h|_{L^2})\\ &&\times |W^\ep_{q+s}(D)W_{\gamma/2}f|_{L^2}
 + |g|_{L^2_{\gamma+3}}( \mathrm{1}_{2s>1}| W^\ep_{q+s-1}(D)h|_{L^2}+\mathrm{1}_{2s=1} W^\ep_{q-s+\log}(D)h|_{L^2}\\&&+\mathrm{1}_{2s<1}| W^\eps_{q-s}(D)h|_{L^2})| W^\ep_{q+s}(D)f|_{L^2}
 + |W^\ep_{q+s-1 }(D)g|_{L^2}|h|_{L^2}|W^\ep_{q+s}(D)f|_{L^2}
 +|g|_{L^1}|h|_{L^2}|f|_{L^2};\\
  &&(ii).\,\sum_{j\ge-1} 2^{2qj} |\lr{\mathfrak{F}_{j}Q^\ep(g,h)-Q^\ep(g,\mathfrak{F}_{j}h), \mathfrak{F}_{j}f}_v| \lesssim (|W_{\gamma+3}g|_{H^1} |\lr{D}^{q-1}W^\ep_{s}(D)h|_{L^2}+|g|_{H^1} \\&&\qquad\times|\lr{D}^{q-1}W^\ep_{s}(D)W_{\gamma/2+\f52}h|_{L^2})|W^\ep_{q+s}(D)W_{\gamma/2}f|_{L^2} + |g|_{L^2}(  \mathrm{1}_{2s=1}|W^\ep_{q-s+\log}(D)W_{\gamma/2+\f52}h|_{L^2}\\&&\qquad+\mathrm{1}_{2s<1}|W^{\eps}_{q-s}(D)W_{\gamma/2+\f52}h|_{L^2})|W^\ep_{q+s}(D)W_{\gamma/2}f|_{L^2}
 + |g|_{L^2_{\gamma+3}}( \mathrm{1}_{2s=1} W^\ep_{q-s+\log}(D)h|_{L^2}\\&&\qquad +\mathrm{1}_{2s<1}| W^\eps_{q-s}(D)h|_{L^2})| W^\ep_{q+s}(D)f|_{L^2}  
+ |\lr{D}^{q-1+\eta}W^\ep_{s}(D)g|_{L^2}|h|_{L^2}|\lr{D}^qW^\ep_{s}(D)f|_{L^2}
 +|g|_{L^1}|h|_{L^2}|f|_{L^2}.\eeno
\end{lem}
\begin{proof}  We observe that
 \ben\label{psdecom} &&\lr{Q(g,h), f}_v\nonumber\\&&=\sum_{k\ge-1}\sum_{m\ge-1}\langle Q_k(\mathcal{P}_mg, h), f \rangle_v\notag =\sum_{m\le k-N_0}\langle Q_k(\mathcal{P}_m g, \tilde{\mathcal{P}}_kh), \tilde{\mathcal{P}}_kf \rangle_v +
\sum_{m\ge k+N_0}\langle Q_k(\mathcal{P}_{m} g, \tilde{\mathcal{P}}_mh), \tilde{\mathcal{P}}_mf \rangle_v\notag\\&&\quad+\sum_{|m-k|\le N_0}\langle Q_k( \mathcal{P}_{m} g, \mathcal{U}_{k+N_0}h), \mathcal{U}_{k+N_0}f \rangle_v=\sum_{k\ge N_0-1}\langle Q_k(\mathcal{U}_{k-N_0} g, \tilde{\mathcal{P}}_kh), \tilde{\mathcal{P}}_kf \rangle_v \\&&\quad+
\sum_{m\ge k+N_0}\langle Q_k(\mathcal{P}_{m} g, \tilde{\mathcal{P}}_mh), \tilde{\mathcal{P}}_mf \rangle_v\notag +\sum_{|m-k|\le N_0}\langle Q_k( \mathcal{P}_{m} g, \mathcal{U}_{k+N_0}h), \mathcal{U}_{k+N_0}f \rangle_v. \notag \een
Then by  Lemma \ref{baslem2}, Lemma \ref{comforsym2}, \eqref{func6}, \eqref{func7} and \eqref{func8} we have
 \beno &&\sum_{j\le -\log_2 \epsilon} 2^{2qj} |\lr{\mathfrak{F}_{j}Q^\ep(g,h)-Q^\ep(g,\mathfrak{F}_{j}h), \mathfrak{F}_{j}f}_v|+\sum_{j\ge -\log_2 \epsilon} \ep^{-2q} |\lr{\mathfrak{F}_{j}Q^\ep(g,h)-Q^\ep(g,\mathfrak{F}_{j}h), \mathfrak{F}_{j}f}_v|\\
 &&\lesssim  |g|_{L^2}( \mathrm{1}_{2s>1}|W_{\gamma/2+\f52}W^\ep_{q+s-1}(D)h|_{L^2}+\mathrm{1}_{2s=1}|W_{\gamma/2+\f52}W^\ep_{q-s+\log}(D)h|_{L^2}+\mathrm{1}_{2s<1}|W_{\gamma/2+\f52}\\&&\times W^\ep_{q-s}(D)h|_{L^2})|W_{\gamma/2}W^\ep_{q+s}(D)f|_{L^2}
 + |g|_{L^2_{\gamma+3}}( \mathrm{1}_{2s>1}| W^\ep_{q+s-1}(D)h|_{L^2}+\mathrm{1}_{2s=1} W^\ep_{q-s+\log}(D)h|_{L^2} \\
 &&\quad+\mathrm{1}_{2s<1}|W^\ep_{q-s}(D) h|_{L^2})| W^\ep_{2s}(D)f|_{L^2}
+|W^\ep_{q+s-1}(D)g|_{L^2}|h|_{L^2}|W^\ep_{q+s}(D)f|_{L^2}
 +|g|_{L^1}|h|_{L^2}|f|_{L^2}.
\eeno
  This ends the proof of the first result of the lemma. The second result can be obtained by the similar argument and we skip the details.
\end{proof}

 \subsection{Trilinear estimates in $\TT^3\times\R^3$}

We will give the estimates to the nonlinear terms involved in the energy estimates. We begin with a useful lemma:
\begin{lem} \label{dsums}
Suppose $0<\delta_2\ll1$ and $a,b\ge0$ with $a+b=\f32+\delta_2$. Then we have
\ben\label{dsum1} |\sum_{k\in\ZZ^3}\sum_{p\in\ZZ^3}
    A_pB_{k-p}C_k|\lesssim \big(\sum_{p\in \ZZ^3} |p|^{2a} A_p^2\big)^{\f12}\big(\sum_{p\in \ZZ^3} |p|^{2b} B_p^2\big)^{\f12}\big(\sum_{p\in \ZZ^3} C_p^2\big)^{\f12}. \een
As an application, suppose
\beno \mathcal{I}&\eqdefa&\sum_{k\in\ZZ^3}\sum_{p\in\ZZ^3}
(|p|^{a} A_p)( |k-p|^{b}B_{k-p})C_k,
\eeno
 where $a,b\ge0$ with $a+b=m+n\varrho$. For  $0\le \delta_3\le a$,   we have
\begin{enumerate}
\item  Case1: $m+n\varrho\le \f32$.   It holds
  \beno |\mathcal{I}|&\le&  \left\{\begin{aligned}  &\big(\big(\sum_{p\in \ZZ^3} |p|^{2(\f32+\delta_2+\delta_3)} A_p^2\big)^{\f12}\big(\sum_{p\in \ZZ^3} |p|^{2(m+n\varrho-\delta_3)} B_p^2\big)^{\f12}\big(\sum_{p\in \ZZ^3} C_p^2\big)^{\f12},\\
   &\big(\sum_{p\in \ZZ^3} |p|^{2(m+n\varrho)} A_p^2\big)^{\f12}\big(\sum_{p\in \ZZ^3} |p|^{2(\f32+\delta_2)} B_p^2\big)^{\f12}\big(\sum_{p\in \ZZ^3} C_p^2\big)^{\f12}.\end{aligned} \right.
\eeno

\item Case 2: $m+n\varrho> \f32$.  For $\tilde{\delta}\le \f32+\delta_2$, it holds

\beno|\mathcal{I}|&\le&\big(\sum_{p\in \ZZ^3} |p|^{2(\f32+\delta_2+\delta_3)} A_p^2\big)^{\f12}\big(\sum_{p\in \ZZ^3} |p|^{2(m+n\varrho-\delta_3)} B_p^2\big)^{\f12}\big(\sum_{p\in \ZZ^3} C_p^2\big)^{\f12}\\&&+\big(\sum_{p\in \ZZ^3} |p|^{2(m+n\varrho+\tilde{\delta})} A_p^2\big)^{\f12}\big(\sum_{p\in \ZZ^3} |p|^{2(\f32+\delta_2-\tilde{\delta})} B_p^2\big)^{\f12}\big(\sum_{p\in \ZZ^3} C_p^2\big)^{\f12}.\eeno
\end{enumerate}
\end{lem}
\begin{proof}
It is easy to check that \beno \sum_{k\in\ZZ^3}\sum_{p\in\ZZ^3}
    A_pB_{k-p}C_k=(\sum_{2|p|\le|k|}+ \sum_{2|k|\le|p|}+\sum_{|k|/2<|p|<2|k|})A_pB_{k-p}C_k.\eeno
In the case of $2|p|\le|k|$, we have $|p|\lesssim |k|\sim|k-p|$ which implies that
\beno |\sum_{2|p|\le|k|}A_pB_{k-p}C_k|&\lesssim& \sum_{2|p|\le|k|} |p|^{-\f32-\delta_2} (|p|^a |A_p|)(|k-p|^{b}|B_{k-p}|)|C_k|\\&\lesssim&
\big(\sum_{p\in \ZZ^3} |p|^{2a} A_p^2\big)^{\f12}\big(\sum_{p\in \ZZ^3} |p|^{2b} B_p^2\big)^{\f12}\big(\sum_{p\in \ZZ^3} C_p^2\big)^{\f12}. \eeno
By the symmetric structure, we can handle the case $2|k|\le|p|$. Next we turn to the case $|k|/2<|p|<2|k|$. In this situation, we have $|k-p|\lesssim|k|\sim|p|$. It yileds
\beno |\sum_{2|p|\le|k|}A_pB_{k-p}C_k|&\lesssim& \sum_{2|p|\le|k|} |k-p|^{-\f32-\delta_2} (|p|^a |A_p|)(|k-p|^{b}|B_{k-p}|)|C_k|\\&\lesssim&
\big(\sum_{p\in \ZZ^3} |p|^{2a} A_p^2\big)^{\f12}\big(\sum_{p\in \ZZ^3} |p|^{2b} B_p^2\big)^{\f12}\big(\sum_{p\in \ZZ^3} C_p^2\big)^{\f12}. \eeno
It completes the proof of \eqref{dsum1}.

The estimates for $\mathcal{I}$  follows   the similar argument.
It ends the proof of the lemma.
\end{proof}

Let us introduce the  translation  and finite difference operators for $x$ variable: 
\ben\label{transop}  && T_k^n h\eqdefa \left\{\begin{aligned} & \ T_kh\eqdefa h(\cdot+k), \quad\mbox{if}\quad n>0;\\
	&  h,\quad\mbox{if}\quad n=0.\end{aligned}\right. \\ &\label{fdop}&\bar{\triangle}^{n\varrho}_k h\eqdefa \left\{\begin{aligned} & \big(T_kh(t,x,v)-h(t,x,v)\big)|k|^{-3/2-n\varrho}, \quad\mbox{if}\quad n>0;\\
	&  h,\quad\mbox{if}\quad n=0.\end{aligned}\right.  
\een
Combining the lower bounds in Corollary \ref{basic-estimates} and Lemma \ref{dsums}, we have

\begin{cor}\label{basic-estimates1} Suppose $g, h$ and $f$ are smooth functions and  let $\delta_2\ll1$ and $|\alpha|=m$ with $m\in \N$.  If $g$ is a non-negative function verifying \eqref{lbcondi}, we   have the coercivity estimates:
\beno   &&\sum_{|\alpha|=m}\int_{\TT^6} \lr{Q^\eps(g,\bar{\triangle}^{n\varrho}_k\pa_x^\alpha h) W_{m,n},  \bar{\triangle}^{n\varrho}_k\pa_x^\alpha h) W_{m,n}}_vdxdk
\lesssim -\f13\int_{\TT^3}\mathcal{E}_{g}^\gamma(W_{m,n}  |D_x|^{m+n\varrho} h )dx-\mathcal{C}_1(c_1,c_2)\\&& \times(\|W^\ep_s(D)W_{\gamma/2}W_{m,n}   |D_x|^{m+n\varrho} h\|_{L^2}^2
+\int_{\TT^3}\mathcal{E}_{\mu}^{0,\eps}(W_{\gamma/2}W_{m,n}   |D_x|^{m+n\varrho} h)dx)+
\mathcal{C}_2(c_1,c_2)\|W_{\gamma/2}W_{m,n}  h\|_{H^{m+n\varrho}_xL^2}^2\\&&+\|g\|_{H^{\f32+\delta_2}_xL^2_{\gamma+4}}^2\|W_{m,n} h\|_{H^{m+n\varrho}_xL^2}^2 +\mathrm{1}_{m+n\varrho\le\f32}\|W_{\gamma/2}W_{m,n}g\|_{H_x^{m+n\varrho}L^2}^2\|  h\|_{H_x^{\f32+\delta_2}L^2_4}^2\\&&
+\mathrm{1}_{m+n\varrho>\f32}\|W_{\gamma/2}W_{m,n}g\|_{H_x^{\f32+\delta_2}L^2}^2\|  h\|_{H_x^{m+n\varrho}L^2_4}^2.
\eeno
\end{cor}
\begin{proof} For simplicity, we denote $W_{m,n}$ by $W_l$.
We first observe that
\ben\label{gainRegu}\int_{\TT^3}  \f{|e^{iq\cdot k}-1|^2}{|k|^{3+2\varrho}}dk\sim |q|^{2\varrho},\een
which implies that
$ \int_{\TT^3\times\TT^3} (\bar{\triangle}_k^\varrho F-\bar{\triangle}_k^\varrho G)^2 dxdk\sim \sum_{q\in \ZZ^3} |q|^{2\varrho} (\hat{F}(q)-\hat{G}(q))^2
 = ||D_x|^\varrho (F-G)|^2_{L^2_x},$
 where we use the notation that $\hat{f}$ denotes the Fourier transform with respect to $x$ variables.
 Then 
we arrive at
 $\int_{\TT^6} \mathcal{E}^\gamma_g(W_l \bar{\triangle}_k^\varrho h)dx dk  \sim  \int_{\TT^3}  \mathcal{E}^\gamma_g(W_l |D_x|^\varrho  h) dx$,  
$\int_{\TT^3}\|W^\eps_s(D) W_{l}\bar{\triangle}_k^\varrho h\|_{L^2}^2 dk\sim \|W^\eps_s(D) W_{l}|D_x|^\varrho h\|_{L^2}^2$ and $ 
\int_{\TT^6}  \mathcal{E}^{0,\eps}_\mu(W_l \bar{\triangle}_k^\varrho h) dxdk  \sim   \int_{\TT^3} \mathcal{E}^{0,\eps}_\mu(W_l |D_x|^\varrho  h)dx. $

Thanks to these facts, by comparing desired results to the coercivity estimate  \eqref{corbaesti1} obtained in Corollary \ref{basic-estimates}, we only need to take care of the terms:
$|g|^2_{L^1_{\gamma+2s}}|h|_{L^2_{l}}^2$ and $|g|_{L^2_{l+\gamma/2}}^2|h|_{L^1_2}^2$ appearing in \eqref{corbaesti1}. It is easy to check that  the first term can be estimated by
$$\int_{\TT^3}|g|^2_{L^1_{\gamma+2s}}|h|_{L^2_{l}}^2dx\lesssim \|g\|_{H^{\f32+\delta_2}_xL^2_{\gamma+4}}^2\|h\|_{L^2_l}^2.$$
Notice that the second  term comes from the estimates for
$I_2^2$ and $II_2$ in the proof of Lemma \ref{commforweight}. Thus to get the desired results, we should first apply Plancherel formula with respect to $x$ variable to $I_2^2$ and $II_2$ and then use the upper bounds for  $v$ variable. Then we get the desired results thanks to Lemma \ref{dsums} and ends the proof of the corollary.
\end{proof}

 Moreover, we have the following upper bounds:
 \begin{cor}\label{basic-estimates3}  Suppose $g, h$ and $f$ are smooth functions and  let $\delta_2\ll1$ and $\alpha=\alpha_1+\alpha_2$ with $|\alpha|=m\in \N$.   
 	
 (i).	If $0\le \delta_3\le m-|\alpha_2|$ and $\tilde{\delta}\le \f32+\delta_2$, then
 \beno   &&\big|\int_{\TT^3}\int_{\TT^3}\lr{Q^\eps(\pa^{\alpha_1}_xg,\pa^{\alpha_2}_x\bar{\triangle}^{n\varrho}_kh)W_{m,n},  \bar{\triangle}^{n\varrho}_kf}_v dxdk\big| \lesssim
  \big(\|g\|_{H_x^{\f32+\delta_2+\delta_3}L^2_{\gamma+4}}\|W^\ep_s(D)W_{m,n}W_{\gamma/2+2s} h\|_{H_x^{m+n\varrho-\delta_3}L^2}\\&&+
   \mathrm{1}_{m+n\varrho>\f32}\|g\|_{H_x^{m+n\varrho+\tilde{\delta}}L^2_{\gamma+4}} \|W^\ep_s(D)W_{m,n}W_{\gamma/2+2s} h\|_{H_x^{\f32+\delta_2-\tilde{\delta}}L^2} +\mathrm{1}_{m+n\varrho>\f32} \|W_{m,n}W_{\gamma/2}g\|_{H^{\f32+\delta_2+\delta_3}} \\&&\times \|  h\|_{H^{m+n\varrho}_xL^2_4}
 +\|W_{m,n}W_{\gamma/2}g\|_{H^{m+n\varrho}_xL^2} \|h\|_{H^{\f32+\delta_2}_xL^2_4}\big)\|W^\ep_s(D)W_{\gamma/2} f\|_{H^{n\varrho}_xL^2},
 \eeno
 
 (ii). If $0\le \delta_3\le m+n\varrho-|\alpha_2|$ and $\tilde{\delta}\le \f32+\delta_2$, then \beno
 &&|\int_{\TT^3}\int_{\TT^3}\lr{Q^\eps(\pa^{\alpha_1}_x\bar{\triangle}^{n\varrho}_kg,T^n_k\pa^{\alpha_2}_xh)W_{m,n},  \bar{\triangle}^{n\varrho}_kf}_v dxdk|
 \lesssim \big(\|g\|_{H_x^{\f32+\delta_2+\delta_3}L^2_{\gamma+4}}\|W^\ep_s(D)W_{m,n}W_{\gamma/2+2s} h\|_{H_x^{m+n\varrho-\delta_3}L^2}\\&&+ \mathrm{1}_{m+n\varrho>\f32}\|g\|_{H_x^{m+n\varrho+\tilde{\delta}}L^2_{\gamma+4}} \|W^\ep_s(D)W_{m,n}W_{\gamma/2+2s} h\|_{H_x^{\f32+\delta_2-\tilde{\delta}}L^2} +\mathrm{1}_{m+n\varrho>\f32} \|W_{m,n}W_{\gamma/2}g\|_{H^{\f32+\delta_2+\delta_3}}\\&&\times \|  h\|_{H^{m+n\varrho}_xL^2_4}
 +\|W_{m,n}W_{\gamma/2}g\|_{H^{m+n\varrho}_xL^2} \|h\|_{H^{\f32+\delta_2}_xL^2_4}\big)\|W^\ep_s(D)W_{\gamma/2} f\|_{H^{n\varrho}_xL^2}.
 \eeno
 \end{cor}

\begin{proof}
Thanks to the Plancherel theorem, we have
$ \int_{\TT^3} g(x)h(x)f(x)dx=\sum_{q\in\ZZ^3}\sum_{p\in\ZZ^3} \hat{g}(p) \hat{h} (q-p)\bar{\hat{f}}(q), $ where
$\hat{f}$ denotes the Fourier transform with respect to $x$ variables.
Then 
   we get \beno   &&\int_{\TT^3}\lr{Q^\eps(\pa^{\alpha_1}_xg,\pa^{\alpha_2}_x\bar{\triangle}^{n\varrho}_kh)W_{m,n},  \bar{\triangle}^{n\varrho}_kf}_v dx\\
   &&= \sum_{q\in\ZZ^3}\sum_{p\in\ZZ^3} \bigg(\int_{\TT^3}(\mathrm{1}_{n\neq0}\f{(e^{i(q-p)k}-1)(e^{iqk}-1)}{|k|^{3+2n\varrho}}+\mathrm{1}_{n=0})dk\bigg)
   p^{\alpha_1}(q-p)^{\alpha_2}  \lr{Q^\eps(\hat{g}(p),\hat{h} (q-p))W_{m,n}, \bar{\hat{f}}(q)}_v\eeno
By $(i)$ of Corollary \ref{basic-estimates} and \eqref{gainRegu},    we deduce that
   \beno   &&\big|\int_{\TT^3}\int_{\TT^3}\lr{Q^\eps(\pa^{\alpha_1}_xg,\pa^{\alpha_2}_x\bar{\triangle}^{n\varrho}_kh)W_{m,n},  \bar{\triangle}^{n\varrho}_kf}_v dxdk\big|\\&&\lesssim\sum_{q\in\ZZ^3}\sum_{p\in\ZZ^3}
   |p|^{|\alpha_1|}|q-p|^{|\alpha_2|+n\varrho} |q|^{n\varrho}\big(  |\hat{g}(p)|_{L^2_{\gamma+4}} |W^\ep_s(D)W_{m,n}W_{\gamma/2+2s} \hat{h} (q-p)|_{L^2}|W^\ep_s(D)W_{\gamma/2} \hat{f} (q)|_{L^2} \\&&\quad+|W_{m,n}W_{\gamma/2}\hat{g}(p)|_{L^2}|\hat{h}(q-p)|_{L^2_4}|\hat{f}(q)|_{L^2_{\gamma/2}}\big) \eqdefa \mathcal{S}.\eeno
Then due to Lemma \ref{dsums},
if $m+n\varrho\le\f32$, we have  
\beno
\mathcal{S}&\lesssim& \big(
\|g\|_{H_x^{\f32+\delta_2+\delta_3}L^2_{\gamma+4}}\|W^\ep_s(D)W_{m,n}W_{\gamma/2+2s} h\|_{H_x^{m+n\varrho-\delta_3}L^2}\\&&\quad   +\|W_{m,n}W_{\gamma/2}g\|_{H^{m+n\varrho}_xL^2} \|h\|_{H^{\f32+\delta_2}_xL^2_4}\big)\|W^\ep_s(D)W_{\gamma/2} f\|_{H^{n\varrho}_xL^2}.
\eeno
And if $m+n\varrho>\f32$, we have
\beno
&&\mathcal{S}\lesssim \big(\|g\|_{H_x^{m+n\varrho+\tilde{\delta}}L^2_{\gamma+4}} \|W^\ep_s(D)W_{m,n}W_{\gamma/2+2s} h\|_{H_x^{\f32+\delta_2-\tilde{\delta}}L^2}+\|g\|_{H_x^{\f32+\delta_2+\delta_3}L^2_{\gamma+4}}\\&&\times\|W^\ep_s(D)W_{m,n}W_{\gamma/2+2s} h\|_{H_x^{m+n\varrho-\delta_3}L^2}+ \|W_{m,n}W_{\gamma/2}g\|_{H^{\f32+\delta_2}} \|  h\|_{H^{m+n\varrho}_xL^2_4}
\\&&+\|W_{m,n}W_{\gamma/2}g\|_{H^{m+n\varrho}_xL^2} \|h\|_{H^{\f32+\delta_2}_xL^2_4}\big)\|W^\ep_s(D)W_{\gamma/2} f\|_{H^{n\varrho}_xL^2}.
\eeno
Combining estimates for $\mathcal{S}$ in the above, we get $(i)$. The second result can be derived in a similar way.
 We ends the proof of the upper bounds.
\end{proof}

\begin{cor}\label{basic-estimates2}  Suppose $g, h$ and $f$ are smooth functions and  let $\delta_2\ll1, a\ge0$ and $\alpha=\alpha_1+\alpha_2$ with $|\alpha|=m\in \N$.  
	
	(i). If $0\le \delta_3\le m-|\alpha_2|$, then
 \beno &&\|Q^\eps(\pa^{\alpha_1}_xg, \bar{\triangle}^{n\varrho}_k\pa^{\alpha_2}_xh)W_{m,n+1}W_{\gamma/2+a}\|_{L^2_{x,k}H^{-1}_v}^2
\lesssim\mathrm{1}_{m+n\rho>\f32}\|W_{m,n+1}W_{\f32\gamma+a} g\|^2_{H_x^{\f32+\delta_2}L^2}\|  h\|^2_{H_x^{m+n\varrho}L^2_4}
\\&&\qquad +\|g\|^2_{H_x^{\f32+\delta_2+\delta_3}L^2_{\gamma+4}}\|W^\ep_s(D)W_{m,n+1}W_{\f32\gamma+2s+a} h\|^2_{H_x^{m+n\varrho-\delta_3}L^2} +\mathrm{1}_{m+n\rho>\f32}\|g\|^2_{H_x^{m+n\varrho}L^2_{\gamma+4}}\\&&\qquad \times\|W^\ep_s(D)W_{m,n+1}W_{\f32\gamma+2s+a} h\|^2_{H_x^{\f32+\delta_2}L^2}
+  \|W_{m,n+1}W_{\f32\gamma+a}g\|^2_{H^{m+n\varrho}_xL^2} \|h\|^2_{H^{\f32+\delta_2}_xL^2_4}.\eeno   
 (ii). If $0\le \delta_3\le m+n\varrho-|\alpha_2|$, then
\beno
&& \|Q^\eps(\bar{\triangle}^{n\varrho}_k\pa^{\alpha_1}_xg, T_k^n\pa^{\alpha_2}_xh)W_{m,n+1}W_{\gamma/2+a}\|_{L^2_{x,k}H^{-1}_v}^2
\lesssim\mathrm{1}_{m+n\rho>\f32}\|W_{m,n+1}W_{\f32\gamma+a} g\|^2_{H_x^{\f32+\delta_2}L^2}\|  h\|^2_{H_x^{m+n\varrho}L^2_4}
\\&&\qquad+\|g\|^2_{H_x^{\f32+\delta_2+\delta_3}L^2_{\gamma+4}}\|W^\ep_s(D)W_{m,n+1}W_{\f32\gamma+2s+a} h\|^2_{H_x^{m+n\varrho-\delta_3}L^2}+\mathrm{1}_{m+n\rho>\f32}\|g\|^2_{H_x^{m+n\varrho}L^2_{\gamma+4}}\\&& \qquad\times\|W^\ep_s(D)W_{m,n+1}W_{\f32\gamma+2s+a} h\|^2_{H_x^{\f32+\delta_2}L^2}
+  \|W_{m,n+1}W_{\f32\gamma+a}g\|^2_{H^{m+n\varrho}_xL^2} \|h\|^2_{H^{\f32+\delta_2}_xL^2_4}.\eeno
 (iii). If $\alpha_1=0$, then
 \beno &&\|Q^\eps(g, \bar{\triangle}^{n\varrho}_k\pa^{\alpha}_xh)W_{m,n+1}W_{\gamma/2+a}\|_{L^2_{x,k}H^{-1}_v}^2 \lesssim   (\eta_1^2+\epsilon^{2(1-s)} \|g\|^2_{H_x^{\f32+\delta_2}L^2_{\gamma+4}})\times\|W^\ep_s(D)W_{m,n+1}W_{\f32\gamma+2s+a}h\|^2_{H_x^{m+n\varrho}L^2} \\ &&\quad + \eta_1^{-\f{2(2s-1)}{1-s}} \|g\|^{\f{2(2s-1)}{1-s}}_{H_x^{\f32+\delta_2}L^2_{\gamma+4}}\|W_{m,n+1}W_{\f32\gamma+2s+a}h\|^2_{H_x^{m+n\varrho}L^2}
  +
 \|g\|^2_{H_x^{\f32+\delta_2}L^2_{\gamma+4}}\| W_{m,n+1}W_{\f32\gamma+2s+a} h\|^2_{H_x^{m+n\varrho}L^2}\\&&\quad +\mathrm{1}_{m+n\varrho>\f32}\|W_{m,n+1}W_{\f32\gamma+a} g\|^2_{H_x^{\f32+\delta_2}L^2}\|  h\|^2_{H_x^{m+n\varrho}L^2_4}+\|W_{m,n+1}W_{\f32\gamma+a} g\|^2_{H_x^{m+n\varrho}L^2}\|  h\|^2_{H_x^{\f32+\delta_2}L^2_4}.  \eeno
\end{cor}

\begin{proof} We prove the desired results by duality.  In fact,  by $(ii)$ of Corollary \ref{basic-estimates}, we have  \beno
&&\bigg| \int_{\TT^3}\int_{\TT^3}\lr{Q^\eps(\pa^{\alpha_1}_xg, \bar{\triangle}^{n\varrho}_k\pa^{\alpha_2}_xh)W_{m,n+1}W_{\gamma/2+a},   f}_v dxdk\bigg|\\
   &&=\bigg| \sum_{q\in\ZZ^3}\sum_{p\in\ZZ^3} \bigg(\int_{\TT^3}(\mathrm{1}_{n\neq0}\f{(e^{i(q-p)k}-1) }{|k|^{3/2+n\varrho}}+\mathrm{1}_{n=0})dk\bigg)
   p^{\alpha_1}(q-p)^{\alpha_2}  \lr{Q^\eps(\hat{g}(p),\hat{h} (q-p))W_{m,n+1}W_{\gamma/2+a}, \bar{\hat{f}}(q)}_v\bigg|   \\&&\lesssim  \sum_{q\in\ZZ^3}\sum_{p\in\ZZ^3} |p|^{\alpha_1}|q-p|^{\alpha_2+n\varrho} \big(|\hat{g}(p)|_{L^2_{\gamma+4}}|W^\ep_s(D)W_{m,n+1}W_{\gamma/2+a}W_{\gamma+2s}\hat{h}(q-p)|_{L^2}
\\&&\quad+|W_{m,n+1}W_{\gamma/2+a}\hat{g}(p)|_{L^2_{\gamma}}|\hat{h}(q-p)|_{L^2_4}\big)|W^\ep_s(D)\hat{f}(q)|_{L^2_{k,v}}.
\eeno
   Then due to  Lemma \ref{dsums}, we get the first result. The last two results can be obtained in a similar way thanks to $(ii)$ and $(iii)$ of Corollary \ref{basic-estimates} and Lemma \ref{dsums}.  We skip the details here and complete the proof of the corollary.
 \end{proof}

\subsection{Hypo-elliptic estimate for the transport equation}

In this subsection we will show the hypo-elliptic estimates for the transport equation which reads:
\begin{equation}
\partial_t f(t,x,v)+v\cdot\nabla_x f(t,x,v)=g(t,x,v)\label{treq}.
\end{equation}

\begin{lem}\label{hypoellipticity}
Let $f\in L^2([0,T]\times\mathbb{T}^3\times\mathbb{R}^3)$ be a
solution of the transport equation \eqref{treq}. Suppose that $g\in L^2([0,T]\times\mathbb{T}^3;H^{-1}(\mathbb{R}^3_v))$. If we further assume that $f_\phi\in
L^2([0,T]\times\mathbb{T}^3;H^s(\mathbb{R}^3_v))$ for some
$0<s<1$, then  for any $l<-\frac{3}{2}$ and $\eta>0$, we have 
\begin{equation}
\int_0^T \|W_l f_\phi\|^2_{H^{\frac{s}{4(4+s)}}_xL^2_v}dt\lesssim\eta^{-8} \|f|_{t=0}\|_{L^2}^2+\eta^{2s}
\|f_\phi\|^2_{L^2([0,T]; L^2_xH^s_v)}+\eta^{-8}(\|g\|_{L^2([0,T]; L^2_xH^{-1}_v)}^2+\|f\|_{L^2([0,T]; L^2)}^2).
\label{xre}
\end{equation}
\end{lem}

\begin{proof}  Recall  that $
T_k f_\phi(t,x,v)=f_\phi(t,x+k,v)$ and 
 $
\bar{\triangle}_k^{\varrho} f_\phi(t,x,v)=\big(T_k f_\phi(t,x,v)-f_\phi(t,x,v)\big)|k|^{-3/2-\varrho},
$ with $\varrho=\frac{s}{4(4+s)}$.
Using \eqref{gainRegu}, we observe that
\begin{eqnarray}
W_l f_\phi\in L^2([0,T];H^{\varrho}_xL^2_v)&\Leftrightarrow 
&\|W_lf_\phi\|_{L^2([0,T];L^2)}^2+\int_{0}^T\int_{\mathbb{R}^3}\int_{\mathbb{T}^3}\int_{\mathbb{T}^3}\langle
v\rangle^{2l}|\bar{\triangle}_k^{\varrho}
f_\phi|^2dtdvdxdk<+\infty
\nonumber\\
&\Leftrightarrow &\|W_lf_\phi\|_{L^2([0,T];L^2)}^2+\int_{0}^T\int_{\mathbb{R}^3}\langle
v\rangle^{2l}\left(\sum_{m\in\mathbb{Z}^3}|m|^{2\varrho}\left|\hat{f_\phi}(t,m,v)\right|^2\right)
dtdv<+\infty,\label{equv}
\end{eqnarray}
where $\hat{f_\phi}(t,m,v)$ is the Fourier transform of $f_\phi$ with respect to  $x$ variable.
\par
We now turn to prove \eqref{equv}. Let $\chi (v)=\mathfrak{F}_v^{-1}(1-\phi)$ where
$\mathfrak{F}_v$ and $\mathfrak{F}_v^{-1}$ are the Fourier transform for $v$ variable and its inverse.  Then it is easy to check $\int_{\mathbb{R}^3}\chi(v)dv=1$. For any
$\eta>0$, we denote the regularizing sequence $\chi_{\eta}$
by
$\chi_{\eta}(v)=\eta^{-3}\chi\left(\frac{v}{\eta}\right)$
and write
\begin{equation}
\hat f_\phi(t,m,v)=[\hat f_\phi(t,m,v)- (\hat f_\phi(t,m,\cdot)\ast_v
\chi_{\eta})(v)]+ (\hat f_\phi(t,m,\cdot)\ast_v \chi_{\eta })
(v).\label{dec}
\end{equation}
We point out   that $\eta$   will be
chosen later and it will depend on $|m|$.
\par
We use Minkowski and Cauchy-Schwartz inequalities to get
\begin{eqnarray*}
&&\int_{\mathbb{R}^3}\langle v\rangle^{2l}|\hat f_\phi(t,m,v)- (\hat f_\phi(t,m,\cdot)\ast_v \chi_{\eta})(v)|^2 dv 
\lesssim \int_{\mathbb{R}^3}\left|\int_{\mathbb{R}^3}[\hat f_\phi(t,m,v)-\hat f_\phi(t,m,v-u)]\,\chi_{\eta }(u)du\right|^2 dv\\
&\lesssim& \left(\int_{\mathbb{R}^3}\left(\int_{\mathbb{R}^3}|\hat f_\phi(t,m,v)-\hat f_\phi(t,m,v-u)|^2dv\right)^{1/2}\chi_{\eta }(u)du\right)^2 \\
&\lesssim& \left(\int_{\mathbb{R}^3}\chi_{\eta}^2(u)|u|^{3+2s}du\right)\left(\int_{\mathbb{R}^6}\frac{|\hat f_\phi(t,m,v)-\hat f_\phi(t,m,v-u)|^2}{|u|^{3+2s}}dudv\right) 
\lesssim \eta
^{2s}\left|\hat{f_\phi}(t,m,\cdot)\right|^2_{H^s}.
\end{eqnarray*}
Then it gives
\begin{eqnarray}
&&\int_{0}^T\int_{\mathbb{R}^3}\langle v\rangle^{2l}\left(\sum_{m\in\mathbb{Z}^3}|m|^{2\varrho}\left|\hat f_\phi(t,m,v)- (\hat f_\phi(t,m,\cdot)\ast_v \chi_{\eta })(v)\right|^2\right)dtdv\nonumber\\
&\lesssim&\int_{0}^T\left(\sum_{m\in\mathbb{Z}^3}|m|^{ 2\varrho}\eta^{2s}\left|\hat{f_\phi}(t,m,\cdot)\right|^2_{H^s}\right)dt.\label{est1}
\end{eqnarray}
\par
For the second term of the right-hand side of \eqref{dec}, we recall
that $g\in L^2([0,T]\times\mathbb{T}^3;H^{-1}(\mathbb{R}^3_v))$
implies that
$
g(t,x,v)=g_0(t,x,v)+\sum_{j=1}^3\partial_{v_j}h_j(t,x,v),
$
where $g_0(t,x,v)=\mathfrak{F}^{-1}[(1+|\xi|)^{-1}\mathfrak{F}
g](t,x,v)$ and $h_j(t,x,v)=-R_j g_0(t,x,v)$, $j=1,2,3$. Here
$R_j$
is the Riesz transform with respect to $v$ variable. Then, one has $g_0,h_j\in
L^2([0,T]\times\mathbb{T}^3\times\mathbb{R}^3)$,\,\,$j=1,2,3.$
Following the proof of (2.16) in  Theorem 2.1 (averaging lemma) of \cite{bd3} and the fact $\hat f_\phi=(1-\phi(\epsilon D)) \hat f=\hat f\ast_v\chi_{\epsilon}$,
we can deduce that
\begin{eqnarray*}
&& \int_{0}^T|(\hat f_\phi(t,m, \cdot)\ast_v\chi_{\eta})(v)|^2dt
=\int_{0}^T|(\hat f(t,m, \cdot)\ast_v(\chi_{\epsilon}\ast\chi_{\eta}))(v)|^2dt\\
&\lesssim&|m|^{-\frac{1}{2}}\left(\|(\chi_{\epsilon}\ast\chi_{\eta})(v-u)(1+|u|^2)\|_{L^{\infty}_u}
+\|\nabla (\chi_{\epsilon}\ast\chi_{\eta})(v-u)(1+|u|^2)\|_{L^{\infty}_u}\right)^2\\
&&\times\left[ |\hat f(0,m, \cdot)|^2_{L^2}+\int_{0}^{T}(|\hat f(t,m, \cdot)|^2_{L^2}+|\hat
g_0(t,m,\cdot)|^2_{L^2}
+\sum_{j=1}^3|\hat
h_j(t,m,\cdot)|^2_{L^2 })dt\right].
\end{eqnarray*}
From the facts that if $\eps\ll\eta$, $\chi_{\epsilon}\ast\chi_{\eta}=\chi_{\eta}$ and if $\eta\ll\eps$, $\chi_{\epsilon}\ast\chi_{\eta}=\chi_{\eps}$,
together with the estimate $\| \chi_{\epsilon}(v-u)(1+|u|^2)\|_{L^{\infty}_u}\lesssim\eps^{-3}(1+|v|^2)$,
we get that
$\|(\chi_{\epsilon}\ast\chi_{\eta})(v-u)(1+|u|^2)\|_{L^{\infty}_u}\lesssim\eta^{-3}(1+|v|^2)$, which yields
\ben
&&\qquad\int_{0}^T\int_{\mathbb{R}^3}\langle v\rangle^{2l}\left(\sum_{m\in\mathbb{Z}^3}|m|^{ 2\varrho}\left|(\hat f_\phi(t,m,\cdot)\ast_v \chi_{\eta })(v)\right|^2\right)dtdv\\
&\lesssim&\sum_{m\in\mathbb{Z}^3}|m|^{2\varrho-\frac{1}{2}}(\eta^{-6}+\eta^{-8})\bigg[ |\hat f(0,m, \cdot)|^2_{L^2}+\int_{0}^{T}(|\hat f(t,m, \cdot)|^2_{L^2}+|\hat
g_0(t,m,\cdot)|^2_{L^2}
 +\sum_{j=1}^3|\hat h_j(t,m,\cdot)|^2_{L^2 })dt\bigg]\nonumber.\label{est2}
\een
\par
Thanks to the fact $\|W_lf_\phi\|_{L^2([0,T];L^2)}\le \|f\|_{L^2([0,T];L^2)}$, we complete the
proof of Lemma \ref{hypoellipticity} if we choose $\eta:=\eta|m|^{-\frac{1}{4(4+s)}}$ in \eqref{est1} and \eqref{est2}.
\end{proof}

\section{ {\it A priori} estimates for the linear equation}

In this section, we will focus on {\it a priori} estimates for the linear equation: 
\begin{equation}\label{EAPLN:h-eqaution} \left\{
    \begin{aligned}
&\partial_t h+v\cdot\nabla_x h=Q^\ep(f,h)+Q^\ep(h,g)\\
&h|_{t=0}=h_0.\end{aligned}\right.
\end{equation}
In what follows, we assume that $f$ and $g$ are smooth and bounded functions.

\subsection{$L^1$ and $L^2$ moment estimates for the equation} In this subsection, we will  give the estimates on the propagation of the moments.

\subsubsection{$L^1$ moment estimates to the equation} We begin with two useful lemmas which are related to the Povnzer's inequality.
\begin{lem}\label{Povnzer1}  Suppose   that $l\in \R^+$, $\ep\le l^{-1/2-a}$ with $a>0$ and $E(\theta)\eqdefa \lr{v}^2\cos^2(\theta/2)+\lr{v_*}^2\sin^2(\theta/2)$
Then there exists a constant $c\ge0$ such that
\ben\label{Povnzer-ineq1}
\int_{\sigma\in \SS^2} \big((E(\theta))^{l/2}-\lr{v}^l\big)b^\epsilon(\cos\theta) d\sigma&\lesssim& -cl^s\lr{v}^l+l^s\big(\lr{v_*}^2\lr{v}^{l-2}+\lr{v}^2\lr{v_*}^{l-2}\big)+l^{-1} \lr{v_*}^l.
\een
\end{lem}

\begin{proof}
We first recall that  Gamma function $\Gamma(x)$ and   Beta function $B(x,y)$ are defined by
$\Gamma(x)=\int_0^\infty t^{x-1}e^{-t}dt$and $ B(x,y)=\int_0^1 t^{x-1}(1-t)^{y-1}dt.$
Observe that for $p\ge1$ and $k_p=[(p+1)/2]$,  
\beno &&\sum_{k=0}^{k_p-1} \frac{\Gamma(p+1)}{\Gamma(k+1)\Gamma(p+1-k)} (x^ky^{p-k}+x^{p-k}y^k) 
 \le (x+y)^p
 \le  \sum_{k=0}^{k_p} \frac{\Gamma(p+1)}{\Gamma(k+1)\Gamma(p+1-k)}  (x^ky^{p-k}+x^{p-k}y^k),  \eeno
where $x,y\ge0$(see \cite{LuMouhot}),  we get that
\beno
&&(E(\theta))^{l/2}-\lr{v}^l\le \sum_{k=1}^{k_{l/2}} \frac{\Gamma(l/2+1)}{\Gamma(k+1)\Gamma(l/2+1-k)} \big[ \big( \lr{v}^2\cos^2(\theta/2)\big)^k\big(\lr{v_*}^2\sin^2(\theta/2)\big)^{l/2-k}\\&&\qquad+\big( \lr{v}^2\cos^2(\theta/2)\big)^{l/2-k}\big(\lr{v_*}^2\sin^2(\theta/2)\big)^k\big]+\big((\cos^2(\theta/2))^{l/2}-1\big)\lr{v}^l+(\sin^2(\theta/2))^{l/2}\lr{v_*}^l\\
&&\qquad\eqdefa K_1+K_2+K_3.
\eeno

Suppose $I_i\eqdefa\int_{\sigma\in\SS^2} K_i b^\epsilon d\sigma.$   For the term $I_1$, by direct calculation, we have
\beno  &&I_1\lesssim \sum_{k=1}^{k_{l/2}}\frac{\Gamma(l/2+1)}{\Gamma(k+1)\Gamma(l/2+1-k)} \int_0^{\pi/2}
[(\cos^2(\theta/2))^k(\sin^2(\theta/2))^{l/2-k}\\&&\qquad+(\cos^2(\theta/2))^{l/2-k}(\sin^2(\theta/2))^{k}]\sin\theta^{-1-2s} d\theta
\big(\lr{v_*}^2\lr{v}^{l-2}+\lr{v}^2\lr{v_*}^{l-2}\big)\\&&\lesssim\bigg(\sum_{k=1}^{k_{l/2}}\frac{\Gamma(l/2+1)}{\Gamma(k+1)\Gamma(l/2+1-k)}  \int_0^{\f12}
[(1-t)^{k-s-1}t^{l/2-k-s-1}+t^{k-s-1}(1-t)^{l/2-k-s-1}]dt\bigg)  \big(\lr{v_*}^2\lr{v}^{l-2}\\
&&+\lr{v}^2\lr{v_*}^{l-2}\big)\lesssim \bigg(\sum_{k=1}^{k_{l/2}}\frac{\Gamma(l/2+1)}{\Gamma(k+1)\Gamma(l/2+1-k)} \frac{\Gamma(k-s)\Gamma(l/2-k-s)}{\Gamma(l/2-2s)}\bigg)\big(\lr{v_*}^2\lr{v}^{l-2}+\lr{v}^2\lr{v_*}^{l-2}\big)\\
&& \lesssim l^s  \big(\lr{v_*}^2\lr{v}^{l-2}+\lr{v}^2\lr{v_*}^{l-2}\big).
\eeno
 We remark that in the above inequalities we use the facts
\beno  \int_0^1 t^k(1-t)^{p}dt=B(k+1,p+1)=\frac{\Gamma(k+1)\Gamma(p+1)}{\Gamma(k+p+2)} \mbox{
and}\, B(x,y)\sim \sqrt{2\pi}\f{x^{x-\f12}y^{y-\f12}}{(x+y)^{x+y-\f12}},\eeno  which can be derived by Stirling's approximation.

Note that $(\cos^2(\theta/2))^{l/2}\le 1$ and the condition $\ep\le l^{-1/2-a}$. For sufficiently small $\eta>0$,  we get
\beno I_2&\lesssim& \int_{\theta\sim \eta l^{-\f12}} \big((\cos^2(\theta/2))^{l/2}-1\big) \sin\theta^{-1-2s} d\theta \lr{v}^l \lesssim \bigg[-\int_{\theta\sim \eta l^{-\f12}} (\sin(\theta/2))^2(\sin\theta)^{-1-2s} d\theta \\&&+\sum_{k=2}^{[l/2]}
\frac{\Gamma(l/2+1)}{\Gamma(k+1)\Gamma(l/2+1-k)} \int_{\theta\sim \eta l^{-\f12}}(\sin(\theta/2))^{2k}(\sin\theta)^{-1-2s}d\theta\bigg] \lr{v}^l
\lesssim -C\eta^{2-2s}l^s(1-C\eta^2)  \lr{v}^l.\eeno
Finally, it is easy to check that
$  I_3\lesssim \int_0^{\pi/2} (\sin(\theta/2))^{l-1-2s}d\theta \lr{v_*}^l\lesssim l^{-1} \lr{v_*}^l.$

Putting together the estimates for $I_i$(i=1,2,3), we are led to the desired result.
\end{proof}

\begin{lem}\label{Povnzer2} Under the conditions in Lemma \ref{Povnzer1},  it holds
\beno  \int_{\sigma\in \SS^2} (\lr{v'}^l-\lr{v}^l)b^\epsilon(\cos\theta) d\sigma
&\lesssim& -l^s\lr{v}^l+l^s\big(\lr{v_*}^2\lr{v}^{l-2}+\lr{v}^2\lr{v_*}^{l-2}\big)+l^{-1} \lr{v_*}^l\\&&
+l^2a_l \big(\lr{v_*}^2+\lr{v}^{2}\big)^{l/2-2}\big(|v_*|^2|v|^2-(v_*\cdot v)^2\big),
\eeno
where $a_l\eqdefa \int_0^{\pi/2} \big(\int_0^1 t(1-\f{t}4\sin^2\theta)^{l/2-2}dt\big)b^\ep(\cos\theta)\sin^3\theta d\theta$.
\end{lem}

\begin{proof}
We follow the notations used in \cite{LuMouhot} to set $\mathbf{h}= (v+v_*)/|v+v_*|$, $\mathbf{n}=(v-v_*)/|v-v_*|$ and $\mathbf{j}=\frac{\mathbf{h}-(\mathbf{h}\cdot\mathbf{n})\mathbf{n}}{\sqrt{1-(\mathbf{h}\cdot\mathbf{n})}}$. Then we have
$ \sigma=\cos\theta \mathbf{n}+\sin\theta \omega,$ with $\omega\in \SS^1(\mathbf{n})$.  This implies that
$\mathbf{h}\cdot\sigma=(\mathbf{h}\cdot\mathbf{n})\cos\theta+\sqrt{1-(\mathbf{h}\cdot\mathbf{n})}\sin\theta(\mathbf{j}\cdot \omega). $

Thanks to the $\sigma$-representation \eqref{sigma-rep}, we derive that
$  \lr{v'}^2=E(\theta)+\sin\theta(\mathbf{j}\cdot\omega)\tilde{h}$ and $ \lr{v_*'}^2=E(\pi-\theta)-\sin\theta(\mathbf{j}\cdot\omega)\tilde{h}
$
where $\tilde{h}=\sqrt{|v_*|^2|v|^2-(v_*\cdot v)^2}$.
By Taylor expansion, it yields that
\beno \lr{v'}^l&=&\big(E(\theta)+\sin\theta(\mathbf{j}\cdot\omega)\tilde{h}\big)^{l/2}=(E(\theta))^{l/2}+\f{l}2(E(\theta))^{\f{l}2-1}\tilde{h}\sin\theta(\mathbf{j}\cdot\omega)\nonumber\\&&+\f{l}2\f{l-2}{2}\int_0^1 (1-t)\big(E(\theta)+t\tilde{h}\sin\theta (\mathbf{j}\cdot\omega)\big)^{\f{l}2-2}dt
\tilde{h}^2\sin^2\theta(\mathbf{j}\cdot\omega)^2.\eeno
Since $E(\theta)+t\tilde{h}\sin\theta (\mathbf{j}\cdot\omega)\le E(\theta)+tE(\pi-\theta)\le (\lr{v_*}^2+\lr{v}^2)(1-\f{1-t}2\sin^2\theta)$,   we deduce that
\beno  \lr{v'}^l-\lr{v}^l&\le& (E(\theta))^{l/2}-\lr{v}^l+\f{l}2(E(\theta))^{\f{l}2-1}\tilde{h}\sin\theta(\mathbf{j}\cdot\omega)\\&&+l^2 \bigg(\int_0^1 t(1-\f{t}4\sin^2\theta)^{l/2-2}dt \bigg)\big(\lr{v_*}^2+\lr{v}^{2}\big)^{l/2-2}\tilde{h}^2(\mathbf{j}\cdot\omega)^2 \sin^2\theta. \eeno
Now using  Lemma \ref{Povnzer1} and the fact $\int_{\SS^1(\mathbf{n})} (\mathbf{j}\cdot \omega) d\omega=0$, we arrive at the result.
\end{proof}

Now we are in a position to state the estimates for the moments.
\begin{lem}\label{L1md} Let $l\in \R^+$  verify $l>N_s+2$(recalling that $N_s=\f{2s}{1-s}$) and $\ep\le l^{-1/2-a}$ with $a>0$.  Assume that $h$ is a solution to  \eqref{EAPLN:h-eqaution}. Then there exits a  universal constant $c$ such that  for $\eta,\eta_1>0$,
\beno  &&\f{d}{dt} \|h\|_{L^1_l}+cl^s\int f_*|h|\lr{v}^l |v-v_*|^\gamma dv_*dvdx
\lesssim   \int \bigg(2^l\big(|h|_{L^1_{l+\gamma-2}}|f|_{L^1_{\gamma+2}} +
|h|_{L^1_{\gamma+2}} |f|_{L^1_{l+\gamma-2}}\big)
+l^{-1} |f|_{L^1_{l+\gamma}}|h|_{L^1_\gamma} \\
 &&+\eta^{-2s}(|h|_{L^1_\gamma}  |W^\eps_{2s+\eta_1}(D)g |_{L^2_{l+\gamma+2+N_s}}+|W^\eps_{2s+\eta_1}(D)g|_{L^2_{l+2}}|h|_{L^1_{N_s+\gamma}})+|h|_{L^1_{l+\gamma}}(\eta^{2-2s}|W^\eps_{2s+\eta_1}(D)g |_{L^2_{2}}
 \\&&+l^{-1}|W^\eps_{\eta_1}(D)g|_{L^2_{2}})+2^l(|h|_{L^1_{\gamma+2}}|W^\eps_{\eta_1}(D)g|_{L^2_{l+\gamma}}+|h|_{L^1_{l+\gamma-2}}|W^\eps_{\eta_1}(D)g|_{L^2_{\gamma+4}})+l^{-1}|W^\eps_{\eta_1}(D)g|_{L^2_{\gamma+2}}|h|_{L^1_l}\bigg)dx.\eeno
\end{lem}

\begin{proof}
By multiplying both sides of \eqref{EAPLN:h-eqaution}  by $(\sgn h)\lr{v}^l$ and integrating over $v$ and $x$, we
obtain that
\begin{equation}\label{E:L1}
\frac{d}{dt}\|h\|_{L^1_l}=\iint (\sgn h)\lr{v}^l Q^\ep(f,h)  dvdx+\iint (\sgn h)\lr{v}^l
Q^\ep(h,g) dvdx.
\end{equation}
Observing that $\iint (\sgn h)\lr{v}^l Q^\ep(f,h)  dvdx\le
 \iiiint f_* |h|(\lr{v'}^l-\lr{v}^l)B^\ep d\sigma dv_*dvdx,
$
we get that 
\ben\label{E:L1-upper-bound}
 \frac{d}{dt}\|h\|_{L^1_l}&\leq & \int f_* |h|(\lr{v'}^l-\lr{v}^l)B^\ep d\sigma dv_*dvdx
 +\int h_*g((\sgn h') \lr{v'}^l-(\sgn h)\lr{v}^l) B^\ep d\sigma dv_*dvdx\notag\\
 &\eqdefa& I+II.
\een
We will give estimates for $I$ and $II$ term by term.

{\it  Step 1: Estimate of  $I$.}
Thanks to Lemma \ref{Povnzer2}, we derive that
\beno  I&\lesssim&  \iiint_{\R^9} f_*|h||v-v_*|^\gamma\bigg(-l^s\lr{v}^l+l^s\big(\lr{v_*}^2\lr{v}^{l-2}+\lr{v}^2\lr{v_*}^{l-2}\big)+l^{-1} \lr{v_*}^l\nonumber\\&&
+l^2a_l \big(\lr{v_*}^2+\lr{v}^{2}\big)^{l/2-2}\big(|v_*|^2|v|^2-(v_*\cdot v)^2\big)\bigg)dv_*dvdx\\
&\lesssim& -l^s\int f_*|h|\lr{v}^l |v-v_*|^\gamma dv_*dvdx+ \int 2^l\big(|h|_{L^1_{l+\gamma-2}}|f|_{L^1_{\gamma+2}} +
|h|_{L^1_{\gamma+2}} |f|_{L^1_{l+\gamma-2}}\big)
+l^{-1} |f|_{L^1_{l+\gamma}}|h|_{L^1_\gamma}dx.
 \eeno

 {\it Step 2:  Estimate of  II. }
We split $II$ into two parts that $II=II_1+II_2$ where
$II_1\eqdefa\int h_*g_\phi(\sgn h '\lr{v'}^l-\sgn h\lr{v}^l) B^\ep  d\sigma dv_*dvdx$ and $
  II_2\eqdefa
 \int h_* g^\phi(\sgn h '\lr{v'}^l-\sgn h\lr{v}^l) B^\ep  d\sigma dv_*dvdx.$

{\it Step 2.1: Estimate of $II_1$}. Set $g_\phi^j\eqdefa \mathfrak{F}_jg_\phi$. For any fixed $v,v_*$, $\eta>0$, we write $B^\ep =B^{\ep}_{>}+B^{\ep}_{\leq}$ where 
$B^\ep_{\le}=B^\ep(1-\phi (\frac{2^j\sin\theta}{\eta/\lr{v-v_*}^{\alpha}}))$ and $B^\ep_{>}=B^\ep \phi (\frac{2^j\sin\theta}{\eta/\lr{v-v_*}^{\alpha}})$ with  $\alpha=(1-s)^{-1}$. 
We remark that $B^{\ep}_{>}$ and $ B^{\ep}_{\leq}$ denote the
kernels with the restriction that the derivation angle $\theta$ is bigger and not bigger than
$ (2^{-j}\eta)/\lr{v-v_*}^{\alpha}$ respectively.  Then we have \beno II_1&=&\sum_{j\le [\log_2 \f1{\epsilon}]+1} \iiiint h_*g^j_\phi(\sgn h '\lr{v'}^l-\sgn h\lr{v}^l) B^\ep  d\sigma dv_*dvdx\\
&=&\sum_{j\le [\log_2 \f1{\epsilon}]+1}\bigg( \iiiint h_*\big(g^j_\phi\sgn h \lr{v}^l\big)'-g^j_\phi\sgn h\lr{v}^l) B^\ep  d\sigma dv_*dvdx
 +\iiiint h_*\big(g^j_\phi-(g^j_\phi)'\big)(\sgn h \lr{v}^l)' \\&&\times B^\ep_{\leq}  d\sigma dv_*dvdx+\iiiint h_* g^j_\phi (\sgn h \lr{v}^l)' B^\ep_{>}  d\sigma dv_*dvdx
 -\iiiint h_*\big (g^j_\phi)' (\sgn h \lr{v}^l)' B^\ep_{>}  d\sigma dv_*dvdx\\
&\eqdefa&\sum_{j\le [\log_2 \f1{\epsilon}]+1}\sum_{i=1}^4 II_{1,i}^j.
\eeno

{\it \underline{Estimate of $II_{1,1}^j$.}}Thanks to \eqref{canclem}, we have
\begin{equation*}
\begin{split}
 II_{1,1}^j&=2\pi\iint (g_\phi^j(\sgn h)\lr{v}^l)  \int_0^{\frac{\pi}{2}}{\Big [}
\frac{1}{\cos^3(\theta/2)}B^{\ep} (\frac{|v-v_*|}{\cos(\theta/2)},\cos\theta)
-B^{\ep} (|v-v_*|,\cos\theta) {\Big ]}\sin\theta d\theta dv_* dv.
\end{split}
\end{equation*}
From the fact
 $ \bigg|\frac{1}{\cos^3(\theta/2)}B^{\ep}(\frac{|v-v_*|}{\cos(\theta/2)},\cos\theta)
-B^{\ep}(|v-v_*|,\cos\theta)\bigg|\lesssim  \theta^{-2s}|v-v_*|^\gamma,$ we derive that
$ |II_{1,1}^j|\lesssim  \int |h|_{L^1_{\gamma}}|g^j_\phi|_{L^1_{l+\gamma}}dx.$

{\it \underline{Estimate of $II_{1,2}^j$.}}
 We use Taylor expansion
\begin{equation}\label{E:Taylor-g}
\begin{split}
g(v)-g(v')=(v-v')\cdot \nabla g(v')
+|v-v'|^2\int_0^1 (1-\kappa)D^2g(v'+\kappa(v-v'))\cdot(\frac{v-v'}{|v-v'|},\frac{v-v'}{|v-v'|}) d\kappa,
\end{split}
\end{equation}
to write $II_{1,2}^j=II_{1,21}^j+II_{1,22}^j$ where
\begin{equation}
\begin{split}
II_{1,21}^j&\eqdefa\iiiint h_*(v-v')\cdot\nabla g^j_\phi(v')(\sgn h\lr{v}^l)' B^{\ep}_{\leq}(v-v_*,\sigma)d\sigma dv_*dvdx.\end{split}
\end{equation}
We claim that $II_{1,21}^j=0$. To see that,
for each $v_*$ and $\sigma$, let $\psi_{\sigma}(v')$ represent the inverse transform
$v'\to\psi_{\sigma}(v')=v$ (see~\cite{ADVW00}) one has
 $
{\Big |}\frac{dv'}{dv}{\Big |}=\frac{1}{4}\lr{\frac{v'-v_*}{|v'-v_*|},\sigma}^2
$
and
\begin{equation}
\begin{split}
&\iint (v-v')\cdot\nabla g_\phi^j(v')(\sgn h\lr{v}^l)' B^{\ep}_{\leq}(v-v_*,\sigma)  d\sigma dv \\
&=4\iint \frac{(\psi_{\sigma}(v)-v)}{\lr{\frac{v'-v_*}{|v'-v_*|},\sigma}^2}
\cdot\nabla g_\phi^j(v)(\sgn h\lr{v}^l)B^{\ep}_{\leq}(\psi_{\sigma}(v)-v_*,\sigma)  d\sigma dv=0
\end{split}
\end{equation}
The last equality comes from the symmetry property of $\psi_{\sigma}(v)$ with respect to
$\sigma$.  

To estimate $II_{1,22}^j$, we will use  change of variable $u=\kappa v'+(1-\kappa)v$.
From~\cite{ADVW00}, we know that
$
{\Big |} \frac{du}{dv} {\Big |}=(1-\frac{\kappa}{2})^2{\Big \{}(1-\frac{\kappa}{2})-
\frac{\kappa}{2}\lr{\frac{v-v_*}{|v-v_*|},\sigma}  {\Big \}}
$
is bounded above and below. For the kernel $B^{\ep}_{\leq}$, we have
\beno
 |v'-v_*|\leq |u-v_*|\leq |v-v_*| \;,\;|v-v_*|\leq (1+\ep)|v'-v_*|,\,
\frac{\theta}{2}\leq \tilde{\theta}\leq \theta,
\eeno
where $\cos\tilde{\theta}=|u-v_*|^{-1}\lr{u-v_*,\sigma}$.  Combining above observations,
we have
\beno
&&|II_{1,2}^j|=|II_{1,22}^j|=\big| \int_0^1(1-\kappa)\iiiint  h_* |v-v'|^2 D^2g_\phi^j(u)(\sgn h\lr{v}^l)' B^{\ep}_{\leq}(v-v_*,\sigma) {\Big |} \frac{du}{dv} {\Big |}  d\sigma dv_*dudxd\kappa\big|\\
&&\lesssim \eta^{2-2s}2^{(2s-2)j}\iiint |h_*||u-v_*|^{2+\gamma-\alpha(2-2s)} |D^2g_\phi^j(u)|\max\{\lr{v_*}^l,\lr{u}^l\} dv_*dudx\\
&&\lesssim  \eta^{2-2s}2^{(2s-2)j}\int |h|_{L^1_{l+\gamma}} |D^2 g_\phi^j|_{L^1}+|h|_{L^1}|D^2g_\phi^j|_{L^1_{l+\gamma}}dx.
\eeno

{\it\underline{ Estimate of $II_{1,3}^j$.}} We observe that
\beno
II_{1,3}^j&=&\iiiint h_*g^j_\phi(\sgn h\lr{v}^l)'B^{\ep}_{\ge}d\sigma dv_*dvdx\\
&\le &\iiiint |h_*||g^j_\phi|\big[ (\lr{v}^l)'-\lr{v}^l\big]B^{\ep}_{\ge}d\sigma dv_*dvdx +\iiiint |h_*||g^j_\phi| \lr{v}^l B^{\ep}_{\ge}d\sigma dv_*dvdx.
\eeno
Following the proof of Lemma   \ref{Povnzer1} and  Lemma  \ref{Povnzer2}, we conclude that
\beno
II_{1,3}^j&\lesssim& \int \big(2^l|g_\phi^j|_{L^1_{l+\gamma-2}}|h|_{L^1_{\gamma+2}}
+2^l|g_\phi^j|_{L^1_{\gamma+2}}|h|_{L^1_{l+\gamma-2}}+l^{-1}(|g_\phi^j|_{L^1}|h|_{L^1_{l+\gamma}}+|g_\phi^j|_{L^1_\gamma}|h|_{L^1_{l}}) \big) dx\\
&&\quad+ \eta^{-2s}2^{2sj}\int \big(|g_\phi^j|_{ L^1_{l+\gamma+N_s}}|h|_{L^1}+|g_\phi^j|_{L^1_{l}}
|h|_{L^1_{N_s+\gamma}} \big)dx.
\eeno

{\it\underline{ Estimate of $II_{1,4}^j$.}} Direct calculation gives
 \beno
II_{1,4}^j\lesssim\eta^{-2s}2^{2sj}\int \big(|g_\phi^j|_{ L^1_{l+\gamma+2\alpha s}}|h|_{L^1}+|g_\phi^j|_{L^1_{l}}
|h|_{L^1_{\gamma+2\alpha s}} \big)dx 
\lesssim \eta^{-2s}2^{2sj}\int \big(|g_\phi^j|_{ L^1_{l+\gamma+N_s}}|h|_{L^1}+|g_\phi^j|_{L^1_{l}}
|h|_{L^1_{N_s+\gamma}} \big)dx.
\eeno

Thanks to Theorem \ref{baslem3} and the fact that $g_\phi^j=\mathfrak{F}_j g_\phi$,
patching together all the above estimates will give
\beno II_1&\lesssim& \int \bigg(\eta^{-2s}(|h|_{L^1_\gamma}  |g_\phi|_{H^{2s+\eta_1}_{l+\gamma+2+N_s}}+|g_\phi|_{H^{2s+\eta_1}_{l+2}}|h|_{L^1_{N_s+\gamma}})+|h|_{L^1_{l+\gamma}}(\eta^{2-2s}|g_\phi|_{H^{2s+\eta_1}_{2}}
+l^{-1}|g_\phi|_{H^{\eta_1}_{2}})\\&&+2^l(|h|_{L^1_{\gamma+2}}|g_\phi|_{H^{\eta_1}_{l+\gamma}}+|h|_{L^1_{l+\gamma-2}}|g_\phi|_{H^{\eta_1}_{\gamma+4}})+l^{-1}|g_\phi|_{H^{\eta_1}_{\gamma+2}}|h|_{L^1_l}\bigg)dx.\eeno

{\it Step 2.2: Estimate of $II_2$.} Following the argument applied to $II_{1,3}^j$, we get that
\beno
&&II_2\le \iiint \big[|h_*||g^\phi|(\lr{v'}^l-\lr{v}^l)+2|h_*||g^\phi|\lr{v}^l\big] B^\eps d\sigma dv_*dvdx\lesssim\int\big(2^l|g^\phi|_{L^1_{l+\gamma-2}}|h|_{L^1_{\gamma+2}}\\
&&
+2^l|g^\phi|_{L^1_{\gamma+2}}|h|_{L^1_{l+\gamma-2}}+l^{-1}(|g^\phi|_{L^1}|h|_{L^1_{l+\gamma}}+|g^\phi|_{L^1_\gamma}|h|_{L^1_{l}}) \big) dx  +\eps^{-2s}\int |h|_{L^1_{\gamma}}|g^\phi|_{L^1_{l+\gamma}}dx.
\eeno

Combine the estimates for $II_1$ and $II_2$, then we finally arrive at
\beno
II&\lesssim& \int \bigg(\eta^{-2s}(|h|_{L^1_\gamma}  |W^\eps_{2s+\eta_1}(D)g |_{L^2_{l+\gamma+2+N_s}}+|W^\eps_{2s+\eta_1}(D)g|_{L^2_{l+2}}|h|_{L^1_{N_s+\gamma}})+|h|_{L^1_{l+\gamma}}(\eta^{2-2s}|W^\eps_{2s+\eta_1}(D)g |_{L^2_{2}}
\\&&+l^{-1}|W^\eps_{\eta_1}(D)g|_{L^2_{2}})+2^l(|h|_{L^1_{\gamma+2}}|W^\eps_{\eta_1}(D)g|_{L^2_{l+\gamma}}+|h|_{L^1_{l+\gamma-2}}|W^\eps_{\eta_1}(D)g|_{L^2_{\gamma+4}})+l^{-1}|W^\eps_{\eta_1}(D)g|_{L^2_{\gamma+2}}|h|_{L^1_l}\bigg)dx,
\eeno
which is enough to get the desired result thanks to the estimate in {\it Step 1}.
 \end{proof}

Then $L^1$-moment estimate for the equation can be stated as follows.
\begin{prop}\label{L1i} Suppose that $W_{-1,0}\in \mathbb{W}_I(N,\kappa,\varrho,\delta_1,q_1,q_2)(\,or\,\, \mathbb{W}_{II}(N,\kappa,\varrho,\delta_1))$ with $W_{-1,0}=W_{l_1}$ and $\ep\le l_1^{-1/2-a}$. Then there exists a constant $A _1(c_1,c_2)$ which is in proportion to $\mathcal{C}_3(c_1,c_2)$ defined in Proposition \ref{lbLdelta} such that for $\eta,\eta_1>0$,
\beno&&(i).\, \|h(t)\|_{L^1_{l_1}}+ \int_0^t (l_1^sA _1(c_1,c_2)-\eta^{2-2s}\|W^\eps_{2s+\eta_1}(D)g \|_{H^{\f32+\delta_1}_xL^2_{2}}
-2l_1^{-1}\|W^\eps_{\eta_1}(D)g\|_{H^{\f32+\delta_1}_xL^2_{\gamma+2}})\|h\|_{L^1_{l_1+\gamma}}d\tau \\
&&\le \|h_0\|_{L^1_{l_1}}+C(l_1) \bigg[\int_0^t   \big(\|h\|_{L^1_{l_1+\gamma-2}}(\|f\|_{L^\infty_xL^1_{\gamma+2}}+\|W^\eps_{\eta_1}(D)g\|_{H^{\f32+\delta_1}_xL^2_{\gamma+4}})  + \|h\|_{L^\infty_xL^1_{\gamma+2}}( \|f\|_{L^1_{l_1+\gamma-2}}\\&&\quad+\|W^\eps_{\eta_1}(D)g\|_{L^2_{l_1+\gamma}})   \big)d\tau\bigg]+
\int_0^t ( Cl_1^{-1}\|h\|_{L^\infty_xL^1_\gamma}\|f\|_{L^1_{l_1+\gamma}}+\eta^{-2s}(\|h\|_{L^\infty_xL^1_\gamma}  \|W^\eps_{2s+\eta_1}(D)g \|_{L^2_{l_1+\gamma+2+N_s}}\\&&\quad+\|W^\eps_{2s+\eta_1}(D)g\|_{L^2_{l_1+2}}\|h\|_{L^\infty_x L^1_{N_s+\gamma}})  ) d\tau;\\
 &&(ii).\,\|h(t)\|^2_{L^1_{l_1-\gamma}}+ \int_0^t l_1^sA _1(c_1,c_2)-\eta^{2-2s}\|W^\eps_{2s+\eta_1}(D)g \|_{H^{\f32+\delta_1}_xL^2_{2}}
	-2l_1^{-1}\|W^\eps_{\eta_1}(D)g\|_{H^{\f32+\delta_1}_xL^2_{\gamma+2}})\\
	&&\times\|h\|_{L^1_{l_1}}\|h\|_{L^1_{l_1-\gamma}}d\tau \le \|h_0\|_{L^1_{l_1-\gamma}}^2+C(l_1) \bigg[\int_0^t   \big(\|h\|_{L^1_{l_1-\gamma}}\|h\|_{L^1_{l_1-2}}(\|f\|_{L^\infty_xL^1_{\gamma+2}}+\|W^\eps_{\eta_1}(D)g\|_{H^{\f32+\delta_1}_xL^2_{\gamma+4}})\\&&\quad  + \|h\|_{L^1_{l_1-\gamma}}\|h\|_{L^\infty_xL^1_{\gamma+2}}( \|f\|_{L^1_{l_1-2}}+\|W^\eps_{\eta_1}(D)g\|_{L^2_{l_1}})   \big)d\tau\bigg]+
	\int_0^t ( Cl_1^{-1}\|h\|_{L^1_{l_1-\gamma}}\|h\|_{L^\infty_xL^1_\gamma}\|f\|_{L^1_{l_1 }}+\eta^{-2s} \\&&\quad\times \|h\|_{L^1_{l_1-\gamma}} (\|h\|_{L^\infty_xL^1_\gamma}\|W^\eps_{2s+\eta_1}(D)g \|_{L^2_{l_1 +2+N_s}}+\|W^\eps_{2s+\eta_1}(D)g\|_{L^2_{l_1-\gamma+2}}\|h\|_{L^\infty_x L^1_{N_s+\gamma}}) ) d\tau.
	\eeno
	\end{prop}

\subsubsection{$L^2$ moment estimates for the equation}
In  this subsection, we will give the basic energy estimates in $L^2_l$ space. We first recall
\beno  \pa_t(hW_l)+v\cdot\na_x (hW_l)=Q^\eps(f, h)W_l +Q^\eps(h, g)W_l.\eeno
By multiplying $hW_l$ and taking inner product, we have
\beno
\f{d}{dt}\|h\|_{L^2_l}^2=\int \lr{Q^\eps(f, h)W_l, hW_l}_v+\lr{Q^\eps(h,g)W_l, hW_l}_v dx.
\eeno
Applying $(i)$ and $(iv)$ of  Corollary \ref{basic-estimates} to the righthand side of the equality, we have
\beno &&\f{d}{dt}\|h\|_{L^2_l}^2+  \f16\int_{\TT^3}\mathcal{E}_{f}^\gamma(hW_l)dx+\mathcal{C}_4(c_1,c_2)(\|W^\ep_s(D)W_{l+\gamma/2}h\|_{L^2}^2+\int_{\TT^3}\mathcal{E}_{\mu}^{0,\eps}(W_{l+\gamma/2}h)dx+\f12\mathcal{C}_5(c_1,c_2)\delta^{-2s}\\&&\quad\times\|h\|_{L^2_{l+\gamma/2}}^2\lesssim\int \big[\mathcal{C}_{6}(c_1, c_2)\delta^{-6-6s} |h|_{L^1_{2l}}|h|_{L^1_{\gamma}}+|f|_{L^1_{\gamma+2s}}^2|h|_{L^2_{l}}^2+|f|_{L^2_{l+\gamma/2}}^2|h|_{L^1_2}^2 +|h|_{L^1_{\gamma+2s}}^2\\&&\quad \times|W^\ep_s(D)W_{l+\gamma/2+2s} g|_{L^2}^2+
|h|_{L^1_{\gamma+2s}}^2|g|_{L^2_{l}}^2
+|h|_{L^2_{l+\gamma/2}}^2|g|_{L^1_2}^2\big]dx.
 \eeno

 We arrive at
 \begin{prop}\label{L2m}   Suppose that $W_{0,0}\in \mathbb{W}_I(N,\kappa,\varrho,\delta_1,q_1,q_2)(\,or\,\, \mathbb{W}_{II}(N,\kappa,\varrho,\delta_1))$ with $W_{0,0}=W_{l_2}$. Then there exist constants  $A_2(c_1,c_2)\sim \mathcal{C}_4(c_1,c_2),  A_3(c_1,c_2)\sim \mathcal{C}_5(c_1,c_2), A_4(c_1,c_2)\sim \mathcal{C}_6(c_1,c_2)$ such that
 \beno &&\|h(t)\|_{L^2_{l_2}}^2+\f16\int_0^t\int_{\TT^3} \mathcal{E}_f^\gamma(hW_{l_2})dxd\tau
+A_2(c_1,c_2)
 \int_0^t  \bigg[\big(\|W^\ep_s(D)W_{l_2+\gamma/2}h\|_{L^2}^2+\int_{\TT^3}\mathcal{E}_{\mu}^{0,\eps}\big(W_{l_2+\gamma/2}h\big)dx\bigg]d\tau \\&&\qquad +A_3(c_1,c_2) \delta^{-2s}\int_0^t \|h\|_{L^2_{l_2+\gamma/2}}^2 d\tau
\\&&\le \|h_0\|_{L^2_{l_2}}^2+A_4(c_1, c_2)\delta^{-6-6s}\int_0^t  \|h\|_{L^1_{2l_2}}\|h\|_{L^\infty_xL^1_{\gamma}}d\tau
+C_E\int_0^t \big(\|f\|_{L^\infty_xL^1_{\gamma+2s}}^2\|h\|_{L^2_{l_2}}^2+\|f\|_{L^2_{l_2+\gamma/2}}^2\|h\|_{L^\infty_xL^1_2}^2\\&&\quad
+\|h\|_{L^\infty_xL^1_{\gamma+2s}}^2\|W^\ep_s(D)W_{l_2+\gamma/2+2s} g\|_{L^2}^2+
\|h\|_{L^\infty_xL^1_{\gamma+2s}}^2\|g\|_{L^2_{l_2}}^2
 +\|h\|_{L^2_{l_2+\gamma/2}}^2\|g\|_{L^\infty_xL^1_2}^2)d\tau.\eeno  
 \end{prop}

\subsection{Gain and propagation of  derivatives for $x$ and $v$ variables} In this subsection, we will show the gain of fractional derivatives for $x$ variable due to hypo-elliptic property of the transport equation and also the propagation of   derivatives for $v$ variable.
\subsubsection{Gain of fractional derivatives for $x$ variable}
By $(iii)$ of Corollary \ref{basic-estimates},  for sufficiently small $\eta$, we have
\beno  && |Q(f, h)W_l +Q(h, g)W_l|_{H^{-1}}\lesssim (\eta+ \ep^{1-s}|f|_{L^1_{\gamma+2s}})|W^\ep_s(D)W_{l+\gamma+2s}h|_{L^2}+\eta^{-\f{2s-1}{1-s}}|f|_{L^1_{\gamma+2s}}^{\f{2s-1}{1-s}}|h|_{L^2_{l+\gamma+2s}}\\&&\qquad\qquad\qquad\qquad+|f|_{L^1_{\gamma+2s}}|h|_{L^2_{l+\gamma}}
+|f|_{L^2_{l+\gamma}}|h|_{L^1_2}+|h|_{L^1_{\gamma+2s}}|W^\ep_s(D)W_{l+\gamma+2s}g|_{L^2}
+|h|_{L^2_{l+\gamma}}|g|_{L^1_2}.\eeno
From this together with  Lemma \ref{hypoellipticity}, we get that
\beno && \|W_{-d_1}(hW_l)_\phi\|^2_{L^2_TH^{\f{s}{4(4+s)}}_xL^2}
\lesssim \eta^{-8}\|h_0\|_{L^2_l}^2+\eta^{2s}\|(hW_l)_\phi\|_{L^2_TL^2_xH^s}^2  +\eta^{-8}
\bigg(\int_0^T\int_{\TT^3}\big[|h|_{L^2_l}^2+ \mathrm{1}_{2s>1}(\eta_1^2\\&&\quad+ \ep^{2(1-s)}|f|^2_{L^1_{\gamma+2s}}) |W^\ep_s(D)W_{l+\gamma+2s}h|^2_{L^2}+\eta_1^{-2\f{2s-1}{1-s}}|f|_{L^1_{\gamma+2s}}^{2\f{2s-1}{1-s}}|h|^2_{L^2_{l+\gamma+2s}}+\big(|f|^2_{L^1_{\gamma+2s}}|h|^2_{L^2_{l+\gamma}}
+|f|^2_{L^2_{l+\gamma}}|h|^2_{L^1_2}\\&&\quad +|h|^2_{L^1_{\gamma+2s}}|W^\ep_s(D)W_{l+\gamma+2s}g|^2_{L^2}
+|h|^2_{L^2_{l+\gamma}}|g|^2_{L^1_2}\big)\big]dxdt\bigg).
\eeno

We arrive at
\begin{prop}\label{smx01} Suppose that $W_{m,n}\in \mathbb{W}_I(N,\kappa,\varrho,\delta_1,q_1,q_2)(\,or\,\, \mathbb{W}_{II}(N,\kappa,\varrho,\delta_1))$. Then
\beno && \int_0^t \|W_{-d_1}(W_{0,1}W_{\gamma/2+d_1+d_2}h)_\phi\|^2_{H^{\varrho}_xL^2}d\tau 
 \lesssim \eta^{-8}\|W_{0,1}W_{\gamma/2+d_1+d_2}h_0\|_{L^2}^2+ \eta^{-8}\int_0^t \|W_{0,1}W_{\gamma/2+d_1+d_2}h\|_{L^2}^2d\tau\\&&+\eta^{2s}\int_0^t \|W^\eps_s(D)(W_{0,1}W_{\gamma/2+d_1+d_2}h)\|_{L^2}^2d\tau
 +\eta^{-8}\bigg( \int_0^t \mathrm{1}_{2s>1} (\eta_1^2+ \ep^{2(1-s)} \|f\|^2_{L^\infty_xL^1_{\gamma+2s}})\|W^\ep_s(D)W_{0,1}\\&&\times W_{\f32\gamma+2s+d_1+d_2}h\|^2_{L^2}d\tau+\int_0^t \big[\eta_1^{-2\f{2s-1}{1-s}}(\|f\|_{L^\infty_xL^1_{\gamma+2s}}^{2\f{2s-1}{1-s}}+\|f\|_{L^\infty_xL^1_{\gamma+2s}}^{2} +\|g\|_{L^\infty_xL^1_2})\|W_{0,1}W_{\f32\gamma+2s+d_1+d_2}h\|^2_{L^2}\\&& 
+\|W_{0,1}W_{\f32\gamma+d_1+d_2}f\|^2_{L^2}\|h\|^2_{L^\infty_xL^1_2} +\|h\|^2_{L^\infty_xL^1_{\gamma+2s}}\|W^\ep_s(D)W_{0,1}W_{\f32\gamma+2s+d_1+d_2}g\|^2_{L^2} \big]d\tau\bigg).
\eeno
\end{prop}

\subsubsection{Propagation of  the regularity for $v$ variable.} We have two results. The first one reads: 
\begin{prop}\label{Envq}
	For $q\ge s$ and $\eta>0$, it holds
\beno
&& V^{q,\eps}(h(t))+\mathcal{C}_1(c_1,c_2)\int_0^t \|W^{\ep}_{q+s}(D) W_{\gamma/2}h\|^2_{L^2}d\tau 
 \le  V^{q,\eps}(h_0)+ \mathcal{C}_2(c_1,c_2)\\&&\times\int_0^t\|W^\ep_q(D)W_{\gamma/2}h\|_{L^2}^2d\tau+C_E\int_0^t \bigg(\|W^{\ep}_{(q-1)^+}(D)h\|_{H^1_xL^2}^2
+\|h\|_{L^\infty_xL^1_{\gamma+2s}}^2 \|W^{\ep}_{q+s}(D)W_{\gamma/2+2s}g\|_{L^2}^2\\&&+ \|f\|_{L^\infty_xL^2}^2( \mathrm{1}_{2s>1}\|W^\ep_{q+s-1}(D)W_{\gamma/2+\f52}h\|^2_{L^2}+\mathrm{1}_{2s=1}\|W^\ep_{q-s+\log}(D)W_{\gamma/2+\f52}h\|^2_{L^2}+ \| f\|_{L^\infty_xL^1}^2 \|h\|_{L^2}^2\\&&+\mathrm{1}_{2s<1}\|W^{\eps}_{q-s}(D)W_{\gamma/2+\f52}h\|_{L^2}^2)
+ \|f\|_{L^\infty_xL^2_{\gamma+3}}^2( \mathrm{1}_{2s>1}\| W^\ep_{q+s-1}(D)h\|^2_{L^2}  +\mathrm{1}_{2s=1} \|W^\ep_{q-s+\log}(D)h\|_{L^2}^2 +\mathrm{1}_{2s<1} \\&&\times\|W^\eps_{q-s}(D) h\|_{L^2}^2)  + \|W^\ep_{q+s-1+\eta}(D)f\|_{L^2}^2 \| h\|_{L^\infty_xL^2}^2+ \|h\|_{L^{\infty}_xL^2}^2( \mathrm{1}_{2s>1}\|W^\ep_{q+s-1}(D)W_{\gamma/2+\f52}g\|^2_{L^2}+\mathrm{1}_{2s=1}\\&&\times \|W^\ep_{q-s+\log}(D)W_{\gamma/2+\f52}g\|^2_{L^2}+\mathrm{1}_{2s<1}\|W^\eps_{q-s}(D)W_{\gamma/2+\f52}g\|_{L^2}^2)
 + \|h\|_{L^\infty_xL^2_{\gamma+3}}( \mathrm{1}_{2s>1}\| W^\ep_{q+s-1}(D)g\|^2_{L^2}\\&&+\mathrm{1}_{2s=1} \|W^\ep_{q-s+\log}(D)g\|_{L^2}^2+\mathrm{1}_{2s<1}\|W^\eps_{q-s}(D)g\|_{L^2}^2)   + \|W^\ep_{q+s-1+\eta}(D)h\|_{L^2}^2 \| g\|_{L^\infty_xL^2}^2+ \| h\|_{L^2_2}^2 \|g\|_{L^\infty_xL^2}^2\bigg)d\tau.
\eeno
\end{prop}
Let us sketch the proof of the proposition. By frequency localization, we first  observe that
\beno  \pa_t (\mathfrak{F}_j h) +v\cdot\na_x(\mathfrak{F}_j h)&=&[v\cdot \na_x, \mathfrak{F}_j]h+Q^\ep(f, \mathfrak{F}_j h)+\big[\mathfrak{F}_j Q^\ep(f, h)-Q^\ep(f, \mathfrak{F}_j h)\big] \\
&&+Q^\ep(h,\mathfrak{F}_j g)+\big[\mathfrak{F}_j Q^\ep(h,g)-Q^\ep(h,\mathfrak{F}_j g)\big]\eqdefa \sum_{i=1}^5 R_i.\eeno
For $R_1$, by using the  estimates $ |[v_i,\mathfrak{F}_j]h|_{L^2}\le 2^{-j}|(\pa_i\varphi)(2^{-j} D)h|_{L^2}$ and  $ \|W^\ep_q(D)h\|_{L^2}^2\sim \sum\limits_{j\le |\log \epsilon|}2^{2qj} \|\mathfrak{F}_j h\|_{L^2}^2+\sum\limits_{j\ge |\log \epsilon|} \ep^{-2q} \|\mathfrak{F}_j h\|_{L^2}^2$, we  derive that
$$\sum_{j\le |\log \epsilon|}2^{2qj}| ([v_i,\mathfrak{F}_j]\pa_{x_i}h,\mathfrak{F}_j h)| +\sum_{j\ge |\log \epsilon|} \ep^{-2q}| ([v_i,\mathfrak{F}_j]\pa_{x_i}h,\mathfrak{F}_j h)|\lesssim \|W^{\ep}_{(q-1)^+}h\|_{H^1_xL^2}\|W^\ep_q(D)h\|_{L^2}.$$
For $R_2$, we may apply $(vi)$ of Corollary \ref{basic-estimates}   and \eqref{func8} to get the coercivity estimate.  For $R_3,R_4$ and $R_5$, we may treat them by
$(i)$ of Corollary \ref{basic-estimates}, $(i)$ of Lemma \ref{comforsym3} and \eqref{func8}.  Summarizing  all the estimates, we  get   the propagation of  the partial regularity for $v$ variable.

Next we want to show that the regularity for $v$ variable can be propagated. To get that, we only need to modify the proof of the previous proposition. More precisely, we shall apply $(v)$ of  Corollary \ref{basic-estimates} to $R_2$  and apply $(ii)$ of Lemma \ref{comforsym3} to  $R_3$ and $R_5$ to renew the corresponding estimates.  Finally by the interpolation inequality $|f|_{H^{q-1}}\le |f|_{L^2}^{\f1{q+1}}|f|_{H^q}^{\f{q}{q+1}}$, we will get   

\begin{prop}\label{Enfvq}
	For $q\ge2$ and $\eta>0$, it holds
\beno
&& V^q(h(t))+\mathcal{C}_1(c_1,c_2)\int_0^t \|W^{\ep}_{s}(D) W_{\gamma/2}h\|^2_{L^2_xH^{q}}d\tau+\delta^{-2s}\|h\|_{L^2_xH^{q}_{\gamma/2}}^2\\
&&\le V^q(h_0)+ \mathcal{C}_6(c_1,c_2)\delta^{-4-2s-4(q-1)}\int_0^t\|h\|_{L^2_{\gamma/2}}^2d\tau+C_E\int_0^t \bigg(\|h\|_{H^1_xH^{q-1}}^2
+\int_{\TT^3} (|W_{\gamma+3}f|_{H^1}^2 \\&&\times |W^\ep_{s}(D)h|^2_{H^{q-1}}+|f|_{H^1}^2 |W^\ep_{s}(D)W_{\gamma/2+\f52}h|_{H^{q-1}}^2)dx+ \|f\|_{L^\infty_xL^2}^2( \mathrm{1}_{2s=1}\|W^\ep_{q-s+\log}(D)W_{\gamma/2+\f52}h\|^2_{L^2}\\&&+\mathrm{1}_{2s<1}\|W^{\eps}_{q-s}(D)W_{\gamma/2+\f52}h\|_{L^2}^2) 
+ \|f\|_{L^\infty_xL^2_{\gamma+3}}^2( \mathrm{1}_{2s=1} \|W^\ep_{q-s+\log}(D)h\|_{L^2}^2 +\mathrm{1}_{2s<1}  \|W^\eps_{q-s}(D) h\|_{L^2}^2) \\&& + \|W^\ep_{s}(D)f\|_{L^2_xH^{q-1+\eta}}^2 \| h\|_{L^\infty_xL^2}^2 + \| f\|_{L^\infty_xL^1}^2 \|h\|_{L^2}^2+\|h\|_{L^\infty_xL^1_{\gamma+2s}}^2 \|W^{\ep}_{s}(D)W_{\gamma/2+2s}g\|_{L^2_xH^{q}}^2
\\&&+\int_{\TT^3} (|W_{\gamma/2+3}h|_{H^1}^2 |W^\ep_{s}(D)g|^2_{H^{q-1}}+|h|_{H^1}^2 |W^\ep_{s}(D)W_{\gamma/2+\f52}g|_{H^{q}})^2dx+ \|W^\ep_{s}(D)h\|_{L^2_xH^{q-1+\eta}}^2 \| g\|_{L^\infty_xL^2}^2\\&&+ \| h\|_{L^2_2}^2 \|g\|_{L^\infty_xL^2}^2+ \|h\|_{L^{\infty}_xL^2}^2(\mathrm{1}_{2s=1}\|W^\ep_{q-s+\log}(D)W_{\gamma/2+\f52}g\|^2_{L^2}+\mathrm{1}_{2s<1}\|W^\eps_{q-s}(D)W_{\gamma/2+\f52}g\|_{L^2}^2)
 \\&&+ \|h\|_{L^\infty_xL^2_{\gamma+3}} (\mathrm{1}_{2s=1} \|W^\ep_{q-s+\log}(D)g\|_{L^2}^2+\mathrm{1}_{2s<1}\|W^\eps_{q-s}(D)g\|_{L^2}^2) \bigg)d\tau.
\eeno 
\end{prop}

\subsection{High order energy estimates}
 We recall  notations:
$W_{m,n}=\lr{v}^{l_{m,n}}$ and $|\alpha|=m\le N$.
 Observe that if $h$ is a solution to \eqref{EAPLN:h-eqaution}, then $W_{m,n} \bar{\triangle}^{n\varrho}_k\pa_x^\alpha h$  and  $W_{m,n+1} W_{\gamma/2+2} \bar{\triangle}^{n\varrho}_k\pa_x^\alpha h$ solve  the   equations:
\beno &&(1).\,\pa_t(W_{m,n} \bar{\triangle}^{n\varrho}_k\pa_x^\alpha h)+v\cdot \na_x(W_{m,n} \bar{\triangle}^{n\varrho}_k\pa_x^\alpha h)\\
&&= Q^\eps(f,\bar{\triangle}^{n\varrho}_k\pa_x^\alpha h) W_{m,n} +\sum_{|\alpha_1|\ge1;\alpha_1+\alpha_2=\alpha} Q^\eps(\pa_x^{\alpha_1}f,  \bar{\triangle}^{n\varrho}_k\pa_x^{\alpha_2}h)W_{m,n}+\sum_{\alpha_1+\alpha_2=\alpha} Q^\eps(\bar{\triangle}^{n\varrho}_k\pa_x^{\alpha_1}f,  T_k^n\pa_x^{\alpha_2}h)W_{m,n}\\
&&+\sum_{\alpha_1+\alpha_2=\alpha}Q^\eps(\pa_x^{\alpha_1}h, \bar{\triangle}^{n\varrho}_k\pa_x^{\alpha_2}g)W_{m,n}+\sum_{\alpha_1+\alpha_2=\alpha}Q^\eps(\bar{\triangle}^{n\varrho}_k\pa_x^{\alpha_1}h,  T_k^n\pa_x^{\alpha_2}g)W_{m,n}\eqdefa \sum_{i=1}^5 R_i;
 \\ &&(2).\,\pa_t(W_{m,n+1} W_{\gamma/2+2} \bar{\triangle}^{n\varrho}_k\pa_x^\alpha h)+v\cdot \na_x(W_{m,n+1} W_{\gamma/2+2}\bar{\triangle}^{n\varrho}_k\pa_x^\alpha h) = Q^\eps(f,\bar{\triangle}^{n\varrho}_k\pa_x^\alpha h) W_{m,n+1} W_{\gamma/2+2}\\&& +\sum_{|\alpha_1|\ge1;\alpha_1+\alpha_2=\alpha} Q^\eps(\pa_x^{\alpha_1}f,  \bar{\triangle}^{n\varrho}_k\pa_x^{\alpha_2}h)W_{m,n+1} W_{\gamma/2+2} 
 +\sum_{\alpha_1+\alpha_2=\alpha} Q^\eps(\bar{\triangle}^{n\varrho}_k\pa_x^{\alpha_1}f,  T_k^n\pa_x^{\alpha_2}h)W_{m,n+1} W_{\gamma/2+2}\\&& +\sum_{\alpha_1+\alpha_2=\alpha}Q^\eps(\pa_x^{\alpha_1}h, \bar{\triangle}^{n\varrho}_k\pa_x^{\alpha_2}g)W_{m,n+1} W_{\gamma/2+2}+\sum_{\alpha_1+\alpha_2=\alpha}Q^\eps(\bar{\triangle}^{n\varrho}_k\pa_x^{\alpha_1}h,  T_k^n\pa_x^{\alpha_2}g)W_{m,n+1} W_{\gamma/2+2}\eqdefa \sum_{i=6}^{10} R_i.
 \eeno
  where  $T_k^n$ and $\bar{\triangle}^{n\varrho}_k$ are defined in \eqref{transop} and \eqref{fdop}.  In what follows, we want to derive the high order estimates for \eqref{EAPLN:h-eqaution} from these two equations.

    To obtain the propagation of $\|W_{m,n} h\|_{H^{m+n\varrho}_xL^2}$ , the key point is to give the bounds for $\int R_iW_{m,n} \bar{\triangle}^{n\varrho}_k\pa_x^\alpha h dxdkdv$\\($i=1,\dots,5$). For the term involving $R_1$, one may apply Corollary \ref{basic-estimates1} to obtain the coercivity estimate. We can apply $(i)$ of Corollary \ref{basic-estimates3} to the terms involving $R_2$ and $R_4$ and apply $(ii)$ of Corollary \ref{basic-estimates3} to the terms involving $R_3$ and $R_5$ to get the upper bounds.  To get the hypo-elliptic estimate of $W_{m,n+1} W_{\gamma/2+2} \bar{\triangle}^{n\varrho}_k\pa_x^\alpha h$,
  we will employee $(iii)$ of Corollary \ref{basic-estimates2} to handle $R_6$, (i) of Corollary \ref{basic-estimates2} to bound $R_7$ and $R_9$ and $(ii)$ of Corollary \ref{basic-estimates2} to treat $R_8$ and $R_{10}$. Finally applying directly  Lemma \ref{hypoellipticity} to the equation of $W_{m,n+1} W_{\gamma/2+2} \bar{\triangle}^{n\varrho}_k\pa_x^\alpha h$ will yield the desired result.  Our results can be summarized as follows:

\subsubsection{Propagation and gain of regularity for $x$ variable in the case of $\{m=0,0< n\le N_{\varrho,1}\}$, $\{m=1, n=0\}$ and $\{m=1, 0<n\le [1/2\varrho]\}$}

 \begin{prop}\label{Enxmn1}  Suppose that $W_{m,n}\in \mathbb{W}_I(N,\kappa,\varrho,\delta_1,q_1,q_2)(\,or\,\, \mathbb{W}_{II}(N,\kappa,\varrho,\delta_1))$. Then
\beno && \|W_{m,n}    h(t)\|^2_{H^{m+n\varrho}_xL^2}+\f{c_o}3\int_0^t\bigg(\int_{\TT^3} \mathcal{E}_f^\gamma(W_{m,n}    |D_x|^{m+n\varrho} h)dx\\&&\qquad+c_oA_5(c_1, c_2)\big(\|W^\epsilon_s(D)W_{\gamma/2}W_{m,n}   h\|_{H^{m+n\varrho}_xL^2}^2+\int_{\TT^3} \mathcal{E}^{0,\eps}_\mu(W_{\gamma/2}W_{m,n}  |D_x|^{m+n\varrho} h)dx\big)\bigg)d\tau\\
&&\le  \|W_{m,n}  h_0\|^2_{H^{m+n\varrho}_xL^2}+A_6(c_1,c_2)\int_0^t\|W_{\gamma/2}W_{m,n}  h\|_{H^{m+n\varrho}_xL^2}^2d\tau+C_E\int_0^t \bigg(\|f\|_{H^{\f32+\delta_1}_xL^2_{\gamma+4}}^2 \|W_{m,n}  h\|_{H^{m+n\varrho}_xL^2}^2\\&&
+\|W_{\gamma/2}W_{m,n}f\|_{H_x^{m+n\varrho}L^2}^2\|  h\|_{H_x^{\f32+\delta_1 }L^2_4}^2
+\|W^\ep_s(D)W_{m,n}W_{\gamma/2+2s} h\|^2_{H_x^{ m+n\varrho-\delta_1}L^2}\|f\|_{H_x^{\f32+2\delta_1}L^2_{\gamma+4}}^2  \\&&+ \|h\|^2_{H_x^{\f32+\delta_1}L^2_{\gamma+4}}\|W^\ep_s(D)W_{m,n}W_{\gamma/2+2s} g\|^2_{H_x^{m+n\varrho}L^2}
 +\| W_{m,n}W_{\gamma/2}h\|^2_{H^{m+n\varrho}_xL^2}\|g\|^2_{H^{\f32+\delta_1}_xL^2_4}\bigg)d\tau,\eeno
  where $A_5(c_1,c_2)\sim \mathcal{C}_1(c_1,c_2), A_6(c_1,c_2)\sim \mathcal{C}_2(c_1,c_2)$ and $c_o$ is a small and universal constant.
 \end{prop}

 \begin{prop}\label{smxmn1}  Suppose that $W_{m,n}\in \mathbb{W}_I(N,\kappa,\varrho,\delta_1,q_1,q_2)(\,or\,\, \mathbb{W}_{II}(N,\kappa,\varrho,\delta_1))$. Then
\beno  &&\int_0^t \|W_{-d_1}(W_{m,n+1}W_{\gamma/2+d_1+d_2}    h)_\phi\|_{H^{m+(n+1)\varrho}_xL^2}^2 d\tau 
 \lesssim \eta^{-8}\|W_{m,n+1}W_{\gamma/2+d_1+d_2}   h_0\|^2_{H^{m+n\varrho}_xL^2}+ \eta^{2s}\\&&\times\int_0^t  \|W^\ep_s(D)W_{m,n+1}W_{\gamma/2+d_1+d_2}  h\|_{H^{m+n\varrho}_xL^2}^2 d\tau +\eta^{-8}\int_0^t  \| W_{m,n+1}W_{\gamma/2+d_1+d_2}   h\|_{H^{m+n\varrho}_xL^2}^2 d\tau
  +\eta^{-8}\\&&\times\bigg[\int_0^t\bigg(
 \mathrm{1}_{2s>1}(\eta_1^2+\epsilon^{2(1-s)} \|f\|^2_{H_x^{\f32+\delta_1}L^2_{\gamma+4}}) \|W^\ep_s(D)W_{m,n+1}W_{\f32\gamma+2s+d_1+d_2}h\|^2_{H_x^{m+n\varrho}L^2}\notag  + \eta_1^{-\f{2(2s-1)}{1-s}} \\&&\times\|f\|^{\f{2(2s-1)}{1-s}}_{H_x^{\f32+\delta_1}L^2_{\gamma+4}}\|W_{m,n+1}W_{\f32\gamma+2s+d_1+d_2}h\|^2_{H_x^{m+n\varrho}L^2} +
 \|f\|^2_{H_x^{\f32+\delta_1}L^2_{\gamma+4}}\| W_{m,n+1}W_{\f32\gamma+2s+d_1+d_2} h\|^2_{H_x^{m+n\varrho}L^2} \\&&  + \|W_{m,n+1}W_{\f32\gamma+d_1+d_2}f\|^2_{H^{m+n\varrho}_xL^2} \|h\|^2_{H^{\f32+\delta_1}_xL^2_4}+
\|f\|^2_{H_x^{\f32+2\delta_1}L^2_{\gamma+4}}\|W^\ep_s(D)W_{m,n+1}W_{\f32\gamma+2s+d_1+d_2} h\|^2_{H_x^{m+n\rho-\delta_1}L^2}\\&&  + \|h\|^2_{H_x^{\f32+\delta_1}L^2_{\gamma+4}}\|W^\ep_s(D)W_{m,n+1}W_{\f32\gamma+3+2s} g\|_{H_x^{m+n\varrho}L^2}^2+\| W_{m,n+1}W_{\f32\gamma+d_1+d_2}h\|^2_{H^{m+n\varrho}_xL^2}  \|g\|^2_{H^{\f32+\delta_1}_xL^2_4}\bigg)d\tau\bigg].
  \eeno \end{prop}

 \subsubsection{Propagation  of regularity for $x$ variable in the case of  $\{m=1, n\varrho=\f12+\delta_1, \f12+2\delta_1\}$.}
 \begin{prop}\label{Enx32} Suppose that $W_{1,\f12+\delta_1}\in \mathbb{W}_I(N,\kappa,\varrho,\delta_1,q_1,q_2)(\,or\,\, \mathbb{W}_{II}(N,\kappa,\varrho,\delta_1))$. Then
		\beno &&(i).\, \|W_{1,\f12+\delta_1}    h(t)\|^2_{H^{\f32+\delta_1}_xL^2}+\f{c_o}3\int_0^t\bigg(\int_{\TT^3} \mathcal{E}_f^\gamma(W_{1,\f12+\delta_1}  |D_x|^{\f32+\delta_1} h)dx\\&&\qquad+c_oA_5(c_1, c_2)\big(\|W^\epsilon_s(D)W_{\gamma/2}W_{1,\f12+\delta_1}     h\|_{H^{\f32+\delta_1}_xL^2}^2+\int_{\TT^3} \mathcal{E}^{0,\eps}_\mu(W_{\gamma/2}W_{1,\f12+\delta_1}   |D_x|^{\f32+\delta_1} h)dx\big)\bigg)d\tau\\
	&&\le  \|W_{1,\f12+\delta_1}  h_0\|^2_{H^{\f32+\delta_1}_xL^2}+A_6(c_1,c_2)\int_0^t\|W_{\gamma/2}W_{1,\f12+\delta_1}  h\|_{H^{\f32+\delta_1}_xL^2}^2d\tau+C_E\int_0^t \bigg(
	\|W_{\gamma/2}W_{1,\f12+\delta_1}f\|_{H_x^{\f32+\delta_1}L^2}^2\\&&\times\|  h\|_{H_x^{\f32+\delta_1 }L^2_4}^2 +\|f\|_{H^{\f32+\delta_1}_xL^2_{\gamma+4}}^2\|W_{1,\f12+\delta_1}  h\|_{H^{\f32+\delta_1}_xL^2}^2
	+ \|W^\ep_s(D)W_{1,\f12+\delta_1}W_{\gamma/2+2s} h\|^2_{H_x^{\f32}L^2}
	\|f\|_{H_x^{\f32+2\delta_1}L^2_{\gamma+4}}^2 \\&& +\|h\|^2_{H_x^{\f32+\delta_1}L^2_{\gamma+4}}\|W^\ep_s(D)W_{1,\f12+\delta_1}W_{\gamma/2+2s} g\|^2_{H_x^{\f32+\delta_1}L^2}+\| W_{1,\f12+\delta_1}W_{\gamma/2}h\|^2_{H^{\f32+\delta_1}_xL^2}\|g\|^2_{H^{\f32+\delta_1}_xL^2_4} \bigg)d\tau.\eeno\beno
 &&(ii).\, \|W_{\gamma+4}    h(t)\|^2_{H^{\f32+2\delta_1}_xL^2}+\f{c_o}3\int_0^t\bigg(\int_{\TT^3} \mathcal{E}_f^\gamma(W_{\gamma+4}  \  |D_x|^{\f32+2\delta_1} h)dx\\&&\qquad+c_oA_5(c_1, c_2)\big(\|W^\epsilon_s(D)W_{\gamma/2}W_{\gamma+4}    h\|_{H^{\f32+2\delta_1}_xL^2}^2+\int_{\TT^3} \mathcal{E}^{0,\eps}_\mu(W_{\gamma/2}W_{\gamma+4}     |D_x|^{\f32+2\delta_1} h)dx\big)\bigg)d\tau\\
&&\le  \|W_{\gamma+4} h_0\|^2_{H^{\f32+2\delta_1}_xL^2}+A_6(c_1,c_2)\int_0^t\|W_{\gamma/2}W_{\gamma+4}  h\|_{H^{\f32+2\delta_1}_xL^2}^2d\tau+C_E\int_0^t \bigg(\|f\|_{H^{\f32+2\delta_1}_xL^2_{\gamma+4}}^2\|W_{\gamma+4} h\|_{H^{\f32+2\delta_1}_xL^2}^2
\\&&+\|W_{\gamma/2}W_{\gamma+4}f\|_{H_x^{\f32+2\delta_1}L^2}^2\|  h\|_{H_x^{\f32+2\delta_1 }L^2_4}^2  
+ \|f\|_{H_x^{\f32+2\delta_1}L^2_{\gamma+4}}^2\|W^\ep_s(D)W_{\gamma+4}W_{\gamma/2+2s} h\|^2_{H_x^{ \f32+\delta_1}L^2}  + \|h\|^2_{H_x^{\f32+2\delta_1}L^2_{\gamma+4}}\\&&\times\|W^\ep_s(D)W_{\gamma+4}W_{\gamma/2+2s} g\|^2_{H_x^{\f32+2\delta_1}L^2}
 +\| W_{\gamma+4}W_{\gamma/2}h\|^2_{H^{\f32+2\delta_1}_xL^2}\|g\|^2_{H^{\f32+2\delta_1}_xL^2_4}\bigg)d\tau.\eeno
 \end{prop}

 \subsubsection{Propagation and gain of regularity for $x$ variable in the case of  $\{m=1, [1/2\varrho]\le  n\le N_{\varrho,1}\}$ and  $\{2\le m\le N, 0\le n\le N_{\varrho,\kappa}\}$}
We have

  \begin{prop}\label{Enxmn2} Suppose that $W_{m,n}\in \mathbb{W}_I(N,\kappa,\varrho,\delta_1,q_1,q_2)(\,or\,\, \mathbb{W}_{II}(N,\kappa,\varrho,\delta_1))$. Then
 \beno && \|W_{m,n} h(t)\|^2_{H^{m+n\varrho}_xL^2}+\f{c_o}3\int_0^t\bigg(\int_{\TT^3} \mathcal{E}_f^\gamma(W_{m,n} |D_x|^{m+n\varrho}   h)dx\\&&\qquad+c_oA_5(c_1, c_2)\big(\|W^\epsilon_s(D)W_{\gamma/2}W_{m,n}  h\|_{H^{m+n\varrho}_xL^2}^2+\int_{\TT^3} \mathcal{E}^{0,\eps}_\mu(W_{\gamma/2}W_{m,n} |D_x|^{m+n\varrho}   h)dx\big)\bigg)d\tau\\
&&\le  \|W_{m,n}   h_0\|^2_{H^{m+n\varrho}_xL^2}+A_6(c_1,c_2)\int_0^t\|W_{\gamma/2}W_{m,n}    h\|_{H^{m+n\varrho}_xL^2}^2d\tau+C_E\int_0^t \bigg(\|W_{m,n}   h\|_{H^{m+n\varrho}_xL^2}^2\\
&& \times\|f\|^2_{H^{\f32+\delta_1}_xL^2_{\gamma+4}} +\|W_{m,n}W_{\gamma/2}f\|^2_{H^{\f32+\delta_1}_xL^2} \|  h\|_{H^{m+n\varrho}_xL^2_4}^2
+\|f\|^2_{H_x^{\f32+2\delta_1}L^2_{\gamma+4}}\|W^\ep_s(D)W_{m,n}W_{\gamma/2+2s} h\|^2_{H_x^{m+n\varrho-\delta_1}L^2}\\&&+ \|f\|^2_{H_x^{m+n\varrho}L^2_{\gamma+4}}\|W^\ep_s(D)W_{m,n}W_{\gamma/2+2s} h\|^2_{H_x^{\f32+\delta_1}L^2}+\|W_{m,n}W_{\gamma/2}f\|^2_{H^{m+n\varrho}_xL^2} \|h\|^2_{H^{\f32+\delta_1}_xL^2_4}+
\|h\|^2_{H_x^{\f32+\delta_1}L^2_{\gamma+4}}\\&&\times\|W^\ep_s(D)W_{m,n}W_{\gamma/2+2s} g\|^2_{H_x^{m+n\varrho}L^2}+ \|h\|^2_{H_x^{m+n\varrho}L^2_{\gamma+4}}\|W^\ep_s(D)W_{m,n}W_{\gamma/2+2s} g\|^2_{H_x^{\f32+\delta_1}L^2}\\ &&+\|W_{m,n}W_{\gamma/2}h\|_{H^{m+n\varrho}_xL^2}^2 \|g\|^2_{H^{\f32+\delta_1}_xL^2_4}
 +\|W_{m,n}W_{\gamma/2}h\|^2_{H^{\f32+\delta_1}_xL^2} \|g\|^2_{H^{m+n\varrho}_xL^2_4}
  \bigg)d\tau.\eeno    \end{prop}
  \begin{prop}\label{smxmn2}  Suppose that $W_{m,n}\in \mathbb{W}_I(N,\kappa,\varrho,\delta_1,q_1,q_2)(\,or\,\, \mathbb{W}_{II}(N,\kappa,\varrho,\delta_1))$. Then
 \beno  &&\int_0^t \|W_{-d_1}(W_{m,n+1}W_{\gamma/2+d_1+d_2}    h)_\phi\|_{H^{m+(n+1)\varrho}_xL^2}^2 d\tau
  \lesssim \eta^{-8}\|W_{m,n+1}W_{\gamma/2+d_1+d_2}    h_0\|^2_{H^{m+n\varrho}_xL^2}+\eta^{2s}\\&&\times \int_0^t  \|W^\ep_s(D)W_{m,n+1}W_{\gamma/2+d_1+d_2}    h\|_{H^{m+n\varrho}_xL^2}^2 d\tau +\eta^{-8} \int_0^t  \|W_{m,n+1}W_{\gamma/2+d_1+d_2}    h\|_{H^{m+n\varrho}_xL^2}^2 d\tau
+\eta^{-8}\\&&\times\bigg[\int_0^t \bigg( \mathrm{1}_{2s>1}(\eta_1^2+\epsilon^{2(1-s)} \|f\|^2_{H_x^{\f32+\delta_1}L^2_{\gamma+4}})\|W^\ep_s(D)W_{m,n+1}W_{\f32\gamma+2s+d_1+d_2}h\|^2_{H_x^{m+n\varrho}L^2}  + \eta_1^{-\f{2(2s-1)}{1-s}} \|f\|^{\f{2(2s-1)}{1-s}}_{H_x^{\f32+\delta_1}L^2_{\gamma+4}}\\&&\times\|W_{m,n+1}W_{\f32\gamma+2s+d_1+d_2}h\|^2_{H_x^{m+n\varrho}L^2}+
 \|f\|^2_{H_x^{\f32+\delta_1}L^2_{\gamma+4}}\| W_{m,n+1}W_{\f32\gamma+2s+d_1+d_2} h\|^2_{H_x^{m+n\varrho}L^2} +\|  h\|^2_{H_x^{m+n\varrho}L^2_4}\\&&\times\|W_{m,n+1}W_{\f32\gamma+d_1+d_2} f\|^2_{H_x^{\f32+\delta_1}L^2}
 + \|f\|^2_{H_x^{\f32+2\delta_1}L^2_{\gamma+4}}\|W^\ep_s(D)W_{m,n+1}W_{\f32\gamma+2s+d_1+d_2} h\|^2_{H_x^{m+n\varrho-\delta_1}L^2}\\&&  +\|W^\ep_s(D)W_{m,n+1}W_{\f32\gamma+2s+d_1+d_2} h\|^2_{H_x^{\f32+\delta_1}L^2} \|f\|^2_{H_x^{m+n\varrho}L^2_{\gamma+4}}
 +  \|W_{m,n+1}W_{\f32\gamma+d_1+d_2}f\|^2_{H^{m+n\varrho}_xL^2} \|h\|^2_{H^{\f32+\delta_1}_xL^2_4}\\&&+\|W^\ep_s(D)W_{m,n+1}W_{\f32\gamma+2s+d_1+d_2} g\|^2_{H_x^{m+n\varrho}L^2} 
\|h\|^2_{H_x^{\f32+\delta_1}L^2_{\gamma+4} }+\|W^\ep_s(D)W_{m,n+1}W_{\f32\gamma+2s+d_1+d_2} g\|^2_{H_x^{\f32+\delta_1}L^2} \\&&\times\|h\|^2_{H_x^{m+n\varrho}L^2_{\gamma+4}}+  \|W_{m,n+1}W_{\f32\gamma+d_1+d_2}h\|^2_{H^{m+n\varrho}_xL^2} \|g\|^2_{H^{\f32+\delta_1}_xL^2_4}
+ \| W_{m,n+1}W_{\f32\gamma+d_1+d_2}h\|^2_{H^{\f32+\delta_1}_xL^2}\\&&\times \|g\|^2_{H^{m+n\varrho}_xL^2_4}   \bigg)d\tau\bigg].\eeno
  \end{prop}

\subsubsection{Propagation   of regularity for  $x$ variable in the case of  $\{m=N, n\varrho=\kappa\}$} We have
\begin{prop}\label{EnNk} Suppose that $W_{N,\kappa}\in \mathbb{W}_I(N,\kappa,\varrho,\delta_1,q_1,q_2)(\,or\,\, \mathbb{W}_{II}(N,\kappa,\varrho,\delta_1))$. Then
		\beno && \|W_{N,\kappa}   h(t)\|^2_{H^{N+\kappa}_xL^2}+\f{c_o}3\int_0^t\bigg(\int_{\TT^3} \mathcal{E}_f^\gamma(W_{N,\kappa}     |D_x|^{N+\kappa} h)dx\eeno\beno&&\qquad+c_oA_5(c_1, c_2)\big(\|W^\epsilon_s(D)W_{\gamma/2}W_{N,\kappa}   h\|_{H^{N+\kappa}_xL^2}^2+\int_{\TT^3} \mathcal{E}^{0,\eps}_\mu(W_{\gamma/2}W_{N,\kappa}  |D_x|^{N+\kappa} h)dx\big)\bigg)d\tau\\
	&&\le  \|W_{N,\kappa}  h_0\|^2_{H^{N+\kappa}_xL^2}+A_6(c_1,c_2)\int_0^t\|W_{\gamma/2}W_{N,\kappa}  h\|_{H^{N+\kappa}_xL^2}^2d\tau+C_E\int_0^t \bigg(\|f\|_{H^{\f32+\delta_1}_xL^2_{\gamma+4}}^2\\&&\times\|W_{N,\kappa}  h\|_{H^{N+\kappa}_xL^2}^2\
	+\|W_{\gamma/2}W_{N,\kappa}f\|_{H_x^{\f32+\delta_1}L^2}^2\|  h\|_{H_x^{N+\kappa}L^2_4}^2
	+ \|f\|_{H_x^{N+\kappa}L^2_{\gamma+4}}^2\|W^\ep_s(D)W_{N,\kappa}W_{\gamma/2+2s} h\|^2_{H_x^{ \f32+\delta_1}L^2}  \\&& 	+ \|f\|_{H_x^{\f32+2\delta_1}L^2_{\gamma+4}}^2\|W^\ep_s(D)W_{N,\kappa}W_{\gamma/2+2s} h\|^2_{H_x^{N+\kappa-\delta_1 }L^2}  + \|W_{\gamma/2}W_{N,\kappa}f\|_{H_x^{N+\kappa}L^2}^2\|  h\|_{H_x^{\f32+\delta_1 }L^2_4}^2\\&& + \|h\|^2_{H_x^{\f32+\delta_1}L^2_{\gamma+4}}\|W^\ep_s(D)W_{N,\kappa}W_{\gamma/2+2s} g\|^2_{H_x^{N+\kappa}L^2}
	+\|h\|^2_{H_x^{N+\kappa}L^2_{\gamma+4}}\|W^\ep_s(D)W_{N,\kappa}W_{\gamma/2+2s} g\|^2_{H_x^{\f32+\delta_1}L^2} \\&& +\| W_{N,\kappa}W_{\gamma/2}h\|^2_{H^{N+\kappa}_xL^2}\|g\|^2_{H^{\f32+\delta_1}_xL^2_4}+\| W_{N,\kappa}W_{\gamma/2}h\|^2_{H^{\f32+\delta_1}_xL^2}\|g\|^2_{H^{ N+\kappa}_xL^2_4}\bigg)d\tau.\eeno
\end{prop}

 \subsubsection{Mixed energy estimates for the equation} It is easy to check that 
  $\pa_x^\alpha\bar{\triangle}^{a}_k\mathfrak{F}_j h$ with $a\in (0,1)$ solves
  \beno && \pa_t (\bar{\triangle}^{a}_k\pa_x^\alpha\mathfrak{F}_j h) +v\cdot\na_x(\bar{\triangle}^{a}_k\pa_x^\alpha \mathfrak{F}_j h)=[v\cdot \na_x, \mathfrak{F}_j]\bar{\triangle}^{a}_k\pa_x^\alpha h+Q^\ep(f, \bar{\triangle}^{a}_k\pa_x^\alpha \mathfrak{F}_j h)\\&&\quad+\sum_{|\alpha_1|\ge1;\alpha_1+\alpha_2=\alpha}  Q^\ep(\pa_x^{\alpha_1} f, \bar{\triangle}^{a}_k\pa_x^{\alpha_2} \mathfrak{F}_j h)+\sum_{ \alpha_1+\alpha_2=\alpha} Q^\ep(\pa_x^{\alpha_1} \bar{\triangle}^{a}_kf,  \pa_x^{\alpha_2}T_k\mathfrak{F}_j h)\\\qquad&&+\sum_{ \alpha_1+\alpha_2=\alpha}
  \bigg( \big[\mathfrak{F}_j Q^\ep(\pa_x^{\alpha_1}f, \bar{\triangle}^{a}_k\pa_x^{\alpha_2} h)-Q^\ep(\pa_x^{\alpha_1}f, \bar{\triangle}^{a}_k\pa_x^{\alpha_2} \mathfrak{F}_j h)\big]+ \big[\mathfrak{F}_j Q^\ep(\pa_x^{\alpha_1}\bar{\triangle}^{a}_kf, \pa_x^{\alpha_2}T_k h)\\\qquad&&-Q^\ep(\pa_x^{\alpha_1}\bar{\triangle}^{a}_k f, \pa_x^{\alpha_2}T_k \mathfrak{F}_j h)\big] +Q^\ep(\pa_x^{\alpha_1} h,\bar{\triangle}^{a}_k\pa_x^{\alpha_2} \mathfrak{F}_j g)+Q^\ep(\pa_x^{\alpha_1}\bar{\triangle}^{a}_k h,\pa_x^{\alpha_2}T_k\mathfrak{F}_j g)+\big[\mathfrak{F}_j Q^\ep(\pa_x^{\alpha_1}h,\bar{\triangle}^{a}_k\pa_x^{\alpha_2} g)\\\qquad&&-Q^\ep(\pa_x^{\alpha_1}h,\bar{\triangle}^{a}_k\pa_x^{\alpha_2} \mathfrak{F}_j g)\big]+\big[\mathfrak{F}_j Q^\ep(\pa_x^{\alpha_1}\bar{\triangle}^{a}_k h,\pa_x^{\alpha_2}T_k g)-Q^\ep(\pa_x^{\alpha_1}\bar{\triangle}^{a}_k h,\mathfrak{F}_j \pa_x^{\alpha_2}T_k g)\big]\bigg)
  \eqdef \sum_{i=1}^{10}R_i.\eeno
Then we have    the propagation of  $\|W^\eps_{q}(D)h\|_{H^{\f32+\delta_1}_xL^2}$:
 \begin{prop}\label{Enmix} For $q\ge s$ and $\eta>0$, we have
 \beno
 && \|W^\ep_q(D)h(t)\|_{H^{\f32+\delta_1}_xL^2}^2+\mathcal{C}_1(c_1,c_2)\int_0^t \|W^{\ep}_{q+s}(D) W_{\gamma/2}h\|^2_{H^{\f32+\delta_1}_xL^2}d\tau\\
 &&\le \|W^\ep_q(D)h_0\|_{H^{\f32+\delta_1}_xL^2}^2+ \mathcal{C}_2(c_1,c_2)\int_0^t\|W^\ep_q(D)W_{\gamma/2}h\|_{H^{\f32+\delta_1}_xL^2}^2d\tau+C_E\int_0^t \bigg(\|W^{\ep}_{(q-1)^+}(D)h\|_{H^{\f52+\delta_1}_xL^2}^2
 \\&&\quad+\|f\|_{H^{\f32+2\delta_1}_xL^2_{\gamma+4}}^2\|W^{\ep}_{q+s}(D)W_{\gamma/2+2s}h\|_{H^{\f32}_xL^2}^2+\|h\|_{H^{\f32+\delta_1}_xL^2_{\gamma+4}}^2\|W^{\ep}_{q+s}(D)W_{\gamma/2+2s}g\|_{H^{\f32+\delta_1}_xL^2}^2\\&&\quad+ \|f\|_{H^{\f32+\delta_1}_xL^2}^2( \mathrm{1}_{2s>1}\|W^\ep_{q+s-1}(D)W_{\gamma/2+\f52}h\|^2_{H^{\f32+\delta_1}_xL^2}+\mathrm{1}_{2s=1}\|W^\ep_{q-s+\log}(D)W_{\gamma/2+\f52}h\|^2_{H^{\f32+\delta_1}_xL^2}\\&&\quad+\mathrm{1}_{2s<1}\|W^{\eps}_{q-s}W_{\gamma/2+\f52}h\|_{H^{\f32+\delta_1}_xL^2}^2)
 + \|f\|_{H^{\f32+\delta_1}_xL^2_{\gamma+3}}( \mathrm{1}_{2s>1}\| W^\ep_{q+s-1}(D)h\|^2_{H^{\f32+\delta_1}_xL^2}\\
&&\quad+\mathrm{1}_{2s=1} \|W^\ep_{q-s+\log}(D)h\|_{H^{\f32+\delta_1}_xL^22}^2+\mathrm{1}_{2s<1}\|W^\eps_{q-s}(D) h\|_{H^{\f32+\delta_1}_xL^2}^2)  + \|W^\ep_{q+s-1+\eta}(D)f\|_{H^{\f32+\delta_1}_xL^2}^2\\&&\quad\times \| h\|_{H^{\f32+\delta_1}_xL^2}^2+  \| f\|_{H^{\f32+\delta_1}_xL^2_2}^2 \|h\|_{H^{\f32+\delta_1}_xL^2}^2
 +\|h\|_{H^{\f32+\delta_1}_xL^2}^2( \mathrm{1}_{2s>1}\|W^\ep_{q+s-1}(D)W_{\gamma/2+\f52}g\|^2_{H^{\f32+\delta_1}_xL^2}\\&&\quad+\mathrm{1}_{2s=1}\|W^\ep_{q-s+\log}(D)W_{\gamma/2+\f25}g\|^2_{H^{\f32+\delta_1}_xL^2}+\mathrm{1}_{2s<1}\|W^\eps_{q-s}(D)W_{\gamma/2+\f52}g\|_{H^{\f32+\delta_1}_xL^2}^2)
 \\&&\quad+ \|h\|_{H^{\f32+\delta_1}_xL^2_{\gamma+3}}( \mathrm{1}_{2s>1}\| W^\ep_{q+s-1}(D)g\|^2_{H^{\f32+\delta_1}_xL^2}+\mathrm{1}_{2s=1} \|W^\ep_{q-s+\log}(D)g\|_{H^{\f32+\delta_1}_xL^2}^2\\ &&\quad+\mathrm{1}_{2s<1}\|W^\eps_{q-s}(D)g\|_{H^{\f32+\delta_1}_xL^2}^2)
+ \|W^\ep_{q+s-1+\eta}(D)h\|_{H^{\f32+\delta_1}_xL^2}^2 \| g\|_{H^{\f32+\delta_1}_xL^2}^2+ \| h\|_{H^{\f32+\delta_1}_xL^2}^2 \|g\|_{H^{\f32+\delta_1}_xL^2}^2\bigg)d\tau.
 \eeno
\end{prop}
Let us sketch the proof of the proposition. We first remark that $R_1$ can be handled by copying the argument in Proposition \ref{Envq}. The coercivity estimate for $R_2$ can be derived from Corollary \ref{basic-estimates1} and \eqref{func8}. Next we apply Corollary \ref{basic-estimates3} and \eqref{func8} to $R_3,R_4,R_7$ and $R_8$ to get the corresponding estimates. Finally  $R_5, R_6, R_9$ and $R_{10}$ can be treated by Plancherel equality with respect to $x$ variable and also by $(i)$ of Lemma \ref{comforsym3}.

 \section{Well-posedness of the Boltzmann equation in weighted Sobolev spaces}
 In this section, we will give   rigorous proofs to the well-posedness for the Botlzmann equations with angular cutoff and without cutoff. To do that, we first 
 show that the linear Boltzmann equation admits a unique and non-negative solution in weighted Sobolev spaces. Meanwhile we show that the energy estimates can be closed in the function space $\mathbb{E}^{N,\kappa,\eps}$. Next
 by the standard Picard iteration scheme and the estimates obtained for the linear equation, we  give the proof of the well-posedness for the nonlinear Boltzmann equation with angular cutoff. Based on the uniform bounds obtained from the previous steps, we finally get the well-posedness for the equation without angular cutoff.

\subsection{Well-posedness of the linear Boltzmann equation} Suppose that $\mathbb{E}^{N,\kappa,\eps}$ and $\mathbb{E}^{1,\f12+2\delta_1,\eps}$ are function spaces associated to $\mathbb{W}_{I}(N,\kappa,\varrho,\delta_1,q_1,q_2)$ where  $  \mathbb{W}_{I}(N,\kappa,\varrho,\delta_1,q_1,q_2)=\{W_{1.\f12+\delta_1}, W_{1.\f12+2\delta_1}\}\cup \{W_{m,n}\}_{(m,n)\in \mathbb{I}_x(N,
	\kappa)}$ with $N+\kappa\ge \f52+\delta_1$ and $ W_{0,-1}=W_{l_1},  W_{0,0}=W_{l_2}$. In this subsection we will prove the existence of non-negative solutions to the linear Boltzmann equation 
\ben\label{EpLinB}  \left\{
\begin{aligned}
	& \pa_t f+v\cdot \na_x f =Q^{\ep}(g,f),\\
	&f|_{t=0}=f_0,\end{aligned}\right. \een
where $g$ is a non-negative function verifying the conditions \ben\label{lubforg}
&&\inf_{x\in\TT^3, t\in[0,T]}|g|_{L^1}\ge  c_1; \sup_{ t\in[0,T]}\|g\|_{H^{\f32+2\delta_1}_xL^2_{\gamma+4}}^2\le c_2; \sup_{ t\in[0,T]}V^{q_1,\eps}(g(t))\le 4V^{q_1,\eps}(g(0));\\
&& \sup_{t\in[0,T]} \mathbb{E}^{1,\f12+2\delta_1,\eps}(g(t))+A_7\int_0^T\mathbb{D}_2^{1,\f12+2\delta_1,\eps}(g(\tau))d\tau\nonumber\\&&\quad+A_8\int_0^T\mathbb{D}_1^{1,\f12+2\delta_1,\eps}(g(\tau))d\tau\le 4M_1;
\int_0^T \mathbb{D}^{1,\f12+2\delta_1,\eps}_3(g(\tau))d\tau \le 4A_{9} M_1\label{322regu};\\
&& \sup_{t\in[0,T]} \mathbb{E}^{N,\kappa,\eps}(g(t)) +A_7\int_0^T\mathbb{D}_2^{N, \kappa,\eps}(g(\tau))d\tau \le 4M_2;
\int_0^T \mathbb{D}^{N,\kappa,\eps}_3(g(\tau))d\tau \le 4A_{9} M_2;\label{52regu} \een
where $T\le \mathfrak{L}$, \ben\label{defiAi}\quad &&A_7\eqdefa \f12\min\{A_2,\mathcal{C}_1,c_oA_5\}, A_8\eqdefa\min\{A_1,A_3\}, A_9\eqdefa C_E \bigg(\f{c_0A_7A_5}{20(N_{\varrho,1}+2)(2C_Ec_0A_5+1)}\bigg)^{-\f{4}{s}}+1.\een

To achieve the goal, 
 our strategy lies in the  
construction of the approximate equations  and   {\it a priori} estimates for the linear equation \eqref{EAPLN:h-eqaution}.

\subsubsection{Well-posedness for the approximate equation (I)} We first want to solve the  equation:
\ben\label{NepLinB}  \left\{
    \begin{aligned}
 & \pa_t f+v\cdot \na_x f+\eta\lr{v}^{\gamma+2s}f=Q^{\ep}_N(g,f),\\
&f|_{t=0}=f_0,\end{aligned}\right. \een
where $Q^{\ep}_N(g,f)\eqdefa\int_{\sigma,v_*} |v-v_*|^{\gamma}\mathrm{1}_{|v-v_*|\le N} b^{\eps}(\cos\theta)(g_*'f'-g_*f)d\sigma dv_*\eqdefa Q^{\ep+}_N(g,f)-L_N^\ep(g)f.$
We have the proposition:
\begin{prop} Let  $W_{l_1},W_{l_2},W_{2,0}\in\mathbb{W}_{I}(N,\kappa,\varrho,\delta_1,q_1,q_2)$. Suppose $g$ is a non-negative function verifying the conditions \eqref{lubforg},\eqref{322regu} and \eqref{52regu}.  Then \eqref{NepLinB} admits a unique and non-negative solution in  $L^\infty([0,T]; L^1_xL^1_{l_1}\cap L^2_xL^2_{l_2}\cap H^2_xL^2_{l_{2,0}})\cap L^1([0,T]; L^1_xL^1_{l_1+\gamma})\cap L^2([0,T]; L^2_xL^2_{l_2+\gamma/2}\cap H^2_xL^2_{l_{2,0}+\gamma/2})$ if $f_0\in L^1_xL^1_{l_1}\cap L^2_xL^2_{l_2}\cap H^2_xL^2_{l_{2,0}}$.
\end{prop}
\begin{proof} To prove the proposition, we introduce the approximate equation to \eqref{NepLinB}:
\ben\label{ANepLinB}  \left\{
    \begin{aligned}
 & \pa_t f^n+v\cdot \na_x f^n+\eta\langle v\rangle^{\gamma+2s}f^n+L_N^\ep(g)f^n=Q^{\ep+}_N(g,f^{n-1}),\\
&f^n|_{t=0}=f_0,\\
&f^0=f_0.\end{aligned}\right. \een
The equation \eqref{ANepLinB} is easily solved by characteristic method in frequency space. Moreover, by Duhamel formula,  we get for each $n$, $f^n\ge0$. Next we will prove the well-posedness for \eqref{NepLinB}. The proof falls into two steps.

{\it Step 1: Uniform bounds for $\{f^n\}_{n\in \N}$.}
We shall give the uniform bounds for $\{f^n\}_{n\in \N}$ in the space $L^\infty([0,T]; L^1_xL^1_{l_1}\cap L^2_xL^2_{l_2}\cap H^2_xL^2_{l_{2,0}})$.

{\it  \underline{$L^1_l$-estimate.}} It is not difficult to check that
\beno   \f{d}{dt} \|f^n\|_{L^1_{l_1}}+\eta \|f^n\|_{L^1_{l_1+\gamma+2s}} 
 \le\bigg|\int_{} |v-v_*|^\gamma\mathrm{1}_{|v-v_*|\le N}b^\ep(\cos\theta)g_*f^{n-1}\lr{v'}^{l_1}d\sigma dv_*dvdx \bigg|.\eeno
By using $\lr{v'}^l\lesssim \lr{v}^l+|v'-v|^l\lesssim \lr{v}^l+\theta^l(\lr{v}^l+\lr{v_*}^l)$, we infer that
\beno   \f{d}{dt} \|f^n\|_{L^1_{l_1}}+\eta \|f^n\|_{L^1_{l_1+\gamma+2s}} 
 &\lesssim& N^\gamma (\ep^{-2s}\|g\|_{L^\infty_xL^1}\|f^{n-1}\|_{L^1_{l_1}}+\|g\|_{L^1_{l_1}}\|f^{n-1}\|_{L^\infty_xL^1}
) \\
&\lesssim& N^\gamma (\ep^{-2s}\|g\|_{H^2_xL^2_2}\|f^{n-1}\|_{L^1_{l_1}}+\|g\|_{L^\infty([0,T];L^1_{l_1})}^2+\|f^{n-1}\|_{H^2_xL^2_2}^2
).\eeno

{\it \underline{ $L^2_l$-estimate.}} Using the fact $(L_N^\ep(g)f^nW_{l_2},f^nW_{l_2})\ge0$, we first have
\beno  \f12\f{d}{dt} \|f^n\|_{L^2_{l_2}}^2+\eta \|f^n\|_{L^2_{l_2+\gamma/2+s}}^2 
&\le&\bigg|\int_{} |v-v_*|^\gamma\mathrm{1}_{|v-v_*|\le N}b^\ep(\cos\theta)g_*f^{n-1}W_{l_2}'(fW_{l_2})'d\sigma dv_*dvdx \bigg|\eeno

 Using the fact $\lr{v'}^l\lesssim \lr{v}^l+\theta^l(\lr{v}^l+\lr{v_*}^l)$ and the argument applied to $B$ in {\it Step 1.2} of the proof of Lemma \ref{commforweight}, we derive that
\beno   \f{d}{dt} \|f^n\|_{L^2_{l_2}}^2+\eta \|f^n\|_{L^2_{l_2+\gamma/2+s}}^2 
&\lesssim& N^\gamma (\ep^{-2s}\|g\|_{L^\infty_xL^1}\|f^{n-1}\|_{L^2_{l_2}}\|f^{n}\|_{L^2_{l_2}}+\|g\|_{L^2_{l_2}}\|f^{n-1}\|_{L^\infty_xL^1}\|f^{n}\|_{L^2_{l_2}})\\
&\lesssim& N^\gamma (\ep^{-2s}\|g\|_{H^2_xL^2_2}\|f^{n-1}\|_{L^2_{l_2}}\|f^{n}\|_{L^2_{l_2}}+\|g\|_{L^2_{l_2}}\|f^{n-1}\|_{H^2_xL^2_2}\|f^{n}\|_{L^2_{l_2}}).
\eeno

{\it \underline{$H_x^2L^2_{l_{2,0}}$-estimate.}}  It is easy to check
\beno &&\pa_t \pa^\alpha_xf^n+v\cdot \na_x \pa^\alpha_xf^n+\eta\lr{v}^{\gamma+2s}\pa^\alpha_xf^n+L_N^\ep(g)\pa^\alpha_xf^n
\\&&=\sum_{|\alpha_1|\ge1;
\alpha_1+\alpha_2=\alpha} L_N^\ep(\pa^{\alpha_1}g)\pa^{\alpha_2}_xf^n+\sum_{\alpha_1+\alpha_2=\alpha}Q^{\ep+}_N(\pa^{\alpha_1}_xg, \pa^{\alpha_2}_xf^{n-1}).\eeno
By applying the similar argument used in   $L^2_l$ estimate, we have
\beno  && \f{d}{dt} \|f^n\|_{H^2_xL^2_{l_{2,0}}}^2+\eta \|f^n\|_{H^2_xL^2_{l_{2,0}+\gamma/2+s}}^2 
\lesssim  \sum_{|\alpha_1|\ge1;
\alpha_1+\alpha_2=\alpha} N^\gamma \int |\pa_x^{\alpha_1}g_*|b^\ep(\cos\theta) |\pa_x^{\alpha_2}f^{n-1}| (W'_{2,0})^2\\&&\qquad\times|(\pa_x^\alpha f^{n}|)'d\sigma dv_*dvdx 
  +   \sum_{
\alpha_1+\alpha_2=\alpha} N^\gamma \int |\pa_x^{\alpha_1}g_*|b^\ep(\cos\theta)  |\pa_x^{\alpha_2}f^{n-1}| (W_{2,0})^2|\pa_x^\alpha f^{n}|d\sigma dv_*dvdx \\
 &&\qquad\lesssim N^\gamma\big(\ep^{-2s}\|g\|_{H^2_xL^2_2}\|f^{n-1}\|_{H^2_xL^2_{l_{2,0}}}+\|g\|_{H^2_xL^2_{l_{2,0}}}\|f^{n-1}\|_{H^2_xL^2_{2}}\big)\|f^n\|_{H^2_xL^2_{l_{2,0}}}.
\eeno

{\it \underline{Closing the energy estimates}.} We set
$ E^n(f)(t)\eqdefa\|f^n(t)\|_{L^1_{l_1}}+\|f^n(t)\|_{L^2_{l_2}}^2+\|f^n(t)\|_{H^2_xL^2_{l_{2,0}}}^2+\|g\|_{L^\infty([0,T];L^1_{l_1})}^2$ and $ E_0\eqdefa\|f_0\|_{L^1_{l_1}}+\|f_0\|_{L^2_{l_2}}^2+\|f_0\|_{H^2_xL^2_{l_{2,0}}}^2+\|g\|_{L^\infty([0,T];L^1_{l_1})}^2$.
It is easy to see $E^0(f)(t)=E^n(f)(0)= E_0$. The estimates in the above can be summarized as follows:
\beno \f{d}{dt} E^n(t)\lesssim C(N,\ep, \|g\|_{L^\infty([0,T]; L^2_{l_2}\cap H^2_xL^2_{l_{2,0}}}) (E^n(t)+E^{n-1}(t)).\eeno
which together with \eqref{322regu} and \eqref{52regu} imply that for $t\le T$,
\ben\label{NepLinBub}  E^n(t)\lesssim e^{CT}E_0\sum_{i=0}^n \f{(e^{CT})^i }{i!}\lesssim C(T,N,\ep, \|g\|_{L^\infty([0,T]; L^2_{l_2}\cap H^2_xL^2_{l_{2,0}}}).\een

{\it Step 2: Convergence of  $\{f^n\}_{n\in \N}$. } We want to prove that $\{f^n\}_{n\in \N}$ is a Cauchy sequence in $L^\infty([0,T]; L^2)$. By setting $h^n=f^n-f^{n-1}$, we have 
\beno  \left\{
    \begin{aligned}
 & \pa_t h^n+v\cdot \na_x h^n+\eta\lr{v}^{\gamma+2s}h^n+L_N^\ep(g)h^n=Q^{\ep+}_N(g,h^{n-1}),\\
&h^n|_{t=0}=0. \end{aligned}\right. \eeno
Due to the energy estimates, we have
\beno  \f{d}{dt}\|h^n\|_{L^2}^2+\eta \|h^n\|^2_{L^2_{\gamma/2+s}}\lesssim N^\gamma \ep^{-2s}\|g\|_{H^2_xL^2_2}\|h^{n-1}\|_{L^2}\|h^n\|_{L^2},\eeno
from which  together with \eqref{322regu} and \eqref{52regu} yield that for $t\le T$,
$ \|h^n(t)\|_{L^2}^2\lesssim \f{(C(T, \ep, N,\|g\|_{L^\infty([0,T];H^2_xL^2_2)}T)^{n-1}}{(n-1)!}.$
Thus we deduce that $\{f^n\}_{n\in \N}$ is a Cauchy sequence in $L^\infty([0,T]; L^2)$.  This together with \eqref{NepLinBub} imply the existence   of non-negative solutions to the equation \eqref{NepLinB}.
\end{proof}

\subsubsection{Well-posedness for the approximate equation (II)} We  want to solve the  equation:
\ben\label{EepLinB}  \left\{
    \begin{aligned}
 & \pa_t f+v\cdot \na_x f+\eta\lr{v}^{\gamma+2s}f=Q^{\ep}(g,f),\\
&f|_{t=0}=f_0.\end{aligned}\right. \een
We have the proposition:
\begin{prop}\label{wepoepl} Let  $W_{l_1},W_{l_2},W_{2,0}\in\mathbb{W}_{I}(N,\kappa,\varrho,\delta_1,q_1,q_2)$.  Suppose $g$ is a non-negative function verifying conditions \eqref{lubforg},\eqref{322regu} and \eqref{52regu}. Then \eqref{EepLinB} admits a unique and non-negative solution in  $L^\infty([0,T]; L^1_xL^1_{l_1}\cap L^2_xL^2_{l_2}\cap H^2_xL^2_{l_{2,0}})\cap L^1([0,T]; L^1_xL^1_{l_1+\gamma})\cap L^2([0,T]; L^2_xL^2_{l_2+\gamma/2}\cap H^2_xL^2_{l_{2,0}+\gamma/2})$ if $f_0\in L^1_xL^1_{l_1}\cap L^2_xL^2_{l_2}\cap H^2_xL^2_{l_{2,0}}$.
\end{prop}
\begin{proof} To prove the result, we introduce the following approximate equation
\ben\label{AEepLinB}  \left\{
    \begin{aligned}
 & \pa_t f^n+v\cdot \na_x f^n+\eta\lr{v}^{\gamma+2s}f^n=Q^{\ep}_n(g,f^n)\eqdefa Q^{\ep+}_n(g,f^n)-Q^{\ep-}_n(g,f^n),\\
 &f^n|_{t=0}=f_0.\end{aligned}\right. \een
 To get the result, it suffices to prove the uniform bounds and the convergence for the sequence $\{f^n\}_{n\in\N}$.

 {\it Step 1: Uniform bounds for $\{f^n\}_{n\in\N}$ }. We want to give the uniform bounds for $\{f^n\}_{n\in\N}$ in the space $L^\infty([0,T]; L^1_{l_1}\cap L^2_{l_2}\cap H^2_xL^2_{l_{2,0}})$.
 
{\it \underline{$L^1_{l_1}$-estimate.}} By change of variables, we get
 \beno \f{d}{dt}\|f^n\|_{L^1_{l_1}}+\eta \|f^n\|_{L^1_{l_1}}=\int_{\TT^3\times\R^6\times \SS^2} b^\eps(\cos\theta) |v-v_*|^\gamma \mathrm{1}_{|v-v_*|\le n}g_*f^n(\lr{v'}^{l_1}-\lr{v}^{l_1})d\sigma dv_*dvdx.\eeno
 Thanks to Lemma \ref{Povnzer2}, we infer that
 \beno \f{d}{dt}\|f^n\|_{L^1_{l_1}}+\eta \|f^n\|_{L^1_{l_1+\gamma+2s}}\lesssim l^{-1}_1\|g\|_{L^1_{l_1+\gamma}}\|f^n\|_{L^\infty_xL^1_\gamma}+C(l_1)
\big( \|g\|_{L^1_{l_1}}\|f^n\|_{L^\infty_xL^1_{\gamma+2}}+\|f^n\|_{L^1_{l_1}}\|g\|_{L^\infty_xL^1_{\gamma+2}}\big).\eeno

{\it \underline{$L^2_{l_2}$-estimate}.}  
We first claim that
\ben\label{estimateforqneps} &&\big|\lr{ Q_n^\eps( g, h), W_lf}_v\big|  \lesssim \ep^{-2s} |g |_{ L^2_{\gamma+2}}|h|_{ L^2_{l+\gamma/2}}|f|_{L^2_{\gamma/2}}+|g|_{L^2_{l+\gamma/2}}|h|_{L^2_{\gamma+2}}|f|_{L^2_{\gamma/2}}.\een
In fact, it is easy to check that $\big|\lr{ Q_n^{\eps-}( g, h), W_lf}_v\big|  \lesssim \ep^{-2s} |g _{ L^2_{\gamma+2}}|h|_{ L^2_{l+\gamma/2}}|f|_{L^2_{\gamma/2}}$. Observe that  
\beno |\lr{Q^{\ep+}_n(g,h),W_lf}|= \bigg|\int b^\epsilon |v-v_*|^\gamma\mathrm{1}_{|v-v_*|\le n}g_*hf'(W_l)'d\sigma dv_*dv \bigg|. \eeno
By using the fact $\lr{v'}^l\lesssim \lr{v}^l+|v'-v|^l\lesssim (1+\theta^l)\lr{v}^l+ \lr{v_*}^l$, we copy the argument used  for $B$ in {\it Step 1.2} of the proof of Lemma \ref{commforweight} to the righthand side of the equality. From this, we conclude the claim. Now we derive that
 \beno \f{d}{dt}\|f^n\|_{L^2_{l_2}}^2+\eta\|f^n\|_{L^2_{l_2+\gamma/2+s}}^2\lesssim \|g\|_{L^\infty_xL^2_{\gamma+2}}\|f^n\|_{L^2_{l_2+\gamma/2}}^2+\|g\|_{L^2_{l_2+\gamma/2}}\|f^n\|_{L^\infty_xL^2_{\gamma+2}}\|f^n\|_{L^2_{l_2+\gamma/2}}.\eeno
By the interpolation inequality that $ \|f\|^2_{L^2_{l+\gamma/2}}\lesssim R^\gamma\|f\|_{L^2_l}^2+R^{-2s}\|f\|_{L^2_{l+\gamma/2+s}}^2,$
 we deduce that
 \beno  \f{d}{dt}\|f^n\|_{L^2_{l_2}}^2+\eta\|f^n\|_{L^2_{l_2+\gamma/2+s}}^2\lesssim (\|g\|_{L^\infty_xL^2_{\gamma+2}}+1)(R^{-2s}\|f^n\|_{L^2_{l_2+\gamma/2+s}}^2+R^\gamma\|f^n\|_{L^2_{l_2}}^2)+\|g\|^2_{L^2_{l_2+\gamma/2}}\|f^n\|_{L^\infty_xL^2_{\gamma+2}}^2.\eeno
 Choose $ C(\|g\|_{L^\infty_xL^2_{\gamma+2}}+1)R^{-2s}=\eta/2$, then we  derive that
  \beno  \f{d}{dt}\|f^n\|_{L^2_{l_2}}^2+\eta\|f^n\|_{L^2_{l_2+\gamma/2+s}}^2\lesssim  \|g\|^2_{L^2_{l_2+\gamma/2}}\|f^n\|_{H^2_xL^2_4}^2+C(\eta,\|g\|_{L^\infty([0,T];H^2_xL^2_{\gamma+2})})\|f^n\|_{L^2_{l_2}}^2.\eeno

 {\it \underline{High order estimate.}} It is easy to check that for $|\alpha|\le 2$,
\beno &&\pa_t \pa^\alpha_xf^n+v\cdot \na_x \pa^\alpha_xf^n+\eta\lr{v}^{\gamma+2s}\pa^\alpha_xf^n
 =Q^\ep_n(g,\pa^\alpha_xf^n)+\sum_{|\alpha_1|\ge1;
\alpha_1+\alpha_2=\alpha} Q_n^\ep(\pa^{\alpha_1}_xg,\pa^{\alpha_2}_xf^n).\eeno
Thanks to Plancheral equality with respect to $x$ variable, \eqref{estimateforqneps} as well as Lemma \ref{dsums}, we can derive that  for $|\alpha|\le 2$, it holds
\beno &&\big|\int_{\TT^3\times\R^3} Q_n^{\ep}(\pa^{\alpha_1}_xg,\pa^{\alpha_2}_xh)f dvdx\big|\lesssim \ep^{-2s}\|g\|_{H^2_xL^2_{\gamma+2}}\|h\|_{H^2_xL^2_{l+\gamma/2}}\|f\|_{L^2_{\gamma/2}}+\|g\|_{H^2_xL^2_{l+\gamma/2}}\|h\|_{H^2_xL^2_{\gamma+2}}\|f\|_{H^2_{\gamma/2}}.\eeno
By applying this estimate and the fact $ \|f\|^2_{L^2_{l+\gamma/2}}\lesssim R^\gamma\|f\|_{L^2_l}^2+R^{-2s}\|f\|_{L^2_{l+\gamma/2+s}}^2,$
we get
\beno   \f{d}{dt} \|f^n\|_{H^2_xL^2_{l_{2,0}}}^2+\eta \|f^n\|_{H^2_xL^2_{l_{2,0}+\gamma/2+s}}^2 \lesssim \big[C(\ep,\eta, \|g\|_{L^\infty([0,T];H^2_xL^2_{\gamma+2})})+\|g\|^2_{H^2_xL^2_{l_{2,0}+\gamma/2}}\big]\|f^n\|_{H^2_xL^2_{l_{2,0} }}^2. \eeno

{\it \underline{Closing the energy estimates}.} 
By  setting
$  E^n(t)\eqdefa\|f^n(t)\|_{L^1_{l_1}}+\|f^n(t)\|_{L^2_{l_2}}^2+\|f^n(t)\|_{H^2_xL^2_{l_{2,0}}}^2\\+\|g\|_{L^1([0,T];L^1_{l_1+\gamma})}+\|g\|^2_{L^\infty([0,T];L^1_{l_1})},$
 we obtain that for $t\le T$,
\beno \f{d}{dt}E^n\lesssim \big[C(\ep,\eta, \|g\|_{L^\infty([0,T];L^1_{l_1})}, \|g\|_{L^\infty([0,T];H^2_xL^2_{\gamma+2})})+(\|g\|_{L^1_{l_1+\gamma}}+\|g\|^2_{L^2_xL^2_{l_2+\gamma/2}}+\|g\|_{H^2_xL^2_{l_{2,0}+\gamma/2}}^2)\big]E^n.\eeno
It implies
\beno &&E^n\lesssim  C(\ep,\eta, \|g\|_{L^\infty([0,T];L^1_{l_1})}, \|g\|_{L^\infty([0,T];H^2_xL^2_{\gamma+2})},\|g\|_{L^1([0,T]; L^1_{l_1+\gamma})}, \|g\|_{L^2([0,T];L^2_xL^2_{l_2+\gamma/2})},\\&&\qquad\quad\|g\|_{L^2([0,T];H^2_xL^2_{l_{2,0}+\gamma/2})})\eqdefa \tilde{C}_e.  \eeno

{\it Step 2: Convergence of $\{f^n\}_{n\in \N}$.} We will prove that the sequence $\{f^n\}_{n\in \N}$ is a Cauchy sequence in the space $L^\infty([0,T];L^1)$. Let $h^n=f^n-f^{n-1}$. Then it solves
\beno  \left\{
    \begin{aligned}
 & \pa_t h^n+v\cdot \na_x h^n+\eta\lr{v}^{\gamma+2s}h^n=Q^{\ep}_n(g,h^n)+Q^{\ep}_{n,n-1}(g, f^{n-1}),\\
 &h^n|_{t=0}=0,\end{aligned}\right. \eeno
where $Q^{\ep}_{n,n-1}(g, f^{n-1})=\int_{  \SS^2\times \R^3} b^\ep(\cos\theta)|v-v_*|^{\gamma}\mathrm{1}_{n-1\le|v-v_*|\le n} (g_*'(f^{n-1})'-g_*f^{n-1})d\sigma dv_*$.
 It is easy to check
 \beno \f{d}{dt}\|h^n\|_{L^1}+\eta\|h^n\|_{L^1_{\gamma+s}}&\lesssim& \int_{\TT^3\times\R^6\times\SS^2} b^\ep(\cos\theta)|v-v_*|^{\gamma+2}\mathrm{1}_{n-1\le|v-v_*|\le n} g_*f^{n-1}d\sigma dv_*dvdx\\&\lesssim& C(\eps)\f1{n^2} \|g\|_{H^2_xL^2_{\gamma+4}}\|f^{n-1}\|_{L^1_{\gamma+4}}.\eeno
 Then we deduce that for $t\le T$,
$  \|h^n\|_{L^\infty([0,T];L^1)}\lesssim \tilde{C}_e\f{1}{n^2},$
 which is enough to prove that $\{f^n\}_{n\in \N}$ is a Cauchy sequence in $L^\infty([0,T];L^1)$.

Combining the results in {\it Step 1} and {\it Step 2}, we complete the proof of the proposition.
 \end{proof}

\subsubsection{Energy estimates to the linear equation \eqref{EpLinB} in function space $\mathbb{E}^{N,\kappa,\eps}$}  In this subsection, we want to close the energy estimates in $\mathbb{E}^{N,\kappa,\eps}$ for the linear Boltzmann equation \eqref{EpLinB}.
Before stating our main results, we give several propositions which will be used frequently.

\begin{prop}\label{intx1}  Suppose that $W_{m,n}\in \mathbb{W}_I(N,\kappa,\varrho,\delta_1,q_1,q_2)(\,or\,\, \mathbb{W}_{II}(N,\kappa,\varrho,\delta_1))$. Then for any smooth function $h$,  it holds
	\beno \|W_{m,n}W_{\gamma/2}h\|_{H^{m+n\varrho}_xL^2}^2&\le& ((W_{m,n})^2W_{\gamma})(\eta^{-1/(2d_2)})\|h\|_{H^{m+n\varrho}_xL^2}^2+
	\eta\|W_{-d_1}(W_{m,n}W_{\gamma/2}W_{d_1+d_2}h)_\phi\|_{H^{m+n\varrho}_xL^2}^2\\&&
	+\epsilon^{2s}\|W^\eps_s(D)W_{m,n}W_{\gamma/2}h\|_{H^{m+n\varrho}_xL^2}^2.
	\eeno
	In other words, we have
	\beno \|W_{m,n}W_{\gamma/2}h\|_{H^{m+n\varrho}_xL^2}^2&\lesssim & C(\eta^{-1}, W_{m,n}) E^{m,n,\eps}(h)+\eta D^{m,n-1,\eps}_3(h)+\eps^{2s} D^{m,n,\eps}_2(h).
	\eeno
\end{prop}
\begin{proof}
	It is easy to check
	\beno
	\|W_{m,n}W_{\gamma/2}h\|_{H^{m+n\varrho}_xL^2}&\le& \|\psi_R W_{m,n}W_{\gamma/2}h\|_{H^{m+n\varrho}_xL^2}+\|(1-\psi_R)W_{m,n}W_{\gamma/2}h\|_{H^{m+n\varrho}_xL^2}\\&\le&
	W_{m,n}W_{\gamma/2}(R)\|h\|_{H^{m+n\varrho}_xL^2}+
	\|(1-\psi_R)W_{-(d_1+d_2)}\big(W_{m,n}W_{\gamma/2}W_{d_1+d_2}h\big)_\phi\|_{H^{m+n\varrho}_xL^2}\\&&+\|(W_{m,n}W_{\gamma/2}h)^\phi\|_{H^{m+n\varrho}_xL^2}+
	\|(1-\psi_R)[W_{d_1+d_2}, \phi(\eps D)]W_{m,n}W_{\gamma/2}W_{d_1+d_2}h\|_{H^{m+n\varrho}_xL^2}
	\eeno
	By Lemma \ref{baslem2}, we have  $|[W_{-(d_1+d_2)}, \phi(\eps D)]h|_{L^2}\lesssim \eps|h|_{L^2_{-(d_1+d_2)-1}}$, which implies 
	\beno \|W_{m,n}W_{\gamma/2}h\|_{H^{m+n\varrho}_xL^2}&\le&
	(W_{m,n}W_{\gamma/2})(R)\|h\|_{H^{m+n\varrho}_xL^2}+R^{-d_2}\|W_{-d_1}(W_{m,n}W_{\gamma/2}W_{d_1+d_2}h)_\phi\|_{H^{m+n\varrho}_xL^2}\\&&
	+\epsilon^{s}\|W^\eps_s(D)W_{m,n}W_{\gamma/2}h\|_{H^{m+n\varrho}_xL^2}
	+\eps\|W_{m,n}W_{\gamma/2-1}h\|_{H^{m+n\varrho}_xL^2}.
	\eeno	
	By choosing $\eta^{\f12}=R^{-d_2}$ and using the notations introduced in Section \ref{notspace}, we derive the desired results.
\end{proof}

\begin{rmk}\label{intxbod1}
	By the similar argument, we also have
	\beno &&(i).\,\|W_{1,\f12+\delta_1}W_{\gamma/2}h\|_{H^{\f32+\delta_1}_xL^2}^2+
	\|W_{\gamma+4}W_{\gamma/2}h\|_{H^{\f32+2\delta_1}_xL^2}^2\\&&\lesssim C(\eta^{-1}, W_{m,n}) (E^{1,\f12+\delta_1,\eps}(h)+E^{1,\f12+2\delta_1,\eps}(h))+\eta D_3^{1,N_{\varrho,2},\eps}(h)+\eps^{2s}(D^{1,\f12+\delta_1,\eps}_2(h)+D^{1,\f12+2\delta_1,\eps}_2(h));\\
	&&(ii).\,\|W_{N,\kappa}W_{\gamma/2}h\|_{H^{N+\kappa}_xL^2}^2\lesssim  C(\eta^{-1}, W_{m,n}) E^{N,\kappa,\eps}(h)+\eta D_3^{1,N_{\varrho,\kappa},\eps}(h)+\eps^{2s}D^{N,\kappa,\eps}_2(h).
	\eeno
	Indeed,  using (W-5) of Definition \ref{ws1}  
	and the fact $ \f32+\delta_1=1+(N_{\varrho,2}+1)\varrho+\delta_1-N_d\le1+(N_{\varrho,2}+1)\varrho-N_d/2$, we can prove that  
	\beno
	&&\|W_{1,\f12+\delta_1}W_{\gamma/2}h\|_{H^{\f32+\delta_1}_xL^2}\le
	W_{1,\f12+\delta_1}W_{\gamma/2}(R)\|h\|_{H^{\f32+\delta_1}_xL^2}+
	R^{-d_2/2}\| W_{-d_1}\big(W_{1,\f12+\delta_1}W_{\gamma/2}W_{d_1+d_2/2}h\big)_\phi\|_{H^{\f32+\delta_1}_xL^2}\\&&\quad+\epsilon^{2s}\|W^\eps_s(D)W_{1,\f12+\delta_1}W_{\gamma/2}h\|_{H^{\f32+\delta_1}_xL^2}^2
	+\eps^2\|W_{1,\f12+\delta_1}W_{\gamma/2-1}h\|_{H^{\f32+\delta_1}_xL^2}^2,\eeno
	which is enough to get the first result. We remark that the second one can be obtained by the same way.	
\end{rmk}

\begin{prop}\label{intx2}  Suppose that $W_{m,n}\in \mathbb{W}_I(N,\kappa,\varrho,\delta_1,q_1,q_2)(\,or\,\, \mathbb{W}_{II}(N,\kappa,\varrho,\delta_1))$. Then for any smooth function $h$  we have
	\beno  &&\|W^\eps_s(D)W_{m,n}W_{\gamma/2+2s}h\|_{H^{m+n\varrho-\delta_1}_xL^2} \lesssim \eta(\|W^\eps_s(D)W_{m,n}W_{\gamma/2}h\|_{H^{m+n\varrho}_xL^2}+\|W^\eps_s(D)W_{m,n-1}W_{\gamma/2}h\|_{L^2}
	\\&&\qquad+\|W^\eps_s(D)W_{m,n-1}W_{\gamma/2 -d_3}h\|_{H^{m+(n-1)\varrho}_xL^2}
	)+ \f{\eta^{-2^{N_{\varrho,\delta_1}}}}{2^{N_{\varrho,\delta_1}}}(\eta^{-1-2^{N_{\varrho,\delta_1}}}2^{-N_{\varrho,\delta_1}})^{\f{m+(n-1)\varrho-\delta_1}{\delta_1}}\\&&\qquad\times(W_{\gamma/2-d_3}W_{m,n-1})\big((\eta^{-1-2^{N_{\varrho,\delta_1}}}2^{-N_{\varrho,\delta_1}})^{\f{m+(n-1)\varrho-\delta_1}{\delta_1 d_3}}\big)\|W^\eps_s(D)h\|_{L^2}.\eeno
	In other words, we have
	\beno \|W^\eps_s(D)W_{m,n}W_{\gamma/2+2s}h\|_{H^{m+n\varrho-\delta_1}_xL^2}^2\lesssim C(\eta^{-1},W_{m,n}) \|W^\eps_s(D)h\|_{L^2}^2+\eta(D^{m,n-1,\eps}_2(h)+D^{m,n,\eps}_2(h)). \eeno
\end{prop}
\begin{proof}
	We observe  that for any $J\in\N$,
	\beno \|W_l h\|_{H^{m-\delta_1}_xL^2}\lesssim \|(W_{2^Jl})h\|_{H^{m-(2^J-1)\delta_1}_xL^2}^{2^{-J}L^2}\|h\|_{H^m_xL^2}^{1-2^{-J}L^2}. \eeno
	By choosing that $(2^J-1)\delta_1\ge \varrho+\delta_1$, that is, $J\ge [\log_2(\varrho/\delta_1+2)]+1=N_{\varrho,\delta_1}$, we have
	\beno &&\|W^\eps_s(D)W_{m,n}W_{\gamma/2+2s}h\|_{H^{m+n\varrho-\delta_1}_xL^2} \lesssim \f{\eta^{-2^{N_{\varrho,\delta_1}}}}{2^{N_{\varrho,\delta_1}}}\|W^\eps_s(D)W_{m,n}W_{\gamma/2}(W_{2s})^{2^{N_{\varrho,\delta_1}}}h\|_{H^{m+(n-1)\varrho-\delta_1}_xL^2}\\&&\qquad\qquad\qquad\qquad+
	\eta \|W^\eps_s(D)W_{m,n}W_{\gamma/2}h\|_{H^{m+n\varrho}_xL^2}.
	\eeno
	By   (W-4) of Definition \ref{ws1},
	we get that
	\beno &&\|W^\eps_s(D)W_{m,n}W_{\gamma/2}(W_{2s})^{2^{N_{\varrho,\delta_1}}}h\|_{H^{m+(n-1)\varrho-\delta_1}_xL^2}\lesssim \|W^\eps_s(D)W_{m,n-1}W_{\gamma/2}W_{-d_3}h\|_{H^{m+(n-1)\varrho-\delta_1}_xL^2}\\
	&&\lesssim K^{-\delta_1}\|W^\eps_s(D)W_{m,n-1}W_{\gamma/2}W_{-d_3}h\|_{H^{m+(n-1)\varrho}_xL^2}
	+K^{m+(n-1)\varrho-\delta_1}\|W^\eps_s(D)W_{m,n-1}W_{\gamma/2-d_3}h\|_{L^2}.
	\eeno
	Next we focus on the term $\|W^\eps_s(D)W_{m,n-1}W_{\gamma/2}W_{-d_3}h\|_{L^2}$. It is easy to check that
	\beno \|W^\eps_s(D)W_{m,n-1}W_{\gamma/2-d_3}h\|_{L^2}\lesssim (W_{\gamma/2-d_3}W_{m,n-1})(R)\|W^\eps_s(D)h\|_{L^2}+R^{-d_3}\|W^\eps_s(D)W_{m,n-1}W_{\gamma/2}h\|_{L^2}^2.
	\eeno
	Choose $R^{-d_3}K^{m+(n-1)\varrho-\delta_1}=1$, then we have
	\beno
	&&\|W^\eps_s(D)W_{m,n}W_{\gamma/2}(W_{2s})^{2^{N_{\varrho,\delta_1}}}h\|_{H^{m+(n-1)\varrho-\delta_1}_xL^2}\lesssim K^{-\delta_1}(\|W^\eps_s(D)W_{m,n-1}W_{\gamma/2 -d_3}h\|_{H^{m+(n-1)\varrho}_xL^2}\\&&\qquad+
	\|W^\eps_s(D)W_{m,n-1}W_{\gamma/2}h\|_{L^2})+K^{m+(n-1)\varrho-\delta_1}(W_{\gamma/2-d_3}W_{m,n-1})(K^{\f{m+(n-1)\varrho}{d_3}})\|W^\eps_s(D)h\|_{L^2}.
	\eeno
	The desired result is concluded by choosing $K^{-\delta_1}\f{\eta^{-2^{N_{\varrho,\delta_1}}}}{2^{N_{\varrho,\delta_1}}}=\eta$ and combining all the estimates. We end the proof of the proposition. 
\end{proof}	
\begin{rmk}\label{intxbod2} Thanks to (W-4) and (W-5) of Definition \ref{ws1}, by the similar argument, we have
	\beno &&\|W^\eps_s(D)W_{1,\f12+\delta_1}W_{\gamma/2+2s}h\|_{H^{\f32}_xL^2}^2+
	\|W^\eps_s(D)W_{\gamma+4}W_{\gamma/2+2s}h\|_{H^{\f32+\delta_1}_xL^2}^2\\&&\lesssim  C(\eta^{-1},W_{m,n}) \|W^\eps_s(D)h\|_{L^2}^2+\eta(D^{1,N_{\varrho,2},\eps}_2(h)+D^{1,\f12+\delta_1,\eps}_2(h)+D^{1,\f12+2\delta_1,\eps}_2(h)),\\
	&&\|W^\eps_s(D)W_{N,\kappa}W_{\gamma/2+2s}h\|_{H^{N+\kappa}_xL^2}^2\lesssim C(\eta^{-1},W_{m,n}) \|W^\eps_s(D)h\|^2_{L^2}+\eta(D^{N,N_{\varrho,\kappa,\eps}}_2(h)+D^{N,\kappa,\eps}_2(h)). \eeno
\end{rmk}

Now we state our estimates for \eqref{EpLinB} in function space $\mathbb{E}^{N,\kappa,\eps}$.
\begin{thm}\label{En-Priori}  Suppose that function spaces $\mathbb{E}^{N,\kappa,\eps}$and   $\mathbb{E}^{1,\f12+2\delta_1,\eps}$, the well-prepared sequence $\mathbb{W}_{I}(N,\kappa,\varrho,\delta_1,q_1,q_2)$  and $f_0$ verify all the conditions stated in Theorem \ref{thmwepo}.  Let $g$ verify \eqref{lubforg}, \eqref{322regu} and \eqref{52regu} and $f$ be a solution to \eqref{EpLinB} with the initial data $f_0$.  
	\begin{enumerate}
		\item If   $ \eps\le\min\{[(\f{A_5c_o  A_7}{200A_6(N_{\varrho,1}+2)(2C_Ec_0A_5+1)})^{1+4/s}c_2^{-1}]^{1/(2(1-s))}, l_1^{-\f12-\eta},(\f{A_7}{20(N_{\varrho,1}+2)A_6})^{\f1{2s}}\}$ with $\eta>0,$  	   then there exists a   constant $\mathcal{C}_{E,1}=C(c_1,c_2, M_1, \mathbb{W}_{I},\mathfrak{L})$   defined in \eqref{e321}    such that  
		\beno  &&\sup_{t\in[0,T]} \mathbb{E}^{1,\f12+2\delta_1,\eps}(f(t))+\int_0^T \mathbb{D}^{1,\f12+2\delta_1,\eps}(f(\tau))d\tau \le \mathcal{C}_{E,1},\\
		&&   \sup_{t\in[0,T]} \mathbb{E}^{N,\kappa,\eps}(f(t))+\int_0^{T} \mathbb{D}^{N,\kappa,\eps}(f(\tau))d\tau \le C(T, C_{E,1}, c_1,c_2, M_1, M_2, \mathbb{W}_{I}). \eeno
		
		\item Suppose that $\epsilon$ verifies \eqref{Rstreps}. There exists a time
		$T^*=T^*(c_1,c_2, M_1, M_2, \mathbb{W}_{I})\le T$
		such that  for $t\in [0,T^*]$, $f$ verifies  \eqref{lubforg}, \eqref{322regu} and \eqref{52regu}.
	\end{enumerate}
\end{thm}

\begin{proof}  We divide our proof into several steps.
	
	{\it Step 1: Closing   the energy estimates in space $\mathbb{E}^{1,\f12+2\delta_1,\eps}$.} We first show that the energy estimates can be closed in $\mathbb{E}^{1,\f12+2\delta_1,\eps}$. To do that, we first revisit the estimates obtained in the previous section.
	
	{\it Step 1.1: $L^1$ moment estimates.}
	Thanks to Proposition \ref{L1i}, the $L^1$ moment estimates can be summarized as follows:
	\beno&&(i).\, \|f(t)\|_{L^1_{l_1}}+ l_1^sA _1(c_1,c_2)\int_0^t \|f\|_{L^1_{l_1+\gamma}}d\tau\le \|f_0\|_{L^1_{l_1}} +C(l_1) \bigg[\int_0^t \big((c_2+4M_1)\mathbb{E}^{1,\f12+2\delta_1,\eps}(f(\tau))d \tau\bigg]\\
	&&\qquad\qquad\qquad\qquad\qquad+C_El_1^{-1-s} 
	\int_0^t  \|f\|_{H^{\f32+2\delta_1}_xL^2_{\gamma+4}} \mathbb{D}_1^{1,\f12+2\delta_1,\eps}(g(\tau)) d\tau;
	\\ &&(ii).\, \|f(t)\|_{L^1_{l_1-\gamma}}^2+ l_1^sA _1(c_1,c_2)\int_0^t   \|f\|_{L^1_{l_1}}\|f\|_{L^1_{l_1-\gamma}}d\tau \\
	&&\le \|f_0\|_{L^1_{l_1-\gamma}}^2+C(l_1) \bigg[\int_0^t (c_2+M_1+l^{-1})\mathbb{E}^{1,\f12+2\delta_1,\eps}(f(\tau)) d\tau\bigg].
	\eeno
	
	{\it Step 1.2: $L^2$ moment estimates.} Thanks to Proposition \ref{L2m}, we have
	\beno  &&\|f(t)\|_{L^2_{l_2}}^2+\f16\int_0^t\int_{\TT^3} \mathcal{E}_g^{\gamma,\eps}(fW_{l_2})dxd\tau
	+A_2(c_1,c_2)
	\int_0^t  \bigg[\big(\|W^\ep_s(D)W_{l_2+\gamma/2}f\|_{L^2}^2 \\&&\qquad +\int_{\TT^3}\mathcal{E}_{\mu}^{0,\eps}\big(W_{l_2+\gamma/2}f\big)dx\bigg]d\tau+A_3(c_1,c_2) \delta^{-2s}\int_0^t \|f\|_{L^2_{l_2+\gamma/2}}^2 d\tau\le \|f_0\|_{L^2_{l_2}}^2+(A_4(c_1, c_2)\delta^{-6-6s}
	\\&&\qquad+C_Ec_2)\int_0^t \mathbb{E}^{1,\f12+2\delta_1,\eps}(f(\tau))   d\tau
	+C_E\delta^{2s} \int_0^t \mathbb{D}_1^{1,\f12+2\delta_1,\eps}(g(\tau)) \|f\|_{H^{\f32+2\delta_1}_xL^2_{\gamma+4}}^2
	d\tau.
	\eeno
	
	{\it Step 1.3: Propagation of the regularity with the symbol $W^\eps_q(D)$.} Due  to \eqref{func4} and Proposition \ref{Envq},
	we have
	\beno
	&& V^{q_1,\eps}(f(t))+\mathcal{C}_1(c_1,c_2)\int_0^t \|W^{\ep}_{q_1+s}(D) W_{\gamma/2}f\|^2_{L^2}d\tau \le   V^{q_1,\eps}(f_0)+ \mathcal{C}_2(c_1,c_2)\int_0^t\|W^\ep_{q_1}(D)W_{\gamma/2}f\|_{ L^2}^2d\tau\\&&\,+C_E\int_0^t \bigg(\|W^\eps_{(q_1-1)^+}(D)f\|_{H^{1}_xL^2}^2
	+ \|g\|_{H^{\f32+\delta_1}_xL^2}^2( \mathrm{1}_{2s>1}\|W^\ep_{q_1+s-1}(D)W_{\gamma/2+ \f52}f\|^2_{ L^2}+ \|W^\ep_{q_1-s+\log}(D)W_{\gamma/2+\f52}f\|^2_{L^2} \\&&\,\times\mathrm{1}_{2s=1}+\mathrm{1}_{2s<1}\|W^{\eps}_{q_1-s}(D)W_{\gamma/2+\f52}f\|_{ L^2}^2)
	+ \|g\|_{H^{\f32+\delta_1}_xL^2_{\gamma+3}}^2( \mathrm{1}_{2s>1}\| W^\ep_{q+s-1}(D)f\|^2_{ L^2} +\mathrm{1}_{2s=1} \|W^\ep_{q_1-s+\log}(D)f\|_{ L^2}^2\\&&\,+\mathrm{1}_{2s<1}\|W^\eps_{q_1-s}(D) f\|_{L^2}^2)  + \|W^\ep_{q_1+s-1+\eta}(D)g\|_{ L^2}^2  \| f\|_{H^{\f32+\delta_1}_xL^2}^2+  \| g\|_{H^{\f32+\delta_1}_xL^2_2}^2 \|f\|_{L^2}^2
	\bigg)d\tau.
	\eeno
	
	It is not difficult to check that
	\beno    |W^\eps_{q_1}(D)f|_{L^2_l}&\lesssim& |f|_{L^2_l}^{\f{s}{q_1+s}}|W^\eps_{q_1+s}(D)f|_{L^2_l}^{\f{q_1}{q_1+s}};\\
	\|W^\eps_{(q_1-1)^+}(D)f\|_{H^1_xL^2}&\lesssim&\|W^\eps_{\big((q_1-1)^+-(2^{n_1}-1)(1+s-q_1)\big)^+}(D)f\|_{H^1_xL^2}^{2^{-n_1}}\|W^\eps_{s}(D)f\|_{H^1_xL^2}^{1-2^{-n_1}};\eeno\beno
	|W^\ep_{q_1+s-1}(D)W_{\gamma/2+\f52}f|^2_{L^2}&\lesssim&
	|W^\ep_{(q_1+s-2^{n_2})^+}(D)W_{\gamma/2}W_{\f522^{n_2}}f|^{2^{-n_2}}_{L^2}|W^\ep_{q_1+s}(D)W_{\gamma/2} f|^{1-2^{-n_2}}_{L^2};\\
	|W^\ep_{q_1-s+\eta}(D)W_{\gamma/2+\f52}f|^2_{L^2}&\lesssim&
	|W^\ep_{(q_1-s+\eta-(2^{n_3}-1)(2s-\eta))^+}(D)W_{\gamma/2}\f52W_{2^{n_3}}f|^{2^{-n_3}}_{L^2}|W^\ep_{q_1+s}(D)W_{\gamma/2} f|^{1-2^{-n_3}}_{L^2}.
	\eeno
	Here $\eta$ can be chosen to be sufficiently small. 
	By choosing $n_1=[\log_2(\f{s}{1+s-q_1} )]+1$ if $1<q_1<1+s$ and  $n_2=n_3=N_{q_1,s,1}=\max\{[\log_2(q_1+s)]+1,[\log_2\big(\f{q_1+s}{2s}(1+\delta_1)\big)]+1\}$, and using (W-6) of Definition \ref{ws1}, we conclude that	
	\ben\label{Eforq1}
	&&  V^{q_1,\eps}(f(t))+\f12\mathcal{C}_1(c_1,c_2)\int_0^t \|W^{\ep}_{q_1+s}(D) W_{\gamma/2}f\|^2_{L^2}d\tau\notag\\
	&&\le V^{q_1,\eps}(f_0)+\eta\int_0^t \mathbb{D}_2^{1,\f12+\delta_1,\eps}(f)d\tau +C(C_E, c_1, c_2, M_1, \eta^{-1}, \mathbb{W}_{I})\int_0^t \mathbb{E}^{1,\f12+2\delta_1,\eps}(f(\tau))   d\tau.
	\een
	
	{\it  Step 1.4: High order estimates.} Thanks to Proposition \ref{Enxmn1}, we get
	that \beno && \|W_{m,n}  f(t)\|^2_{H^{m+n\varrho}_xL^2}+\f{c_o}3\int_0^t\bigg(\int_{\TT^3} \mathcal{E}_g^{\gamma,\eps}(W_{m,n}     |D_x|^{m+n\varrho} f)dx\\&&\qquad+c_oA_5(c_1, c_2)\big(\|W^\epsilon_s(D)W_{\gamma/2}W_{m,n}      f\|_{H^{m+n\varrho}_xL^2}^2+\int_{\TT^3} \mathcal{E}^{0,\eps}_\mu(W_{\gamma/2}W_{m,n}    |D_x|^{m+n\varrho} f)dx\big)\bigg)d\tau\\
	&&\le  \|W_{m,n}  f_0\|^2_{H^{m+n\varrho}_xL^2}+A_6(c_1,c_2)\int_0^t\big( C(\eta_1^{-1}, \mathbb{W}_{I}) \mathbb{E}^{1,\f12+2\delta_1,\eps}(f)+\eta_1\mathbb{D}^{1,\f12+2\delta_1,\eps}_3(f)+\eps^{2s}\mathbb{D}^{1,\f12+2\delta_1,\eps}_2(f)\big)d\tau\\&&\quad+C_E\int_0^t \bigg( c_2\mathbb{E}^{1,\f12+2\delta_1,\eps}(f)
	+ \big(C(\eta_2^{-1}, \mathbb{W}_{I}) M_1+\eta_2\mathbb{D}^{1,\f12+2\delta_1,\eps}_3(g)+\eps^{2s}\mathbb{D}^{1,\f12+2\delta_1,\eps}_2(g)\big)\|f\|_{H_x^{\f32+\delta_1 }L^2_4}^2 \\&&\quad
	+ c_2(\eta_3\mathbb{D}^{1,\f12+2\delta_1,\eps}_2(f)+C(\eta_3^{-1},\mathbb{W}_{I}) \mathbb{E}^{1,\f12+2\delta_1,\eps}(f)\bigg)d\tau.
	\eeno
	
	Then we conclude that if $W_{m,n}\in \mathbb{W}_{I}(N,\kappa,\varrho,\delta_1,q_1,q_2)$ with $m+n\rho\le \f32+2\delta_1$,
	\beno &&(i).\quad \|W_{m,n}  f(t)\|^2_{H^{m+n\varrho}_xL^2}+\f{c_o}3\int_0^t\bigg(\int_{\TT^3} \mathcal{E}_g^{\gamma,\eps}(W_{m,n}     |D_x|^{m+n\varrho} f)dx\\&&\qquad+c_oA_5(c_1, c_2)\big(\|W^\epsilon_s(D)W_{\gamma/2}W_{m,n}    f\|_{H^{m+n\varrho}_xL^2}^2+\int_{\TT^3} \mathcal{E}^{0,\eps}_\mu(W_{\gamma/2}W_{m,n}   |D_x|^{m+n\varrho} f)dx\big)\bigg)d\tau\\
	&&\le \|W_{m,n}  f_0\|^2_{H^{m+n\varrho}_xL^2}+C( \mathbb{W}_{I},\eta_1^{-1},\eta_2^{-1},\eta_3^{-1},c_1,c_2,M_1)\int_0^t \mathbb{E}^{1,\f12+2\delta_1,\eps}(f)d\tau
	+A_6\eta_1 \int_0^t \mathbb{D}^{1,\f12+2\delta_1,\eps}_3(f)d\tau\\&&\quad+(A_6\eps^{2s}+C_Ec_2\eta_3)\int_0^t \mathbb{D}^{1,\f12+2\delta_1,\eps}_2(f)d\tau +\int_0^t(\eta_2\mathbb{D}^{1,\f12+2\delta_1,\eps}_3(g)+\eps^{2s}\mathbb{D}^{1,\f12+2\delta_1,\eps}_2(g)\big)\mathbb{E}^{1,\f12+2\delta_1,\eps}(f)d\tau,\\ 
  &&(ii).  \|W_{1,\f12+\delta_1}   f(t)\|^2_{H^{\f32+\delta_1}_xL^2}+\|W_{\gamma+4}   f(t)\|^2_{H^{\f32+2\delta_1}_xL^2}
	 +\f{c_o}3\int_0^t\bigg(\int_{\TT^3} \big(\mathcal{E}_g^{\gamma,\eps}(W_{1,\f12+\delta_1}  |D_x|^{\f32+\delta_1} f)\\&& +\mathcal{E}_g^{\gamma,\eps}(W_{\gamma+4}  \  |D_x|^{\f32+2\delta_1} f) \big)dx +c_oA_5(c_1, c_2)\big(\|W^\epsilon_s(D)W_{\gamma/2}W_{1,\f12+\delta_1}   \|_{H^{\f32+\delta_1}_xL^2}^2+\|W^\epsilon_s(D)W_{\gamma/2}W_{\gamma+4}  f\|_{H^{\f32+2\delta_2}_xL^2}^2\\&& +\int_{\TT^3}( \mathcal{E}^{0,\eps}_\mu(W_{\gamma/2}W_{1,\f12+\delta_1}   |D_x|^{\f32+\delta_1} f)+\mathcal{E}^{0,\eps}_\mu(W_{\gamma/2}W_{\gamma+4}    |D_x|^{\f32+2\delta_1} f) )dx\big)\bigg)d\tau\le  \|W_{1,\f12+2\delta_1}  f_0\|^2_{H^{\f32+\delta_1}_xL^2}\\
	&&+\|W_{\gamma+4}  f_0\|^2_{H^{\f32+2\delta_1}_xL^2}+C( \mathbb{W}_{I},\eta_1^{-1},\eta_2^{-1},\eta_3^{-1},c_1,c_2,M_1)\int_0^t \mathbb{E}^{1,\f12+2\delta_1,\eps}(f)d\tau
	+A_6\eta_1 \int_0^t \mathbb{D}^{1,\f12+2\delta_1,\eps}_3(f)d\tau\\&&\quad+(A_6\eps^{2s}+C_Ec_2\eta_3)\int_0^t \mathbb{D}^{1,\f12+2\delta_1,\eps}_2(f)d\tau +\int_0^t(\eta_2\mathbb{D}^{1,\f12+2\delta_1,\eps}_3(g)+\eps^{2s}\mathbb{D}^{1,\f12+2\delta_1,\eps}_2(g)\big)\mathbb{E}^{1,\f12+2\delta_1,\eps}(f)d\tau.\eeno
	
	{\it Step 1.5:  Smoothing estimates for $x$ variable.} We rewrite the smoothing estimates for $x$ variable. By taking
	$\eta\sim (c_o \eta_2 \f{A_5}{A_6})^{1/(2s)}, \eta_1\sim (c_o \eta_2 \f{A_5}{A_6})^{1/2+2/s},\eps\le [(c_o \eta_2 \f{A_5}{A_6})^{1+4/s}c_2^{-1}]^{1/(2(1-s))}$ in Proposition \ref{smxmn1} , which imply that $\eta^{-8}\sim (c_o \eta_2 \f{A_5}{A_6})^{-4/s}$,
	we  have
	\beno  &&2A_6(c_1,c_2)\int_0^t \|W_{-d_1}(W_{m,n+1}W_{\gamma/2+d_1+d_2}    f)_\phi\|_{H^{m+(n+1)\varrho}_xL^2}^2 d\tau\eeno\beno
	&& \le 2A_6C_E(c_o \eta_2 \f{A_5}{A_6})^{-4/s}\|W_{m,n}W_{-d_3}   f_0\|^2_{H^{m+n\varrho}_xL^2}+ 2C_Ec_0\eta_2A_5\int_0^t  \|W^\ep_s(D)W_{m,n}W_{\gamma/2}W_{-d_3}  f\|_{H^{m+n\varrho}_xL^2}^2 d\tau\\&&\quad + C_E(c_o \eta_2 \f{A_5}{A_6})^{-4/s} \big(1+c_2^{\f{(2s-1)}{1-s}}((c_o \eta_2 \f{A_5}{A_6})^{-\f{(s+4)(2s-1)}{s(1-s)}}+c_2+4M_1\big)\int_0^t  \| W_{m,n}W_{-d_3}    f\|_{H^{m+n\varrho}_xL^2}^2 d\tau
	\\&&\quad+C_E(c_o \eta_2 \f{A_5}{A_6})^{-4/s} \int_0^t
	(c_2\eta_3\mathbb{D}^{1,\f12+2\delta_1,\eps}_2(f)+C(\eta^{-1}_3,\mathbb{W}_{I})\|W^\eps(D)f\|_{L^2}^2) d\tau .
	\eeno
	Choose $C_E(c_o \eta_2 \f{A_5}{A_6})^{-4/s}c_2\eta_3=\eta_2$, then for $m+n\varrho\le 1+N_{\varrho,1}\varrho$, we conclude that
	\ben\label{esti-d3}
	&&2A_6(c_1,c_2)\int_0^t \|W_{-d_1}(W_{m,n+1}W_{\gamma/2+d_1+d_2}    f)_\phi\|_{H^{m+(n+1)\varrho}_xL^2}^2 d\tau\nonumber\\
	&& \le 2A_6C_E(c_o \eta_2 \f{A_5}{A_6})^{-4/s}\|W_{m,n}W_{-d_3}   f_0\|^2_{H^{m+n\varrho}_xL^2}+ (2C_Ec_0A_5+1)\eta_2\int_0^t \mathbb{D}^{1,\f12+2\delta_1,\eps}_2(f)d\tau\\&&\quad+C(\eta_2^{-1},\mathbb{W}_{I},c_1,c_2,M_1)\int_0^t\mathbb{E}^{1,\f12+2\delta_1,\eps}(f)  d\tau.\nonumber
	\een
	It implies that if $(2C_Ec_0A_5+1)\eta_2(N_{\varrho,1}+2)2\le A_7/100$, then  
	\beno &&2A_6(c_1,c_2)\int_0^t \|W_{-d_1}(W_{m,n+1}W_{\gamma/2+d_1+d_2}    f)_\phi\|_{H^{m+(n+1)\varrho}_xL^2}^2 d\tau \le C(c_1,c_2)\|W_{m,n}W_{-d_3}   f_0\|^2_{H^{m+n\varrho}_xL^2}\\&&\quad+  \f{A_7}{200(N_{\varrho,1}+2)}\int_0^t \mathbb{D}^{1,\f12+2\delta_1,\eps}_2(f)d\tau+C(W_{m,n},c_1,c_2,M_1)\int_0^t\mathbb{E}^{1,\f12+2\delta_1,\eps}(f)  d\tau.\eeno
	Moreover the conditions $\eps\le [(c_o \eta_2 \f{A_5}{A_6})^{1+4/s}c_2^{-1}]^{1/(2(1-s))}$ and $(2C_Ec_0A_5+1)\eta_2(N_{\varrho,1}+2)2\le A_7/100$ yield that
	\ben\label{aumforeps1}\eps\le
	[(\f{A_5c_o  A_7}{200A_6(N_{\varrho,1}+2)(2C_Ec_0A_5+1)})^{1+4/s}c_2^{-1}]^{1/(2(1-s))}.\een

	{\it Step 1.6: Closing the estimates.}
	Now patch together all the estimates from {\it Step 1.1} to {\it Step 1.4}, then for $\eta\ll1$, we  arrive at
	\ben\label{continousE1}
	&&(i).\,\mathbb{E}^{1,\f12+2\delta_1,\eps}(f(t))+\f{c_o}3\int_0^t\mathbb{D}_g^{1,\f12+2\delta_1,\eps}(f(\tau))d\tau+A_7\int_0^t\mathbb{D}_2^{1,\f12+2\delta_1,\eps}(f(\tau))d\tau +A_8\int_0^t\mathbb{D}_1^{1,\f12+2\delta_1,\eps}(f(\tau))d\tau \nonumber\\
	&&\le   \mathbb{E}^{1,\f12+2\delta_1,\eps}(f_0)+C(\delta,\eta^{-1},c_1,c_2,W_{m,n},M_1)\int_0^t
	\mathbb{E}^{1,\f12+2\delta_1,\eps}(f)d\tau
	+C_El_1^{-1-s} 
	\int_0^t  \|f\|_{H^{\f32+2\delta_1}_xL^2_{\gamma+4}}\nonumber \\&&\qquad\times\mathbb{D}_1^{1,\f12+2\delta_1,\eps}(g(\tau)) d\tau
	+C_E \delta^{2s}\int_0^t \mathbb{D}_1^{1,\f12+2\delta_1,\eps}(g(\tau)) \|f\|_{H^{\f32+2\delta_1}_xL^2_{\gamma+4}}^2
	d\tau+ 2(N_{\varrho,1}+2) \\&&\qquad\times\int_0^t\bigg[\big(C(\eta^{-1}, \mathbb{W}_{I}) M_1+\eta\mathbb{D}^{1,\f12+2\delta_1,\eps}(g)+\eps^{2s}\mathbb{D}^{1,\f12+2\delta_1,\eps}_2(g)\big)\mathbb{E}^{1,\f12+2\delta_1,\eps}(f)+A_6\eta \mathbb{D}^{1,\f12+2\delta_1,\eps}_3(f)\bigg]d\tau. \notag \een 
	From this together with the estimate in {\it Step 1.5}, if \eqref{aumforeps1} and  $A_6\eps^{2s}\le A_7/(N_{\rho,2}+2)$ hold, we drive that \beno
	&&\mathbb{E}^{1,\f12+2\delta_1,\eps}(f(t))+\f{c_o}3\int_0^t\mathbb{D}_g^{1,\f12+2\delta_1,\eps}(f(\tau))d\tau+A_7\int_0^t\mathbb{D}_2^{1,\f12+2\delta_1,\eps}(f(\tau))d\tau +A_8\int_0^t\mathbb{D}_1^{1,\f12+2\delta_1,\eps}(f(\tau))d\tau\\&&+\f34A_6\int_0^t\mathbb{D}_3^{1,\f12+2\delta_1,\eps}(f(\tau))d\tau \le C(c_1,c_2)\mathbb{E}^{1,\f12+2\delta_1,\eps}(f_0)+C(c_1,c_2,\mathbb{W}_{I},M_1)\int_0^t
	\mathbb{E}^{1,\f12+2\delta_1,\eps}(f)d\tau\\&&+
	C(c_1,c_2) \int_0^t   \bigg (\mathbb{D}_1^{1,\f12+2\delta_1,\eps}(g(\tau)) +\mathbb{D}_2^{1,\f12+2\delta_1,\eps}(g(\tau))+\mathbb{D}_3^{1,\f12+2\delta_1,\eps}(g(\tau)\bigg) \mathbb{E}^{1,\f12+2\delta_1,\eps}(f)d\tau.  \nonumber
	\eeno
	By Gronwall inequality, we deduce that for $t\in [0,T]$ with $T\le \mathfrak{L}$,
	\ben\label{e321} &&\mathbb{E}^{1,\f12+2\delta_1,\eps}(f(t))+\int_0^t \mathbb{D}^{1,\f12+2\delta_1,\eps}(f(\tau))d\tau\nonumber\\ &&\le
	2\exp\{C(c_1,c_2,\mathbb{W}_{I},M_1)\mathfrak{L}+C(c_1,c_2)4M_1\}\tilde{C}(c_1,c_2,\mathbb{W}_{I},M_1,\mathfrak{L})
	\eqdefa C_{E,1}. \een
	It gives the proof of the first result of  $(1)$ in the theorem.
	
	{\it Step 2: Continuous bounds with respect to the initial data.}
	We want to prove that    the continuous bounds with respect to the initial data can be obtained if  we shrink the time interval.  
	
	We first give the new bound for $\|f\|_{H^{\f32+2\delta_1}_xL^2_{\gamma+4}}$. Thanks to Proposition \ref{Enx32}, Remark \ref{intxbod1} and the bounds obtained in  {\it Step 1},
	we derive that for $t\in[0,T]$,
	\beno
	\|f(t)\|_{H^{\f32+2\delta_1}_xL^2_{\gamma+4}}^2&\le& c_2/2+
	(A_6+c_2C_{E,1})\big(C(\eta_1,\mathbb{W}_{I})C_{E,1}t+(\eta_1+\eps^{2s})C_{E,1}\big)+C_Ec_2C_{E,1}t\\&&+c_2\big(C(\eta_2,\mathbb{W}_{I})C_{E,1}t+\eta_2C_{E,1}\big).
	\eeno
	Then with the help of \eqref{Rstreps}, it implies that  there exists a time $T_1=T_1(C_{E,1},W_{m,n},c_1,c_2)\le T$ such that
	for $t\in [0,T_1]$, \ben\label{esf32}\|f(t)\|_{H^{\f32+2\delta_1}_xL^2_{\gamma+4}}^2\le c_2.\een
	
	Now we can improve the estimate for $\mathbb{E}^{1,\f12+2\delta_1,\eps}(f(t))$. Use conditions \eqref{e321},\eqref{esf32}  and \eqref{Rstrdelta},
	then we obtain that  $\int_0^t \mathbb{E}^{1,\f12+2\delta_1,\eps}(f(\tau))d\tau\le C_{E,1}t$, $C_El_1^{-1-s}
	\int_0^t  \|f\|_{H^{\f32+2\delta_1}_xL^2_{\gamma+4}} \mathbb{D}_1^{1,\f12+2\delta_1,\eps}(g(\tau)) d\tau\le  C_El_1^{-1-s}\sqrt{c_2}\\\times A_8^{-1}4M_1\le M_1/10, C_E\delta^{2s}\int_0^t \mathbb{D}_1^{1,\f12+2\delta_1,\eps}(g(\tau)) \|f\|_{H^{\f32+2\delta_1}_xL^2_{\gamma+4}}^2
	d\tau\le C_E\delta^{2s}c_2A_8^{-1}4M_1\le M_1/10. $ 
	Then  \eqref{Rstreps} and  \eqref{continousE1}  imply that  there exists a time $T_2=T_2(C_{E,1},W_{m,n},c_1,c_2)\le T_1$   such that for $t\le T_2$,
	\ben\label{e322}
	&&\mathbb{E}^{1,\f12+2\delta_1,\eps}(f(t)) +A_7\int_0^t\mathbb{D}_2^{1,\f12+2\delta_1,\eps}(f(\tau))d\tau +A_8\int_0^t\mathbb{D}_1^{1,\f12+2\delta_1,\eps}(f(\tau))d\tau \le 4M_1.\een
	From this together with \eqref{Eforq1}, we may also derive that $\sup_{t\in[0,T_2]}V^{q_1,\eps}(f(t))\le 4V^{q_1,\eps}(f_0)$.
	
	Next we go back to \eqref{esti-d3} to improve the estimate of $\int_0^t\mathbb{D}_3^{1,\f12+2\delta_1,\eps}(f(\tau))d\tau$. Choose $\eta_2$ verifying that
	$ 2(N_{\varrho,1}+2)(2A_6)^{-1}(2C_Ec_0A_5+1)\eta_24A_{7}^{-1}= 1/5, $ we have
	\beno && \int_0^t \|W_{-d_1}(W_{m,n+1}W_{\gamma/2+d_1+d_2}    f)_\phi\|_{H^{m+(n+1)\varrho}_xL^2}^2 d\tau\\
	&& \le A_9\|W_{m,n}W_{-d_3}   f_0\|^2_{H^{m+n\varrho}_xL^2}+ \f{1} {10(N_{\varrho,1}+2)}M_1+C(\eta_2^{-1},\mathbb{W}_{I},c_1,c_2,M_1)\int_0^t\mathbb{E}^{1,\f12+2\delta_1,\eps}(f)  d\tau.
	\eeno
 It implies that there exists a time $T_3=T_3(W_{m,n}, c_1,c_2)\le T_2$ such that
	$ \int_0^{T_3} \mathbb{D}^{1,\f12+2\delta_1,\eps}_3(f(\tau))d\tau \le 4A_9 M_1
	$. 
	From this together with \eqref{e322} imply that\eqref{322regu} holds for $f$ in the time interval $[0, T_3]$.  
	
	{\it Step 3: Energy estimates in $\mathbb{E}^{N,\kappa,\eps}$ with $N+\kappa\ge \f52+\delta_1$.}  We will use the inductive method to prove the propagation of the smoothness.
	
	{\it Step 3.1: Propagation of the regularity in   space $E^{m,n,\eps}$ with $m+n\varrho\ge 1+(N_{\varrho,2}+1)\varrho$.}
	Assume that for   $t\in [0,T^*_{m,n-1}=T^*_{m,n-1}(T,c_1,c_2, M_1, M_2, W_{m,n})]\le T_3$,
	it hold
	\beno && \mathbb{E}^{1,\f12+2\delta_1,\eps}(f(t))\le 4M_1\le 4M_2; \int_0^t D^{m,n-1,\eps}_3(f(\tau))d\tau\le A_9M_2+\f{A_9}{N(N_{\rho,1}+2)}M_2;\\
	&&E^{m,n-1,\eps}(f(t))+A_7\int_0^t D^{m,n-1,\eps}_2(f(\tau))d\tau\le E^{m,n-1,\eps}(f_0)+\f{1}{N(N_{\rho,1}+2)}M_2\le 4M_2.
	\eeno
	Next we will show   that there exists a time $T^*_{m,n}=T^*_{m,n}(T,c_1,c_2, M_1, M_2, W_{m,n})\le T^*_{m,n-1}$ such that
	\ben &&\label{in1}   \int_0^t D^{m,n,\eps}_3(f(\tau))d\tau\le A_9M_2+\f{A_9}{N(N_{\rho,1}+2)}M_2;\\
	&&\label{in2}E^{m,n,\eps}(f(t))+A_7\int_0^t D^{m,n,\eps}_2(f(\tau))d\tau\le E^{m,n,\eps}(f_0)+\f{1}{N(N_{\rho,1}+2)}M_2\le 4M_2.
	\een
	
	Thanks to Proposition \ref{Enxmn2} and the inductive assumptions, we first have
	\beno  &&E^{m,n,\eps}(f(t))+c_oA_5\int_0^t D^{m,n,\eps}_2(f)d\tau+\f{c_o}3\int_0^t D_g^{m,n,\eps}(f)d\tau\\
	&&\le E^{m,n,\eps}(f_0)+(A_6+C_Ec_2)\int_0^t\big[C(\eta^{-1}_1,\mathbb{W}_{I})E^{m,n,\eps}(f)+\eta_1D_3^{m,n-1,\eps}(f)+\eps^{2s}D^{m,n,\eps}_2(f)\big]d\tau+C_E(c_2\\&&+4M_1)\int_0^tE^{m,n,\eps}(f)d\tau+C_E(c_2+4M_2)
	\int_0^t\big[C(\eta_2^{-1},\mathbb{W}_{I})\|W^\eps_s(D)f\|_{L^2}^2+\eta_2(D_2^{m,n,\eps}(f)+D_2^{m,n-1,\eps}(f))\big]d\tau\eeno\beno
	&&\le E^{m,n,\eps}(f_0)+(A_6+C_Ec_2)\eta_14A_9M_2+\big((A_6+C_Ec_2)\eps^{2s}+C_E(c_2+4M_1)\eta_2\big)\int_0^t D_2^{m,n,\eps}(f)d\tau\\&&+\big[(A_6+C_Ec_2)C(\eta^{-1}_1,\mathbb{W}_{I})+C_E(c_2+4M_1)\big]\int_0^t
	E^{m,n,\eps}(f)d\tau+C_E(c_2+4M_2)C(\eta_2^{-1},\mathbb{W}_{I})4M_1t\\&&+C_{E}(c_2+4M_2)\eta_24M_2.
	\eeno
	Thanks to \eqref{Rstreps}, we may choose $\eta_2$ sufficiently small such that  $(A_6+C_Ec_2)\eps^{2s}+C_E(c_2+4M_1)\eta_2\le A_7$. Then  Gronwall's inequality implies that for $t\in[0,T^*_{m,n-1}]$,
$E^{m,n,\eps}(f(t))+A_7\int_0^t D^{m,n,\eps}_2(f)d\tau+\f{c_o}3\int_0^t D_g^{m,n,\eps}(f)d\tau \le C(c_1,c_2,M_1,M_2,\mathbb{W}_{I})\eqdefa C_{E,2}.$
	It yields \beno
	&&E^{m,n,\eps}(f(t))+A_7\int_0^t D^{m,n,\eps}_2(f)d\tau+\f{c_o}3\int_0^t D_g^{m,n,\eps}(f)d\tau\\
	&&\le E^{m,n,\eps}(f_0)+(A_6+C_Ec_2)\eta_14A_9M_2+ \big[(A_6+C_Ec_2)C(\eta^{-1}_1,\mathbb{W}_{I})+C_E(c_2+4M_1)\big] C_{E,2}t\\&&+C_E(c_2+4M_2)C(\eta_2^{-1},\mathbb{W}_{I})4M_1t+C_{E}(c_2+4M_2)\eta_24M_2. \eeno
	Then there exists a time $T^*_{m,n,1}=T^*_{m,n,1}(T,c_1,c_2, M_1, M_2, \mathbb{W}_{I})\le T^*_{m,n-1}$ such that \eqref{in2} holds.

	Next we give the estimate for $\int_0^t D^{m,n,\eps}_3(f)d\tau$. Thanks to Proposition \ref{smxmn2}, we have
	\beno  \int_0^t D_3^{m,n,\eps}(f)d\tau&
	\le &C_E\eta^{-8}E^{m,n,\eps}(f_0)+C_E\eta^{2s}\int_0^t D_2^{m,n,\eps}(f)d\tau+C_E\eta^{-8}\int_0^t E^{m,n,\eps}(f)d\tau\\&&+C_E\eta^{-8}(\eta_1^2+\eps^{2(1-s)}c_2)D_2^{m,n,\eps}(f)d\tau+C_E\eta^{-8-\f{2(2s-1)}{1-s}}c_2^{\f{2s-1}{1-s}}\int_0^t E^{m,n,\eps}(f)d\tau\\&&+C_E\eta^{-8}(c_2+4M_1)\int_0^t E^{m,n,\eps}(f)d\tau+C_E\eta^{-8}c_24M_2t\\&&+C_E\eta^{-8}(c_2+4M_2)\int_0^t\big[C(\eta_3^{-1},\mathbb{W}_{I})  \|W^\eps_s(D)f\|_{L^2}^2+\eta_3(D_2^{m,n,\eps}(f)+D_2^{m,n-1,\eps}(f))\big]d\tau.
	\eeno
	By following the argument applied in {\it Step 1.5}, we derive that for $t\in [0,T^*_{m,n,1}]$
	\beno
	\int_0^t D_3^{m,n,\eps}(f)d\tau&
	\le & A_9E^{m,n,\eps}(f_0)+(2C_Ec_0A_5+1)\eta_2\int_0^t (D_2^{m,n,\eps}(f)+D_2^{m,n-1,\eps}(f))d\tau\\&&\quad+C(\eta_2^{-1},\mathbb{W}_{I},c_1,c_2,M_2)4M_2t. \eeno
	This implies that there exists a time $T^*_{m,n}=T^*_{m,n}(T,c_1,c_2, M_1, M_2, \mathbb{W}_{I})\le T^*_{m,n,1}$ such that \eqref{in1} holds. We complete the   the inductive argument.
	
	{\it Step 3.2: Propagation of the mixed regularity.}  Thanks to Proposition \ref{Enmix},
	we have
	\beno
	&& \|W^\ep_{q_2}(D)f(t)\|_{H^{\f32+\delta_1}_xL^2}^2+\mathcal{C}_1(c_1,c_2)\int_0^t \|W^{\ep}_{q_2+s}(D) W_{\gamma/2}f\|^2_{H^{\f32+\delta_1}_xL^2}d\tau\\
	&&\le \|W^\ep_{q_2}(D)f_0\|_{H^{\f32+\delta_1}_xL^2}^2+ \mathcal{C}_2(c_1,c_2)\int_0^t\|W^\ep_{q_2}(D)W_{\gamma/2}f\|_{H^{\f32+\delta_1}_xL^2}^2d\tau+C_E\int_0^t \bigg(\|W^\eps_{(q_2-1)^+}(D)f\|_{H^{\f52+\delta_1}_xL^2}^2
	\\&&\quad+\|g\|_{H^{\f32+2\delta_1}_xL^2_{\gamma+4}}^2\|W^{\ep}_{q_2+s}(D)W_{\gamma/2+2s}f\|_{H^{\f32}_xL^2}^2 + \|g\|_{H^{\f32+\delta_1}_xL^2}^2( \mathrm{1}_{2s>1}\|W^\ep_{q_2+s-1}(D)W_{\gamma/2+\f52}f\|^2_{H^{\f32+\delta_1}_xL^2}\\&&\quad+\mathrm{1}_{2s=1}\|W^\ep_{q_2-s+\log}(D)W_{\gamma/2+\f52}f\|^2_{H^{\f32+\delta_1}_xL^2}+\mathrm{1}_{2s<1}\|W^{\eps}_{q_2-s}(D)W_{\gamma/2+\f52}f\|_{H^{\f32+\delta_1}_xL^2}^2)\\&&\quad
	+ \|g\|^2_{H^{\f32+\delta_1}_xL^2_{\gamma+3}}( \mathrm{1}_{2s>1}\| W^\ep_{q_2+s-1}(D)f\|^2_{H^{\f32+\delta_1}_xL^2} +\mathrm{1}_{2s=1} \|W^\ep_{q_2-s+\log}(D)f\|_{H^{\f32+\delta_1}_xL^2}^2+\mathrm{1}_{2s<1}\\&&\quad\times\|W^\eps_{q_2-s}(D) f\|_{H^{\f32+\delta_1}_xL^2}^2)  + \|W^\ep_{q_2+s-1+\eta}(D)g\|_{H^{\f32+\delta_1}_xL^2}^2 \| f\|_{H^{\f32+\delta_1}_xL^2}^2+  \| g\|_{H^{\f32+\delta_1}_xL^2_2}^2 \|f\|_{H^{\f32+\delta_1}_xL^2}^2
	\bigg)d\tau.
	\eeno
	
	It is not difficult to check that
	\beno |W^\eps_{q_2}(D)f|_{L^2_l}&\lesssim& |f|_{L^2_l}^{\f{s}{q_2+s}}|W^\eps_{q_2+s}(D)f|_{L^2_l}^{\f{q_2}{q_2+s}};\\
	\|W^\eps_{(q_2-1)^+}(D)f\|_{H^{\f52+\delta_1}_xL^2}&\lesssim&\|W^\eps_{\big((q_2-1)^+-(2^{n_0}-1)(1+s-q_2)\big)^+}(D)f\|_{H^{\f52+\delta_1}_xL^2}^{2^{-n_0}}\|W^\eps_{s}(D)f\|_{H^{\f52+\delta_1}_xL^2}^{1-2^{-n_0}};\eeno\beno
	\|W^\eps_{q_2+s}(D)W_{\gamma/2+2s}f\|_{H^{\f32}_xL^2}&\lesssim& \|W^\eps_{q_2+s}(D)W_{\gamma/2}W_{2^{n_1}2s}f\|_{H^{\f32- (2^{n_1}-1)\delta_1}_xL^2}^{2^{-n_1}}\|W^\eps_{q_2+s}(D)W_{\gamma/2}f\|_{H^{\f32+\delta_1}_xL^2}^{1-2^{n_1}};\\
	\|W^\eps_{q_2+s}(D)W_{\gamma/2}W_{2^{n_1}2s}f\|_{L^2}&\lesssim& \|W^\eps_{(q_2+s-(q_1-q_2)(2^{n_2}-1))^+}(D)W_{\gamma/2}W_{2^{n_1+n_2}2s}f\|_{ L^2}^{2^{-n_2}}\|W^\eps_{q_1+s}(D)W_{\gamma/2}f\|_{L^2}^{1-2^{-n_2}};\\
	|W^\ep_{q_2+s-1}(D)W_{\gamma/2+\f52}f|^2_{L^2}&\lesssim&
	|W^\ep_{(q_2+s-2^{n_3})^+}(D)W_{\gamma/2}W_{\f522^{n_3} }f|^{2^{-n_3}}_{L^2}|W^\ep_{q_2+s}(D)W_{\gamma/2} f|^{1-2^{-n_3}}_{L^2};\\
	|W^\ep_{q_2-s+\eta}(D)W_{\gamma/2+\f52}f|^2_{L^2}&\lesssim&
	|W^\ep_{(q_2-s+\eta-(2^{n_4}-1)(2s-\eta))^+}(D)W_{\gamma/2}W_{\f522^{n_4}}f|^{2^{-n_4}}_{L^2}|W^\ep_{q_2+s}(D)W_{\gamma/2} f|^{1-2^{-n_4}}_{L^2}.
	\eeno
	
	By choosing $n_0=[\log_2(\f{s}{1+s-q_2})]+1$, $n_1=N_{\delta_1}=[\log_2(3(2\delta_1)^{-1}+1)]+1, n_2=N_{q_1,q_2,s}=[\log_2(\f{q_2+s}{q_1-q_2}+1)]+1$ and $n_3=n_4=N_{q_2,s,1}=\max\{[\log_2(q_2+s)]+1,[\log_2\big(\f{q_2+s}{2s}(1+\delta_1)\big)]+1\}$, and using (W-6) of Definition \ref{ws1}, we conclude that
	\beno
	&& \|W^\ep_{q_2}(D)f(t)\|_{H^{\f32+\delta_1}_xL^2}^2+\mathcal{C}_1(c_1,c_2)\int_0^t \|W^{\ep}_{q_2+s}(D) W_{\gamma/2}f\|^2_{H^{\f32+\delta_1}_xL^2}d\tau\\
	&&\le \|W^\ep_{q_2}(D)f_0\|_{H^{\f32+\delta_1}_xL^2}^2+
	\eta_1\int_0^t \|W^{\ep}_{q_2+s}(D) W_{\gamma/2}f\|^2_{H^{\f32+\delta_1}_xL^2}d\tau
	+C(\eta_1,c_1,c_2)4M_2t+C_E\big(4M_2t+
	\\&&\quad+\int_0^tc_2(\eta_2\|W^{\ep}_{q_2+s}(D) W_{\gamma/2}f\|^2_{H^{\f32+\delta_1}_xL^2}+\eta_34M_2t+C(\eta_2,\eta_3)4M_2t)+
	c_24M_2t+c_2^2t\big).
	\eeno
	Then it yields that there exists a time $T_{q_2}=T_{q_2}(c_1,c_2,M_2,\mathbb{W}_{I})\le T_{N,\kappa}$ such that
	\beno \|W^\ep_{q_2}(D)f(t)\|_{H^{\f32+\delta_1}_xL^2}^2+\f12\mathcal{C}_1(c_1,c_2)\int_0^t \|W^{\ep}_{q_2+s}(D) W_{\gamma/2}f\|^2_{H^{\f32+\delta_1}_xL^2}d\tau
	\le 4\|W^\ep_{q_2}(D)f_0\|_{H^{\f32+\delta_1}_xL^2}^2\le 4M_2.\eeno

	{\it Step 3.3: The lower bound of the density. } By Theorem \ref{ubofQe}, it is easy to see
	\beno \f{d}{dt}|f|_{L^1}&\ge& -C_E\|f\|_{H^{\f52+\delta_1}_xL^2_3}-|Q^{\eps}(g,f)|_{L^1}\\
	&\ge&-C_E(\|f\|_{H^{\f52+\delta_1}_xL^2_3}+\|g\|_{H_x^{\f32+\delta_1}L^2_{\gamma+2s+4}}
	\|W^{\eps}_{2s}(D)f\|_{L^2_{\gamma+2s+2}}).
	\eeno
	
	Observe that
	$ \|W^{\eps}_{2s}(D)f\|_{L^2_{\gamma+2s+2}}\lesssim \|W^{\eps}_{(2s-(2^J-1)(q_1-s))^+}(D)W_{\gamma}W_{2^J(2s+2)}f\|_{L^2}^{2^{-J}}\|W^{\eps}_{q_1+s}(D)W_{\gamma/2}f\|_{L^2}^{1-2^{-J}}. $
	By choosing $J\ge \big[\log_2(\f{q_1+s}{q_1-s})\big]+1=N_{q_1,s,2}$ and using (W-6) of Definition \ref{ws1}, we have
	\beno \|W^{\eps}_{2s}(D)f\|_{L^2_{\gamma+2s+2}}&\lesssim& \| W_{\gamma}W_{2^{N_{q_1,s,2}}(2s+2)}f\|_{L^2}^{2^{-N_{q_1,s,2}}}\|W^{\eps}_{q_1+s}(D)W_{\gamma/2}f\|_{L^2}^{1-2^{-N_{q_1,s,2}}}\\&\lesssim&
	\|W_{l_2}f\|_{L^2}^{2^{-N_{q_1,s,2}}}\|W^{\eps}_{q_1+s}(D)W_{\gamma/2}f\|_{L^2}^{1-2^{-N_{q_1,s,2}}}.\eeno
	It implies that for $t\in [0,T_{N,\kappa}]$,
	$|f(t)|_{L^1}\ge |f_0|_{L^1}-2C_E\sqrt{M_2}(t+2\sqrt{M_1}\sqrt{t}).$
	
	Let $T^*=\min\{T_{q_2}, \f{c_1}{4C_E\sqrt{M_2}}, \big(\f{c_1}{8C_EM_2}\big)^2\}$. Then for $t\in[0,T^*]$,
	$|f(t)|_{L^2}\ge c_1$. We remark that \eqref{esf32} and \eqref{322regu} are proved in {\it Step 2}.  \eqref{52regu} can be obtained by combining the results from {\it Step 3.1} and  {\it Step 3.2}. 
	Then we complete the proof of $(2)$ in the theorem.
	
	Finally let us complete the proof of the second result in $(1)$. Actually it is easily obtained by the similar argument applied in {\it Step 3.1} and  {\it Step 3.2}. We skip the details here.
\end{proof}

 Now we are in a position to prove the well-posedness for the linear equation \eqref{EepLinB}.

\begin{thm}\label{Wepo-epslin}
	  Suppose that function spaces $\mathbb{E}^{N,\kappa,\eps}$and   $\mathbb{E}^{1,\f12+2\delta_1,\eps}$, the well-prepared sequence $\mathbb{W}_{I}(N,\kappa,\varrho,\delta_1,q_1,q_2)$  and $f_0$ verify all the conditions stated in Theorem \ref{thmwepo}.  Let $g$ verify \eqref{lubforg}, \eqref{322regu} and \eqref{52regu}.  
	 Then if $\eps$ verifies \eqref{Rstreps},   \eqref{EepLinB} admits a unique and non-negative solution in  $L^\infty([0,T];  
	 \mathbb{E}^{N,
\kappa,\eps})$  Moreover, there exists a time $T^*=T^*(T,c_1,c_2,M_1,M_2,\mathbb{W}_{I})\le T$ such that for $t\in[0,T^*]$, $f$ verifies \eqref{lubforg}, \eqref{322regu} and \eqref{52regu}.
 \end{thm}

\begin{proof} We construct the approximate  equation as follows:
 \beno  \left\{
 \begin{aligned}
 	& \pa_t f^n+v\cdot \na_x f+\f{1}{n}\lr{v}^{\gamma+2s}f^n=Q^{\ep}(g,f^n),\\
 	&f^n|_{t=0}=f_0.\end{aligned}\right. \eeno
 Thanks to Proposition \ref{wepoepl}, we know that the approximate equation admits a unique and non-negative solution. By the observation that the term $\f{1}{n}\lr{v}^{\gamma+2s}f^n$  plays no harmful role in  the energy estimates and hypo-elliptic estimates in the proof of Theorem \ref{En-Priori}, we   obtain that for each $n$,
 \beno \sup_{t\in[0,T]} \mathbb{E}^{1,\f12+2\delta_1,\eps}(f^n(t))+  \mathbb{E}^{N,\kappa,\eps}(f^n(t))\le C(c_1,c_2,M_1,M_2,\mathbb{W}_{I},\mathfrak{L}),\eeno
 and there exists a time $T^*=T^*(\mathfrak{L},c_1,c_2,M_2,\mathbb{W}_{I})$ such that
 \eqref{lubforg}, \eqref{322regu} and \eqref{52regu} hold uniformly for $f^n$ in the interval $t\in [0,T^*]$.

 Next we prove that $\{f^n\}_{n\in \N}$ is a Cauchy sequence in $L^1_{l}$ space where $l$ satisfies that $l+\gamma\le l_{1,\f12+\delta_1}$. Set $h^n\eqdefa f^{n+1}-f^n$, then $h^n$ solves
 \beno \pa_t h^n+v\cdot \na_x h^n+\f1{n}\lr{v}^{\gamma+2s}h^n=Q^\eps(g, h^n)+\f1{n(n+1)}\lr{v}^{\gamma+2s}f^{n+1}.\eeno
 By Proposition \ref{L1i}, one has
 \beno \f{d}{dt}\|h^n\|_{L^1_l}+c(c_1,c_2)\|h^n\|_{L^1_{l+\gamma}}
 \lesssim C(\|g\|_{L^\infty_xL^1_{l+\gamma}})\|h^n\|_{L^1_l}+\f{1}{n^2}\|f^{n+1}\|_{L^1_{l+\gamma+2s}}.\eeno
  Gronwall's inequality  implies that for $t\in[0,T]$,
$\|h^n\|_{L^\infty([0,T];L^1_{l})}\lesssim \f1{n^2}C(c_1,c_2,M_1,M_2,\mathbb{W}_{I}), $
 which yields that $\{f^n\}_{n\in \N}$ is a Cauchy sequence in $L^\infty([0,T];L^1_{l})$ space.
   In other words, we get the existence of the solution $f$ to the linear equation \eqref{EepLinB}.
  Finally let us check  that  \eqref{lubforg}, \eqref{322regu} and \eqref{52regu} hold for $f$ in the interval $t\in [0,T^*]$. From the fact that $f^n$ strongly converges to $f$
  in $L^\infty([0,T];L^1_{l})$ together with Fatou Lemma, interpolation inequality and uniform boundedness principle,  the uniform bounds for $f^n$ will yield the desired results except the term
  $\int_{\TT^3} \mathcal{E}^{0,\eps}_\mu(W_{N,\kappa}W_{\gamma/2}|D_x|^{N+\kappa}f)dx$ in $\mathbb{D}_2^{N, \kappa,\eps}$. However, this estimate can be achieved if we go back to the equation $W_{m,n} \bar{\triangle}^{n\varrho}_k\pa_x^\alpha f$  with $|\alpha|=N-1$. In this situation, $W_{m,n} \bar{\triangle}^{n\varrho}_k\pa_x^\alpha f$ can be chosen as a test function to the equation of $W_{m,n} \bar{\triangle}^{n\varrho}_k\pa_x^\alpha f$. Then the standard energy estimate will yield the control of $\int_{\TT^3} \mathcal{E}^{0,\eps}_\mu(W_{N,\kappa}W_{\gamma/2}|D_x|^{N+\kappa}f)dx$.
We complete the proof of the theorem.
 \end{proof}

\subsection{Proof of  Theorem \ref{thmwepo}(Part I)} Based on the estimates for the linear equation, we will use Picard iteration scheme to get the well-posedness for the non-linear equation \eqref{NepsB}.
 
\begin{proof}[Proof of  Theorem \ref{thmwepo}(Part I)] We split the proof into three steps.

{\it Step 1: Iteration scheme and uniform bounds.}	
Let $f^0=c_1M_{1,0,1}$ and
\beno  \left\{
\begin{aligned}
	& \pa_t f^n+v\cdot \na_x f^n =Q^{\ep}(f^{n-1},f^n),\\
	&f^n|_{t=0}=f_0.\end{aligned}\right. \eeno
It is easy to check that
\beno &&|f^0|_{L^1}=c_1, |f^0|_{L^1_2\cap L\log L}\le c_3\\ &&\mathbb{E}^{N,\kappa,\eps}(f^0)+A_7\int_0^t\mathbb{D}_2^{N, \kappa,\eps}(f^0)d\tau  +A_8\int_0^t\mathbb{D}_1^{N, \kappa,\eps}(f^0)d\tau
\\
&&=\mathbb{E}^{N,\kappa,\eps}(f^0)+A_7\mathbb{D}_2^{N, \kappa,\eps}(f^0)t  +A_8 \mathbb{D}_1^{N, \kappa,\eps}(f^0)t \le  C(l_1,l_2).
\eeno

Then by the proof of Theorem \ref{En-Priori}, we deduce that there exists a time $T_1=T_1(M_2,c_1,c_2,\mathbb{W}_{I})$ such that for $t\in [0, T_1]$,
\beno &&\inf_{x\in\TT^3, t\in[0,T_1]}|f^1|_{L^1}\ge  c_1; \sup_{ t\in[0,T_1]}\|f^1\|_{H^{\f32+2\delta_1}_xL^2_{\gamma+4}}^2\le c_2;\sup_{ t\in[0,T_1]}V^{q_1,\eps}(f^1(t))\le 4V^{q_1,\eps}(f_0);\\
&& \sup_{t\in[0,T_1]} \mathbb{E}^{1,\f12+2\delta_1,\eps}(f^1(t))+A_7(c_1,c_3)\int_0^{T_1}\mathbb{D}_2^{1,\f12+2\delta_1,\eps}(f^1(\tau))d\tau\nonumber\\&&\quad+A_8(c_1,c_3)\int_0^{T_1}\mathbb{D}_1^{1,\f12+2\delta_1,\eps}(f^1(\tau))d\tau\le 4M_1;
\int_0^{T_1} \mathbb{D}^{1,\f12+2\delta_1,\eps}_3(f^1(\tau))d\tau \le 4A_{9}(c_1,c_3) M_1;\\
&& \sup_{t\in[0,T_1]} \mathbb{E}^{N,\kappa,\eps}(f^1(t)) +A_7(c_1,c_3)\int_0^{T_1}\mathbb{D}_2^{N, \kappa,\eps}(f^1(\tau))d\tau\nonumber\\&&\quad+A_8(c_1,c_3)\int_0^{T_1}\mathbb{D}_1^{N, \kappa,\eps}(f^1(\tau))d\tau\le 4M_2;
\int_0^{T_1} \mathbb{D}^{N,\kappa,\eps}_3(f^1(\tau))d\tau \le 4A_{9}(c_1,c_3) M_2.
\eeno
We emphasize that $c_3$ actually is a universal constant which is not related to the size of the initial data. Thus in the next procedure, we deduce that there exists a time $T_2=T_2(M_2,c_1,c_2,\mathbb{W}_{I})\le T_1$ such that for $t\in [0, T_2]$,
\beno &&\inf_{x\in\TT^3, t\in[0,T_2]}|f^2|_{L^1}\ge  c_1; \sup_{ t\in[0,T_2]}\|f^2\|_{H^{\f32+2\delta_1}_xL^2_{\gamma+4}}^2\le c_2; \sup_{ t\in[0,T_2]}V^{q_1,\eps}(f^2(t))\le 4V^{q_1,\eps}(f_0);\\
&& \sup_{t\in[0,T_2]} \mathbb{E}^{1,\f12+2\delta_1,\eps}(f^2(t))+A_7(c_1,c_2)\int_0^{T_2}\mathbb{D}_2^{1,\f12+2\delta_1,\eps}(f^2(\tau))d\tau\nonumber\\&&\quad+A_8(c_1,c_2)\int_0^{T_2}\mathbb{D}_1^{1,\f12+2\delta_1,\eps}(f^2(\tau))d\tau\le 4M_1;
\int_0^{T_2} \mathbb{D}^{1,\f12+2\delta_1,\eps}_3(f^2(\tau))d\tau \le 4A_{9}(c_1,c_2) M_1;\\
&& \sup_{t\in[0,T_2]} \mathbb{E}^{N,\kappa,\eps}(f^2(t)) +A_7(c_1,c_2)\int_0^{T_2}\mathbb{D}_2^{N, \kappa,\eps}(f^2(\tau))d\tau \\&&\quad+A_8(c_1,c_2)\int_0^{T_2}\mathbb{D}_1^{N, \kappa,\eps}(f^2(\tau))d\tau\le 4M_2;
\int_0^{T_2} \mathbb{D}^{N,\kappa,\eps}_3(f^2(\tau))d\tau \le 4A_{9}(c_1,c_2) M_2.
\eeno

It is obvious that $f^2$ verifies \eqref{lubforg}, \eqref{322regu} and \eqref{52regu} in the time interval $[0,T_2]$.   Then thanks to Theorem \ref{Wepo-epslin},
 we derive that there exists a time  $T_3=T_3(T_2,c_1,c_2,M_1,M_2,\mathbb{W}_{I})\le T_2$,
 such that \eqref{lubforg}, \eqref{322regu} and \eqref{52regu} hold for $f^3$ in the time interval $[0,T_3]$.

Now let us focus on the lifespan $T_3$. From the proof of Theorem \ref{En-Priori}, the estimate \eqref{e321} shows  that $C_{E,1}$ can be chosen to be independent of $T_2$ if we use the condition $T_2\le T_1$. As a direct consequence, we get that there exists a time $T_u=T_u(c_1,c_2,M_1,M_2, \mathbb{W}_{I},T_1)$ such that  $T_3=\min\{T_u,T_2\}$ if we carefully check the argument of {\it Step 2} and {\it Step 3} in the proof of Theorem \ref{En-Priori}. By the iteration scheme and the condition $T_3\le T_1$, the similar argument yields that there exists a time
$T_4=\min\{T_u,T_3\}$ such that \eqref{lubforg}, \eqref{322regu} and \eqref{52regu} hold for $f^4$ in the time interval $[0,T_4]$. It is easy to see that $T_4=\min\{T_u,T_3\}=T_3$.
 It implies that there exists a common lifespan $T_3$ such that  \eqref{lubforg}, \eqref{322regu} and \eqref{52regu} hold for the sequence $\{f^n\}_{n\ge3}$  in the time interval $[0,T_3]$.
 \medskip

{\it Step 2: Existence and uniqueness.}  Thanks to \eqref{unicon1},  there exists a constant $l$ verifies
 \beno \max\{A_1(c_1,c_2)^{-\f1s},  \mathrm{1}_{\gamma\neq0}8C_I (c_2^{\f12}+V^{q_1,\eps}(f_0)),N_s+2\}<l\le l_{1,\f12+\delta_1}2^{-N_{q_2,s,2}}-N_s-2-\gamma.\eeno  In other words, $l$ satisfies
 $N_s+2\le l$, $2^{N_{q_2,s,2}}(l+\gamma+2+N_s)\le l_{1,\f12+\delta_1}$, $\mathrm{1}_{\gamma\neq0}8C_Il^{-1}  (c_2^{\f12}+V^{q_1,\eps}(f_0))\le 1$ and $l^sA_1(c_1,c_2)>1$. Next we will prove that $\{f^n\}_{n\ge3}$ is a Cauchy sequence in the space $L^\infty([0,T_3];L^1_l)$.

Let $h^n\eqdefa f^{n+1}-f^n$, then   $h^n$ solves
\beno \pa_t h^n+v\cdot \na_x h^n=Q^\eps(f^{n+1}, h^n)+Q^\eps(h^{n-1},f^{n}).\eeno
By  Lemma \ref{L1md}, one has
\beno
&&\f{d}{dt}\|h^n\|_{L^1_l}+l^sA_1(c_1,c_2)\|h^n\|_{L^1_{l+\gamma}}
 \le  C(l)\|h^n\|_{L^1_l}\|f^{n+1}\|_{L^\infty_xL^1_{l+\gamma}}+
 C(l,\eta^{-1})\|h^{n-1}\|_{L^1_l}\\&&\times\|W^\eps_{2s+\eta_1}(D)f^n\|_{H^{\f32+\delta_1}_xL^2_{l+\gamma+2+N_s}}+C\|h^{n-1}\|_{L^1_{l+\gamma}}(\eta^{2-2s}\|W^\eps_{2s+\eta_1}(D)f^n \|_{H^{\f32+\delta_1}_xL^2_{2}}
 +l^{-1}\|W^\eps_{\eta_1}(D)f^n
\|_{H^{\f32+\delta_1}_xL^2_{ 2}}).
 \eeno
Notice that if $2^n\ge \f{q_2}{q_2-2s-\eta_1}$, that is, $n\ge N_{q_2,s,2}$, one has
\ben &&(i).\,\|W^\eps_{2s+\eta_1}(D)f^n\|_{H^{\f32+\delta_1}_xL^2_{l+\gamma+2+N_s}}
\lesssim \|W^\eps_{\big(2s+\eta_1-(2^{n}-1)(q_2-2s-\eta_1)\big)^+}(D)W_{2^n(l+\gamma+2+N_s)}f^n\|_{H^{\f32+\delta_1}_xL^2_{}}^{2^{-n}}\notag\\&&\quad
 \times
\|W^\eps_{q_2}(D)f^n\|_{H^{\f32+\delta_1}_xL^2_{}}^{1-2^{-n}}\lesssim
\|W_{1,\f12+\delta_1}f^n\|_{H^{\f32+\delta_1}_xL^2_{}}^{2^{-n}} 
\|W^\eps_{q_2}(D)f^n\|_{H^{\f32+\delta_1}_xL^2_{}}^{1-2^{-n}}\lesssim 2\sqrt{M_2},\notag\\
  &&(ii).\, Cl^{-1}\|W^\eps_{\eta_1}(D)f^n
\|_{H^{\f32+\delta_1}_xL^2_{2}} 
\le C_{I}l^{-1} \| f^n
\|_{H^{\f32+\delta_1}_xL^2_{\gamma+4}}^{\f12}\|W^\eps_{2\eta_1}(D)f^n
\|_{H^{\f32+\delta_1}_xL^2_{ }}^{\f12}\le C_{I}l^{-1} \| f^n
\|_{H^{\f32+\delta_1}_xL^2_{\gamma+4}}^{\f12}\notag\\
&&\qquad\times(\|W^\eps_{q_1}(D)f^n
\|_{L^2}+\| f^n
\|_{H^{\f32+2\delta_1}_xL^2})^{\f12}\le  C_Il^{-1} (c_2^{\f12}+  2V^{q_1,\eps}(f_0)) \le 1,\label{defici} \een
where $C_I$ is a universal constant independent of $\eps$(see Lemma \ref{func}).
 Denote $L=l^sA_1(c_1,c_2)>1$, then we have
  \ben\label{wbwc1}
 &&\quad\f{d}{dt}\|h^n\|_{L^1_l}+L\|h^n\|_{L^1_{l+\gamma}} 
 \le  C (l,M_2)\|h^n\|_{L^1_l} +
 C(l,M_2)\|h^{n-1}\|_{L^1_l} +\|h^{n-1}\|_{L^1_{l+\gamma}}.
 \een
Let $X^{n}(t)=e^{-C(l,M_2)t}\|h^n\|_{L^1_l}$ and $Y^{n}(t)=e^{-C(l,M_2)t}\|h^n\|_{L^1_{l+\gamma}}$.  Then one has
\beno \f{d}{dt} X^n+LY^n\le CX^{n-1}+Y^{n-1}.\eeno
 Now set $S^n=\sum_{i=0}^{n-4} L^{-i}X^{n-i}$ with $n\ge4$, then we infer that
\beno S^n(t)&\le& C\int_0^t S^{n-1}(\tau)d\tau+L^{-(n-4)}\int_{0}^{t} (CX^3+Y^3)d\tau \le C\int_0^t S^{n-1}(\tau)d\tau+L^{-(n-4)}\tilde{C}t,
 \eeno
which implies that
\beno \sum_{n\ge4}S^n(t)&\le&\sum_{n\ge4}\bigg( \big(\sup\limits_{t\in[0,T^*]} X^4(t)\big)C^{n-4}\f{t^{n-4}}{(n-4)!}+(C+1)\tilde{C}\sum_{i=1}^{n-4}L^{-(n-3-i)}\f{t^i}{i!}\bigg)
\\&\le& \sum_{n\ge4}  \tilde{C} \f{(Ct)^{n-4}}{(n-4)!}+(C+1)\tilde{C}\sum_{n\ge4}L^{-(n-3)}\sum_{i=1}^{n-4}\f{(Lt)^i}{i!}. \eeno
We deduce that for $t\le T_3$, $\sum\limits_{n\ge3}X^{n}(t)<\infty$ since $X^n(t)\le S^n(t)$. Thus we get that $\{h^n\}_{n\ge3}$  is a Cauchy sequence in $L^\infty([0,T_3];L^1_l)$. Suppose that for $t\le T_3$,
$ \lim_{n\rightarrow \infty}\|f^n(t)-f(t)\|_{L^1_l}=0. $  Then $f$ is a solution to \eqref{NepsB}.  Thanks to the uniform bounds obtained from {\it Step 1},  the similar argument applied in Theorem \ref{Wepo-epslin} will imply that  \eqref{lubforg}, \eqref{322regu} and \eqref{52regu} hold for $f$ in the time interval $[0,T_3]$.   

For the uniqueness of the equation, if set $h=f_1-f_2$ where $f_1$ and $f_2$ denote two solutions to \eqref{NepsB} with the same initial data, then thanks to \eqref{wbwc1}, we arrive at
$   \f{d}{dt}\|h\|_{L^1_l}
 \le  C (l,M_2)\|h\|_{L^1_l}. $
 From this together with Gronwall inequality, we are led to the uniqueness. \smallskip
 
 {\it Step 3: Propagation of the regularity for $v$ variable.} Suppose $q\le N+\kappa$. Then by the interpolation inequalities 
$ |f|_{H^{q-1}_l}\le |f|_{L^2_{ql}}^{\f1q}|f|_{H^q}^{1-\f1q}$ and $|f|_{H^{1}_l}\le |f|_{L^2_{\f{q}{q-1}l}}^{1-\f1q}|f|_{H^q}^{\f1q}$, we have 
\beno  &&\int_{\TT^3}(|W_{\gamma/2+\f52}f|_{H^1}^2|W^\ep_{s}(D)f|^2_{H^{q-1}}+|f|_{H^1}^2 |W^\ep_{s}(D)W_{\gamma/2+\f52}f|_{H^{q-1}}^2) dx\\
&&\lesssim \eta^{-q+1}(\|W^\eps_s(D)f\|_{H^{\f32+\delta_1}_xL^2}^2+\|W^{\eps}_s(D)W_{(\gamma/2+\f52)q}f\|_{H^{\f32+\delta_1}_xL^2}^2)\|f\|_{L^2_xH^q}^2+\eta \|W_{\gamma+5}f\|_{H^{\f32+\delta_1}_xL^2}^2\|W^{\eps}_s(D)f\|_{L^2_xH^q}^2\\
&&\le \f{\mathcal{C}_1(c_1,c_2)}{100} \|W^\eps_s(D)W_{\gamma/2}f\|_{L^2_xH^q}^2+C(c_1,c_2)\|W_{\gamma+5}f\|_{H^{\f32+\delta_1}_xL^2}^{2(q-1)}
(\|W^\eps_{q_1}(D)f\|_{H^{\f32+\delta_1}_xL^2}^2\\
&&\quad+\|W^{\eps}_{q_1}(D)f\|_{H^{\f32+\delta_1}_xL^2}\|W_{(\gamma+5)q}f\|_{H^{\f32+\delta_1}_xL^2})\|f\|_{L^2_xH^q}^2. \eeno 
From  the conditions $\f522^{N_{q,s,1}} \le l_{2}$ and $(\gamma+5)q\le l_{1,\f12+\delta_1}$,  the Proposition \ref{Enfvq} as well as the interpolation inequalities used in Step 1.1.3 of the proof of Theorem \ref{En-Priori},  we deduce from that for any $t\in [0,T_3]$,
 \beno  &&V^{q}(f(t))+\mathcal{C}_1(c_1,c_2)\int_0^t \big(\|W^\eps_s(D)W_{\gamma/2}f\|_{L^2_xH^q}^2+\delta^{-2s}\|W_{\gamma/2}f\|_{L^2_xH^q}^2)d\tau\\
 &&\le V^q(f_0)\int_0^t  (\mathcal{C}_6(c_1,c_2)\delta^{-4-2s-4(q-1)}\|f\|_{L^2_{\gamma/2}}^2+C_E\|f\|_{H^{N+\kappa}_xL^2}^2)d\tau\\&&+ C(c_1,c_2) \int_0^t (1+\mathbb{E}^{N,\kappa,\eps}(f))^{q}\|f\|_{L^2_xH^q}^2 d\tau+\int_0^t \mathbb{E}^{N,\kappa,\eps}(f)\mathbb{D}^{N,\kappa,\eps}_2(f)d\tau.
 \eeno
 Gronwall inequality will yield that $f\in C([0,T_3]; V^q)$.
It ends the proof of the first part of the theorem.
\end{proof}

\subsection{Proof of Theorem \ref{thmwepo}(Part II)}
Now we are in a position to complete the proof of Theorem \ref{thmwepo} thanks to the uniform bounds obtained from the previous result.

\begin{proof}[Proof of Theorem \ref{thmwepo}(Part II)] The proof falls into three steps.
	
{\it Step 1: Existence of the solution for the equation without cutoff.}	Suppose that $f^n$ is a solution to
\beno  \left\{
\begin{aligned}
	& \pa_t f+v\cdot \na_x f =Q^{\f1{n}}(f,f),\\
	&f|_{t=0}=f_0.\end{aligned}\right. \eeno
	Then by the first part of  Theorem \ref{thmwepo}, we arrive at that there exists a common lifespan $T^*$ such that  if $\f1n$ verifies \eqref{Rstreps} then $f^n$ has uniform bounds as follows
	\beno
	 &&\inf_{x\in\TT^3, t\in[0,T^*]}|f^n(t)|_{L^1}\ge  c_1; \sup_{ t\in[0,T^*]}\|f^n(t)\|_{H^{\f32+2\delta_1}_xL^2_{\gamma+4}}^2\le c_2; \sup_{ t\in[0,T^*]}V^{q_1,\f1{n}}(f^n(t))\le 4V^{q_1,\f1{n}}(f_0);\\
	&& \sup_{t\in[0,T_*]} \mathbb{E}^{1,\f12+2\delta_1,\f1{n}}(f^n(t))+A_7(c_1,c_2)\int_0^{T_*}\mathbb{D}_2^{1,\f12+2\delta_1,\f1{n}}(f^n(\tau))d\tau\nonumber\eeno\beno&&\quad+A_8(c_1,c_2)\int_0^{T_*}\mathbb{D}_1^{1,\f12+2\delta_1,\f1{n}}(f^n(\tau))d\tau\le 4M_1;
	\int_0^{T_*} \mathbb{D}^{1,\f12+2\delta_1,\f1{n}}_3(f^n(\tau))d\tau \le 4A_{9}(c_1,c_2) M_1;\\
	&& \sup_{t\in[0,T_*]} \mathbb{E}^{N,\kappa,\f1{n}}(f^n(t)) +A_7(c_1,c_2)\int_0^{T_*}\mathbb{D}_2^{N, \kappa,\f1{n}}(f^n(\tau))d\tau \\&&\quad+A_8(c_1,c_2)\int_0^{T_*}\mathbb{D}_1^{N, \kappa,\f1{n}}(f^n(\tau))d\tau\le 4M_2;
	\int_0^{T_*} \mathbb{D}^{N,\kappa,\f1{n}}_3(f^n(\tau))d\tau \le 4A_{9}(c_1,c_2) M_2.
	\eeno

	 Thanks to the renormalized theory for the equation(see \cite{dl}), we infer that there exists a sub-sequence $\{n_k\}$ and $f(t)$ such that
$\lim_{k\rightarrow\infty}\|f^{n_k}-f\|_{L^1((0,T^*)\times \TT^3\times\R^3)}=0. $
	It is easy to check that for any test function $\Psi$,
	\beno
	&&\bigg|\int_0^t\int_{\TT^3} \lr{Q^{\f1{n_k}}(f^{n_k},f^{n_k})-Q(f,f),\Psi}_v dxd\tau\bigg| 
 \lesssim \int_0^t\|f^{n_k}-f\|_{L^1_{\gamma+2}}\big(\|\Psi\|_{L^\infty_xW^{2,\infty}}+\|\Psi\|_{L^\infty_xH^{2s}}\big)\\&&\qquad\times(\|f^{n_k}\|_{L^\infty_xL^2}+\|f\|_{L^\infty_xL^1_{\gamma+2}})d\tau +n_k^{2s-2}\|f\|_{L^1_{\gamma+2}}\|f\|_{L^\infty_xL^1_{\gamma+2}}\|\Psi\|_{L^\infty_xW^{2,\infty}}t.
	\eeno
	It implies that $f$ is a non-negative weak solution to the nonlinear Boltzmann equation in the time interval $[0,T^*]$. Thanks to the uniform bounds obtained for $f^n$,  the  argument applied in Theorem \ref{Wepo-epslin} will yield  that $f$ verifies all the estimates in the theorem.	This means that $f$ is a classical solution to the equation.
	
{\it Step 2: Uniqueness.}	Suppose that $f_1,f_2$ are two solutions to the equation with the same initial data $f_0$, then $h\eqdefa f_1-f_2$ solves
	\beno \pa_t h+v\cdot \na_xh =Q(f_1,h)+Q(h,f_2). \eeno
Thanks to \eqref{unicon1},  there exists a constant $l$ verifies $N_s+2\le l$, $2^{N_{q_2,s,2}}(l+\gamma+2+N_s)\le l_{1,\f12+\delta_1}$, $\mathrm{1}_{\gamma\neq0}8C_Il^{-1}  (c_2^{\f12}+V^{q_1}(f_0))\le 1$ and $l^sA_1(c_1,c_2)>1$.
Due to Lemma \ref{L1md}, we have  
	\beno
	&&\f{d}{dt}\|h\|_{L^1_l}+(l^sA_1(c_1,c_2)-\eta^{2-2s}\| f_2 \|_{H^{\f32+\delta_1}_xH^{2s+\eta}_{2}}
	-l^{-1}\| f_2
	\|_{H^{\f32+\delta_1}_xH^\eta_{ 2}})\|h\|_{L^1_{l+\gamma}}
	\\&&\le  C(l)\|h\|_{L^1_l}\|f_1\|_{L^\infty_xL^1_{l+\gamma}}+
	C(l,\eta^{-1})\|h\|_{L^1_l}\|f_2\|_{H^{\f32+\delta_1}_xH^{2s+\eta}_{l+\gamma+2+N_s}}.
	\eeno Then by the   argument   used  to get \eqref{wbwc1}, we deduce that
	$\f{d}{dt}\|h\|_{L^1_l}\le C\|h\|_{L^1_l}. $
	Then  Gronwall inequality implies the uniqueness result.
	
{\it Step 3: Asymptotic formula.}	    Set $R^\eps=\eps^{2s-2}(f-f^\eps)$ and then it solves
	\beno \pa_tR^\eps+v\cdot \na_x R^\eps=Q^\eps(f,R^\eps)+Q^\eps(R^\eps, f^\eps)+\eps^{2s-2}(Q(f,f)-Q^\eps(f,f)).\eeno
	Thanks to Lemma \ref{L1md}, it is not difficult to check that
	\beno
	&&\f{d}{dt}\|R^\eps\|_{L^1_l}+(l^sA_1(c_1,c_2)-\eta^{2-2s}\| f_2 \|_{H^{\f32+\delta_1}_xH^{2s+\eta}_{2}}
	-l^{-1}\| f_2
	\|_{H^{\f32+\delta_1}_xH^\eta_{ 2}})\|R^\eps\|_{L^1_{l+\gamma}}
	\\&&\lesssim C(l)\|R^\eps\|_{L^1_l}\|f\|_{L^\infty_xL^1_{l+\gamma}}+
	C(l,\eta^{-1})\|R^\eps\|_{L^1_l}\|W^\eps_{2s+\eta}(D)f^\eps\|_{H^{\f32+\delta_1}_xL^2_{l+\gamma+2+N_s}} +\|f\|_{L^\infty_xL^1_{l+\gamma}}\|f\|_{L^1_{l+\gamma}}.
	\eeno
	Choose that $l$ verifies \eqref{unicon1}, then we obtain that
$ \|f-f^\epsilon\|_{L^\infty([0,T^*];L^1_l)}=O(\eps^{2-2s})$. It ends the proof of the theorem.
	\end{proof}

\section{Global dynamics of the Boltzmann equation}
In this section, we will consider the global dynamics of the Boltzmann equation \eqref{Boltz eq} under the  assumption  that the solution $f$  verifies \eqref{dynacondi}.  Our key observation lies in the energy-entropy method. To carry out our strategy, we divide our proof into four parts. In the first part, we will derive the entropy dissipation inequality following the argument due to Villani(see \cite{villani}). In the second part, we will investigate the dissipation estimates for the hydrodynamical fields. In the third part, we will show that under the condition
 \eqref{dynacondi}, regularity of the solution can be propagated. In the last part, we will show the new  mechanism for convergence to the equilibrium.
 
  Without loss of the generality, we assume that 
 the conserved quantities verify
\beno \int_{\TT^3\times\R^3} fdvdx=1; \int_{\TT^3\times\R^3} fvdvdx=0; \int_{\TT^3\times\R^3} f\f{|v|^2}2dvdx=\f32,\eeno
which imply that the hydrodynamical fields satisfy
\ben\label{normalized}  \int_{\TT^3} \rho dx=1; \int_{\TT^3} \rho udx=0; \int_{\TT^3} (\rho T+\f13\rho|u|^2)dx=1,\een
 and $M_{f}=M_{1,0,1}\eqdefa M$.

  \subsection{Entropy dissipation inequality}  In this subsection, we  will prove the following lemma:
 \begin{lem}\label{etrdiss} Suppose $f$ is a solution to the equation \eqref{Boltz eq} verifying $f\ge K_0 \exp\{-A_0 |v|^{q_0}\}$. Then for any $\delta, m,R>0$ and $0<a<1$, it holds
 \beno  D(f)&\eqdefa& \iint (f_*'f'-f_*f)\log\f{f_*'f'}{f_*f}|v-v_*|^\gamma b(\cos\theta)d\sigma dv_* dv dx\\
 &\ge& K^{-1}(\pi/2)^{-1-2s}\bigg[ \delta^{\gamma}R^{-(2-\gamma)}C_{K_0} H(f| M^f_{\rho,u,T})-\delta^{\gamma}R^{-(2m+4-\gamma)}\big(H(f|M)+\|f-M\|_{L^2_xL^2_{q_0+2m+4}}^2 \\
 &&\quad\times (1+\|f\|_{L^\infty_xL^2_{5q_0+4m+6}})+H(f|M)^a(1+\|f\|_{L^2_xL^2_{2(1-a)^{-1}(q_0+m+1)+q_0+2}})^{1-a}\big)\\&&-\delta^{\gamma+\f32}(\|f-M\|_{L^2_{q_0}}^2+\|f-M\|_{L^{\infty}_xL^2_{2q_0}}H(f|M))\bigg].
\eeno
When $\gamma=2$, then $a=1$.
 \end{lem}
 \begin{proof} The proof is inspired by the work \cite{villani} due to Villani.  We first notice that
 \beno |v-v_*|^\gamma b(\cos\theta)&\ge& K^{-1}|v-v_*|^\gamma\theta^{-1-2s}\ge K^{-1}(\pi/2)^{-1-2s}|v-v_*|^\gamma\\
 &\ge& K^{-1}(\pi/2)^{-1-2s}\big[\delta^{\gamma}R^{-(2-\gamma)}\big((1+|v-v_*|^2)-(1+|v-v_*|^2)\mathrm{1}_{|v-v_*|\ge R}\big)\\&&\quad-\delta^\gamma(1+\delta)^\gamma\mathrm{1}_{|v-v_*|\le\delta}\big].\eeno

 Suppose $
 D_1(f)\eqdefa  \iint (f_*'f'-f_*f)\log\f{f_*'f'}{f_*f}(1+|v-v_*|^2)d\sigma dv_* dv dx$,
 $D_2(f)\eqdefa  \iint (f_*'f'-f_*f)\log\f{f_*'f'}{f_*f}(1+|v-v_*|^2)\mathrm{1}_{|v-v_*|\ge R}d\sigma dv_* dv dx$ and
$D_3(f)\eqdefa  \iint (f_*'f'-f_*f)\log\f{f_*'f'}{f_*f}\mathrm{1}_{|v-v_*|\le\delta}d\sigma dv_* dv dx$.
 Then we get
$ D(f)\gtrsim D_1(f)-D_2(f)-D_3(f).$

 {\it Step 1: Estimate of $D_1(f)$}. Thanks to Theorem 2.1  in \cite{villani},  one has
 $ D_1(f)\gtrsim  T_*(f)H(f|M_{\rho,u,T}^f), $
 where $T_*(f)=\min\limits_{e\in \SS^2,x\in\TT^3} \int_{\R^3} f (v\cdot e)^2dv$.
 It is easy to check that
 \beno T_*(f)&\ge& K_0\int_{\R^3} \exp\{-K_0|v|^{q_0}\}  (v\cdot e)^2dv 
 \ge\f23K_0\int_0^\infty r^4\exp\{-K_0r^{q_0}\} dr\eqdefa C_{K_0}. \eeno

 {\it Step 2: Estimate of $D_2(f)$}.
From the lower bound condition and the fact(see Lemma 4.3 in \cite{villani}),
$  (X-Z)\log \f{X}{Z}\le C\max\{1, \log \f{X}{Z}, \log \f{Z}{X}\}((X-Y)\log \f{Y}{Z}, (Y-Z)\log \f{Y}{Z}),$ we easily deduce that
 \beno  (f_*'f'-f_*f)\log\f{f_*'f'}{f_*f}&\lesssim& (1+\log K_0+A(\lr{v'}^{q_0}+\lr{v_*'}^{q_0})(f_*'f'-M'_*M')\log\f{f_*'f'}{M'_*M'} \\
&& +(1+\log K_0+A(\lr{v}^{q_0}+\lr{v_*}^{q_0})(f_*f-M_*M)\log\f{f_*'f'}{M_*M}.  \eeno
  From this together with the symmetric property, we get
 \beno
 D_2(f)&\lesssim& R^{-2m-2} \int_{|v-v_*|\ge R} (\lr{v}+ \lr{v_*})^{q_0+2m+2}(f_*f-M_*M)\log \f{ff_*}{M_*M} dvdv_*dx\\
 &\lesssim&2R^{-2m-2} \int_{|v-v_*|\ge R} (\lr{v}+ \lr{v_*})^{q_0+2m+2}(f_*f-M_*M)\log \f{f}{M} dvdv_*dx.
 \eeno
 Due to the  decomposition
 \beno (f_*f-M_*M)\log \f{f}{M}&=&(f-M)_*\log \f{f}{M}(f-M)+M_*\log \f{f}{M}(f-M)\eeno\beno&&+(f-M)_*(M\log \f{M}{f}-M+f)+(f-M)_*(f-M), \eeno
the estimate of $D_2(f)$ will be split into four terms denoted by $I_i(i=1,2,3,4)$.  We first observe that
\beno I_1&\eqdefa&  \int_{|v-v_*|\ge R} (\lr{v}+ \lr{v_*})^{q_0+2m+2}(f-M)_*\log \f{f}{M}(f-M) dvdv_*dx\eeno\beno
&=& \int_{|v-v_*|\ge R} (\lr{v}+ \lr{v_*})^{q_0+2m+2} (f-M)_*[f\log \f{f}{M}-f+M+M\log \f{M}{f}-M+f]dvdv_*dx\\
&\lesssim&\|f-M\|_{L^\infty_xL^2_{q_0+2m+4}}\iint  \lr{v}^{q_0+2m+2}[f\log \f{f}{M}-f+M+M\log \f{M}{f}-M+f]dvdx.
\eeno
Using the lower bound condition and the inequality
$ M\log \f{M}{f}-M+f\le C\max\{1,\log\f{M}{f}\}(f\log \f{f}{M}-f+M),$
 we infer that
$ I_1\lesssim \|f-M\|_{L^2_xL^2_{q_0+2m+4}}H(f|M)^{\f12}(1+\|f\|_{L^\infty_xL^2_{5q_0+4m+6}})^{\f12}.
$

 Let $I_2\eqdefa \int_{|v-v_*|\ge R} (\lr{v}+ \lr{v_*})^{q_0+2m+2} M_*\log \f{f}{M}(f-M) dvdv_*dx$,
 $I_3\eqdefa \int_{|v-v_*|\ge R} (\lr{v}+ \lr{v_*})^{q_0+2m+2}(f-M)_*(M\log \f{M}{f}-M+f) dvdv_*dx$ and
 $I_4 \eqdefa \int_{|v-v_*|\ge R} (\lr{v}+ \lr{v_*})^{q_0+2m+2}(f-M)_*(f-M) dvdv_*dx$.
 Then by similar argument which is used to handle $I_1$, we may derive
that
\beno I_2+I_3+I_4&\lesssim& H(f|M)^{a}(1+\|f\|_{L^2_xL^2_{(1-a)^{-1}(2q_0+2m+2)+q_0+2}})^{1-a}\\&&+\|f-M\|_{L^2_xL^2_{q_0+2m+4}}H(f|M)^{\f12}(1+\|f\|_{L^\infty_xL^2_{5q_0+4m+6}})^{\f12}+\|f-M\|_{L^2_xL^2_{q_0+2m+4}}^2,
\eeno
which implies that
\beno D_2(f)&\lesssim& H(f|M)^{a}(1+\|f\|_{L^2_xL^2_{a^{-1}(2q_0+2m+2)+q_0+2}})^{1-a}\\&&+\|f-M\|_{L^2_xL^2_{q_0+2m+4}}H(f|M)^{\f12}(1+\|f\|_{L^\infty_xL^2_{5q_0+4m+6}})^{\f12}+\|f-M\|_{L^2_xL^2_{q_0+2m+4}}^2. \eeno

 {\it Step 3: Estimate of $D_3(f)$}. The argument used to the estimate of $D_2(f)$  can be applied to give the upper bound for $D_3(f)$. The only difference is that integral domain verifies $\lr{v}\sim \lr{v_*}$ thanks to  the condition  $|v-v_*|\le \delta$. We are led to
 \beno
 D_3(f)\lesssim \delta^{\f32}\big(\|f-M\|_{L^2_{q_0}}^2+\|f-M\|_{L^\infty_xL^2_{2q_0}}H(f|M)\big).
  \eeno

Combine all the estimates and then we will conclude the result. It  ends the proof of the lemma.
 \end{proof}

 \subsection{Dissipation estimates for the hydrodynamical fields}
By formal hydrodynamics, we first recall that $\rho, u$ and $T$ defined in \eqref{hydrofield} verify
\ben\label{dynamic1}
&\pa_t \rho +u\cdot\na_x \rho+\rho\na_x\cdot u=0;\\
&\pa_t  u+u\cdot\na_x u+\na_x T+\f{T\na_x\rho}{\rho}+\f{\na_x\cdot D}{\rho} =0, \label{dynamic2}
\\
&\pa_t T+u\cdot\na_x T+\f23T\na_x\cdot u+\f23\f{1}\rho (\na_xu:D+\na_x\cdot R)=0,\label{dynamic3}
\een
where  $R(x)=\f12\int_{\R^3} f|v-u|^2 (v-u)dv$ and $D=(D_{i,j})_{3\times3}$ with
$ D_{ij}=\int_{\R^3} f((v-u)_i(v-u)_j-\f13|v-u|^2\delta_{ij})dv.$

Let $M_{\rho u T}$ be any smooth local Maxwellian with parameters $\rho,u$ and $T$ which may depend on $t$ and $x$. Then we obtain that
\beno 
&&\pa_t(f-M_{\rho u T})+v\cdot \na_x(f-M_{\rho u T})\\&&=-M_{\rho u T}\bigg\{\bigg[\f{\pa_t\rho+u\cdot \na_x \rho}{\rho}-\f32\f{\pa_tT+u\cdot\na_x T}{T}\bigg]+\f{v-u}{\sqrt{T}}\cdot\bigg[\sqrt{T}\f{\na_x \rho}{\rho}-\f32\f{\na_x T}{\sqrt{T}}+
\f{\pa_tu+u\cdot\na_x u}{\sqrt{T}}\bigg]\\&&\quad+
\sum_{1\le i<j\le 3}\bigg(\f{v-u}{\sqrt{T}}\bigg)_i \bigg(\f{v-u}{\sqrt{T}}\bigg)_j[\pa_{x_i}u_j+\pa_{x_j}u_i]+\sum_{1\le i\le 3}\bigg(\f{v-u}{\sqrt{T}}\bigg)_i^2\bigg[\pa_{x_i}u_i+\f12\f{\pa_tT+u\cdot\na_x T}{T}\bigg]\\&&\quad+
\bigg|\f{v-u}{\sqrt{T}}\bigg|^2\f{v-u}{\sqrt{T}}\cdot \f{\na_x T}{2\sqrt{T}}
\bigg\}+Q(f, f-M_{\rho u T})+Q(f-M_{\rho u T},M_{\rho u T}).
\eeno

Let us abuse the notations to set  $M_{\rho,  u, \lr{T}_x}^f\eqdefa\f{\rho e^{-\f{|v-u|^2}{2\lr{T}_x}}}{(2\pi \lr{T}_x)^{\f32}} $ and  $M_{\rho,  \lr{u}_x, \lr{T}_x}^f\eqdefa\f{\rho e^{-\f{|v-\lr{u}_x|^2}{2\lr{T}_x}}}{(2\pi \lr{T}_x )^{\f32}} $ where $\rho, u$ are the hydrodynamical fields  associated to $f$, $\lr{T}_x= \int_{\TT^3} T dx$ and $\lr{u}_x= \int_{\TT^3} u dx$. Then thanks to (\ref{dynamic1}-\ref{dynamic3}), we infer that 
\beno 
&&(i).\,\pa_t(f-M^f_{\rho, u, T})+v\cdot \na_x(f-M^f_{\rho, u, T})\\&&=-M^f_{\rho, u, T}\mathbf{P}_1\bigg(\f{v-u}{\sqrt{T}}\bigg)+Q(f, f-M^f_{\rho, u, T})+Q(f-M^f_{\rho, u, T},M^f_{\rho, u, T}),\\
&&(ii).\,\pa_t(f-M^f_{\rho, u, \lr{T}_x})+v\cdot \na_x(f-M^f_{\rho, u, \lr{T}_x})\eeno\beno&&=-M^f_{\rho, u, \lr{T}_x}\mathbf{P}_2\bigg(\f{v-u}{\sqrt{\lr{T}_x}}\bigg)+Q(f, f-M^f_{\rho, u, \lr{T}_x})+Q(f-M^f_{\rho, u, \lr{T}_x},M^f_{\rho, u, \lr{T}_x}),\\ 
&&(iii).\,\pa_t(f-M^f_{\rho, \lr{u}_x, \lr{T}_x})+v\cdot \na_x(f-M^f_{\rho, \lr{u}_x, \lr{T}_x})\\
&&
=-M^f_{\rho, \lr{u}_x, \lr{T}_x}\mathbf{P}_3\bigg(\f{v-\lr{u}_x}{\sqrt{\lr{T}_x}}\bigg)
+Q(f, f-M^f_{\rho, \lr{u}_x, \lr{T}_x})+Q(f-M^f_{\rho,\lr{u}_x,\lr{T}_x},M^f_{\rho,\lr{u}_x,\lr{T}_x}),
\eeno
where
\beno \mathbf{P}_1\bigg(\f{v-u}{\sqrt{T}}\bigg)&\eqdefa& \bigg[\f{\na_x u: D}{\rho T}+\f{\na_x\cdot R}{\rho T}\bigg]+\f{v-u}{\sqrt{T}}\cdot\bigg[-\f52\f{\na_x T}{\sqrt{T}}- \f{\na_x\cdot D}{\rho\sqrt{T}} \bigg]+
\sum_{1\le i<j\le 3}\bigg(\f{v-u}{\sqrt{T}}\bigg)_i \bigg(\f{v-u}{\sqrt{T}}\bigg)_j\\&&\quad\times[\pa_{x_i}u_j+\pa_{x_j}u_i]+\sum_{1\le i\le 3}\bigg(\f{v-u}{\sqrt{T}}\bigg)_i^2\bigg[\pa_{x_i}u_i-\f13\na_x\cdot u-\f{\na_x u: D}{3\rho T}-\f{\na_x\cdot R}{3\rho T}\bigg] \\&&\quad+
\bigg|\f{v-u}{\sqrt{T}}\bigg|^2\f{v-u}{\sqrt{T}}\cdot \f{\na_x T}{2\sqrt{T}},\\
\mathbf{P}_2\bigg(\f{v-u}{\sqrt{\lr{T}_x}}\bigg)&\eqdefa&  \bigg(-\na_x\cdot u+\f32\f{\lr{-\f13T\na_x\cdot u+\f23\f1{\rho}(\na_xu:D+\na_x\cdot R)}_x}{\lr{T}_x}\bigg)+\f{v-u}{\sqrt{\lr
{T}_x}}\cdot \bigg[\sqrt{\lr{T}_x}\f{\na_x\rho}{\rho}-\f{\na_x T}{\sqrt{\lr{T}_x}}\\&&-\f{T\na_x \rho+\na_x\cdot D}{\rho\sqrt{\lr{T}_x}} \bigg]+
\sum_{1\le i<j\le 3}\f{(v-u)_i}{\sqrt{\lr{T}_x}} \f{(v-u)_j}{\sqrt{\lr{T}_x}}[\pa_{x_i}u_j+\pa_{x_j}u_i]+\sum_{1\le i\le 3}\bigg(\f{(v-u)_i}{\sqrt{\lr{T}_x}}\bigg)^2\\&&
\times\bigg(\pa_{x_i}u_i -\f12\f{\lr{-\f13T\na_x\cdot u+\f23\f1{\rho}(\na_xu:D+\na_x\cdot R)}_x}{\lr{T}_x}\bigg), \\
\mathbf{P}_3\bigg(\f{v-\lr{u}_x}{\sqrt{\lr{T}_x}}\bigg)&\eqdefa& \bigg(-\na_x\cdot u+\f32\f{\lr{-\f13T\na_x\cdot u+\f23\f1{\rho}(\na_xu:D+\na_x\cdot R)}_x}{\lr{T}_x}\bigg)+\f{v-\lr{u}_x}{\sqrt{\lr
{T}_x}}\cdot \bigg[\sqrt{\lr{T}_x}\f{\na_x\rho}{\rho}\eeno\beno-\f{\lr{u\cdot\na_x u+\na_x T+\f1{\rho}(T\na_x \rho+\na_x\cdot D)}_x}{\sqrt{\lr{T}_x}}  \bigg]  +\sum_{1\le i\le 3}\bigg(\f{(v-u)_i}{\sqrt{\lr{T}_x}}\bigg)^2\bigg( -\f12\f{\lr{-\f13T\na_x\cdot u+\f23\f1{\rho}(\na_xu:D+\na_x\cdot R)}_x}{\lr{T}_x}\bigg).
  \eeno
Here we use the notation: $A:B\eqdefa\sum\limits_{i,j=1}^3a_{ij}b_{ij}$ if $A=(a_{ij})_{3\times3}$ and $B=(b_{ij})_{3\times3}$.

  Before showing the dissipation estimates for the hydrodynamical fileds, we first give the estimates on  the hydrodynamical fields $\rho, u, T$ and $D, R$ in (\ref{dynamic1}-\ref{dynamic3}). 

\begin{prop}\label{Conput} Suppose that the solution $f$ to the equation \eqref{Boltz eq} verifies \eqref{dynacondi}. Then we have
\beno  &&(i). \quad T\ge c(c_1,c_2)>0; \|\rho-1\|_{H^{\f32+\delta_1}}+\|u\|_{H^{\f32+\delta_1}}+\|T-1\|_{H^{\f32+\delta_1}}\le  C(c_1,c_2); \\
&&(ii). \quad  \| D\|_{H^{\f32+\delta_1}}\le  C(c_1,c_2); \|D\|_{L^2}\le C(c_1,c_2) \|f-M^f_{\rho,u,T}\|_{L^2_xL^2_4};\\
&&(iii). \quad \| R\|_{H^{\f32+\delta_1}}\le  C(c_1,c_2); \|R\|_{L^2}\le C(c_1,c_2) \|f-M^f_{\rho,u,T}\|_{L^2_xL^2_5};\\
&&(iv). \quad \|(\na_x a)\na |D_x|^{-2} (b-\lr{b}_x))\|_{L^2}+\|a\na^2|D_x|^{-2}(b-\lr{b}_x)\|_{L^2}\le C\|a\|_{H^{\f32+\delta_1}}\|b-\lr{b}_x\|_{L^2};\\
&&(v).\quad \| |D_x|^{-1}(a\na b- \lr{a\na b}_x)\|_{L^2}\le C(  \|a\|_{H^{\f32+\delta_1}} \|b\|_{L^2}+\|a\|_{L^2}\|b\|_{H^{\f32+\delta_1}});\\ &&(vi). \quad |M_{\rho u T}\mathbf{P}(\f{v-u}{\sqrt{T}})|_{H^{2s}_{4}}\le C(c_1,c_2);\\
 &&(vii).\quad |\big(Q(f, f-M_{\rho u T}), M_{\rho u T}\mathbf{P}(\f{v-u}{\sqrt{T}})F(\rho, u, T)\big)|+|\big(Q( f-M_{\rho u T}, M_{\rho u T}), M_{\rho u T}\mathbf{P}(\f{v-u}{\sqrt{T}})F(\rho, u, T)\big)|\\
 &&\qquad\qquad\le  C(c_1,c_2)
 \|f-M_{\rho u T}\|_{L^2_{\gamma+4}}\|F(\rho, u, T)\|_{L^2},
\eeno
where $a,b$ are functions depending on $x$ variable, $\mathbf{P}$ is a polynomial function and $F(\rho, u, T)$ is a function depending only on $\rho, u, T$.
\end{prop}

\begin{proof} Let us give the proofs to the results term by term.
 We first remark that the inequalities in  $(i)$ are easily obtained by the definitions of the hydrodynamical fields and the assumption \eqref{dynacondi}.  
 
(1). For the results in $(ii)$ and $(iii)$,  we first observe that
$ \|D\|_{H^{\f32+\delta_1}}\lesssim \|f\|_{H^{\f32+\delta_1}_xL^2_4}(1+\|u\|_{H^{\f32+\delta_1}}) ^2$.
Notice that $D=\int_{\R^3} (f-M^f_{\rho,u,T})((v-u)\otimes (v-u)-\f13|v-u|^2\mathbf{I}_3)dv$. We have
$\|D\|_{L^2}\lesssim \|f-M^f_{\rho,u,T}\|_{L^2_4}(1+\|u\|_{L^\infty})^2,
$
which implies the results in $(ii)$. Since $R$ enjoys the similar structure as that for $D$, the estimates in $(iii)$ are easily followed. 

(2). It is easy to see that $(iv)$ is obtained by the fact that $\|ab\|_{L^2}\le \|a\|_{H^{j}}\|b\|_{H^{k}}$ where $j+k=\f32+\delta_1$ with $j,k\ge0$. We remark that this fact can be derived by  Lemma \ref{dsums}. As for $(v)$, we may copy the proof of Lemma  \ref{dsums} to show that 
\beno   |\sum_{m\in\ZZ^3, m\neq 0}\sum_{p\in\ZZ^3}
  |m|^{-1}  A_p(|m-p|B_{m-p})C_m|&\lesssim &\big(\sum_{p\in \ZZ^3} |p|^{2(\f32+\delta_1)} A_p^2\big)^{\f12}\big(\sum_{p\in \ZZ^3}   B_p^2\big)^{\f12}\big(\sum_{p\in \ZZ^3} C_p^2\big)^{\f12}\\&&+\big(\sum_{p\in \ZZ^3}  A_p^2\big)^{\f12}\big(\sum_{p\in \ZZ^3}  |p|^{2(\f32+\delta_1)} B_p^2\big)^{\f12}\big(\sum_{p\in \ZZ^3} C_p^2\big)^{\f12}, \eeno
which is enough to yield the desired result in $(v)$.

(3). Observe that
$ \lr{v}^4=(|v|^2+1)^2=(|v-u+u|^2+1)^2=a(v-u)+b(v-u)c(u)+d(u), $ 
 where $a,b,c,d$ are polynomial functions.
 Then we have
  \beno &&|M_{\rho u T}\mathbf{P}(\f{v-u}{\sqrt{T}})|_{H^{2s}_{4}}^2=\int_{\R^3} |\lr{D_v}^{2s} \lr{v}^4 M_{\rho u T}\mathbf{P}(\f{v-u}{\sqrt{T}})|^2dv =\int_{\R^3} |\lr{D_v}^{2s}  (a(v-u)+b(v-u)c(u)\\&&\quad+d(u))M_{\rho u T}\mathbf{P}(\f{v-u}{\sqrt{T}})|^2dv \lesssim  \rho^2\int_{\R^3}|\lr{\xi/\sqrt{T}}^{2s} e^{-iu\xi/\sqrt{T}} \mathcal{F}_{\xi}(M(a(\sqrt{T}\cdot)+b(\sqrt{T}\cdot)c(u)+d(u))\mathbf{P})(  \xi)|^2d\xi
 \\&&\le C(\|\rho\|_{L^\infty},\|u\|_{L^\infty},\|T\|_{L^\infty}),
\eeno 
which implies $(vi)$.
Next we recall that 
$ |\lr{Q(g,h), f}_v|\lesssim |g|_{L^2_{\gamma+4}}|h|_{L^2}|f|_{H^{2s}_{\gamma+2s}}$. From this together with the result in $(vi)$, we derive the desired result in $(vii)$.
	
We complete the proof of the proposition.	
 \end{proof}

Now we are in a position to prove
\begin{lem}\label{Disp1}
	 Suppose that the solution $f$  to the equation \eqref{Boltz eq} verifies  \eqref{dynacondi}. Then there exist constants $C=C(c_1,c_2)$ and $r_i(i=1,2,3)$ such that  
	 \beno 
	 &&(i).\, \f{d}{dt} \bigg(f-M^f_{\rho, u, T}, M^f_{\rho, u, T}\f{(2\pi T)^{\f32}}{\rho^2}\bigg(-5+\bigg|\f{v-u}{\sqrt{T}}\bigg|^2\bigg)(v-u)\cdot \na_x |D_x|^{-2}(T-\lr{T}_x)  \bigg)
	 +r_1\|T-\lr{T}_x\|_{L^2}^2\\&&\le 
	C \|f-M^f_{\rho,u,T}\|_{L^2_{\gamma+4}}(\|u-\lr{u}_x\|_{L^2}+\|T-\lr{T}_x\|_{L^2}+\|\rho-1\|_{L^2}+\|f-M^f_{\rho, u, T}\|_{L^2_xL^2_5});\\
	&&(ii).\, \f{d}{dt}\sum_{1\le m<n\le3}\bigg(f-M^f_{\rho, u, \lr{T}_x}, M^f_{\rho, u, \lr{T}_x}\f{(2\pi)^{\f32}}{\rho^2}\f{(v-u)_m}{\sqrt{\lr{T}_x}}\f{(v-u)_n}{\sqrt{\lr{T}_x}}|D_x|^{-2}[\pa_{x_m}(u_n-\lr{u_n}_x)+\pa_{x_n}(u_m\\&&\quad-\lr{u_m}_x)]\bigg)+ \f{d}{dt}\bigg(f-M^f_{\rho, u, \lr{T}_x}, M^f_{\rho, u, \lr{T}_x}\f{(2\pi)^{\f32}}{\rho^2} \sum_{1\le i\le 3}(-1+\f{(v-u)_i^2}{\lr{T}_x})|D_x|^{-2}\pa_{x_i}(u_i-\lr{u_i}_x)\bigg)\\&&+r_2\|u-\lr{u}_x\|_{L^2}^2\le C 	 \|f-M^f_{\rho,u,\lr{T}_x}\|_{L^2_xL^2_{\gamma+4}}( \|u-\lr{u}_x\|_{L^2}+\|T-\lr{T}_x\|_{L^2}+\|\rho-1\|_{L^2}+\|f-M^f_{\rho,u,T}\|_{L^2_xL^2_5});\\
	 &&(iii).\,  \f{d}{dt}\sum_{j=1}^3\bigg(f-M^f_{\rho, \lr{u}_x, \lr{T}_x}, M^f_{\rho,\lr{u}_x, \lr{T}_x}(v_j-\lr{u_j}_x)\pa_{x_j}|D_x|^{-2}(\rho-1) \bigg)+r_3\|\rho-1\|_{L^2}^2\\
	 &&\le C 
	 \|f-M^f_{\rho,\lr{u}_x, \lr{T}_x}\|_{L^2_xL^2_{\gamma+4}} ( \|u-\lr{u}_x\|_{L^2}+\|T-\lr{T}_x\|_{L^2}+\|\rho-1\|_{L^2}).
	 \eeno  
\end{lem}
\begin{proof}  Since the proofs of $(i), (ii)$ and $(iii)$ are very similar, we only give the detailed proof to $(i)$. Denote 
	$$\mathbf{P}(\f{v-u}{\sqrt{T}})F(\rho,u,T)\eqdefa \f{(2\pi T)^{\f32}}{\rho^2}\bigg(-5+\bigg|\f{v-u}{\sqrt{T}}\bigg|^2\bigg)(v-u)\cdot \na_x |D_x|^{-2}(T-\lr{T}_x).$$
It is easy to check that 
\beno &&\f{d}{dt}\bigg(f-M^f_{\rho, u, T}, M^f_{\rho, u, T}\mathbf{P}(\f{v-u}{\sqrt{T}})F(\rho,u,T) \bigg)=\bigg(-M^f_{\rho, u, T}\mathbf{P}_1(\f{v-u}{\sqrt{T}}), M^f_{\rho, u, T}\mathbf{P}(\f{v-u}{\sqrt{T}})F(\rho,u,T)\bigg)\\
&&\quad+\bigg(Q(f, f-M^f_{\rho, u, T})+Q(f-M^f_{\rho, u, T},M^f_{\rho, u, T}), M^f_{\rho, u, T}\mathbf{P}(\f{v-u}{\sqrt{T}})F(\rho,u,T)\bigg)\eeno\beno
&&\quad+ \bigg(f-M^f_{\rho, u, T}, M^f_{\rho, u, T}v\cdot \na_x \big(\mathbf{P}(\f{v-u}{\sqrt{T}})F(\rho,u,T) \big)\bigg) +\bigg(f-M^f_{\rho, u, T},   M^f_{\rho, u, T}\mathbf{P}_1(\f{v-u}{\sqrt{T}})\mathbf{P}(\f{v-u}{\sqrt{T}})F(\rho,u,T) \bigg)\\  &&\quad+
\bigg(f-M^f_{\rho, u, T},  M^f_{\rho, u, T} \pa_t\mathbf{P}(\f{v-u}{\sqrt{T}})F(\rho,u,T) +\mathbf{P}(\f{v-u}{\sqrt{T}})\pa_tF(\rho,u,T) \bigg)\eqdef \sum_{i=1}^5I_i.\eeno	

We will give the estimates to $I_i(i=1,2,3,4,5)$ term by term. By Proposition \ref{Conput},  changing of variables and using the condition that if $1\le i,j,k\le 3$, 
$
\int_{\R^3} w_ie^{-|w|^2}dw=\int_{\R^3} w_iw_jw_ke^{-|w|^2}dw=0,
$ we have
\beno &&\bigg(M^f_{\rho, u, T}\mathbf{P}_1(\f{v-u}{\sqrt{T}}), M^f_{\rho, u, T}\mathbf{P}(\f{v-u}{\sqrt{T}})F(\rho,u,T)\bigg)\\
&&=\f12\int_{\TT^3}\int_{\R^3} e^{-|w|^2}(-5+|w|^2)^2 (w\cdot\na_x T)(w\cdot\na_x |D_x|^{-2}(T-\lr{T}_x))  dwdx\\
&&\quad-\int_{\TT^3}\int_{\R^3}\rho^{-1} e^{-|w|^2}(-5+|w|^2) \big(w\cdot(\na_x \cdot D)\big)(w\cdot\na_x|D_x|^{-2}(T-\lr{T}_x)) dwdx\\
&&\ge\sum_{j=1}^3(\f12\int_{\R^3} e^{-|w|^2}|w_j|^2(-5+|w|^2)^2dw)\|\pa_{x_j}|D_x|^{-1}(T-\lr{T}_x)\|_{L^2}^2- C\|  D\|_{L^2}\|T-\lr{T}_x\|_{L^2}\\
&&\ge r_1\|T-\lr{T}_x\|_{L^2}^2- C\|f-M^f_{\rho, u, T}\|_{L^2_xL^2_4}^2,  \eeno 	
	which implies that
$ I_1\le -r_1\|T-\lr{T}_x\|_{L^2}^2+ C\|f-M^f_{\rho, u, T}\|_{L^2_xL^2_4}^2$.

	For $I_2$, by virtue of $(vii)$ of Proposition  \ref{Conput}, it is not difficult to check that 
	\beno |I_2| \le C\|f-M^f_{\rho, u, T}\|_{L^2_xL^2_{\gamma+4}} \|T-\lr{T}_x\|_{L^2}.  \eeno
	
	For $I_3$ and $I_4$, we first observe that the typical terms in the expression of $ \na_x \big(\mathbf{P}(\f{v-u}{\sqrt{T}})F(\rho,u,T) \big)$ and $ \mathbf{P}_1(\f{v-u}{\sqrt{T}})\mathbf{P}(\f{v-u}{\sqrt{T}})F(\rho,u,T)$ are
	$ P_1(\f{v-u}{\sqrt{T}})P_2(\rho,u,T,D)\na a\na_x|D_x|^{-2}(T-\lr{T}_x)$ and $ P_3(\f{v-u}{\sqrt{T}})P_4(\rho,u,T) \\ \na_x^2|D_x|^{-2}(T-\lr{T}_x)$, where $P_i(i=1,2,3,4)$ are polynomial functions and $a=\rho, u, T, D,R$. 
By $(ii-iv)$ of Proposition  \ref{Conput}, we  arrive at	
	\beno  |I_3|+|I_4|\le C \|f-M^f_{\rho, u, T}\|_{L^2_xL^2_2}\|T-\lr{T}_x\|_{L^2}.\eeno
	
For $I_5$, the difficult term lies in the case that the time derivative acts on $\pa_{x_j}|D_x|^{-2}(T-\lr{T}_x)$. The other terms can be handled similarly as those for $I_3$ and $I_4$. By \eqref{dynamic3}, we derive that
\beno&& \pa_{x_j}|D_x|^{-2}\pa_t(T-\lr{T}_x)=-\pa_{x_j}|D_x|^{-2}\bigg( \f23\mathrm{div} (uT-\lr{u}_x\lr{T}_x)+\f13\big[(u-\lr{u}_x)\cdot\na_x (T-\lr{T}_x)-\lr{ u\cdot\na_x T}_x\big]\\&&\quad+\f13\lr{u}_x\cdot\na_x(T-\lr{T}_x) +\f23\big[(\f{1}\rho-1) (\na_x(u-\lr{u}_x):D+\na_x\cdot R)-\lr{ (\f{1}\rho-1) (\na_xu:D+\na_x\cdot R}_x\big]\\&&\quad+\f23 \big[ (\na_x(u-\lr{u}_x):D+\na_x\cdot R)\big]-\lr{   (\na_xu:D+\na_x\cdot R}_x \bigg).\eeno
	Thanks to $(v)$ of Proposition \ref{Conput}, we obtain that 
	\beno \|\pa_{x_j}|D_x|^{-2}\pa_t(T-\lr{T}_x)\|_{L^2}\le C(\|u-\lr{u}_x\|_{L^2}+\|T-\lr{T}_x\|_{L^2}+\|\rho-1\|_{L^2}+\|f-M^f_{\rho, u, T}\|_{L^2_xL^2_5}),\eeno
	which implies that
	\beno  |I_5|\le C\|f-M^f_{\rho, u, T}\|_{L^2_xL^2_2}(\|u-\lr{u}_x\|_{L^2}+\|T-\lr{T}_x\|_{L^2}+\|\rho-1\|_{L^2}+\|f-M^f_{\rho, u, T}\|_{L^2_xL^2_5}).\eeno

Now putting together all the estimates will yield the result in  $(i)$.
We may repeat the similar argument to derive $(ii)$ and $(iii)$. We only emphasize that to get $(ii)$ we will use 
Korn inequality, that is, 
\beno \|\na^{sym} |D_x|^{-1}(u-\lr{u}_x)\|_{L^2}\gtrsim \|\na_x |D_x|^{-1}(u-\lr{u}_x)\|_{L^2}\sim \|u-\lr{u}_x\|_{L^2},\eeno  
where $\na^{sym} u=\f12(\na u+(\na u)^T)$.
We complete the proof of the lemma. 
\end{proof}

As a consequence, we get the full dissipation estimates for the hydrodynamical fields. That is,
\begin{thm}\label{Disp2} Suppose that the solution $f$  to the equation \eqref{Boltz eq}verifies \eqref{dynacondi}. Then there exist a function $M_h(t)$ and a constant $C=C(c_1,c_2)$ verifying that $|M_h(t)|\lesssim \|(f-M)(t)\|_{L^2_xL^2_4}^2$ and 
\beno \f{d}{dt}M_h(t)+\|T-1\|_{L^2}^2+\|u\|_{L^2}^2+\|\rho-1\|_{L^2}^2 \le C\eta^{-3}H(f|M^f_{\rho,u,T})+C\eta (\|f-M\|_{H^1_xL^2_{10}}^2+\|f-M\|_{L^2_xH^1_{10}}^2). \eeno
 \end{thm}
\begin{proof}  Thanks to Lemma \ref{Disp1}, the results  can be written as
\beno  \f{d}{dt} A_1(t)+\f{r_1}2\|T-\lr{T}_x\|_{L^2}^2&\le &
	C\eta^{-2} \|f-M^f_{\rho,u,T}\|_{L^2_xL^2_{\gamma+5}}^2+\eta^2(\|u-\lr{u}_x\|_{L^2}^2+ \|\rho-1\|_{L^2}^2);\\
	 \f{d}{dt} A_2(t)+\f{r_2}2\|u-\lr{u}_x\|_{L^2}^2&\le &
	C\eta^{-1} (\|f-M^f_{\rho,u,T}\|_{L^2_xL^2_{\gamma+5}}^2+\|T-\lr{T}_x\|_{L^2}^2)+\eta \|\rho-1\|_{L^2}^2;\eeno\beno
	 \f{d}{dt} A_3(t)+\f{r_3}2\|\rho-1\|_{L^2}^2&\le &
	C(\|f-M^f_{\rho,u,T}\|_{L^2_xL^2_{\gamma+5}}^2+\|T-\lr{T}_x\|_{L^2}^2+\|u-\lr{u}_x\|_{L^2}^2),
	 \eeno
	where we use the inequalities $\|f-M^f_{\rho,u,\lr{T}_x}\|_{L^2_xL^2_{\gamma+5}}\le \|f-M^f_{\rho,u,T}\|_{L^2_xL^2_{\gamma+5}}+C\|T-\lr{T}_x\|_{L^2}$ and 
	 $\|f-M^f_{\rho,\lr{u}_x,\lr{T}_x}\|_{L^2_xL^2_{\gamma+5}}\le \|f-M^f_{\rho,u,T}\|_{L^2_xL^2_{\gamma+5}}+C(\|T-\lr{T}_x\|_{L^2}+\|u-\lr{u}_x\|_{L^2})$. 
	 
	Let $A(t)=A_1(t)+\f{r_1}4C^{-1}\eta A_2(t)+\eta^{\f32}A_3(t)$ with $\eta$ sufficiently small. Then we have
	 \beno \f{d}{dt} A(t)+\eta^{\f32}\f{r_3}4(\|T-\lr{T}_x\|_{L^2}^2+\|u-\lr{u}_x\|_{L^2}^2+\|\rho-1\|_{L^2}^2)\le C\eta^{-2} \|f-M^f_{\rho,u,T}\|_{L^2_xL^2_{\gamma+5}}^2.  \eeno
	Thanks to \eqref{normalized}, 
we derive that $|\lr{u}_x|=|\int_{\TT^3} (\rho-1)udx|$ and $|-\lr{T}_x+1|=|\int_{T^3} ((\rho-1)T+\f13\rho|u|^2)dx|$,  which yield
$ \|T-\lr{T}_x\|_{L^2}^2+\|u-\lr{u}_x\|_{L^2}^2+\|\rho-1\|_{L^2}^2\ge C(\|T-1\|_{L^2}^2+\|u\|_{L^2}^2+\|\rho-1\|_{L^2}^2).$

Then the theorem is followed by setting $M_h(t)=4\eta^{-\f32}(Cr_3)^{-1}A(t)$ and the inequality
$\|f-M^f_{\rho,u,T}\|^2_{L^2_xL^2_{\gamma+5}}\\ \le \|f-M^f_{\rho,u,T}\|_{L^1}^{\f12}\|(f-M^f_{\rho, u, T})W_{10}\|_{L^3} ^{\f32}
\lesssim \eta_1^{-3}H(f|M^f_{\rho,u,T})+\eta_1 (\|f-M\|_{H^1_xL^2_{10}}^2+\|f-M\|_{L^2_xH^1_{10}}^2).$
\end{proof}

 \subsection{Refined energy estimates for the equation}   Under the assumption \eqref{dynacondi}, we will refine the energy estimates for the equation
 \ben\label{h-equ} \pa_t h+v\cdot \na_x h=Q(f, h)+Q(h,g), \een
 where $h=f-g$. Suppose that $\mathbb{P}_e^{1,\f12+2\delta_1}$ and $\mathbb{P}_e^{N,\kappa}$ are function spaces associated to $\mathbb{W}_{I}(N,\kappa,\varrho,\delta_1)$ where  $  \mathbb{W}_{I}(N,\kappa,\varrho,\delta_1)=\{W_{1.\f12+\delta_1}, W_{1.\f12+2\delta_1}\}\cup \{W_{m,n}\}_{(m,n)\in \mathbb{I}_x(N,
 	\kappa)}$ with $N+\kappa\ge \f52+\delta_1$ and $W_{0,-1}= W_{l_1},  W_{0,0}= W_{l_2}$.  In what follows, we will assume that   
 \ben\label{boufg}&& 8C_E(\|g(t)\|_{H^{\f32+\delta_1}_xH^{\eta_1}_{ \gamma+4}}+c_2)\le l_1^{1+s}A_1(c_1,c_2),\,\|g(t)\|_{H^{\f32+\delta_1}_xL^2_{\gamma+4}}\le C_1,  \notag\\&& \sup_{t\ge0}\big(\|W_{l_1}W_{N_s+3}g(t)\|_{H^{\f32+\delta_1}_xH^{2s+\eta}} +\sum_{(m,n)\in\mathbb{I}_x(1,\f12+\delta_1)}\|W_{m,n}W_{\f32\gamma+2s+4}g(t)\|_{H^{m+n\varrho}_xH^s}\notag\\&&+\sum_{i=1}^2\|W_{1,\f12+i\delta_1}W_{\f32\gamma+2s+4} g\|^2_{H_x^{\f32+i\delta_1}H^s}\big) \le \mathbb{C}(l_1,g).\een

 \subsubsection{Estimates of the propagation and production of $L^1$ and $L^2$  moments.}
 We want to prove:

\begin{lem}\label{refL12} Suppose that \eqref{dynacondi} and \eqref{boufg} hold for f and g respectively and $h$ is a unique and smooth solution to \eqref{h-equ}. If   $\delta=\delta(c_1,c_2,C_1)$ is sufficiently small, then there exists a constant $c_{0,0}=C(\delta, A_2,A_3)$ such that
 \beno
 &&  4A_1^{-1}A_4\delta^{-12(1+s)} \|h(t)\|_{L^1_{l_1}}^2+\|h(t)\|_{L^2_{l_2}}^2+ l_1^sA_4 \delta^{-12(1+s)}   \int_0^t \|h\|_{L^1_{l_1}}\|h\|_{L^1_{l_1+\gamma}}d\tau
 \\&&\quad +A_2 
 \int_0^t  \| W_{l_2+\gamma/2}h\|_{L^2_xH^s}^2  d\tau+\f12A_3  \delta^{-2s} \int_0^t \|h\|_{L^2_{l_2+\gamma/2}}^2 d\tau +c_{0,0}\int_0^t \| W_{0,1}W_{\gamma/2+d_2}h\|^2_{H^{\varrho}_xL^2}d\tau
 \\&&\le \delta^{-12(1+s)}  4A_1^{-1}A_4\|h_0\|_{L^1_{l_1}}^2+\|h_0\|_{L^2_{l_2}}^2+c_{0,0}\|W_{0,1}W_{\gamma/2+d_1+d_2}h_0\|_{L^2}^2+\mathbb{C}(l_1,g)  \int_0^t     (\|h\|^2_{L^1} + \|h\|^2_{H^{\f32+\delta_1}_xL^2_{\gamma+4}})d\tau.
 \eeno
Moreover, the moment production estimates hold for $\gamma>0$, i.e. for any $l_1(=2l_2)>N_s+2 $ and $t\ge t_0>0$,
 \beno  &&\|h(t)\|_{L^1_{l_1}}+\|h(t)\|_{L^2_{l_1/2}}^2\lesssim \big(\f{8C(l_1, c_2)}{ C(c_1,c_2)(1-e^{-C(l_1,c_2)t\gamma/l_1})}\big)^{l_1/\gamma};
 \\&& \int^{t+1}_t  \| W_{l_2+\gamma/2}h\|_{L^2_xH^s}^2  d\tau  +\int^{t+1}_t \| W_{0,1}W_{\gamma/2+d_2}h\|^2_{H^{\varrho}_xL^2}d\tau \le C(t_0). \eeno
 \end{lem}
 \begin{proof} The proof will be concluded   by   $L^1$ and $L^2$ moments estimates.
 
 {\it \underline{$L^1$-moment estimates.}} By Lemma \ref{L1md} and the facts
 \beno &&\int_{\TT^3}|f|_{L^1_{l_1+\gamma-2}}|h|_{L^1_{\gamma+2}}dx\lesssim
 \|h\|_{L^\infty_xL^1_{\gamma+2}}\|h\|_{L^1_{l_1+\gamma-2}}+\|h\|_{L^1_{\gamma+2}}\|g\|_{L^\infty_xL^1_{l_1+\gamma-2}}\lesssim \mathbb{C}(l_1,g)\|h\|_{L^1_{l_1+\gamma-2}},\\
 &&\int_{\TT^3}|f|_{L^1_{l_1+\gamma}}|h|_{L^1_{\gamma}}dx \lesssim  \mathbb{C}(l_1,g)\|h\|_{L^1_{\gamma}}+(C_1+c_2)\|h\|_{L^1_{l_1+\gamma}},
  \eeno
  we infer that
  \beno&& \|h(t)\|_{L^1_{l_1}}+ \int_0^t \bigg(l_1^sA _1(c_1,c_2)-\eta^{2-2s}\|\lr{D}^{2s+\eta_1}g \|_{H^{\f32+\delta_1}_xL^2_{2}}
  -2l^{-1}_1C_E\|\lr{D}^{\eta_1} g\|_{H^{\f32+\delta_1}_xL^2_{\gamma+2}}-2l_1^{-1}C_E\\&&\times(C_1+c_2)\bigg)\|h\|_{L^1_{l_1+\gamma}}d\tau  \le \|h_0\|_{L^1_{l_1}}+\mathbb{C}(l_1,g) \bigg[\int_0^t    \|h\|_{L^1_{l_1+\gamma-2}}   d\tau\bigg]+
 \eta^{-2s} C(c_2)  \int_0^t    \|h\|_{ L^1_{N_s+2}} d\tau.\eeno
Thanks to \eqref{boufg} and interpolation inequality, we deduce that
\beno&& \|h(t)\|_{L^1_{l_1}}+\f12 l_1^sA _1(c_1,c_2 ) \int_0^t \|h\|_{L^1_{l_1+\gamma}}d\tau  
 \le \|h_0\|_{L^1_{l_1}}+\mathbb{C}(l_1,g)  \int_0^t     \|h\|_{L^1} d\tau,\eeno
which implies that
\ben\label{smL1} \f{d}{dt}\|h\|_{L^1_{l_1}}+\f18 l_1^sA _1(c_1,c_2 )\|h\|_{L^1_{l_1}}^{1+\gamma/l_1}\le \mathbb{C}(l_1,g)\|h\|_{L^1_{l_1}}.\een
On the other hand, slight modification of the proof in the above will yield that
 \ben\label{mcL1}&& \|h(t)\|^2_{L^1_{l_1}}+\f12 l_1^sA _1(c_1,c_2 ) \int_0^t \|h\|_{L^1_{l_1+\gamma}}\|h\|_{L^1_{l_1}}d\tau  
\le \|h_0\|^2_{L^1_{l_1}}+\mathbb{C}(l_1,g)  \int_0^t     \|h\|^2_{L^1} d\tau.\een 

{\it \underline{$L^2$-moment estimates.}} By modifying Proposition \ref{L2m} and using the assumptions \eqref{dynacondi} and \eqref{boufg}, we will get
 \ben\label{mcL2} &&\|h(t)\|_{L^2_{l_2}}^2
 +A_2(c_1,c_2)
 \int_0^t   \|W_{l_2+\gamma/2}h\|_{L^2_xH^s}^2  + (A_3(c_1,c_2) \delta^{-2s}-3C_E(C_1+c_2))\int_0^t \|h\|_{L^2_{l_2+\gamma/2}}^2 d\tau\notag
 \\&&\le \|h_0\|_{L^2_{l_2}}^2+A_4(c_1, c_2)\delta^{-12(1+s)}\int_0^t  \|h\|^2_{L^1_{2l_2}}d\tau+  \mathbb{C}(l_1,g)\int_0^t \|h\|^2_{L^\infty_xL^1_{\gamma+2}}d\tau,
\\   && \f{d}{dt}\|h(t)\|_{L^2_{l_2}}^2
 +A_2(c_1,c_2)
  \|W_{l_2+\gamma/2}h\|_{L^2_xH^s}^2  + (A_3(c_1,c_2) \delta^{-2s}-3C_E(C_1+c_2))  \|h\|_{L^2_{l_2+\gamma/2}}^2 
 \notag\\&&\le  C(c_1, c_2)\delta^{-6(1+s)}  \|h\|_{L^1_{2l_2}} + \mathbb{C}(l_1,g) \|h\|^2_{L^2_{\gamma+4}}.\label{smL2}
 \een

 Now we are in a position to prove the lemma. If we choose that $\f12A_3(c_1,c_2) \delta^{-2s}>3C_E(C_1+c_2)$, then by 
 \eqref{mcL1} and \eqref{mcL2} we deduce that
 \beno
 &&  4A_1^{-1}A_4\delta^{-12(1+s)} \|h(t)\|_{L^1_{l_1}}^2+\|h(t)\|_{L^2_{l_2}}^2+ l_1^sA_4 \delta^{-12(1+s)}   \int_0^t \|h\|_{L^1_{l_1}}\|h\|_{L^1_{l_1+\gamma}}d\tau
\\&&\quad +A_2 
 \int_0^t    \| W_{l_2+\gamma/2}h\|_{L^2_xH^s}^2   d\tau+\f12A_3  \delta^{-2s} \int_0^t \|h\|_{L^2_{l_2+\gamma/2}}^2 d\tau
 \\&&\le 4A_1^{-1}A_4 \delta^{-12(1+s)}  \|h_0\|_{L^1_{l_1}}^2+\|h_0\|_{L^2_{l_2}}^2+\mathbb{C}(l_1,g) \big(\int_0^t     \|h\|^2_{L^1} d\tau+ \int_0^t \|h\|^2_{H^{\f32+\delta_1}_xL^2_{\gamma+4}}d\tau\big).
 \eeno
Recalling Proposition \ref{smx01}, we derive that
 \ben\label{sHs}  \int_0^t \| W_{0,1}W_{\gamma/2+d_2}h\|^2_{H^{\varrho}_xL^2}d\tau 
 &\le &  \|W_{0,1}W_{\gamma/2+d_1+d_2}h_0\|_{L^2}^2+C\int_0^t \| W_{0,1}W_{\gamma/2+d_1+d_2}h\|_{L^2_xH^s}^2d\tau\notag \\&&  +C( c_2)\int_0^t  \|W_{0,1}W_{\f32\gamma+2s+d_1+d_2}h\|^2_{L^2}
  +C( l_{0,1})\int_0^t\|h\|^2_{L^2_{\gamma+4}} d\tau.
 \een
 We  conclude  the desired result by combining the last two estimates.

 Next we will prove the moment production estimates. In fact, by \eqref{smL1} and \eqref{smL2}, we have
 \beno &&\f{d}{dt}\big(2C(c_1,c_2)\|h\|_{L^1_{l_1}}+\|h\|_{L^2_{l_2}}^2\big)+\f18 l_1^sA _1(c_1,c_2 )\|h\|_{L^1_{l_1+\gamma}}+\f14A_3(c_1,c_2) \delta^{-2s}  \|h\|_{L^2_{l_2+\gamma/2}}^2
 \\&&
 \le C(\mathbb{C},c_1,c_2)\big(\|h\|_{L^1_{l_1}}+\|h\|_{L^2_{l_2}}^2\big),\eeno
 which together with interpolation inequality will give
 \beno &&\f{d}{dt}\big(2C(c_1,c_2)\|h\|_{L^1_{l_1}}+\|h\|_{L^2_{l_2}}^2\big)+\f18 l_1^sA _1(c_1,c_2 )\|h\|_{L^1_{l_1}}^{1+\gamma/l_1}+\f14A_3(c_1,c_2) \delta^{-2s}  \|h\|_{L^2_{l_2}}^{2(1+\gamma/(2l_2))}
 \\&&
 \le C(\mathbb{C}, c_1,c_2)\big(\|h\|_{L^1_{l_1}}+\|h\|_{L^2_{l_2}}^2\big).\eeno
 It implies that for any $l_1(=2l_2)>N_s+2$,
$\|h\|_{L^1_{l_1}}+\|h\|_{L^2_{l_1/2}}^2\lesssim \big(\f{8C(\mathbb{C},c_1,c_2)}{ C(c_1,c_2)(1-e^{-C(\mathbb{C},c_1,c_2)t\gamma/l_1})}\big)^{l_1/\gamma}.$
From this together with \eqref{smL2} and 
\eqref{sHs}, we get the desired results.
 \end{proof}

  \subsubsection{Refined energy inequality} We want to prove
  \begin{lem}\label{refEn} Suppose that \eqref{dynacondi} and \eqref{boufg} hold for f and g respectively and  $h$ is a unique and smooth solution to \eqref{h-equ}. Let $ X(t)= \mathbb{X}^{1,\f12+2\delta_1,q_1}(h(t))$ or $X(t)=\mathbb{P}_e^{1,\f12+2\delta_1}(h(t))$. Then
one has for any $t_2>t_1\ge0$, there exist constants $b_i=b_i(c_1,c_2,C_1)(i=1,2)$ such that
\beno  &&X(t_2)+\f{b_1}2\int_{t_1}^{t_2} \big(X(\tau)+  \|W_{1,N_{\rho,2}+1}W_{\gamma/2+d_2} h\|_{H^{1+(N_{\rho,2}+1)\varrho}_xL^2}^2\big)d\tau\le b_2X(t_1)+C(\mathbb{C},c_1,c_2)  \int_{t_1}^{t_2}     \|h\|^2_{L^1} d\tau, \eeno
and for all $t\ge0$, $e^{\f{b_1}4t}X(t)+\f{b_1}4\int_0^t e^{\f{b_1}4\tau}X(\tau)d\tau\le b_2X(0)+C(\mathbb{C},c_1,c_2)  \int_0^t    e^{\f{b_1}4\tau} \|h\|^2_{L^1} d\tau.$ As a result, we have
$ X(t)\le b_2e^{-\f{b_1}4t}X(0)+C(\mathbb{C},c_1,c_2)  \int_0^t    e^{-\f{b_1}4(t-\tau)} \|h\|^2_{L^1} d\tau.$
  \end{lem}

 \begin{proof}
 We separate the proof into three steps to get the desired results. We remark that the condition \eqref{boufg} will be used in each step.

 {\it Step 1: Propagation of regularity for $v$ variable.} Thanks to Proposition \ref{Envq} and the interpolation inequalities used in {\it Step 1.3} in the proof of Theorem \ref{En-Priori}, we get that
 \beno
 &&V^{q}(h(t))+\mathcal{C}_1(c_1,c_2)\int_0^t \|W_{\gamma/2}h \|_{L^2_xH^{q_1+s}}^2\\
 &&\le  V^{q}(h_0)+C(c_1,c_2,\eta^{-1})\int_0^t \| h\|^2_{H^{\f32+\delta_1}_xL^2_{\gamma+4}}d\tau+\eta\int_0^t \|h\|_{H^1_xH^{s}}^2d\tau
 +C\int_0^t \|h\|_{L^2_{l_2}}^2d\tau.
 \eeno

 {\it Step2: Propagation of regularity for $x$ variable.} Thanks to Proposition \ref{Enxmn1}, Proposition \ref{Enx32} and Proposition \ref{Enxmn2}, we get
 \beno && \|W_{m,n}    h(t)\|^2_{H_x^{m+n\varrho}L^2}+\|W_{1,\f12+\delta_1}  h(t)\|^2_{H^{\f32+\delta_1}_xL^2}+\|W_{\gamma+4}  h(t)\|^2_{H^{\f32+2\delta_1}_xL^2}
  +c_oA_5(c_1, c_2)\\&&\times\int_0^t\bigg(\| W_{\gamma/2}W_{m,n} h\|_{H^{m+n\varrho}_xH^s}^2 +\| W_{\gamma/2}W_{1,\f12+\delta_1}   h\|_{H^{\f32+\delta_1}_xH^s}^2 +\| W_{\gamma/2}W_{\gamma+4} h\|_{H^{\f32+2\delta_1}_xH^s}^2 \bigg)d\tau\\
 &&\le  \|W_{m,n}  h_0\|^2_{H^{m+n\varrho}_xL^2}+\|W_{1,\f12+\delta_1}  h_0\|^2_{H^{\f32+\delta_1}_xL^2}+\|W_{\gamma+4} h_0\|^2_{H^{\f32+2\delta_1}_xL^2}   +C(c_2,\mathbb{C})\\&&\times\int_0^t \bigg(\|W_{\gamma/2}W_{m,n}  h\|_{H^{m+n\varrho}_xL^2}^2 +\|W_{\gamma/2}W_{1,\f12+\delta_1}h\|_{H_x^{\f32+\delta_1}L^2}^2
+\|W_{\gamma/2}W_{\gamma+4}h\|_{H_x^{\f32+2\delta_1}L^2}^2
 \\&&+  \| W_{m,n}W_{\gamma/2+2s} h\|^2_{H_x^{ m+n\varrho-\delta_1}H^s} +\| W_{1,\f12+\delta_1}W_{\gamma/2+2s} h\|^2_{H_x^{\f32}H^s}+\|W_{\gamma+4}W_{\gamma/2+2s}h\|_{H^{\f32+\delta_1}_xH^s}^2     \bigg)d\tau.\eeno
 With the help of  Proposition \ref{intx1} and Proposition \ref{intx2}, we   conclude that
 \beno && \sum_{0<n\le N_{\rho,1}}\|W_{0,n}  h(t)\|^2_{H^{n\rho}_xL^2}+\sum_{0<n\le N_{\rho,2}}\|W_{1,n}  h(t)\|^2_{H^{1+n\rho}_xL^2}+\|W_{1,\f12+\delta_1}  h(t)\|^2_{H^{\f32+\delta_1}_xL^2}+\|W_{\gamma+4} h(t)\|^2_{H^{\f32+2\delta_1}_xL^2}\\&&
 \qquad+c_oA_5(c_1, c_2)\int_0^t\bigg(\sum_{0<n\le N_{\rho,1}}\|W_{\gamma/2}W_{0,n}  h \|^2_{H^{n\rho}_xL^2}+\sum_{0<n\le N_{\rho,2}}\|W_{\gamma/2}W_{1,n}  h \|^2_{H^{1+n\rho}_xL^2}\\&&\quad+\|W_{\gamma/2}W_{1,\f12+\delta_1}  h \|^2_{H^{\f32+\delta_1}_xL^2}+\|W_{\gamma/2}W_{\gamma+4} h \|^2_{H^{\f32+2\delta_1}_xL^2} \bigg)d\tau\\ &&\le
  \sum_{0<n\le N_{\rho,1}}\|W_{0,n}  h_0\|^2_{H^{n\rho}_xL^2}+\sum_{0<n\le N_{\rho,2}}\|W_{1,n}  h_0\|^2_{H^{1+n\rho}_xL^2}+\|W_{1,\f12+\delta_1}  h_0\|^2_{H^{\f32+\delta_1}_xL^2}+\|W_{\gamma+4} h_0\|^2_{H^{\f32+2\delta_1}_xL^2}
  \\&&\quad+C(c_2, \mathbb{C},\eta^{-1})\bigg(\int_0^t \|h\|^2_{H^{\f32+2\delta_1}_xL^2}d\tau+\int_0^t \|h\|_{L^2_xH^s}^2d\tau\bigg)+\|W_{\gamma/2}W_{\gamma+4}h\|_{H_x^{\f32+2\delta_1}L^2}^2\\&&\quad+\eta
  \bigg(\sum_{m=0,0<n\le N_{\rho,1}}+\sum_{m=1,0<n\le N_{\rho,2}}\bigg)\int_0^t \big(\|W_{m,n}W_{\gamma/2+d_2}h\|^2_{H^{m+n\rho}_xL^2}+\|W_{1,\f12+\delta_1}W_{\gamma/2+d_2}h\|^2_{H^{\f32+\delta_1}_xL^2}\big)d\tau.
 \eeno

  {\it Step3: Gain of regularity for $x$ variable.} Thanks to Proposition \ref{smxmn1}, we get
\beno  &&\int_0^t \|W_{m,n+1}W_{\gamma/2+d_2} h\|_{H^{m+(n+1)\varrho}_xL^2}^2 d\tau \le C(\eta^{-1})\|W_{m,n+1}W_{\gamma/2+d_1+d_2}   h_0\|^2_{H^{m+n\varrho}_xL^2}\\&&+ \eta^{2s}\int_0^t  \| W_{m,n+1}W_{\gamma/2+d_1+d_2}  h\|_{H^{m+n\varrho}_xH^s}^2 d\tau+C(\eta^{-1})\int_0^t  \bigg(  \|W_{m,n+1}W_{\f32\gamma+2s+d_1+d_2}h\|^2_{H_x^{m+n\varrho}L^2}  \\&&+
\|W_{m,n+1}W_{\f32\gamma+2s+d_1+d_2} h\|^2_{H_x^{m+n\rho-\delta_1}H^s}+ \|h\|^2_{H_x^{\f32+\delta_1}L^2_{\gamma+4}} \bigg)d\tau.
\eeno
Due to (W-4)  of Definition \ref{ws1}:
$W_{m,n+1}W_{\f32\gamma+2s+d_1+d_2}W_{d_3}\le W_{m,n}$,  the interpolation inequality and Proposition \ref{intx2}, we derive that
 \beno  &&\int_0^t \|W_{m,n+1}W_{\gamma/2+d_2} h\|_{H^{m+(n+1)\varrho}_xL^2}^2 d\tau \le C(\eta^{-1})\big[\|W_{m,n+1}W_{\gamma/2+d_1+d_2}   h_0\|^2_{H^{m+n\varrho}_xL^2}+C(\eta^{-1})\int_0^t   (\|h\|_{L^2_xH^s}^2\\&&+\|h\|^2_{H_x^{\f32+2\delta_1}L^2_{\gamma+4}})d\tau\big]+\eta\int_0^t  \|W_{m,n} h\|^2_{H^{m+n\rho}_xL^2}d\tau    +\eta^{2s}\sum_{i=0}^1\int_0^t  \| W_{m,n-i}W_{\gamma/2} h\|_{H^{m+(n-i)\varrho}_xH^s}^2 d\tau.
 \eeno

 Combining the estimates in {\it Step 2} and {\it Step 3}, we are led to
  \beno && \sum_{0<n\le N_{\rho,1}}\|W_{0,n}  h(t)\|^2_{H^{n\rho}_xL^2}+\sum_{0<n\le N_{\rho,2}}\|W_{1,n}  h(t)\|^2_{H^{1+n\rho}_xL^2}+\|W_{1,\f12+\delta_1}  h(t)\|^2_{H^{\f32+\delta_1}_xL^2}+\|W_{\gamma+4} h(t)\|^2_{H^{\f32+2\delta_1}_xL^2}\\&&
 +\f12c_oA_5(c_1, c_2)\int_0^t\bigg(\sum_{0<n\le N_{\rho,1}}\big(\|W_{\gamma/2}W_{0,n}  h \|^2_{H^{n\rho}_xL^2}+\|W_{0,n+1}W_{\gamma/2+d_2} h\|_{H^{(n+1)\varrho}_xL^2}^2\big)+\|W_{\gamma/2}W_{\gamma+4} h \|^2_{H^{\f32+2\delta_1}_xL^2} \\&&+\sum_{0<n\le N_{\rho,2}}\big(\|W_{\gamma/2}W_{1,n}  h \|^2_{H^{1+n\rho}_xL^2}+\|W_{1,n+1}W_{\gamma/2+d_2} h\|_{H^{1+(n+1)\varrho}_xL^2}^2\big) +\|W_{\gamma/2}W_{1,\f12+\delta_1}  h \|^2_{H^{\f32+\delta_1}_xL^2}\bigg)d\tau\\&&\le C\bigg(
 \sum_{0<n\le N_{\rho,1}}\|W_{0,n}  h_0\|^2_{H^{n\rho}_xL^2}+\sum_{0<n\le N_{\rho,2}}\|W_{1,n}  h_0\|^2_{H^{1+n\rho}_xL^2}\bigg)+\|W_{1,\f12+\delta_1}  h_0\|^2_{H^{\f32+\delta_1}_xL^2}+\|W_{\gamma+4} h_0\|^2_{H^{\f32+2\delta_1}_xL^2}
 \\&&\quad+C(\mathbb{C},c_2) \bigg(  \int_0^t\|W_{\gamma/2}W_{\gamma+4}h\|_{H_x^{\f32+2\delta_1}L^2}^2d\tau+\int_0^t \|h\|_{L^2_xH^s}^2d\tau\bigg) +\eta
  \int_0^t  \|W_{0,1}W_{\gamma/2+d_2}h\|^2_{H^{\rho}_xL^2} d\tau.
 \eeno
 From this together with Lemma \ref{refL12} and the estimate in {\it Step 1}, we finally conclude that there exists constants $b_i$ such that
 \beno &&X(t)+b_1\int_0^t X(\tau)d\tau+\int_0^t (\|W_{1,N_{\rho,2}+1}W_{\gamma/2+d_2} h\|_{H^{1+(N_{\rho,2}+1)\varrho}_xL^2}^2+\|W_{l_2+\gamma/2}h\|_{L^2_xH^s}^2)d\tau\\
 &&\le b_2X(0)+C( \mathbb{C},c_1,c_2)  \int_0^t\|W_{\gamma/2}W_{\gamma+4}h\|_{H_x^{\f32+2\delta_1}L^2}^2d\tau+C( \mathbb{C},c_1,c_2)\int_0^t     \|h\|^2_{L^1} d\tau.\eeno
 Thanks to   (P-2) of Definition \ref{ws1}, we have $1+(N_{\rho,2}+1)\rho>\f32+2\delta_1$. By interpolation inequalities that
$\|W_{\gamma/2}W_{\gamma+4}h\|_{H_x^{\f32+2\delta_1}L^2}^2\le \eta \|W_{\gamma/2}W_{\gamma+4}h\|_{H_x^{1+(N_{\rho,2}+1)\varrho}L^2}^2+C_\eta  \|W_{\gamma/2}W_{\gamma+4}h\|^2_{L^2},$
$\|W_{\gamma/2}W_{\gamma+4}h\|^2_{L^2}\le C_\eta \|h\|_{L^2}^2+\eta  \|W_{3\gamma+8}h\|^2_{L^2},$
$\|h\|_{L^2}^2\le \eta \|h\|_{L^2_xH^s}^2+C_\eta \int_{\TT^3}|h|^2_{L^1}dx$, and
$\int_{\TT^3}|h|^2_{L^1}dx\le  C_\eta \|h\|_{L^1}^2+\eta \|h\|^2_{L^\infty_xL^2_2},
$
we derive that
$ \|W_{\gamma/2}W_{\gamma+4}h\|_{H_x^{\f32+2\delta_1}L^2}^2\le \eta(X+\|W_{1,N_{\rho,2}+1}W_{\gamma/2+d_2} h\|_{H^{1+(N_{\rho,2}+1)\varrho}_xL^2}^2+\|W_{l_2+\gamma/2}h\|_{L^2_xH^s}^2)+C_\eta\|h\|_{L^1}^2, $
 which implies that
 \beno  &&X(t)+\f{b_1}2\int_0^t \big(X(\tau)+  \|W_{1,N_{\rho,2}+1}W_{\gamma/2+d_2} h\|_{H^{1+(N_{\rho,2}+1)\varrho}_xL^2}^2\big)d\tau\le b_2X(0)+C( \mathbb{C},c_1,c_2)   \int_0^t     \|h\|^2_{L^1} d\tau. \eeno
 By slightly modification,  for any $t_2>t_1\ge0$, we  can get  that
\beno  &&X(t_2)+\f{b_1}2\int_{t_1}^{t_2} \big(X(\tau)+  \|W_{1,N_{\rho,2}+1}W_{\gamma/2+d_2} h\|_{H^{1+(N_{\rho,2}+1)\varrho}_xL^2}^2\big)d\tau\le b_2X(t_1)+C( \mathbb{C},c_1,c_2)   \int_{t_1}^{t_2}     \|h\|^2_{L^1} d\tau. \eeno
If all the estimates are performed for $e^{ct}X(t)$ with $c<b_1/8$, then we have \beno  &&e^{\f{b_1}4t}X(t)+\f{b_1}4\int_0^t e^{\f{b_1}4\tau}X(\tau)d\tau\le b_2X(0)+C( \mathbb{C},c_1,c_2)   \int_0^t    e^{\f{b_1}4\tau} \|h\|^2_{L^1} d\tau. \eeno
 As a corollary, we have
 $ X(t)\le b_2e^{-\f{b_1}4t}X(0)+C( \mathbb{C},c_1,c_2)  \int_0^t    e^{-\f{b_1}4(t-\tau)} \|h\|^2_{L^1} d\tau.$
 We complete the proof of the lemma.
 \end{proof}

 \subsubsection{Propagation of the  regularity}
 We will use inductive method to prove the global-in-time estimates for the solution $h$ to \eqref{h-equ}.
 We will show
 \begin{lem}\label{Regu} Suppose that  \eqref{dynacondi} and \eqref{boufg} hold for f and g respectively and $h$ is a unique and smooth solution to \eqref{h-equ}.
 
  {\bf (1)($\gamma>0$).} For $t\ge t_0>0$ and $m,q,l\in\R^+$,
   $\|h(t)\|_{H^m_xH^q_l}\le C(t_0,m,q,l,c_1,c_2).$
 
 {\bf (2)($\gamma=0$).}   If $\f522^{N_{q,s,1}} \le l_2, q\le N+\kappa$ and $\mathbb{P}_e^{N,\kappa}(h_0)+V^{q}(h_0)<\infty$, then for $t\ge0$,
 \beno \mathbb{P}_e^{N,\kappa}(h(t))+V^{q}(h(t))+\int_t^{t+1}( \mathbb{D}^{N,\kappa}_e(h)+\mathbb{D}^{N,\kappa}_3(h)+ V^{q+s}(W_{\gamma/2}h))d\tau\le C(c_1,c_2,  \mathbb{P}_e^{N,\kappa}(h_0),V^{q}(h_0)).\eeno
Moreover, if $t\ge (N-1)(N_{\varrho,1}+1)+N_{\varrho,\kappa}+[q/s]+3$, it holds
\beno \mathbb{P}_e^{N,\kappa}(h(t))+V^{q}(h(t))+\int_t^{t+1}( \mathbb{D}^{N,\kappa}_e(h)+\mathbb{D}^{N,\kappa}_3(h)+ V^{q+s}(W_{\gamma/2}h))d\tau\le C(c_1,c_2,  \mathbb{P}_e^{1,\f12+2\delta_1}(h_0)). \eeno
\end{lem}
 \begin{proof}
   We will divide the proof into two steps.

 {\it Step 1: Propagation of the regularity for $x$ variable.}
 The result can be concluded as:
 \begin{prop}\label{Regux} \begin{enumerate}
\item {\bf ($\gamma>0$)} For any $t\ge t_0>0$ and $m,n\in\N$, one has
 \beno
 E^{m,n}(h(t))+\int_t^{t+1}( D^{m,n}_2(h)+D^{m,n}_3(h))d\tau\le C(t_0).
 \eeno
 \item {\bf ($\gamma=0$) }If $\mathbb{P}_e^{N,\kappa}(h(t_0))<\infty$, one has for $t\ge t_0\ge0$,
 \beno \mathbb{P}_e^{N,\kappa}(h(t))+\int_t^{t+1}( \mathbb{D}^{N,\kappa}_e(h)+\mathbb{D}^{N,\kappa}_3(h))d\tau\le C(c_1,c_2, \mathbb{P}_e^{N,\kappa}(h(t_0))).\eeno
 And if $t\ge t_0+(N-1)(N_{\rho,1}+1)+N_{\rho,\kappa}+1$,
 \beno \mathbb{P}_e^{N,\kappa}(h(t))+\int_t^{t+1}( \mathbb{D}^{N,\kappa}_e(h)+\mathbb{D}^{N,\kappa}_3(h))d\tau\le   C(c_1,c_2, \mathbb{P}_e^{1,\f12+2\delta_1}(h_0)).\eeno
\end{enumerate}
   \end{prop}

 \begin{proof}  We first claim that
 if $i\le m, j\le n$, the solution $h$ verifies that for any $t\ge t_0\ge0$,
 \beno
 E^{i,j}(h(t))+\int_t^{t+1}( D^{i,j}_2(h(\tau))+D^{i,j}_3(h(\tau)))d\tau\lesssim  1,
 \eeno
 then one has that for $t\ge t_0$,
 \ben\label{indconti1}
 &&\qquad E^{m,n+1}(h(t))+\int_t^{t+1} (D^{m,n+1}_2(h(\tau))+D^{m,n+1}_3(h(\tau)))d\tau \lesssim E^{m,n+1}(h(t_0))+1,\\
\label{indsmest1}&&\qquad
 E^{m,n+1}(h(t))+\int_t^{t+1} (D^{m,n+1}_2(h(\tau))+D^{m,n+1}_3(h(\tau)))d\tau \lesssim \mathrm{1}_{t\in(t_0,t_0+2)} (t-t_0)^{-1}+\mathrm{1}_{t\in[t_0+2,\infty)}.\een
 In particular, if $t\ge t_0+1$,
 $
 E^{m,n+1}(h(t))+\int_t^{t+1} (D^{m,n+1}_2(h)+D^{m,n+1}_3(h))d\tau \lesssim 1.$

Thanks to Proposition \ref{Enxmn1} and Proposition \ref{Enxmn2}, we derive that for $0\le t_2-t_1\le 2$,
\begin{enumerate}
\item if $m+n\rho\le \f32+\delta_1$,
\beno
&&E^{m,n+1}(h(t_2))+c_0A_5(c_1,c_2)\int_{t_1}^{t_2} D^{m,n+1}_2(h)d\tau
 \le E^{m,n+1}(h(t_1))+C(c_1,c_2)\int_{t_1}^{t_2} (D^{m,n}_2(h)\\&&\qquad\qquad+D^{m,n-1}_2(h)+D^{m,n}_3(h)+D^{0,0}_2(h))d\tau+C(c_2,g) \le E^{m,n+1}(h(t_1))+C(c_1,c_2,g),
\eeno
  \item if $m+n\rho\ge \f32+2\delta_1$,
 \beno
&&E^{m,n+1}(h(t_2))+c_0A_5(c_1,c_2)\int_{t_1}^{t_2} D^{m,n+1}_2(h)d\tau
\le E^{m,n+1}(h(t_1))+C(c_1,c_2)\int_{t_1}^{t_2} (D^{m,n}_2(h)\\&&+D^{m,n-1}_2(h)+D^{m,n}_3(h)+D^{0,0}_2(h)+D^{1,\f12+\delta_1}_2(h))d\tau+C(c_2,g) \le E^{m,n+1}(h(t_1))+C(c_1,c_2,g).
\eeno
  \end{enumerate}
These two estimates yield that for $0\le t_2-t_1\le 2$,
 \ben\label{emn1}
 &&E^{m,n+1}(h(t_2))+c_0A_5(c_1,c_2)\int_{t_1}^{t_2} D^{m,n+1}_2(h)d\tau
 \le E^{m,n+1}(h(t_1))+C(c_1,c_2,g),
 \een
 which gives the proof of \eqref{indconti1} with $t\in[t_0,t_0+2]$. 
 
 (i). For $t_2\ge t_0+2$, integrating \eqref{emn1} with respect to $t_1$ over $[t_2-2,t_2-1]$,  we have  
 \beno E^{m,n+1}(h(t_2))\le \int_{t_2-2}^{t_2-1}E^{m,n+1}(h(t_1))dt_1+C(c_1,c_2,g) \le \int_{t_2-2}^{t_2-1}D_3^{m,n}(h(t_1))dt_1+C(c_1,c_2,g)\lesssim1_{t_2\ge t_0+2}, \eeno 
 which implies $E^{m,n+1}(h(t))\lesssim 1+ E^{m,n+1}(h(t_0))\mathrm{1}_{t\in[t_0,t_0+2]}+\mathrm{1}_{t\ge t_0+2}$. 
 
 (ii). For $t_2\in (t_0, t_0+2]$, one has
 \beno \f{t_2-t_0}2E^{m,n+1}(h(t_2))\le \int_{t_0}^{\f{t_2+t_0}2}E^{m,n+1}(h(\tau))d\tau+C(c_1,c_2,g)\le\int_{t_0}^{\f{t_2+t_0}2}D_3^{m,n}(h(\tau))d\tau+C(c_1,c_2,g). \eeno 
It  implies that  
 $ E^{m,n+1}(h(t_2))\lesssim \mathrm{1}_{t\in(t_0,t_0+2)} (t-t_0)^{-1}+\mathrm{1}_{t\ge t_0+2}$. 
 
 (iii). Go back to \eqref{emn1}, then for $t_1\ge t_0$ we get  
 \beno \int_{t_1}^{t_1+1} D^{m,n+1}_2(h)d\tau\lesssim  E^{m,n+1}(h(t_1))+C(c_1,c_2, g)\le C(t_0).  \eeno
 Thanks to Proposition \ref{smxmn1} and Proposition \ref{smxmn2},   the estimate of $\int_{t_1}^{t_2} D^{m,n+1}_3(h)d\tau$ follows the results in (i), (ii) and (iii).
 We complete the proof of the claim.

 Now we are in a position to complete the proof for the desired results.

 (a). For $\gamma>0$,  thanks to Lemma \ref{refL12}, we obtain that the condition of the claim is verified for $(m,n)=(0,0)$. Then the first result follows the inductive method.

 (b). For $\gamma=0$, we first recall that from the proof of Lemma \ref{refEn},
 \beno  \mathbb{P}_e^{1,\f12+2\delta_1}(h(t))+\int_{t}^{t+1} \big(\mathbb{D}_e^{1,\f12+2\delta_1}(h)+\mathbb{D}_3^{1,\f12+2\delta_1}(h)\big)d\tau\lesssim
\mathbb{P}_e^{1,\f12+2\delta_1}(h(t_0))+1. \eeno
Therefore the condition of the claim is verified for $m\le1, n\le  N_{\varrho,2}$. Then the result follows the  claim and the inductive method.  Observe that by \eqref{indsmest1},   $E^{m,n+1}(h(t))\lesssim1$ if $t\ge t_0+1$. Thus if $t\ge t_0+(N-1)(N_{\varrho,1}+1)+N_{\varrho,\kappa}+1$,
$ \mathbb{P}_e^{N,\kappa}(h(t))\lesssim 1$.
  \end{proof}

 {\it Step 2: Propagation of the regularity for $v$ variable.} We want to prove
\begin{prop}\label{Reguv} \begin{enumerate}
\item {\bf ($\gamma>0$)} For any $t\ge t_0>0$ and $m,n\in\N$, one has
 \beno
 V^{ns}(h(t))+\int_t^{t+1} V^{(n+1)s}(W_{\gamma/2}h)d\tau\le C(t_0).
 \eeno
 \item {\bf ($\gamma=0$)} If $\f522^{N_{q,s,1}}\le l_2, q\le N+\kappa$ and $\mathbb{P}_e^{N,\kappa}(h(t_0))+V^{q}(h(t_0))<\infty$, one has for $t\ge t_0\ge0$,
 \beno V^{q}(h(t))+\int_t^{t+1} V^{q+s}(W_{\gamma/2}h)d\tau\le C(c_1,c_2, V^{q}(h(t_0)),\mathbb{P}_e^{N,\kappa}(h(t_0))).\eeno
 In particular,  if $t\ge t_0+[q/s]+2$,
 $ V^{q}(h(t))+\int_t^{t+1} V^{q+s}(W_{\gamma/2}h)d\tau\le C(c_1,c_2,  \mathbb{P}_e^{N,\kappa}(h(t_0))).$
\end{enumerate}
   \end{prop}
\begin{proof}  Recall the definition of $N_{q,s,1}$ in Definition \ref{ws1}. Then we first claim that
if  $2^{N_{(n+1)s,s,1}}\le l_2, (n+1)s\le N+\kappa $  and $t\ge t_0\ge 0$,  \beno
 V^{ns}(h(t))+\int_t^{t+1} \big(V^{(n+1)s}(W_{\gamma/2}h)+V^0(W_{l_2+\gamma/2}h)+D^{N,\kappa}_2(h)\big)d\tau\lesssim 1,
 \eeno
 then  for $t>t_0$, \beno
 V^{(n+1)s}(h(t))+\int_t^{t+1} V^{(n+2)s}(W_{\gamma/2}h)d\tau\lesssim  \mathrm{1}_{t\in(t_0,t_0+2)} (t-t_0)^{-1}+\mathrm{1}_{t\in(t_0+2,\infty)}.\eeno
 In particular,  for $t\ge t_0+1$, it holds
$ V^{(n+1)s}(h(t))+\int_t^{t+1} V^{(n+2)s}(W_{\gamma/2}h)d\tau\lesssim1. $

 Thanks to Proposition \ref{Envq}, the interpolation inequalities used in {\it Step 1.3}\, in the proof of Theorem \ref{En-Priori} as well as  the conditions $\f52
 2^{N_{n(s+1),s,1}}\le l_2, n(s+1)\le N+\kappa$, we obtain that for $0\le t_2-t_1\le 2$,
 \beno
 &&V^{(n+1)s}(h(t_2))+\mathcal{C}_1(c_1,c_2)\int_{t_1}^{t_2}  V^{(n+2)s}(W_{\gamma/2}h) d\tau\le V^{(n+1)s}(h(t_1))+\int_{t_1}^{t_2}\big( V^0(W_{l_2+\gamma/2}h)+D^{N,\kappa}_2(h)\big)d\tau.
  \eeno
It implies that  for $t_2\ge t_0+2$,
 $  V^{(n+1)s}(h(t_2))\lesssim \int_{t_2-2}^{t_2-1}  V^{(n+1)s}(h(t_1)) d\tau+1\lesssim 1.$
 While for $t_2\in(t_0,t_0+2]$,
 \beno  \f12(t_2-t_0)V^{(n+1)s}(h(t_2))\lesssim \int_{t_0}^{(t_2+t_0)/2} V^{(n+1)s}(h(t_1)) d\tau+1\lesssim 1.\eeno
 We conclude the claim by combining the above estimates.

 Now we are in a position to complete the proof of the proposition.

 (a). For $\gamma>0$, thanks to Lemma \ref{refL12}, the condition of the claim  is always verified with $n=0$. Thus the result is obtained by the inductive method.

 (b). For $\gamma=0$, thanks to Proposition \ref{Regux}, the condition of the claim is verified with $n=0$ thanks to   $\mathbb{P}_e^{N,\kappa}(h(t_0))<\infty$. Moreover conditions $\f522^{N_{q,s,1}}\le l_2$ and $q\le N+\kappa$ yield that the claim can be applied for any $n\in\N$ verifying $(n+1)s\le q$. Thus to conclude the desired result, we only need to copy the proof of the claim via replacing the energy  $V^{(n+1)s}(h)$ by the energy $V^{q}(h)$.
\end{proof}

Now we give the proof of Lemma \ref{Regu}.  The results   follow   Proposition \ref{Regux} and Proposition \ref{Reguv}.
 \end{proof}

 \subsection{Proof of Theorem \ref{thmdyna}}
 Now we are in a position to prove Theorem \ref{thmdyna}.
\begin{proof}[Proof of Theorem  \ref{thmdyna}] Without lose of generality, we assume that $M_f=M_{1,0,1}\eqdefa M$ and $h=f-M$. We first consider $\gamma=0$. Since $h_0\in \mathbb{P}_{e,1}^{1,\f12+2\delta_1}$, by Lemma \ref{Regu} and \eqref{evolcondi5}, we derive that $\mathbb{P}_{e,1}^{N,\kappa}(h(t))$ and $V^{q}(h(t)) $ with $N+\kappa\ge q> 6$ are bounded globally in time (for  $t\ge t_0\eqdefa (N-1)(N_{\varrho,1}+1)+N_{\varrho,\kappa}+[q/s]+3$). Observe that the condition $\sup_{t,x}|f|_{L\log L}$ in the main theorem of \cite{Mouhot} can be replaced by the condition $\sup_{t}\|f\|_{H_x^{\f32+2\delta_1}L^2_{4}}$. Then due to \cite{Mouhot}, we have the pointwise lower bound: $f\ge K_0 \exp\{-A_0 |v|^{q_0}\}$ where $K_0, A_0, q_0$ are constants depending on $c_1,c_2,  \mathbb{P}_{e,1}^{N,\kappa}(h(t_0)),V^{q}(h(t_0))$. Then for $t\ge t_0$, we have
	\ben\label{crentp} H(f|M)(t)&\lesssim& \int_{\TT^3\times\R^3} |h|\max\{(1+\|h\|_{L^\infty_{x,v}}), C(K_0,A_0) \} \lr{v}^{\max\{2,	q_0\}} dvdx \\
	&\lesssim&  C(\mathbb{P}_{e,1}^{N,\kappa}(h(t)),V^{q}(h(t)))\|(f-M)(t)\|_{L^1_{q_0+2}},\notag \een
	thanks to the mean-value theorem and the lower bound of $f$.

  Due to the condition \eqref{evolcondi}, we have
$|h|_{H^1_{10}}\lesssim |hW_{0,0}^{(2)}|_{L^2}+|h|_{H^{q_1}}, $
which  implies that
$  \mathbb{X}_2^{1,\f12+2\delta_1,q_1}(h)=\mathbb{P}_{e,2}^{1,\f12+2\delta_1}(h)\cap V^{q_1}(h)\gtrsim \|h\|_{L^2_xH^1_{10}}^2.$
Now we want to  patch together all the results in Lemma \ref{etrdiss}, Theorem \ref{Disp2} and Lemma \ref{refEn}. By using the condition \eqref{evolcondi1} and the notation $X(t)=\mathbb{X}_2^{1,\f12+2\delta_1,q_1}(h(t))$, we arrive at that  there exist constants $b_i(i=1,\dots,6)$ such that for any $\bar{\delta},\bar{\eta}\ll1$,
\ben\label{evol1}
 &&H(f|M)(t_2)+b_1 \bar{\delta}^{\gamma+\f{3(2-\gamma)}{2(2m+4-\gamma)}}\int_{t_1}^{t_2}H(f|M^f_{\rho,u,T})d\tau \nonumber\\&&\le H(f|M)(t_1)+b_2\bar{\delta}^{\gamma+\f32} \int_{t_1}^{t_2}\big(H(f|M)+H(f|M)^a+X\big)d\tau;\\
 && \f{d}{dt} M_h(t)+\|T-1\|_{L^2}^2+\|u\|_{L^2}^2+\|\rho-1\|_{L^2}^2\le  \bar{\eta}^{-3}H(f|M^f_{\rho,u,T})+\bar{\eta} b_3 X;\label{evol2}\\
&&X(t_2)+b_4\int_{t_1}^{t_2}Xd\tau\le b_5 X(t_1)+b_6\int_{t_1}^{t_2}H(f|M)d\tau.\label{evol3}
\een
Suppose  $A=\f{3(2-\gamma)}{2(2m+4-\gamma)}$ and $X(t)\le C$ for all $t\ge0$. Let  $\bar{\eta}=\bar{\delta}^{\f{\eta}4(\f32-A)}$ with $\eta\in(0,1)$. Then from the fact $H(f|M)\sim H(f|M^{f}_{\rho,u,T})+\|T-1\|_{L^2}^2+\|u\|_{L^2}^2+\|\rho-1\|_{L^2}^2 $ if $\rho\sim1$ and $T\sim1$, we derive that there exist constants $b_7$ and $b_8$ such that
\beno  Y(t)&\eqdefa&  H(f|M)(t) +\bar{\delta}^{\gamma+A+\eta(\f32-A)}\big(b_7M_h(t)+b_8X(t)\big) 
 \sim H(f|M)(t)+\bar{\delta}^{\gamma+A+\eta(\f32-A)}X(t), \eeno
and for any $t_1<t_2$,
\ben\label{decme}
&&Y(t_2)+ \f12\bar{\delta}^{\gamma+A+\eta(\f32-A)}\int_{t_1}^{t_2}(b_1H(f| M)+b_4b_7X)d\tau 
\le b_5Y(t_1) +b_2\bar{\delta}^{\gamma+\f32} \int_{t_1}^{t_2}H(f|M)^a d\tau.
 \een
 
 Next we want to prove that $  \lim\limits_{t\rightarrow\infty} H(f|M)(t)=0$ and  $H(f|M)$ is a strictly decreasing function until it vanishes.
 
 (i). Suppose that $\lim\limits_{t\rightarrow\infty}H(f|M)(t)=c\neq0$. Then by  \eqref{decme},
we have
\beno
&&Y(t)+\f12b_1 \bar{\delta}^{\gamma+A+\eta(\f32-A)}ct
\le b_5Y(0)+b_2\bar{\delta}^{\gamma+\f32} tH(f|M)^a(0).
 \eeno
By choosing $\bar{\delta}$ sufficiently small, we will get that for all $t>0$,
$ \f13b_1 \bar{\delta}^{\gamma+A+\eta(\f32-A)}ct
\le b_5Y(0),$
which contradicts the fact that the lefthand side will tend to infinity when $t$ goes to infinity. We conclude the desired result.

(ii). Note that the relative entropy $H(f|M)(t)$ is a decreasing function. Thus if we assume that $H(f|M)(t_1)=H(f|M)(t_2)$, then thanks to \eqref{evol1}, we get that
\beno   b_1 \bar{\delta}^{\gamma+\f{3(2-\gamma)}{2(2m+4-\gamma)}}\int_{t_1}^{t_2}H(f|M^f_{\rho,u.T})d\tau\le  b_2\bar{\delta}^{\gamma+\f32} \int_{t_1}^{t_2}\big(H(f|M)(0)+H(f|M)^a(0)+C\big)d\tau. \eeno
 Since $\bar{\delta}$ is arbitrary small, we obtain that $\int_{t_1}^{t_2}H(f| M^f_{\rho,u,T})d\tau=0$, which implies that $f=M^{f}_{\rho,u,T}$ for $t\in(t_1,t_2)$. Thanks to the result $(i)$ in Lemma  \ref{Disp1}, we first get that $T=\lr{T}_x$ for  $t\in(t_1,t_2)$. It implies that $f=M^{f}_{\rho,u,\lr{T}_x}$.  Due to the result $(ii)$ in Lemma  \ref{Disp1}, we obtain that $u=\lr{u}_x$ for $t\in(t_1,t_2)$ which yields that $f=M^{f}_{\rho,\lr{u}_x,\lr{T}_x}$. Now by result  $(iii)$ in Lemma  \ref{Disp1}, we   get that $\rho=1$
 for  $t\in(t_1, t_2)$. Due to \eqref{normalized}, it is not difficult to check that $u=\lr{u}_x=0$ and $T=\lr{T}_x=1$ for $t\in(t_1,t_2)$. Finally we derive that $f=M_{1,0,1}=M$ for  $t\in(t_1, t_2)$ which implies
  $H(f|M)(t_1)=H(f|M)(t_2)=0$. It implies that the relative entropy $H(f|M)$ is a strictly decreasing function  until it vanishes.

Going further, we want to get the decay rate of the relative entropy. Suppose that at time $t_j$  we have $H(f|M)(t_j)=2^{-j}$. Let $T_j=t_{j+1}-t_j, \eta<\f12$ and
$\bar{\delta}=2^{-jB}$ with $B=\f{2(1-a)}{\f32-A}$. Then it is easy to check that
$B\big(\gamma+A+\eta(\f32-A)\big)+1< (\gamma+\f32)B+a.$
From this together with \eqref{decme}, we obtain that
$ Y(t_{j+1})+ \f18b_12^{-j(B(\gamma+A+\eta(\f32-A))+1)} T_j\le b_5Y(t_j).$
Thanks to the condition \eqref{evolcondi3}, we infer that
\beno  T_j&\le& 8b_5b_1^{-1}2^{j(B(\gamma+A+\eta(\f32-A))+1)} \big(H(f|M)(t_j)+22^{-jB(\gamma+A+\eta(\f32-A))} X(t_j)\big)
\\&\le& 8b_5b_1^{-1}(2^{jB(\gamma+A+\eta(\f32-A))}+22^{j(1-\theta_1)} ).\eeno
We conclude that
$ t_N\le \sum_{j=0}^{N-1}T_j\lesssim 2^{NB(\gamma+A+\eta(\f32-A))}+2^{N(1-\theta_1)}\lesssim 2^{Nc^{-1}}.$
In other words,   $H(f|M)(t_N)=2^{-N}\lesssim  t_N^{-c}$, where $c=\min\{(1-\theta_1)^{-1}, \f{\f32-A}{2(1-a)}( A+\eta(\f32-A))^{-1}\}$. It implies
$
 H(f|M)(t)\lesssim t^{-c}.
$
 Thanks to Lemma \ref{refEn}, we obtain that
$
  \mathbb{P}_{e,1}^{1,\f12+2\delta_1}(h(t))\lesssim t^{-c}
$.
 It concludes the result for $\gamma=0$.

 For the case of $\gamma>0$, thanks to the smoothing estimates in Lemma \ref{Regu}, we infer that the parameter $a$ and $\theta_1$ can be chosen to be close to $1$ arbitrarily. Thus $c$ can be chosen as large as we want. It implies that the order of the decay rate will be  $O(t^{-\infty})$.

 For the case of $\gamma=2$, by modifying the inequalities (\ref{evol1}-\ref{evol3}) (in the sense that we replace $X(t), H(f|M)(t) $ and $M_h(t)$ by $e^{ct} X(t), e^{ct}H(f|M)(t)$ and $e^{ct}M_h(t)$ with $c$ sufficiently small), we can deduce that   for any $t>t_0>0$,
$ e^{ct}Y(t)+c\int_{t_0}^{t} e^{c\tau}Y(\tau) d\tau\le  Y(t_0),$
 which implies the desired result.
\end{proof}

\section{Global-in-time strong stability for the Boltzmann equation}
In this section, we will give the proof of Theorem \ref{thmstab}.

\begin{proof}[Proof of  Theorem \ref{thmstab}] We first give the proof in the case of $\gamma=0$. Observe that $f_0=g_0+h_0$ verifies that $\mathbb{E}^{2,\f12+\delta_1}(f_0)\le\mathbb{E}^{N,\kappa}(f_0)$ with $N+\kappa\ge q>6$ and  $\mathbb{W}_{I}^{(1)}(N,\kappa,\varrho,\delta_1,q^{(1)}_1,q^{(2)}_2)$ verifies \eqref{unicon1} and \eqref{Rstreps}. Then  by Theorem \ref{thmwepo}, the equation \eqref{Boltz eq} admits a unique and smooth solution in $C([0,\tilde{T}];\mathbb{E}^{2,\f12+\delta_1})$. Notice that conditions \eqref{unicon1} and \eqref{Rstreps} depend only on $c_1$ and $c_2$. Then by    the continuity argument, we can assume that there exists a time $T$ which is a maximum time verifying that 
	 	\ben\label{bcri} \mbox{for $t\in [0,T)$,}\quad \rho_f(t,x)> c_1,\quad \|f\|^2_{H^{\f32+2\delta_1}_xL^2_{\gamma+4}}< c_2. \een It is obvious that $T>0$.   
Let $h=f-g$. Notice that the conditions \eqref{evolcondi2} and \eqref{evolcondi4}  implies  \eqref{boufg}. Thus by Lemma \ref{refEn}, we arrive at that for $t\in [0,T]$,
$ \mathbb{P}_{e,2}^{1,\f12+2\delta_1}(h(t))\le  b_2\mathbb{P}_{e,2}^{1,\f12+2\delta_1}(h_0)+C\int_0^t \|h\|_{L^1}^2d\tau. $
From this together with the Gronwall inequality, we are led to that for $t\in[0,T]$,
\ben\label{conti} \mathbb{P}_{e,2}^{1,\f12+2\delta_1}(h(t))\le b_2\mathbb{P}_{e,2}^{1,\f12+2\delta_1}(h_0)e^{Ct}. \een
Recalling that $ \mathbb{P}_{e,2}^{1,\f12+2\delta_1}(h_0)\le  \mathbb{P}_{e,1}^{1,\f12+2\delta_1}(h_0)\le \eta$,  we are in a position to assume that
\beno T^*\eqdefa\sup\{t>0\,\,|\,  \mathbb{P}_{e,2}^{1,\f12+2\delta_1}(h(\tau))\le \eta^{\f12},    \forall \tau\in[0,t] \}. \eeno
Observe that for $t\in [0,T^*]$, the condition \eqref{bcri} always holds, which implies that $T\ge T^*$ due to the definition of $T$.  Thus we have the estimate \eqref{conti} for $t\in [0,T^*]$ which gives that $T^*\ge  C^{-1}(-\ln b_2+\f12|\ln \eta|)$. Let $t_0=(N-1)(N_{\varrho,1}+1)+N_{\varrho,\kappa}+[q/s]+3$. Then $t_0\le T^*$ if $\eta$ is sufficiently small. By Lemma \ref{Regu}, for $t\in [t_0 , T^*]$, we have
\beno \mathbb{P}_{e,1}^{N,\kappa}((f-M_f)(t))+V^q((f-M_f)(t))\lesssim C(c_1,c_2,\mathbb{P}_{e,1}^{1,\f12+2\delta_1}(f_0) ). \eeno
Then due to \cite{Mouhot}, we have the pointwise estimate as follows: for $t\in [ t_0, T^*]$,
$ f\ge K_0e^{-A_0|v|^{q_0}}, $
where $K_0, A_0$ and $q_0$ depend only on $c_1,c_2,\mathbb{P}_{e,1}^{1,\f12+2\delta_1}(g_0)+1$. Let $t_1=(2C)^{-1}(-\ln b_2+\f12|\ln\eta|
)$.
Thanks to \eqref{crentp}, for $t\in [t_1 , T^*]$, we have
$  H(f|M_f)(t)\lesssim \|f-M_f\|_{L^2_{q_0+4}}.
$
Notice that
\beno \mathbb{P}_{e,2}^{1,\f12+2\delta_1}(f-M_f)&\lesssim& \mathbb{P}_{e,2}^{1,\f12+2\delta_1}(h)+\mathbb{P}_{e,2}^{1,\f12+2\delta_1}(M_g-M_f)
+\mathbb{P}_{e,2}^{1,\f12+2\delta_1}(g-M_g).\eeno
From this together with the fact $\mathbb{P}_{e,2}^{1,\f12+2\delta_1}(M_g-M_f)\lesssim \eta $ and Theorem \ref{thmdyna}, we deduce that for $t\in [t_1, T^*]$, we have
$  \mathbb{P}_{e,2}^{1,\f12+2\delta_1}((f-M_f)(t))\lesssim  (1+|\ln \eta|)^{-c},
$
which in turn implies  that for $t\in [t_1, T^*]$,
$H(f|M_f)(t)\le  (1+|\ln \eta|)^{-c/2}.$

Now we choose $t_1$ as a new initial time.
Suppose that \beno T^*_1\eqdefa\sup \{t\ge t_1 \,|\, \mathbb{P}_{e,2}^{1,\f12+2\delta_1}(f-M_f)(\tau))\lesssim  (1+|\ln \eta|)^{-c/3},\forall \tau\in [t_1, t]  \}.\eeno
It is easy to check that \eqref{bcri} still holds for $t\in [t_1, T^*_1]$ and then $T\ge T_1^*\ge T^*$. Thanks to Lemma \ref{refEn}, we deduce that for $t\in [t_1, T_1^*]$, we have
\beno   \mathbb{P}_{e,2}^{1,\f12+2\delta_1}((f-M_f)(t))&\lesssim& \mathbb{P}_{e,2}^{1,\f12+2\delta_1}((f-M_f)(t_1))e^{-c(t-t_1)}+H(f|M_f)(t_1)\lesssim (1+|\ln \eta|)^{-c/2}.
 \eeno
 It means that $T^*_1=\infty$ which yields that \eqref{bcri} holds globally. In other words, the equation \eqref{Boltz eq} admits a unique and global smooth solution $f\in C([0,\infty);\mathbb{P}_{e,1}^{N,\kappa}\cap V^q)$ with the initial data $f_0=g_0+h_0$. We derive that for $t\ge 0$,
 \beno   \mathbb{P}_{e,2}^{1,\f12+2\delta_1}(h(t))&\lesssim&\mathbb{P}_{e,2}^{1,\f12+2\delta_1}((f-M_f)(t))+\mathbb{P}_{e,2}^{1,\f12+2\delta_1}((g-M_g)(t))+\mathbb{P}_{e,2}^{1,\f12+2\delta_1}((M_g-M_f)(t)) \\
 &\lesssim& (1+t)^{-c}+\eta.
 \eeno
 From this together with the definition of $T^*$ and the  fact $T^*\ge  C^{-1}(-\ln b_2+\f12|\ln \eta|)$, we conclude that
 for $t>0$, $\mathbb{P}_{e,2}^{1,\f12+2\delta_1}(h(t)) \lesssim  (1+|\ln \eta|)^{-c}.$ We complete the proof to the desired result for $\gamma=0$.

We may follow the similar argument to prove the result for $\gamma>0$. The only thing that we need to take care of is the existence of $T$ which is defined in the above.  We observe that if \eqref{bcri} holds for $t\in [0,\tilde{T}]$, then by Lemma \ref{refEn} and \eqref{evolcondi6}, we obtain that $\mathbb{X}^{1,\f12+2\delta_1,q_1}_2(f(t))\le  2\big(\bar{\mathcal{C}}+b_2(1+N)+C(\bar{\mathcal{C}},c_1,c_2)\big)$, which together with \eqref{unicon2} imply that \eqref{unicon1} holds for $f(t)$ with $t\in [0,\tilde{T}]$. Therefore the continuity argument can be applied which is enough to derive the existence of $T$ defined in the before. Thanks to the smoothing estimates for $g$, we may copy the similar argument to get the desired result.
\end{proof}

\section{Appendix}

\begin{lem}\label{baslem1}(see \cite{hmuy}) Let $s, r\in \R$ and $a(v), b(v)\in C^\infty$ satisfy for any $\alpha\in \ZZ^3_{+}$,
\beno  |D^\alpha_v a(v)|\le C_{1,\alpha} \langle v\rangle^{r-|\alpha|}, |D^\alpha_\xi b(\xi)|\le C_{2,\alpha} \langle \xi\rangle^{s-|\alpha|} \eeno for constants $C_{1,\alpha}, C_{2,\alpha}$. Then there exists a constant $C$ depending only on $s, r$ and finite numbers of  $C_{1,\alpha}, C_{2,\alpha}$ such that for any $f\in S(\R^3)$,
\beno  |a(v)b(D)f|_{L^2}\le C|\langle D\rangle^s\langle v\rangle^rf|_{L^2}, \quad
|b(D)a(v)f|_{L^2}\le C|\langle v\rangle^r\langle D\rangle^sf|_{L^2}.
\eeno As a direct consequence, we get that  $|\langle D\rangle^m \langle v\rangle^{l} f|_{L^2}\sim | \langle v\rangle^{l} \langle D\rangle^mf|_{L^2}\sim |f|_{H^m_l}.$
\end{lem}

\begin{lem}\label{baslem2}(see \cite{HE16}) Let $l, s, r\in \R$, $M(\xi)\in S^{r}_{1,0}$ and $\Phi(v)\in S^{l}_{1,0}$. Then there exists a constant $C$ such that
$ |[M(D_v), \Phi(v)]f |_{H^s}\le C|f|_{H^{r+s-1}_{l-1}}.$
Moreover, for any $N\in \N$,
\ben\label{psuedoexp} M(D_v)\Phi=\Phi M(D_v)+\sum_{1\le|\alpha|<N}\f1{\alpha !}\Phi_{\alpha}M^{\alpha}(D_v)+r_N(v,D_v),\een
where $\Phi_{\alpha}(v)=\pa_v^\alpha \Phi$, $M^\alpha(\xi)=\pa^\alpha_\xi M(\xi)$ and $ \lr{v}^{N-l}r_{N}(v, \xi)\in S^{r-N}_{1,0}$.
\end{lem}

Now we are in a position to  give the new profiles of the weighted Sobolev spaces $H^m_l(\R^3)$.

\begin{thm}\label{baslem3} Let $m,l\in \R$. Then for   $f\in H^m_l$, 
$ \sum_{k=-1}^\infty 2^{2kl}|\mathcal{P}_{k}f|_{H^m}^2\sim |f|_{H^m_l}^2\sim
\sum_{j=-1}^\infty 2^{2jm} |\mathfrak{F}_{j}f|_{L^2_l}^2 .$
\begin{proof} The first equivalence is proved in \cite{HE16}. The second one will be proven in a similar way. Thanks to the fact
$ |[W_l, 2^{mj}\mathfrak{F}_j]f|\lesssim 2^{-\f12 j} |f|_{H^{m-\f12}_{l-1}}, $
we infer that
$ \sum_{j=-1}^\infty 2^{2jm}|\mathfrak{F}_{j}f|_{L^2_l}^2+|f|_{H^{m-\f12}_{l-1}}^2
 \sim |f|^2_{H^m_l}. $

We first show that the result holds for $l\ge0$. In this case, we have $\sum_{j=-1}^\infty 2^{2jm}|\mathfrak{F}_{j}f|_{L^2_l}^2 \ge |f|^2_{H^m}. $ By induction, we deduce that for any $n\in \N$,
  $ \sum_{j=-1}^\infty 2^{2jm}|\mathfrak{F}_{j}f|_{L^2_l}^2+|f|_{H^{m-\f{n}2}_{l-n}}^2
 \sim |f|^2_{H^m_l}.$
 Choose $n$ large enough, then we get the equivalence.

 Next we turn to the case $l<0$. We only need to prove
$ \sum_{j=-1}^\infty 2^{2jm}|\mathfrak{F}_{j}f|_{L^2_l}^2\gtrsim |f|^2_{H^m_l}. $
 Notice that
 \beno  \big|\int_{\R^3} fg dv\big| 
\lesssim \big(\sum_{j=-1}^\infty 2^{2jm}|\mathfrak{F}_{j}f|_{L^2_l}^2\big)^{\f12}\big(\sum_{j=-1}^\infty 2^{-2jm}|\tilde{\mathfrak{F}}_{j}g|_{L^2_{-l}}^2\big)^{\f12}\lesssim \big(\sum_{j=-1}^\infty 2^{2jm}|\mathfrak{F}_{j}f|_{L^2_l}^2\big)^{\f12}|g|_{H^{-m}_{-l}}.\eeno
 Then it yields
$|\int_{\R^3} \lr{D}^mW_lf g dv|
 \lesssim    \big(\sum_{j=-1}^\infty 2^{2jm}|\mathfrak{F}_{j}f|_{L^2_l}^2\big)^{\f12}| g|_{L^2}$, which ends 
   the proof to the second equivalence.
\end{proof}	
\end{thm}

\begin{thm}\label{thmub}(see \cite{HE16})  Let $w_1,w_2\in \R, a,b\in [0,2s]$ with $w_1+w_2=\gamma+2s$ and  $a+b=2s$.  Then for  smooth functions $g, h$ and
	$f$, we have
	\begin{enumerate}
		\item if $\gamma+2s>0$,     \ben\label{u1} |\langle Q(g,h),f\rangle_v|\lesssim
		(|g|_{L^1_{\gamma+2s+(-w_1)^++(-w_2)^+}}+|g|_{L^2})|h|_{H^{a}_{w_1}}|f|_{H^{b }_{w_2}},
		\een
		\item if $\gamma+2s= 0$,     \ben\label{u2} |\langle Q(g,h),f\rangle_v|\lesssim
		(|g|_{L^1_{w_3}}+|g|_{L^2})|h|_{H^{a}_{w_1}}|f|_{H^{b }_{w_2}},
		\een where $w_3=\max\{\delta,  (-w_1)^++(-w_2)^+\}$ with $\delta>0$ which is sufficiently small,
		\item if $-1<\gamma+2s<0$,     \ben\label{u3} |\langle Q(g,h),f\rangle_v|\lesssim
		(|g|_{L^1_{w_4}}+|g|_{L^2_{-(\gamma+2s)}})|h|_{H^{a}_{w_1}}|f|_{H^{b }_{w_2}},
		\een where  $w_4=\max\{-(\gamma+2s), \gamma+2s+(-w_1)^++(-w_2)^+\}$.
	\end{enumerate}
\end{thm}

\begin{lem}\label{lemub1}(see \cite{HE16}) Recall that $\Phi_k^\gamma(v)$ is defined by \eqref{DefPhi}.  Suppose $N\in \N$ and  $ \gamma+2s>-1$. Let
	\beno   \mathfrak{W}_{k,p,l}^1&\eqdefa&\iint_{\sigma\in \SS^2,v_*,v\in \R^3} \big(\tilde{ \mathfrak{F}}_{p}\Phi_k^\gamma\big)(|v-v_*|)b(\cos\theta) (\mathfrak{F}_{p}g)_*(\mathfrak{F}_{l}h)\big[
	(\tilde{ \mathfrak{F}}_{p}f)'-
	\tilde{ \mathfrak{F}}_{p}f\big]d\sigma dv_* dv,\\
 	\mathfrak{W}_{k,p,l,m}^4&\eqdefa&\iint_{\sigma\in \SS^2,v_*,v\in \R^3}\big(\tilde{ \mathfrak{F}}_{p}\Phi_k^\gamma\big)(|v-v_*|)b(\cos\theta) (\mathfrak{F}_{p}g)_*(\mathfrak{F}_{l}h)\big[
	( \mathfrak{F}_{m}f)'-
	\mathfrak{F}_{m}f\big]d\sigma dv_* dv.
	\eeno

(i) If $l\le p-N_0$, then for $k\ge0$,  
	\beno  |\mathfrak{W}_{k,p,l}^1|&\lesssim&2^{k(\gamma+\f52-N)}(2^{-p(N-2s)}2^{2s(l-p)}+2^{-(N-\f52)p}2^{\f32(l-p)}) |\Phi_0^\gamma|_{H^{N+2}}|\varphi|_{W^{2,\infty}_N} |\mathfrak{F}_{p}g|_{L^1}|\mathfrak{F}_{l}h|_{L^2}|\tilde{ \mathfrak{F}}_{p}f|_{L^2},\\
	 |\mathfrak{W}_{-1,p,l}^1|&\lesssim&
		( 2^{2sl}2^{-p}+2^{\f32l}2^{-(\gamma+3)p})|\mathfrak{F}_{p}g|_{L^{2}} |\mathfrak{F}_{l}h|_{L^2}|\tilde{ \mathfrak{F}}_{p}f|_{L^2}. 
		\eeno 
				(ii). If $|l-p|\le N_0$ and $m<p-2N_0$, then for $k\ge0$, \beno
		 |\mathfrak{W}_{k,p,m}^4| & \lesssim& 2^{2s(m-p)} 2^{(\gamma+\f32-N)k}2^{-p(N-\f52)}|\Phi_0^\gamma|_{H^{N+2}}|\varphi|_{W^{2,\infty}_N} |\mathfrak{F}_{p}g|_{L^1}|
		\tilde{ \mathfrak{F}}_{p}h|_{L^2}|\mathfrak{F}_{m}f|_{L^2},\\
  |\mathfrak{W}_{-1,p,m}^4|&\lesssim& 2^{2sm}2^{-p}|\mathfrak{F}_{p}g|_{L^{2}}|\tilde{ \mathfrak{F}}_{p}h|_{L^2}|\mathfrak{F}_{m}f|_{L^2}. \eeno
		\end{lem}

\end{document}